\documentclass[10pt]{amsart}
\usepackage{latexsym,exscale,enumerate}
\usepackage{mathtools,amsmath,mathscinet,amsfonts,amssymb,amsthm,bbm,euscript}
\usepackage{amscd, stmaryrd}
\usepackage[normalem]{ulem}
\usepackage[all]{xy}
\SelectTips{cm}{}
\usepackage{graphicx}
\usepackage{wrapfig}
\usepackage{xcolor}
\usepackage{bm}

\RequirePackage{color}
\definecolor{references}{rgb}{.7,.1,.6}

\RequirePackage[pdftex,
colorlinks = true,
urlcolor = references, 
citecolor = references, 
linkcolor = references, 
]
{hyperref}

\usepackage{tikz}
\usepackage{tikz-cd}
\usetikzlibrary{math}
\usetikzlibrary{decorations.markings}
\usetikzlibrary{decorations.pathreplacing}
\usetikzlibrary{arrows,shapes,positioning}
\tikzstyle directed=[postaction={decorate,decoration={markings,
    mark=at position #1 with {\arrow{>}}}}]
\tikzstyle rdirected=[postaction={decorate,decoration={markings,
    mark=at position #1 with {\arrow{<}}}}]

\tikzset{anchorbase/.style={baseline={([yshift=-0.5ex]current bounding box.center)}},
    tinynodes/.style={font=\tiny,text height=0.75ex,text depth=0.15ex},
    smallnodes/.style={font=\scriptsize,text height=0.75ex,text depth=0.15ex},
    >={Latex[length=1mm, width=1.5mm]}
  }
\tikzcdset{arrow style=tikz, diagrams={>=stealth}}

  \newcommand{\pu}{to [out=90,in=270]}
  \newcommand{\pr}{to [out=0,in=180]}

\newcommand\mydots{\makebox[1em][c]{$\cdot$\hfil$\cdot$\hfil$\cdot$}}

\allowdisplaybreaks

\colorlet{green}{black!30!green}
\newcommand{\BLUE}[1]{\textcolor[rgb]{0.00,0.00,0.80}{#1}}
\newcommand{\GREEN}[1]{\textcolor[rgb]{0.00,0.70,0.00}{#1}}

\newcommand{\GRAY}[1]{\textcolor[rgb]{0.50,0.50,0.50}{#1}}
\definecolor{CQG}{RGB}{0,153,76}
\definecolor{FS}{RGB}{0,76,153} 
\def\CQGsgn#1{{\color{CQG}#1}}

\newcommand{\arxiv}[1]{\href{https://arxiv.org/abs/#1}{\small  arXiv:#1}}

\newcommand{\googlebooks}[1]{(preview at \href{https://books.google.com/books?id=#1}{google books})}

\newcommand{\numdam}[1]{}

\def\emph#1{{\sl #1\/}}

\def\eg{{\sl e.g.\;}}

\def\cal#1{\mathcal{#1}}

\def\mf#1{\mathfrak{#1}}

\let\hat=\widehat
\let\tilde=\widetilde

\newcommand{\qiso}{\stackrel{\mathrm{qis}}{\cong}}

\newcommand*{\defeq}{=}
\renewcommand{\to}{\rightarrow}

\newcommand{\scs}{\scriptstyle}

\def\B{\mathbb{B}}
\def\C{\mathbb{C}}
\def\D{\mathbb{D}}

\renewcommand{\k}{\mathbb{Q}}
\renewcommand{\L}{\mathbb{L}}
\def\N{\mathbb{N}}

\def\U{\mathbb{U}}
\def\V{\mathbb{V}}
\def\W{\mathbb{W}}
\def\X{\mathbb{X}}
\def\Y{\mathbb{Y}}
\def\Z{\mathbb{Z}}

\renewcommand{\aa}{\bm{a}}
\newcommand{\bb}{\bm{b}}
\newcommand{\cc}{\mathbf{c}}
\newcommand{\rr}{\mathbf{r}}
\newcommand{\kk}{\mathbf{k}}
\newcommand{\ii}{\mathbf{i}}
\newcommand{\KB}{\mathbf{K}}
\newcommand{\LB}{\mathbf{L}}
\newcommand{\UB}{\mathbf{U}}
\newcommand{\DB}{\mathbf{D}}

\newcommand{\AS}{\EuScript A}
\newcommand{\BS}{\EuScript B}
\newcommand{\CS}{\EuScript C}
\newcommand{\DS}{\EuScript D}
\newcommand{\KS}{\EuScript K}
\newcommand{\RS}{\EuScript R}

\newcommand{\US}{\EuScript U}
\newcommand{\USd}{\dot{\EuScript U}}
\newcommand{\VS}{\EuScript V}
\newcommand{\VSred}{\overline{\EuScript V}}

\newcommand{\YS}{\EuScript Y}

\renewcommand{\a}{\alpha}
\renewcommand{\b}{\beta}

\renewcommand{\d}{\delta}
\newcommand{\e}{\varepsilon}

\renewcommand{\phi}{\varphi}
\renewcommand{\theta}{\vartheta}

\newcommand{\thickchi}{\mathcal{X}}

\newcommand{\slnn}[1]{\mf{sl}_{#1}}

\newcommand{\gln}{\mf{gl}_n}
\newcommand{\glN}{\mf{gl}_N}

\newcommand{\glnn}[1]{\mf{gl}_{#1}}

\newcommand{\Bim}{\mathrm{Bim}}
\newcommand{\SBim}{\mathrm{SBim}}
\newcommand{\SSBim}{\mathrm{SSBim}}
\newcommand{\Hilb}{\mathrm{Hilb}}
\newcommand{\FHilb}{\mathrm{FHilb}}
\newcommand{\symg}{\mathfrak{S}}
\newcommand{\Proj}{\mathrm{Proj}}
\newcommand{\Spec}{\mathrm{Spec}}

\newcommand{\E}{{\sf{E}}}

\newcommand{\F}{{\sf{F}}}
\newcommand{\id}{\mathrm{id}}
\newcommand{\Id}{\mathrm{id}}

\newcommand{\oone}{\mathbf{1}}

\newcommand{\Br}{\operatorname{Br}}

\newcommand{\cone}{\operatorname{cone}}

\newcommand{\End}{\operatorname{End}}

\newcommand{\Ext}{\operatorname{Ext}}
\newcommand{\FT}{\operatorname{FT}}
\newcommand{\HH}{\operatorname{HH}}
\newcommand{\HHH}{\operatorname{HHH}}
\newcommand{\Hom}{\operatorname{Hom}}

\newcommand{\im}{{\operatorname{im}}}

\newcommand{\KhR}{\operatorname{KR}}
\newcommand{\largewedge}{\mbox{\Large $\wedge$}}

\newcommand{\rk}{\operatorname{\rk}}
\newcommand{\spann}{\operatorname{span}}

\newcommand{\Sym}{\operatorname{Sym}}
 
\newcommand{\Tor}{\operatorname{Tor}}
\newcommand{\Tr}{\operatorname{Tr}}
\newcommand{\tw}{\operatorname{tw}}

\newcommand{\wt}{\operatorname{wt}}

\newcommand{\op}{\mathrm{op}}

\newcommand{\inv}{^{-1}}

\newcommand{\dg}{\mathrm{dg}}	
\newcommand{\tot}{\mathrm{tot}}

\newcommand{\adeg}{\mathbf{a}}
\newcommand{\qdeg}{\mathbf{q}}
\newcommand{\tdeg}{\mathbf{t}}

\newcommand{\bDelta}{\bar{\Delta}}
\newcommand{\cross}{\mathsf{x}}
\newcommand{\point}{\mathsf{p}}
\newcommand{\comp}{\mathsf{c}}
\newcommand{\perm}{\omega}
\newcommand{\agen}{\beta}
\newcommand{\trans}{\mathsf{t}}
\newcommand{\diffdeg}{\mathsf{d}}

\newcommand{\KR}{\mathrm{KR}}

\newcommand{\YHH}{\YS \!\HH}
\newcommand{\YH}{\YS H}

\newcommand{\M}{\mathbb{M}}
\newcommand{\Fr}{\mathbb{F}}
\newcommand{\leftX}{\mathbb{X}}
\newcommand{\rightX}{\mathbb{X}'}
\newcommand{\rightrightX}{\mathbb{X}''}
\newcommand{\leftM}{\mathbb{M}}
\newcommand{\rightM}{\mathbb{M}'}
\newcommand{\leftL}{\mathbb{L}}
\newcommand{\rightL}{\mathbb{L}'}

\newcommand{\I}{I}

\newcommand{\VFT}{\VS\!\FT}
\newcommand{\VFTmin}{\VS\mathsf{FT}}

\newcommand{\MCS}{\mathrm{MCS}}
\newcommand{\MCSmin}{\mathsf{MCS}}

\newcommand{\KMCS}{\mathrm{KMCS}}
\newcommand{\KMCSmin}{\mathsf{KMCS}}

\newcommand{\YKMCSmin}{\VS\mathsf{KMCS}}

\newcommand{\MCCS}{\mathrm{MCCS}}
\newcommand{\MCCSmin}{\mathsf{MCCS}}

\newcommand{\YMCCSmin}{\VS\mathsf{MCCS}}

\newcommand{\TD}{\operatorname{TD}} 
\newcommand{\VTDmin}{\VS\mathsf{TD}} 
\newcommand{\sqmatrix}[1]{\left[\begin{matrix} #1\end{matrix}\right]}
\newcommand{\PSI}{\mathrm{CM}}

\newcommand{\Dig}{\mathrm{Dig}}

\newcommand{\Key}{\mathrm{Key}}
\newcommand{\hComp}{\star}

\newcommand{\qbinom}[2]{\genfrac[]{0pt}{2}{#1}{#2}}

\newcommand{\Schur}{\mathfrak{s}}

\newcommand{\splitt}{\mathcal{S}}
\newcommand{\spli}{\mathrm{split}}

\newcommand{\uvar}{u}
\newcommand{\YK}{{\VS}\!K}
\newcommand{\ZZ}{\mathcal{Z}}
\newcommand{\cre}{\mathbf{cr}}
\newcommand{\col}{\mathbf{col}}
\newcommand{\zip}{\mathbf{zip}}
\newcommand{\un}{\mathbf{un}}
\newcommand{\sectionval}{D}

\newcommand{\bre}{b} 
\newcommand{\brc}{{\bb}} 
\newcommand{\brcc}{{\cc}} 

\newcommand{\VMCCSmin}{\VS\mathsf{MCCS}}

\newcommand{\hdet}{\operatorname{hdet}}
\newcommand{\FTmin}{\mathsf{FT}}

\newcommand{\yred}{\overline{y}}
\newcommand{\psired}{\overline{\psi}}

\newcommand{\CQGbox}[1]{
\begin{tikzpicture}
\node[draw,  fill=white,rounded corners=4pt,inner sep=3pt] (X) at (0,.75) {$\scs#1$};
\end{tikzpicture}}

\newcommand{\diffslices}{
\tikzmath{
\ylift = 0.2;
\yplus = 0.6;
\xshift = -1;
\fx = -.25;
\fxx = -.2;
\fy = -.5;
\fyy = -.25;
	 } 
\begin{tikzpicture} [scale=.5,fill opacity=0.1,anchorbase]
\path[fill=green, opacity=0.3] (.5+\xshift,1.25+\ylift+\yplus) to [out=135,in=0] (\xshift,1.5+\ylift+\yplus) to [out=180,in=45] 
(-.5+\xshift,1.25+\ylift+\yplus) to [out=225,in=270]  (-2,2.75+\ylift) to [out=90,in=180] (0,4.25+\ylift) to [out=0,in=90] (2,2.75+\ylift) to [out=270,in=315] (.5+\xshift,1.25+\ylift+\yplus) ;
\path[fill=green] (-.5+\xshift,0) to (-.5+\xshift,5) to (-2,6) to (-2,1) to (-.5+\xshift,0);
\path[fill=green] (.5+\xshift,0) to (.5+\xshift,5) to (2,6) to (2,1) to (.5+\xshift,0);
\path[fill=blue] (-3,1) to (-2,1) to [out=20,in=180] (0,1.4) to [out=0,in=160] (2,1) to (3,1) to (3,6) to (2,6) to [out=160,in=0] (0,6.4) to [out=180,in=20] (-2,6) to  (-3,6);
\path[fill=blue]  (3+\xshift+\fx,5+\fy)  to (3+\xshift+\fx,\fy)  to [out=180,in=340] (.5+\xshift,0) to (-.5+\xshift,0) to [out=200,in=0] (-3+\xshift+\fx,\fy) to (-3+\xshift+\fx,5+\fy)to [out=0,in=200] (-.5+\xshift,5) to (.5+\xshift,5) to [out=340,in=180] (3+\xshift+\fx,5+\fy);
	\draw[very thick, red] (.5+\xshift,1.25+\ylift+\yplus) to [out=135,in=0] (0+\xshift,1.5+\ylift+\yplus) to [out=180,in=45] (-.5+\xshift,1.25+\ylift+\yplus);
	\draw[very thick, red, rdirected=.55] (-2,2.75+\ylift) to [out=90,in=180] (0,4.25+\ylift) to [out=0,in=90] (2,2.75+\ylift);
\draw[very thick, red, directed=.80] (-2,6) to (-2,1);
\draw[very thick, red, directed=.15] (2,1) to (2,6);
	\draw[very thick, directed=.44] (3,1) to (2,1) to [out=160,in=0] (0,1.4) to [out=180,in=20] (-2,1) to (-3,1);
\draw[very thick, CQG, directed=.75] (1.25+\xshift/2,.5) to (1.25+\xshift/2,5.5);
\draw[very thick, CQG, rdirected=.75] (-1.25+\xshift/2,.5) to (-1.25+\xshift/2,5.5);
\draw[thick, CQG] (1.25+\xshift/2,1.25+\ylift+\yplus-.2) to [out=90,in=90] (-1.25+\xshift/2,1.25+\ylift+\yplus);
\draw[CQG] (3+\xshift/2,.5) to (-3+\xshift/2,.5) to (-3+\xshift/2,5.5) to (3+\xshift/2,5.5) to (3+\xshift/2,.5); 
\draw[very thick, red, directed=.22] (-.5+\xshift,0) to (-.5+\xshift,5);
\draw[very thick, red, directed=.83] (.5+\xshift,5) to (.5+\xshift,0);
\draw[very thick, red, rdirected=.5] (2,2.75+\ylift) to [out=270,in=315] (.5+\xshift,1.25+\ylift+\yplus);
\draw[very thick, red] (-.5+\xshift,1.25+\ylift+\yplus) to [out=225,in=270] (-2,2.75+\ylift);
	\draw[very thick, directed=.44] (3+\xshift+\fx,\fy) to [out=180,in=340] (.5+\xshift,0) to (-.5+\xshift,0) to [out=200,in=0] (-3+\xshift+\fx,\fy);
	\draw[very thick, directed=.65] (2,1) to  (.5+\xshift,0);
	\draw[very thick, directed=.55] (-.5+\xshift,0) to  (-2,1);
	\draw[very thick, directed=.44] (3,6) to (2,6) to [out=160,in=0] (0,6.4) to [out=180,in=20] (-2,6) to (-3,6);
	\draw[very thick, directed=.44] (3+\xshift+\fx,5+\fy) to [out=180,in=340] (.5+\xshift,5) to (-.5+\xshift,5) to [out=200,in=0] (-3+\xshift+\fx,5+\fy);
	\draw[very thick, directed=.65] (2,6) to  (.5+\xshift,5);
	\draw[very thick, directed=.55] (-.5+\xshift,5) to  (-2,6);
\draw[very thick] (-3,6.04) to (-3,1-.04);
\draw[very thick] (3,6.04) to (3,1-.04);
\draw[very thick] (-3+\xshift+\fx,5+\fy+.04) to (-3+\xshift+\fx,\fy-.04);
\draw[very thick] (3+\xshift+\fx,5+\fy+.04)  to (3+\xshift+\fx,\fy-.04) ;
	\end{tikzpicture}
\qquad , \qquad
\begin{tikzpicture} [scale=.5,fill opacity=0.1,anchorbase]
\path[fill=green, opacity=0.3] (.5+\xshift,1.25+\ylift+\yplus) to [out=135,in=0] (\xshift,1.5+\ylift+\yplus) to [out=180,in=45] 
(-.5+\xshift,1.25+\ylift+\yplus) to [out=225,in=270]  (-2,2.75+\ylift) to [out=90,in=180] (0,4.25+\ylift) to [out=0,in=90] (2,2.75+\ylift) to [out=270,in=315] (.5+\xshift,1.25+\ylift+\yplus) ;
\path[fill=green] (-.5+\xshift,0) to (-.5+\xshift,5) to (-2,6) to (-2,1) to (-.5+\xshift,0);
\path[fill=green] (.5+\xshift,0) to (.5+\xshift,5) to (2,6) to (2,1) to (.5+\xshift,0);
\path[fill=blue] (-3,1) to (-2,1) to [out=20,in=180] (0,1.4) to [out=0,in=160] (2,1) to (3,1) to (3,6) to (2,6) to [out=160,in=0] (0,6.4) to [out=180,in=20] (-2,6) to  (-3,6);
\path[fill=blue]  (3+\xshift+\fx,5+\fy)  to (3+\xshift+\fx,\fy)  to [out=180,in=340] (.5+\xshift,0) to (-.5+\xshift,0) to [out=200,in=0] (-3+\xshift+\fx,\fy) to (-3+\xshift+\fx,5+\fy)to [out=0,in=200] (-.5+\xshift,5) to (.5+\xshift,5) to [out=340,in=180] (3+\xshift+\fx,5+\fy);
\draw[blue] (\xshift+\fxx,\fyy) to (0-\fx,1.75) to (0-\fx,6.75) to (\xshift+\fxx,5+\fyy);
	\draw[very thick, red, rdirected=.55] (-2,2.75+\ylift) to [out=90,in=180] (0,4.25+\ylift) to [out=0,in=90] (2,2.75+\ylift);
\draw[very thick, red, directed=.80] (-2,6) to (-2,1);
\draw[very thick, red, directed=.15] (2,1) to (2,6);
	\draw[very thick, directed=.44] (3,1) to (2,1) to [out=160,in=0] (0,1.4) to [out=180,in=20] (-2,1) to (-3,1);
\draw[very thick, blue, rdirected=.80] (0,6.4) to (0,1.4);
\draw[very thick, blue, rdirected=.80] (\xshift,5) to (\xshift,0);
\draw[thick, blue] (\xshift,1.5+\ylift+\yplus)  \pu (0,4.25+\ylift);
\draw[very thick, red, directed=.22] (-.5+\xshift,0) to (-.5+\xshift,5);
\draw[very thick, red, directed=.83] (.5+\xshift,5) to (.5+\xshift,0);
\draw[very thick, red, rdirected=.5] (2,2.75+\ylift) to [out=270,in=315] (.5+\xshift,1.25+\ylift+\yplus);
\draw[very thick, red] (-.5+\xshift,1.25+\ylift+\yplus) to [out=225,in=270] (-2,2.75+\ylift);
	\draw[very thick, red] (.5+\xshift,1.25+\ylift+\yplus) to [out=135,in=0] (0+\xshift,1.5+\ylift+\yplus) to [out=180,in=45] (-.5+\xshift,1.25+\ylift+\yplus);
	\draw[blue] (\xshift+\fxx,5+\fyy) to (\xshift+\fxx,\fyy); 
	\draw[very thick, directed=.44] (3+\xshift+\fx,\fy) to [out=180,in=340] (.5+\xshift,0) to (-.5+\xshift,0) to [out=200,in=0] (-3+\xshift+\fx,\fy);
	\draw[very thick, directed=.65] (2,1) to  (.5+\xshift,0);
	\draw[very thick, directed=.55] (-.5+\xshift,0) to  (-2,1);
	\draw[very thick, directed=.44] (3,6) to (2,6) to [out=160,in=0] (0,6.4) to [out=180,in=20] (-2,6) to (-3,6);
	\draw[very thick, directed=.44] (3+\xshift+\fx,5+\fy) to [out=180,in=340] (.5+\xshift,5) to (-.5+\xshift,5) to [out=200,in=0] (-3+\xshift+\fx,5+\fy);
	\draw[very thick, directed=.65] (2,6) to  (.5+\xshift,5);
	\draw[very thick, directed=.55] (-.5+\xshift,5) to  (-2,6);
\draw[very thick] (-3,6.04) to (-3,1-.04);
\draw[very thick] (3,6.04) to (3,1-.04);
\draw[very thick] (-3+\xshift+\fx,5+\fy+.04) to (-3+\xshift+\fx,\fy-.04);
\draw[very thick] (3+\xshift+\fx,5+\fy+.04)  to (3+\xshift+\fx,\fy-.04) ;
	\end{tikzpicture}
}

\theoremstyle{plain}
\newtheorem{thm}{Theorem}[section]
\newtheorem{prop}[thm]{Proposition}
\newtheorem{proposition}[thm]{Proposition}
\newtheorem{cor}[thm]{Corollary}
\newtheorem{principle}[thm]{Principle}
\newtheorem{lem}[thm]{Lemma}
\newtheorem{lemma}[thm]{Lemma}
\newtheorem{conj}[thm]{Conjecture}

\theoremstyle{definition}
\newtheorem{defi}[thm]{Definition}
\newtheorem{definition}[thm]{Definition}

\newtheorem{rem}[thm]{Remark}
\newtheorem{remark}[thm]{Remark}

\newtheorem{exa}[thm]{Example}
\newtheorem{example}[thm]{Example}
\newtheorem{conv}[thm]{Convention}
\newtheorem*{exa-nono}{Example}

\newcommand{\HRWsym}{Section 2.1} 
\newcommand{\HRWDefFLA}{Definition 2.3} 
\newcommand{\HRWCQGSSBim}{Proposition 2.18} 
\newcommand{\HRWPropRick}{Proposition 2.25} 
\newcommand{\HRWPropWeb}{Proposition 2.27} 
\newcommand{\HRWDotslide}{Lemma 2.30} 
\newcommand{\shiftedRickard}{Proposition 2.31}
\newcommand{\HRWSkeinrel}{Theorem 3.4}
\newcommand{\MCSMCSmin}{(26)} 
\newcommand{\Zetadiffs}{Proposition 3.10}
\newcommand{\KMCSdiffs}{Proposition 3.12} 
\newcommand{\PropRows}{Proposition 3.20}
\newcommand{\QzeroFT}{Theorem 3.24} 
\newcommand{\PropQs}{Proposition 3.27} 
\newcommand{\IofWk}{Lemma 3.28} 
\newcommand{\MCCSCor}{Corollary 3.29} 

\usepackage{etoolbox}
\tracingpatches
\makeatletter
\patchcmd{\@setref}{\bfseries ??}{\bfseries\color{red} OWE A COFFEE/BEER}{}{}
\makeatother

\renewcommand{\l}{\lambda} 
\renewcommand{\u}{\uvar}
\newcommand{\pv}{\dot{v}}
\newcommand{\pV}{\dot{\V}}

\begin{document}

\author{Matthew Hogancamp}
\address{Department of Mathematics, Northeastern University, 360 Huntington Ave, Boston,
MA 02115, USA}
\email{m.hogancamp@northeastern.edu}

\author{David~E.~V.~Rose}
\address{Department of Mathematics, University of North Carolina, 
Phillips Hall, CB \#3250, UNC-CH, 
Chapel Hill, NC 27599-3250, USA
\href{https://davidev.web.unc.edu/}{davidev.web.unc.edu}}
\email{davidrose@unc.edu}

\author{Paul Wedrich}
\address{P.W.: Max Planck Institute for Mathematics,
Vivatsgasse 7, 53111 Bonn, Germany 
AND Mathematical Institute, University of Bonn,
Endenicher Allee 60, 53115 Bonn, Germany
\href{http://paul.wedrich.at}{paul.wedrich.at}}
\email{p.wedrich@gmail.com}

\title{Link splitting deformation of colored Khovanov--Rozansky homology}

\begin{abstract} 
	We introduce a multi-parameter deformation of the triply-graded Khovanov--Rozansky homology
	of links colored by one-column Young diagrams, generalizing the ``$y$-ified'' link homology 
	of Gorsky--Hogancamp and work of Cautis--Lauda--Sussan.  
	For each link component, the natural set of deformation
	parameters corresponds to interpolation coordinates on the Hilbert scheme of the plane.
	We extend our deformed link homology theory to braids by introducing a monoidal dg
	2-category of curved complexes of type A singular Soergel bimodules.
	Using this framework, we promote to the curved setting the categorical colored skein
	relation from \cite{HRW1} and also the notion of splitting map for the colored full twists on two strands.
	As applications, we compute
	the invariants of colored Hopf links in terms of ideals generated by Haiman
	determinants and use these results to establish general link splitting
	properties for our deformed, colored, triply-graded link homology. Informed
	by this, we formulate several conjectures that have implications for the
	relation between (colored) Khovanov--Rozansky homology and Hilbert schemes.
\end{abstract}

\maketitle

\setcounter{tocdepth}{1}

\tableofcontents

\section{Introduction}

\label{s:introtwo}

The last two decades have seen the introduction of powerful homological 
invariants of knots, links, braids, and tangles, 
which are connected to classical quantum invariants through a
decategorification relationship \cite{Kho, 0409593, KR, MR2421131}.
These invariants are best understood in the context of differential graded categories:
each tangle diagram $\DB$ is assigned a chain complex $\ZZ(\DB)$ over an additive category,
Reidemeister moves between such diagrams are assigned 
specific chain maps that are invertible up to homotopy, 
and movies between diagrams that encode certain braid/tangle cobordisms are assigned 
(generally non-invertible) chain maps that are natural, up to homotopy 
\cite{CMW, Cap, Bla, MR2721032, ETW}. 
In the case that $\DB$ is a knot or link diagram, 
the homology $H_\ZZ(\LB)$ of $\ZZ(\DB)$ is therefore an invariant of the link $\LB$ determined by $\DB$.
In fact, $\ZZ(\DB)$ can typically be equipped with additional structure (see below) and determines 
an invariant $\ZZ(\LB)$ of the corresponding link $\LB$ up to quasi-isomorphism. 
Instances of this higher structure are the subject of this paper.

\subsection{Local operators and monodromy}\label{ss:Mono}

In many cases, a choice of point $\point \in \DB$ equips $\ZZ(\DB)$ 
with an action of a graded-commutative algebra of \emph{local operators} $A_\ZZ$. 
In prototypical examples, $A_\ZZ=\ZZ(\bigcirc)$ can be identified with 
the invariant of the unknot\footnote{More precisely, 
$\ZZ(\bigcirc)$ is typically a free $A_\ZZ$-module of rank 1, so
$A_\ZZ$ and $\ZZ(\bigcirc)$ are isomorphic up to grading shift.}, 
and the action of $A_\ZZ$ at $\point$ is induced by the saddle cobordism 
$\bigcirc\sqcup \DB \to \DB$ that merges a small unknot with $\DB$ 
near the point $\point$. 
We now mention several specific instances of this setup.
For the duration, we work over the rationals $\k$ (see \S \ref{ss:coeff}).

\begin{exa} 
Let $\ZZ(\DB) = C_{\KhR_N}(\DB)$ be the $\glN$ Khovanov--Rozansky complex, whose
homology is the $\glN$ Khovanov--Rozansky homology $H_{\KhR_N}(\LB)$ of the link
$\LB$ determined by $\DB$ \cite{KR}. A choice of $\point \in \DB$ equips
$C_{\KhR_N}(\DB)$ with an action of the graded algebra $A_{\KhR_N}:=H^*(\C
P^{N-1}) \cong \k[x]/(x^N)$. 
The unknot invariant $C_{\KhR_N}(\bigcirc)$ is the free $A_{\KhR_N}$-module
generated by a single element, in degree $1-N$, so we may identify $A_{\KhR_N}$
and $C_{\KhR_N}(\bigcirc)$ up to shift.  
Under this identification, the action of $C_{\KhR_N}(\bigcirc)$ at $\point \in \DB$ 
is induced by the map
\[
	C_{\KhR_N}(\bigcirc) \otimes C_{\KhR_N}(\DB) \xrightarrow{\cong} C_{\KhR_N}(\bigcirc \sqcup \DB) 
\to C_{\KhR_N}(\DB)
\] 
where the first part is due to the monoidality of $C_{\KhR_N}$ and the second
part is induced by the saddle cobordism.
\end{exa}

\begin{remark}
The algebra $A_\ZZ$ is sometimes referred to as the \emph{sheet algebra} of the theory $\ZZ$.
It suggest that elements of $A_\ZZ$ should be visualized as the identity cobordism
of the trivial $(1,1)$-tangle, suitably decorated. 
See e.g. \cite[Corollary 2.4]{MorrisonNieh}, 
where this terminology appears to have originated.
\end{remark}

\begin{exa}\label{exa:ZKR} Let $\DB$ be a closed braid diagram of a link $\LB$,
and let $\ZZ(\DB) = C_{\KR}(\DB)$
be the triply-graded Khovanov--Rozansky complex, whose homology is the
HOMFLYPT homology $\HHH(\LB):=H_{\KhR}(\DB)$ \cite{MR2421131}. 
There are two common choices for the sheet algebra in this case: the
\emph{underived sheet algebra} $\k[x]$, or the \emph{derived sheet algebra}
$\k[x]\otimes \largewedge[\eta]\cong \HH^\bullet(\k[x])$. Here the degrees of
the variables, written multiplicatively following Convention~\ref{conv:wt} below, 
are given by $\wt(x)=\qdeg^2$ and $\wt(\eta)=\adeg \qdeg^{-2}$. 

The action of this sheet algebra is best understood using Khovanov's formulation
of $\HHH(\LB)$ using the Hochschild homology of 
Soergel bimodules \cite{MR2339573}, given that this
homology theory is not functorial with respect to general link cobordisms.
We will refer to the $\adeg$-degree of the variable $\eta$ in the following 
as the \emph{Hochschild-degree}.
\end{exa}

\begin{exa}\label{ex:colored sheet alg} In this paper, we are primarily
interested in colored link homologies and, more specifically, the
$\largewedge$-colored extension of triply-graded Khovanov--Rozansky homology
\cite{MR3687104}.  
This homology theory defines invariants of framed oriented links $\LB$ in which
each component $\comp \in \pi_0(\LB)$ is labeled by a non-negative integer
$b(\comp)$, the \emph{color}, each of which defines a sheet algebra
$A_{\ZZ,b(\comp)}$.   
For $b \in \N$, fix an alphabet $\X^b=\{x_1,\ldots,x_b\}$, then the
$b$-colored sheet algebra is given by the ring of symmetric polynomials
$\Sym(\X^b):=\k[x_1,\ldots,x_b]^{\symg_b}$ and the derived sheet algebra by
$\HH^\bullet(\Sym(\X^b))$.
\end{exa}

In addition to the action of sheet algebras, $\ZZ(\DB)$ is typically also
equipped with higher structures stemming from the fact that choices of different
points $\point_1, \point_2 \in \DB$ that lie in the same
component\footnote{Here, and in the following, we slightly abuse terminology and
identify components $\comp \in \pi_0(\LB)$ of the link determined by $\DB$ with
the corresponding equivalence class of points in the diagram $\DB$. Thus, we can
talk about components of $\DB$.} $\comp \in \pi_0(\LB)$ should induce homotopic
actions of the relevant sheet algebra $A_{\ZZ}$. To be precise, let $\gamma \subset \DB$
be an oriented path from $\point_1$ to $\point_2$.  
For each $a\in A_{\ZZ}$, let $a(\point_i)$ denote the action of $a\in A_{\ZZ}$
at $\point_i\in \DB$. Then, the path $\gamma$ determines a homotopy
$\Psi_{\gamma}(a)$, with
\[
[\d_{\DB}, \Psi_{\gamma}(a)] = a(\point_2)-a(\point_1).
\]
Here $\d_\DB$ is the differential on the complex $\ZZ(\DB)$, 
so the super-commutator $[\d_{\DB},-]$ is the differential on the dg algebra $\End(\ZZ(\DB))$. 
Note, however, that $a(\point_1)$ and $a(\point_2)$ are homotopic in two \emph{different} ways.
We can choose two complementary paths 
$\gamma,\gamma' \colon \point_1 \to \point_2$ in $\DB$ so that traversing $\gamma$ followed by 
the reverse of $\gamma'$ yields a loop.
It follows that the difference
\[
\Psi_a := \Psi_\gamma(a) - \Psi_{\gamma'}(a)
\]
is a closed endomorphism of $\ZZ(\DB)$ of degree $\wt(\Psi_a) = \wt(a)\tdeg^{-1}$
, called the \emph{monodromy} of $a$ along $\DB$. These monodromy endomorphisms,
together with certain higher operations, (should) assemble to give an action of
the Hochschild homology 
$\HH_\bullet(A_\ZZ)$ (itself an algebra, since $A_\ZZ$ is graded-commutative) on
$\ZZ(\DB)$. This should not be confused with the passage to a derived sheet
algebra as in Example~\ref{exa:ZKR}. Instead, the guiding principle is the
following.

\begin{principle}\label{principle}
Each component $\comp \in \pi_0(\LB)$ determines an action
of $\HH_\bullet(A_{\ZZ})$ on $\ZZ(\LB)$, well-defined up to
quasi-isomorphism, in such a way that, for $a \in A_{\ZZ}$, the
K\"ahler differential $d(a) \in \HH_1(A_{\ZZ})$ acts as the monodromy
$\Psi_a$ along $\comp$. The actions along various components assemble to give an
action of $\bigotimes_{\comp \in \pi_0(\LB)}\HH_\bullet(A_{\ZZ})$.
\end{principle}

\begin{remark}
In the above statement, 
we have invoked the well-known fact that for a commutative $\k$-algebra $A$, 
$\HH_1(A)$ is isomorphic to the $A$-module of K\"ahler differentials on $A$, 
i.e.~the free $A$-module generated by symbols $d(a)$ with $a\in A$, 
modulo the relations $d(ab) = ad(b) + bd(a)$ and $d(a+b) = d(a)+d(b)$ for $a,b\in A$ 
and $d(s)=0$ for $s \in \k$.
\end{remark}

\begin{remark}
In the setting of colored link homologies, the statement of Principle
\ref{principle} should be modified accordingly to account for the fact that the
sheet algebras depend on a choice of color.
\end{remark}

We thus refer to $\HH_\bullet(A_\ZZ)$ as the \emph{monodromy algebra} of $A_\ZZ$.
For the duration, 
we restrict to the setting of (colored) triply-graded Khovanov--Rozansky homology for concreteness, 
since this will be the invariant we study in this work.

\begin{example}
Continuing Example~\ref{exa:ZKR}, for the (uncolored) triply-graded
Khovanov--Rozansky homology with sheet algebra $\k[x]$, the associated monodromy
algebra is $\k[x]\otimes \largewedge[\xi]$, where $\wt(x)=\qdeg^2$ and
$\wt(\xi)=\qdeg^2 \tdeg\inv$. Thus, one would expect an action of
$\k[x_1,\ldots,x_r]\otimes \largewedge[\xi_1,\ldots,\xi_r]$ on $C_{\KhR}(\DB)$
when $\DB$ has $r$ components. Such an action was constructed in \cite{GH} in
the form of ``$y$-ified'' Khovanov--Rozansky homology.
\end{example}

\begin{example}\label{ex:colored monodromy alg} For colored Khovanov--Rozansky
homology, the $b$-colored sheet algebra $\k[x_1,\ldots,x_b]^{\symg_b}$ can be
identified with either $\k[p_1,\ldots,p_b]$ or $\k[e_1,\ldots,e_b]$, where $p_i$
and $e_i$ are the power sum and elementary symmetric functions in the alphabet
$\X^b=\{x_1,\ldots,x_b\}$. Thus, the $b$-colored monodromy algebra can be
identified with either $\k[p_1,\ldots,p_b] \otimes
\largewedge[\Xi_1,\ldots,\Xi_b]$ or $\k[e_1,\ldots,e_b] \otimes
\largewedge[\Psi_1,\ldots,\Psi_b]$, where $\wt(p_i)=\wt(e_i)=\qdeg^{2i}$ and
$\wt(\Xi_i)=\wt(\Psi_i)=\qdeg^{2i} \tdeg\inv$.
\end{example}

\subsection{From monodromy to deformation}\label{ss:MonoToDef}

Now, we show how the action of the monodromy algebra $\HH_\bullet(A_\ZZ)$ on
$\ZZ(\DB)$ allows us to construct \emph{link splitting deformations}\footnote{In
this paper, the word \emph{deformation} will always refer to such link splitting
deformations, which are based on monodromy data. These are of an entirely
different nature than the deformations of finite-rank type A link homologies
based on deformations of underlying Frobenius algebras, that were studied by the
second- and third-named authors in \cite{RW}.} of $\ZZ(\DB)$. Again, we work
explicitly with triply-graded Khovanov--Rozansky homology, first in the uncolored
case (considered in \cite{GH}) and then its extension to the colored case which
is achieved in this paper.

Suppose that $\DB$ is a braid closure diagram for an $r$-component oriented link
$\LB$ and let $C_{\KR}(\DB)$ be the Khovanov--Rozansky complex associated to
$\DB$, which carries an action of the monodromy algebra
$\k[x_1,\ldots,x_r]\otimes \largewedge[\xi_1,\ldots,\xi_r]$ on $C_{\KR}(\DB)$.
The deformed (or ``$y$-ified'') complex $\YS C_{\KR}(\DB)$ is constructed from
this action using Koszul duality \cite{BGS} (see the earlier \cite{BGG} for the
specific case of polynomial/exterior algebras). Explicitly, introduce formal
parameters $y_1,\ldots,y_r$ with $\wt(y_i)=\qdeg^{-2} \tdeg^2$ and form the
complex
\begin{equation}\label{eq:GHy}
\YS C_{\KR}(\DB) := C_\KR(\DB)\otimes \k[y_1,\ldots,y_r] 
\, , \quad \d_{\KR} + \sum_{i=1}^r y_i \xi_i.
\end{equation}
In \cite{GH}, it is shown that the homology of $\YS C_{\KR}(\DB)$ 
is a well-defined invariant of the oriented link $\LB$, up to isomorphism.  

The main goal of the present paper is to investigate the colored version of this
invariant. For this, suppose that $\LB$ is a (framed, oriented) \emph{colored}
link, i.e. each component $\comp \in \pi_0(\LB)$ is assigned a color $b(\comp)
\geq 0$. Recall from Example \ref{ex:colored monodromy alg} that the monodromy
algebra associated to a $b$-labeled component is $\k[p_1,\ldots,p_b] \otimes
\largewedge[\Xi_1,\ldots, \Xi_b] \cong \Sym(\X^b) \otimes
\largewedge[\Xi_1,\ldots, \Xi_b]$. The following is a consequence of Lemma
\ref{lemma:alternate YHH}; see Remark \ref{rem:Monodromies}.

\begin{prop}\label{prop:intro}
Let $\DB$ be a diagram for a framed, oriented, colored link $\LB$, presented as a braid closure.
The colored Khovanov--Rozansky complex $C_{\KR}(\DB)$ admits an action of the algebra
\[
\bigotimes_{\comp \in \pi_0(\LB)} \Big(\k[p_1,\ldots,p_{b(\comp)}] \otimes \largewedge[\Xi_1,\ldots,\Xi_{b(\comp)}]\Big) 
\cong (\k[p_{\comp,i}] \otimes \largewedge[\Xi_{\comp,i}])_{\comp \in \pi_0(\LB), 1\leq i\leq b(\comp)}
\]
in which $\Xi_{\comp,k}$ is the monodromy of $\frac{1}{k}p_k$ along $\comp$.
\end{prop}

Using this monodromy action (and Koszul duality), we can build the following deformed complex:
\begin{equation}\label{eq:firstYC}
\YS C_{\KR}(\DB) := C_{\KR}(\DB) \otimes \k[v_{\comp,k}]_{\comp \in \pi_0(\LB), 1\leq k\leq b(\comp)} 
\, , \quad \d_{\KR} + \sum_{\comp,k} v_{\comp,k} \Xi_{\comp,k}.
\end{equation}
In Theorem \ref{thm:YHHH}, we establish the following.
\begin{thm}\label{thm:intro}
The complex $\YS C_{\KR}(\DB)$ is a well-defined invariant of 
the framed, oriented, colored link $\LB$, up to quasi-isomorphism of modules 
over $\k[p_{\comp,i},v_{\comp,j}]_{\comp \in \pi_0(\LB), 1\leq i,j \leq b(\comp)}$.
Consequently, $\YS H_{\KR}(\LB)$ is an invariant of $\LB$
up to isomorphism of $\k[p_{\comp,i},v_{\comp,j}]$-modules.
\end{thm}

\begin{remark}
In this paper, we actually take a different approach to the definition of the deformed complex from \eqref{eq:firstYC}.
Indeed, rather than first constructing the monodromy morphisms for a link diagram $\DB$ and then using them to build $\YS C_{\KR}(\DB)$, 
we instead build the complex $\YS C_{\KR}(\DB)$ from local pieces that encode the homotopies associated with paths 
in $\DB$ that traverse a single crossing. These take the form of certain \emph{curved complexes}, that we discuss next.
Nonetheless, the two approaches are equivalent; see Remark \ref{rem:Monodromies}.
\end{remark}

\subsection{Curving Rickard complexes}\label{ss:curvRick}

The undeformed colored Khovanov--Rozansky complex can be studied at the level of braids 
(without closing up to obtain a link) using the framework of singular Soergel bimodules.  
It will be quintessential to extend our deformed theory to the level of braids as well.  
Similar to the considerations in \cite{GH}, 
the notion of curved complexes of singular Soergel bimodules appears naturally.

Recall the monoidal 2-category\footnote{Throughout, by $2$-category we always mean 
a weak $2$-category, also known as a bicategory.} 
of \emph{singular Soergel bimodules}, denoted $\SSBim$. 
Objects of this category are sequences $\aa=(a_1,\ldots,a_m)$ of positive integers, 
and a 1-morphism from $\bb=(b_1,\ldots,b_{m'})$ to $\aa=(a_1,\ldots, a_m)$ 
is a certain kind of graded bimodule over $(R^{\aa},R^{\bb})$. 
Here, $R^{\aa}$ denotes the ring of polynomials in $\sum_{i=1}^m a_i$ variables
that are invariant with respect to the action of $\symg_{a_1}\times\cdots\times \symg_{a_m}$.  
Specifically, 1-morphisms $\SSBim$ are generated
(with respect to direct sum and summands, grading shift, horizontal composition $\star$, and external tensor product $\boxtimes$) 
by induction and restriction bimodules relating the rings $R^{a,b}$ and $R^{a+b}$. 
The $2$-morphisms are maps of graded bimodules.
See \S \ref{ss:ssbim} for full details.

The $2$-category $\SSBim$ contains the \emph{singular Bott-Samelson bimodules}, 
those bimodules that are constructed from induction and restriction bimodules 
using only grading shift, $\star$, and $\boxtimes$.  
These bimodules can be depicted diagrammatically as certain trivalent graphs called \emph{webs}, 
e.g.:
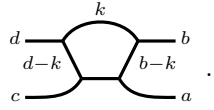
\[
\begin{tikzpicture}[smallnodes,rotate=90,baseline=.1em,scale=1]
\draw[very thick] (0,.25) to [out=150,in=270] (-.25,1) node[left,xshift=2pt]{$c$};
\draw[very thick] (.5,.5) to (.5,1) node[left,xshift=2pt]{$d$};
\draw[very thick] (0,.25) to node[left,yshift=-2pt]{$d{-}k$}  (.5,.5);
\draw[very thick] (0,-.25) to (0,.25);
\draw[very thick] (.5,-.5) to [out=30,in=330] node[above=-2pt]{$k$} (.5,.5);
\draw[very thick] (0,-.25) to node[right,yshift=-2pt] {$b{-}k$}  (.5,-.5);
\draw[very thick] (.5,-1) node[right,xshift=-2pt]{$b$} to (.5,-.5);
\draw[very thick] (-.25,-1)node[right,xshift=-2pt]{$a$} to [out=90,in=210] (0,-.25);
\end{tikzpicture} \, .
\]
(Here, the unlabeled edge has label $a+b-k = c+d-k$). 
In diagrams such as these, a trivalent ``merge'' vertex (when read right-to-left) 
corresponds to a restriction bimodule ${}_{R^{a+b}}(R^{a,b})_{R^{a,b}}$
and a ``split'' vertex corresponds to an induction bimodule ${}_{R^{a,b}}(R^{a,b})_{R^{a+b}}$.

To each colored braid, one can associate a complex of singular Soergel bimodules using the notion of Rickard complexes. 
For instance, to the $(a,b)$-colored elementary crossing with $a\geq b$, one associates a complex of the form
\begin{equation}\label{eq:firstC}
C_{a,b} := 
\left\llbracket
	\begin{tikzpicture}[rotate=90,scale=.5,smallnodes,anchorbase]
		\draw[very thick,->] (1,-1) node[right,xshift=-2pt]{$b$} to [out=90,in=270] (0,1);
		\draw[line width=5pt,color=white] (0,-1) to [out=90,in=270] (1,1);
		\draw[very thick,->] (0,-1) node[right,xshift=-2pt]{$a$} to [out=90,in=270] (1,1);
	\end{tikzpicture}
	\right\rrbracket
:= 
\left( \begin{tikzpicture}[smallnodes,rotate=90,baseline=.1em,scale=.66]
\draw[very thick] (0,.25) to [out=150,in=270] (-.25,1) node[left,xshift=2pt]{$b$};
	\draw[very thick] (.5,.5) to (.5,1) node[left,xshift=2pt]{$a$};
\draw[very thick] (0,.25) to  (.5,.5);
\draw[very thick] (0,-.25) to (0,.25);
\draw[dotted] (.5,-.5) to [out=30,in=330] node[above=-2pt]{$0$} (.5,.5);
\draw[very thick] (0,-.25) to  (.5,-.5);
\draw[very thick] (.5,-1) node[right,xshift=-2pt]{$b$} to (.5,-.5);
\draw[very thick] (-.25,-1)node[right,xshift=-2pt]{$a$} to [out=90,in=210] (0,-.25);
\end{tikzpicture}
\xrightarrow{\d}
\qdeg^{-1} \tdeg
\begin{tikzpicture}[smallnodes,rotate=90,baseline=.1em,scale=.66]
\draw[very thick] (0,.25) to [out=150,in=270] (-.25,1) node[left,xshift=2pt]{$b$};
\draw[very thick] (.5,.5) to (.5,1) node[left,xshift=2pt]{$a$};
\draw[very thick] (0,.25) to  (.5,.5);
\draw[very thick] (0,-.25) to (0,.25);
\draw[very thick] (.5,-.5) to [out=30,in=330] node[above=-2pt]{$1$} (.5,.5);
\draw[very thick] (0,-.25) to  (.5,-.5);
\draw[very thick] (.5,-1) node[right,xshift=-2pt]{$b$} to (.5,-.5);
\draw[very thick] (-.25,-1)node[right,xshift=-2pt]{$a$} to [out=90,in=210] (0,-.25);
\end{tikzpicture}
\xrightarrow{\d}
\cdots \xrightarrow{\d}
\qdeg^{-b} \tdeg^b
\begin{tikzpicture}[smallnodes,rotate=90,baseline=.1em,scale=.66]
\draw[very thick] (0,.25) to [out=150,in=270] (-.25,1) node[left,xshift=2pt]{$b$};
\draw[very thick] (.5,.5) to (.5,1) node[left,xshift=2pt]{$a$};
\draw[very thick] (0,.25) to  (.5,.5);
\draw[very thick] (0,-.25) to (0,.25);
\draw[very thick] (.5,-.5) to [out=30,in=330] node[above=-2pt]{$b$} (.5,.5);
\draw[dotted] (0,-.25) to  (.5,-.5);
\draw[very thick] (.5,-1) node[right,xshift=-2pt]{$b$} to (.5,-.5);
\draw[very thick] (-.25,-1)node[right,xshift=-2pt]{$a$} to [out=90,in=210] (0,-.25);
\end{tikzpicture} \right).
\end{equation}
To each $b$-labeled edge appearing in such diagrams 
(either a colored braid diagram, a web depicting a singular Bott-Samelson, or a composition thereof) 
there is an action of the sheet algebra $\Sym(\X^b)$ from Example \ref{ex:colored sheet alg}
on the associated (complex of) 1-morphism(s).
We will denote the sheet algebra simply by $\Sym(\X_\point)$, 
if we wish to emphasize the point where the sheet algebra is acting (thus $|\X_\point| = b(\comp)$).

It is well-known that the actions on the four endpoints of the Rickard complex 
are homotopic along the strands, see e.g. \cite[Proposition 5.7]{RW}. 
For example, choosing points $\point_1, \point_2, \point'_1,\point'_2$ on $C_{a,b}$ as follows:
\begin{equation}\label{eq:points}
\left\llbracket
\begin{tikzpicture}[rotate=90,scale=.75,smallnodes,anchorbase]
	\draw[very thick] (1,-1) node{$\bullet$} node[right]{$\point'_2$} 
		to [out=90,in=270] (0,1) node{$\bullet$} node[left]{$\point_1$};
	\draw[line width=5pt,color=white] (0,-1) to [out=90,in=270] (1,1);
	\draw[very thick] (0,-1) node{$\bullet$} node[right]{$\point'_1$} 
		to [out=90,in=270] (1,1) node{$\bullet$} node[left]{$\point_2$};
\end{tikzpicture}
\right\rrbracket
\end{equation}
we find that the actions of $f(\X_{\point_2})$ and $f(\X_{\point'_1})$ 
and the actions of $g(\X_{\point_1})$ and $g(\X_{\point'_2})$ 
on $C_{a,b}$ are homotopic, for all $f \in \Sym(\X^a)$ and $g \in \Sym(\X^b)$. 
Hence, 
there exists homotopies $\Psi^{\mathsf{o}}(f)$ and $\Psi^{\mathsf{u}}(g)$ so that 
\[
[\d,\Psi^{\mathsf{o}}(f)] = f(\X_{\point_2}) - f(\X_{\point'_1}) \, , \quad
[\d,\Psi^{\mathsf{u}}(g)] = g(\X_{\point_1}) - g(\X_{\point'_2}) \, .
\]
Following the recipe from \S \ref{ss:MonoToDef}, 
we should incorporate the homotopies for a collection of generators of 
$\Sym(\X^a)$ into our differential, as in \eqref{eq:firstYC}.
One choice of generating set is the collection of \emph{elementary} symmetric functions 
$e_1,\ldots,e_a \in \Sym(\X^a)$. 
Considering the corresponding homotopies 
$\{\Psi^{\mathsf{o}}_{k}\}_{k=1}^a$ and $\{\Psi^{\mathsf{u}}_{k}\}_{k=1}^b$ 
(built in Lemma \ref{lemma:crossing Psi} below)
and extending scalars, we thus can consider $C_{a,b}$, 
equipped with the ``differential''
\begin{equation}\label{eq:totdiff}
\d^{\tot} := \d + \sum_{k=1}^a \Psi^{\mathsf{o}}_{k} u^{\mathsf{o}}_{k} 
	+ \sum_{k=1}^b \Psi^{\mathsf{u}}_{k} u^{\mathsf{u}}_{k}  \in \End(C_{a,b}) \otimes \k[\U]
\end{equation}
where $\U = \{u^{\mathsf{o}}_{k}\}_{1\leq k \leq a} \cup \{u^{\mathsf{u}}_{k}\}_{1\leq k \leq b}$. 
As we show, the homotopies $\Psi^{\mathsf{o}/\mathsf{u}}_{i}$ each square to zero and pairwise anti-commute,
so we find that
\begin{equation}\label{eq:Ecurvature}
(\d^{\tot})^2 
= \sum_{k=1}^a \big(e_k(\X_{\point_2}) - e_k(\X_{\point'_1})\big) u^{\mathsf{o}}_{k} 
	+ \sum_{k=1}^b \big(e_k(\X_{\point_1}) - e_k(\X_{\point'_2})\big) u^{\mathsf{u}}_{k} \, .
\end{equation}
Hence, $\YS C_{a,b} := (C_{a,b}, \d^{\tot})$ is not a chain complex in the traditional sense, 
but rather a \emph{curved complex}. 
Recall that the latter is a generalization of the notion of chain complex, 
and consists of a pair $(X,\d^{\tot})$ where the \emph{curved differential} $\d^{\tot}$
squares to the action of a central element $F \in \End(X) \otimes S$ for some ring $S$. 
We will refer to the curvature in \eqref{eq:Ecurvature}, and its analogue for more general braids, 
as $\Delta e$-curvature.

More generally, using the operations $\hComp$ and $\boxtimes$ in $\SSBim$,
we can associate a Rickard complex $C(\b_\brc)$ to any colored braid $\b_\brc$ 
that is built from complexes $C_{a,b}$ and $C^\vee_{a,b}$ assigned to colored positive and 
negative crossings, as in \eqref{eq:firstC}.
In \S \ref{s:curved rickard etc}, 
we show that these complexes can be deformed in a similar manner to $C_{a,b}$.

\begin{thm}\label{thm:intro1} 
The Rickard complex $C(\b_\brc)$ associated to a colored braid $\b_\brc$ admits 
a deformation to a curved complex $\YS C(\b_\brc)$ with $\Delta e$-curvature.
Such a deformation is unique, up to homotopy equivalence.
\end{thm}

\noindent(See Theorem \ref{thm:braidinvariant} and Lemma \ref{lem:curvinginvertibles} for 
the precise statements.) \smallskip

In fact, the construction of the \emph{curved Rickard complex} $\YS C(\b_\brc)$ 
closely parallels the construction of $C(\b_\brc)$ from the elementary 
pieces $C_{a,b}$ and $C^\vee_{a,b}$. 
Indeed, with the curved complexes $\YS C_{a,b}$ and $\YS C^\vee_{a,b}$ 
associated to positive and negative crossings in hand, 
one need only construct well-defined composition operations to build the 
curved complexes associated to arbitrary braids.
Hence, we establish the following.

\begin{thm}\label{thm:intro2} 
There exists a monoidal dg 2-category $\YS(\SSBim)$ wherein $1$-morphisms are
curved complexes of singular Soergel bimodules with $\Delta e$-curvature.
Appropriate horizontal compositions $\hComp$ and external tensor products $\boxtimes$ of the 
curved complexes $\YS C_{a,b}$ and $\YS C^\vee_{a,b}$ associated with positive and negative crossings
assign a $1$-morphism in $\YS(\SSBim)$ 
to any colored braid (word), which satisfies the braid relations up to canonical homotopy equivalence.
\end{thm}

In our description of $\YS C_{a,b}$ above (and subsequent statements about $\YS C(\b_\brc)$), 
we chose the elementary symmetric functions as the generators of the sheet algebra $\Sym(\X^b)$ 
of a $b$-colored strand. Thus our homotopies $\Psi_k$ encode $\Delta e$-curvature. 
At times, we will find it beneficial to work with curved complexes built from other homotopies, 
which similarly identify sheet algebra actions at the ends of braid strands.
Specifically, let $\point$ and $\point'$ be points at the left and right ends of a $b$-colored braid strand, and
let $N(\X_\point,\X_{\point'}) \lhd \Sym(\X_\point | \X_{\point'}) \cong \Sym(\X_\point) \otimes \Sym(\X_{\point'})$ 
denote the \emph{diagonal ideal}, which is generated by all elements of the form $f(\X_\point) - f(\X_{\point'})$ 
with $f \in \Sym(\X^b)$. In addition to the set of generators
\[
\cal{N}_e := \{e_k(\X_\point) - e_k(\X_{\point'}) \mid 1 \leq k \leq b\}
\]
for $N(\X_\point,\X_{\point'})$, we can also work with the generating sets
\[
\cal{N}_h := \{h_k(\X_\point - \X_{\point'}) \mid 1 \leq k \leq b\} 
\, , \quad
\cal{N}_p := \{p_k(\X_\point) - p_k(\X_{\point'}) \mid 1 \leq k \leq b\} \, .
\]
(See \S \ref{sec:symfns} for details on symmetric functions.) 
By change of variables, it is possible to pass from curved complexes of 
singular Soergel bimodules with $\Delta e$-curvature 
(i.e. curvature modeled on $\cal{N}_e$) to curved complexes with 
$h \Delta$-curvature and $\Delta p$-curvature, modeled on $\cal{N}_h$ 
and $\cal{N}_p$, respectively. See \S \ref{ss:v-curved cxs} and \S \ref{ss:power sum}.

The choice of generators for $N(\X_\point,\X_{\point'})$ is conceptually immaterial, 
but each of the above leads to a notion of curved complex of singular Soergel bimodules 
that is useful in particular instances. 
For example, complexes with $h \Delta$-curvature appear most often ``in the wild,''
and a straightforward change of variables leads to complexes with $\Delta e$-curvature 
that are well-behaved $2$-categorically. 
Passing to $\Delta p$-curvature requires working over a field of characteristic zero, 
but such curvature is best adapted to establishing Markov invariance.
In particular, Proposition \ref{prop:intro} and Theorem \ref{thm:intro} proceed by 
passing from curved Rickard complexes in $\YS(\SSBim)$ to curved complexes 
with (appropriate) $\Delta p$-curvature, before taking braid closure to obtain 
(uncurved) complexes associated to the corresponding link diagram.

\begin{remark}[$y$-variables vs. $u$-variables vs. $v$-variables]
\label{rem:yuv}
In defining the curved differential $\d^{\tot}$ in \eqref{eq:totdiff}, we used
the variable name $u$ to distinguish from the $y$ variables appearing in the
uncolored case \eqref{eq:GHy}. Similarly, we will use the variable names $v$ and
$\pv$ in the setting of $h \Delta$- and $\Delta p$-curvature, respectively. In
each of these settings, these deformation parameters play the role of (but are
\textbf{not} literally equal to) a generating set of symmetric functions in the
uncolored $y$-variables. The relation between the uncolored $y$ deformation
parameters and the $v$ (and $u$) parameters is outlined below and discussed in
detail in \S \ref{ss:ytouv} and \S \ref{ss:interpolation}.
It is best understood in the context of interpolation theory and 
is related to the geometry of the Hilbert scheme $\Hilb_n(\C^2)$.
\end{remark}

\subsection{Relation to previous work}
A notion of curved Rickard complexes, and its application to link homology, 
has been studied by Cautis--Lauda--Sussan (CLS) in \cite{MR4178751}. 
Their construction starts at the level of a categorified quantum group 
$\cal{U}_Q(\slnn{m})$, where $m$ corresponds to the number of braid strands. 
As such, their construction does not see one alphabet $\X_i$ 
per (left) braid boundary point $1\leq i \leq m$, but only the formal difference alphabets $\X_i - \X_{i+1}$. 
Related to this, CLS consider only \emph{one} deformation parameter 
$u$ of weight $\qdeg^{-2}\tdeg^2$. 
On the other hand, our construction starts at the level of
complexes of singular Soergel bimodules 
and considers a \emph{family} of deformation parameters for each strand, 
whose size is given by the strand label. 
These parameters account for higher degree homotopies of Rickard complexes, 
as proposed in \cite[Section 1.3]{MR4178751}. 
We expect that our constructions can be lifted from singular Soergel bimodules 
to a suitable version of the categorified quantum group $\cal{U}_Q(\glnn{m})$.

For a basic comparison of the constructions in \cite{MR4178751} to our construction,
consider the curved Rickard complex associated to the positive crossing \eqref{eq:firstC}. 
Retaining the notation from \eqref{eq:points}, the curvature considered 
by CLS takes the form
\begin{equation}\label{eq:CLScurv}
\big(e_1(\X_{\point_2}) - e_1(\X_{\point'_1})\big) z^{\mathsf{o}} u  - 
\big(e_1(\X_{\point_1}) - e_1(\X_{\point'_2})\big) z^{\mathsf{u}} u
\end{equation}
This expression has only a single deformation parameter $u$, 
that is weighted by scalars $z^{\mathsf{o}}$ and $z^{\mathsf{u}}$, 
and curved complexes with this curvature encode homotopic actions of 
$e_1(\X_{\point_2}) - e_1(\X_{\point_1})$ and $e_1(\X_{\point'_1}) - e_1(\X_{\point'_2})$.
By contrast, our curvature \eqref{eq:Ecurvature} has a family of deformation 
parameters $\{u_k^\mathsf{o}\}_{k=1}^a$ associated to the over strand 
and a family $\{u_k^\mathsf{u}\}_{k=1}^b$ associated to the under strand. 
Curved complexes with this curvature encode the homotopic actions of 
\emph{all} $e_k(\X)$ along \emph{each} strand.
Note that by specializing 
$u_1^\mathsf{o}= z^\mathsf{o} u$, $u_1^\mathsf{u}= -z^\mathsf{u} u$ 
and $u_k^\mathsf{o}=0=u_k^\mathsf{u}$ for all $k>0$ in \eqref{eq:Ecurvature}, 
we recover \eqref{eq:CLScurv}. 
Hence, our link homologies encodes the CLS invariant as a special case, 
see \S \ref{ss:homology with coefficients} and \ref{ss:crossingchange}.

Our multi-parameter curved Rickard complexes give rise to a variety of
deformations of colored, triply-graded Khovanov--Rozansky link homology that
appear closely related to the geometry of the Hilbert scheme $\Hilb_m(\C^2)$,
and its isospectral analogue $X_m$. Indeed, recall that the aforementioned work
of Gorsky--Hogancamp \cite{GH} uses the $y$-ified (uncolored) triply-graded
homology to establish a precise relation between $\YS H_{\KR}(\LB)$ and the
\emph{isospectral Hilbert scheme} $X_m$. There, the variables
$\{x_1,\ldots,x_m,y_1,\ldots,y_m\}$ occurring in the curved Rickard complex
assigned to an $m$-strand braid correspond to local coordinates on an open
subset of $X_m$. In particular, in lowest Hochschild degree the $y$-ified
Khovanov--Rozansky homology of an uncolored unknot is
\[
\YS H_{\KR}^{\text{low}}(\bigcirc) \cong \k[x,y]
\]
which (after extending scalars) is the coordinate ring of $\C^2 = \Hilb_1(\C^2)$. 
In the colored case, we have
\begin{equation}\label{eq:bcoloredU}
\YS H_{\KR}^{\text{low}}(\bigcirc_b) \cong \Sym(\X^b)[v_1,\ldots,v_b]
=\k[e_1(\X^b),\ldots,e_b(\X^b),v_1,\ldots,v_b]
\end{equation}
after changing from the $u$ to $v$ variables, as in Remark \ref{rem:yuv}. As
explained in \S \ref{ss:interpolation}, these $v$-variables are
\emph{interpolation coordinates} for $y$-variables in terms of $x$-variables:
\begin{equation}\label{eq:firstytov}
y_i = \sum_{r=1}^b x_i^{r-1} v_r \, .
\end{equation}
Comparing to \cite[Proposition 3.6.3]{Haiman}, 
we see that the generators 
$\{e_1(\X^b),\ldots,e_b(\X^b),v_1,\ldots,v_b\}$
of $\YS H_{\KR}(\bigcirc_b)$
can be understood as coordinates 
on an open affine subset of $\Hilb_b(\C^2)$.
We will comment further on relations between colored $\YS H_{\KR}(\LB)$ 
and Hilbert schemes
in the following section.

\subsection{Curved colored skein relation, link splitting, and the Hopf link}
\label{sec:introlinksplit}
Important applications of deformed link homologies derive from their controlled
behavior under unlinking, i.e. their \emph{link splitting} properties, see e.g.
Batson--Seed~\cite{BS}. Closer to the present paper, the connection
between (uncolored) $\YS H_{\KR}(\LB)$ and Hilbert schemes mentioned above
relies on the \emph{link splitting map} from the full twist braid to the
identity braid. This map identifies the lowest Hochschild-degree summand of the
homology of the $(m,m)$-torus link with an ideal $I_m \lhd
\k[x_1,\ldots,x_m,y_1,\ldots,y_m] =: \k[\X,\Y]$, and $X_m$ is precisely the
blowup of $(\C^2)^m$ at $\C \otimes_\k I_m$, by work of Haiman \cite[Proposition
3.4.2]{Haiman}.

The (uncolored) link splitting map is determined by a deformation of the categorified 
HOMFLYPT skein relation. In order to understand link splitting behavior in the
colored, triply-graded context, we develop a curved version of the colored skein
relation that we introduced in the companion paper \cite{HRW1}. The culmination
of \S\ref{s:curved skein rel} is the following result. 
Herein, the brackets $\llbracket-\rrbracket_\YS$ denote curved lifts of the
Rickard complexes associated to shown tangled webs, 
and $\tw_{D}$ denotes a twist in the differential by a curved Maurer--Cartan $D$ 
(see \S \ref{ss:pert} for a review of this terminology). 

\begin{thm}[{Corollary \ref{cor:curvedMCCS}}]\label{thm:mainthm} 
Let $\Xi:=\{\xi_1,\dots,\xi_b\}$ be a set of exterior variables
with $\wt(\xi_i)=\tdeg\inv \qdeg^{2i}$, then there exists a homotopy equivalence
\begin{equation}\label{eq:conv}
\tw_{D_1}\left( \bigoplus_{s=0}^b \qdeg^{s(b-1)} \tdeg^s \left\llbracket
\begin{tikzpicture}[scale=.5,smallnodes,rotate=90,anchorbase]
	\draw[very thick] (1,-1) to [out=150,in=270] (0,0); 
	\draw[line width=5pt,color=white] (0,-2) to [out=90,in=270] (.5,0) to [out=90,in=270] (0,2);
	\draw[very thick] (0,-2) node[right=-2pt]{$a$} to [out=90,in=270] (.5,0) 
		to [out=90,in=270] (0,2) node[left=-2pt]{$a$};
	\draw[very thick] (1,1) to (1,2) node[left=-2pt]{$b$};
	\draw[line width=5pt,color=white] (0,0) to [out=90,in=210] (1,1); 
	\draw[very thick] (0,0) to [out=90,in=210] (1,1); 
	\draw[very thick] (1,-2) node[right=-2pt]{$b$} to (1,-1); 
	\draw[very thick] (1,-1) to [out=30,in=330] node[above=-2pt]{$s$} (1,1); 
\end{tikzpicture}
\right\rrbracket_\YS \right)
\simeq 
\tw_{D_2}\left(\qdeg^{b(a-b-1)}\tdeg^b \left\llbracket
\begin{tikzpicture}[scale=.4,smallnodes,anchorbase,rotate=270]
	\draw[very thick] (1,-1) to [out=150,in=270] (0,1) to (0,2) node[right=-2pt]{$b$}; 
	\draw[line width=5pt,color=white] (0,-2) to (0,-1) to [out=90,in=210] (1,1);
	\draw[very thick] (0,-2) node[left=-2pt]{$b$} to (0,-1) to [out=90,in=210] (1,1);
	\draw[very thick] (1,1) to (1,2) node[right=-2pt]{$a$};
	\draw[very thick] (1,-2) node[left=-2pt]{$a$} to (1,-1); 
	\draw[very thick] (1,-1) to [out=30,in=330] node[below=-1pt]{$a{-}b$} (1,1); 
\end{tikzpicture}
\right\rrbracket_\YS \otimes \largewedge[\Xi]\right)
\end{equation}
of curved twisted complexes.
Further, the curved Maurer--Cartan element $D_1$ is one-sided with 
respect to the partial order by the index $s$, 
and the right-hand side is a certain curved Koszul complex.
\end{thm}

For obvious reasons, we will sometimes refer to the left-hand side of 
\eqref{eq:conv} as the complex of \emph{threaded digons}, 
denoted by $\TD_b(a)$. As the notation suggests, it is useful to view this complex 
as a function of the threading $a$-colored strand. 
One consequence of Theorem \ref{thm:mainthm} is that $\TD_b(a) \simeq 0$ when $a<b$
(in that case, the webs on the right-hand side correspond to the zero bimodule).
	
\begin{rem}
An equivalent formulation of the curved colored skein relation \eqref{eq:conv}
that more closely resembles the usual HOMFLYPT skein relation 
(relating positive and negative crossings to their oriented resolution) is as follows:
\[
\tw_{D_1'}\left( \bigoplus_{s=0}^b \qdeg^{s(b-1)} \tdeg^s \left\llbracket
	\begin{tikzpicture}[scale=.5,smallnodes,rotate=90,anchorbase]
		\draw[very thick] (1,-1) to [out=150,in=210] (0,1); 
		\draw[line width=5pt,color=white] (0,-2) to [out=90,in=270] (1,2);
		\draw[very thick] (0,-2) node[right=-2pt]{$a$}to [out=90,in=270] (1,2) node[left=-2pt]{$a$};
		\begin{scope} 
		\clip (1,-.5) rectangle (0,.75);
		\draw[line width=5pt,color=white] (1,-1) to [out=30,in=330] (0,1); 
		\end{scope}
		\draw[very thick] (1,-1) to [out=30,in=330] node[above,yshift=-2pt]{$s$} (0,1); 
		\draw[very thick] (1,-2) node[right=-2pt]{$b$} to (1,-1); 
		\draw[very thick] (0,1) to (0,2) node[left=-2pt]{$b$};
	\end{tikzpicture}
\right\rrbracket_\YS \right)
\simeq 
\tw_{D_2'}\left(
	\qdeg^{-b}\tdeg^b \left\llbracket
	\begin{tikzpicture}[smallnodes,rotate=90,anchorbase,scale=.75]
		\draw[very thick] (0,.25) to [out=150,in=270] (-.25,1) node[left,xshift=2pt]{$b$};
		\draw[very thick] (.5,.5) to (.5,1) node[left,xshift=2pt]{$a$};
		\draw[very thick] (0,.25) to node[right]{$a{-}b$} (.5,.5);
		\draw[very thick] (.5,-1) node[right,xshift=-2pt]{$b$} to [out=90,in=330] (.5,.5);
		\draw[very thick] (-.25,-1)node[right,xshift=-2pt]{$a$} to [out=90,in=270] (0,.25);
	\end{tikzpicture}
\right\rrbracket_\YS \otimes \largewedge[\Xi]\right) \, .
\]
Note that the left-hand side involves complexes that interpolate between a positive and negative crossing.
The twists here have similar properties to those in Theorem \ref{thm:mainthm}.
\end{rem}

\begin{exa}
The Rickard complex for a crossing between a pair of 2-colored strands has the form
\[
C_{2,2}:=\left \llbracket
	\begin{tikzpicture}[rotate=90,scale=.5,smallnodes,anchorbase]
		\draw[very thick,->] (1,-1) node[right,xshift=-2pt]{$2$} to [out=90,in=270] (0,1);
		\draw[line width=5pt,color=white] (0,-1) to [out=90,in=270] (1,1);
		\draw[very thick,->] (0,-1) node[right,xshift=-2pt]{$2$} to [out=90,in=270] (1,1);
	\end{tikzpicture}
\right\rrbracket
=
\left(
\begin{tikzpicture}[smallnodes,rotate=90,anchorbase,scale=.75]
	\draw[very thick] (0,.25) to [out=150,in=270] (-.25,1) node[left,xshift=2pt]{$2$};
	\draw[very thick] (0,.25) to [out=30,in=270] (.5,1) node[left,xshift=2pt]{$2$};
	\draw[very thick] (0,-.25) to (0,.25);
	\draw[very thick] (.5,-1) node[right,xshift=-2pt]{$2$} to  [out=90,in=330] (0,-.25);
	\draw[very thick] (-.25,-1)node[right,xshift=-2pt]{$2$} to [out=90,in=210] (0,-.25);
\end{tikzpicture}
\to  \qdeg\inv  \tdeg
\begin{tikzpicture}[smallnodes,rotate=90,anchorbase,scale=.75]
	\draw[very thick] (0,.25) to [out=150,in=270] (-.25,1) node[left,xshift=2pt]{$2$};
	\draw[very thick] (.5,.5) to (.5,1) node[left,xshift=2pt]{$2$};
	\draw[very thick] (0,.25) to node[left,xshift=2pt,yshift=-1pt]{$1$} (.5,.5);
	\draw[very thick] (0,-.25) to (0,.25);
	\draw[very thick] (.5,-.5) to [out=30,in=330] (.5,.5);
	\draw[very thick] (0,-.25) to node[right,xshift=-2pt,yshift=-1pt]{$1$} (.5,-.5);
	\draw[very thick] (.5,-1) node[right,xshift=-2pt]{$2$} to (.5,-.5);
	\draw[very thick] (-.25,-1)node[right,xshift=-2pt]{$2$} to [out=90,in=210] (0,-.25);
\end{tikzpicture}
\to  \qdeg^{-2} \tdeg^2
\begin{tikzpicture}[smallnodes,rotate=90,anchorbase,scale=.75]
	\draw[very thick] (-.25,-1) node[right,xshift=-2pt]{$2$} to (-.25,1) node[left,xshift=2pt]{$2$};
	\draw[very thick] (.5,-1) node[right,xshift=-2pt]{$2$} to (.5,1) node[left,xshift=2pt]{$2$};
	\end{tikzpicture}
\right)
\]
We denote the webs appearing in this complex as $W_2$, $W_1$ and $W_0$ respectively. 
After basis change in the exterior algebras, 
the curved twisted complex on the right-hand side of \eqref{eq:conv} 
has the following schematic form: 
\begin{equation}
	\label{eq:exaYKMCS}
		\begin{tikzpicture}[anchorbase]
			\node[scale=1] at (5,-2.5){$C_{2,2} \otimes \wedge$
			};
			\draw[->] (5.5,-2.25) to (5.5,-1.25);
			\draw[->] (4.25,-2.5) to (3.25,-2.5);
			\node[scale=.75] at (-3.5,3.5){
				$\left\llbracket
	\begin{tikzpicture}[scale=.5,smallnodes,anchorbase,rotate=90]
			\draw[very thick] (1,0) to [out=90,in=270] (0,1.5)node[left,xshift=2pt]{$2$};
			\draw[line width=5pt,color=white] (1,-1.5) to [out=90,in=270] (0,0) 
			   to [out=90,in=270] (1,1.5);
			\draw[very thick] (1,-1.5)node[right,xshift=-2pt]{$2$}  to [out=90,in=270] (0,0) 
			   to [out=90,in=270] (1,1.5)node[left,xshift=2pt]{$2$};
			\draw[line width=5pt,color=white] (0,-1.5) to [out=90,in=270] (1,0);
			\draw[very thick] (0,-1.5)node[right,xshift=-2pt]{$2$}  to [out=90,in=270] (1,0);
	\end{tikzpicture} 
	\right\rrbracket$
			};
			\draw[->,gray] (-2,3.5) to (-1.5,3.5);
			\node[scale=.75, blue] at (0,3.5){
				$\left\llbracket
	\begin{tikzpicture}[scale=.375,smallnodes,rotate=90,anchorbase]
	\draw[very thick] (1,-1) to [out=150,in=270] (0,0); 
	\draw[line width=5pt,color=white] (0,-2) to [out=90,in=270] (.5,0) to [out=90,in=270] (0,2);
	\draw[very thick] (0,-2) node[right=-2pt]{$2$}to [out=90,in=270] (.5,0) 
		to [out=90,in=270] (0,2) node[left=-2pt]{$2$};
	\draw[line width=5pt,color=white] (0,0) to [out=90,in=210] (1,1); 
	\draw[very thick] (0,0) to [out=90,in=210] (1,1); 
	\draw[very thick] (1,-2) node[right=-2pt]{$2$} to (1,-1); 
	\draw[very thick] (1,-1) to [out=30,in=330] (1,1); 
	\draw[very thick] (1,1) to (1,2) node[left=-2pt]{$2$};
	\end{tikzpicture}
	\right\rrbracket$
			};
			\draw[->,gray] (1.5,3.5) to (2,3.5);
			\node[scale=.75, green] at (3.5,3.5){
				$\left\llbracket
	\begin{tikzpicture}[scale=.5,smallnodes,anchorbase,rotate=90]
		\draw[very thick] (0,-1.5) node[right,xshift=-2pt]{$2$} to (0,1.5)node[left,xshift=2pt]{$2$};
		\draw[very thick] (1,-1.5) node[right,xshift=-2pt]{$2$} to (1,1.5)node[left,xshift=2pt]{$2$};
	\end{tikzpicture} 
	\right\rrbracket$
			};
			\draw[dotted] (1.75,4) to (1.75,3) to [out=270,in=180] (4,1.75) to (5,1.75);
			\draw[dotted] (-1.75,4) to (-1.75,3) to [out=270,in=180] (2,-.75) to (5,-.75);
			\node[scale=1] at (0,0.5){
		\begin{tikzcd}[row sep=2em,column sep=-1.5em]
			& 
			W_{2}\otimes \zeta^{(2)}_1\zeta^{(2)}_2	
			\arrow[ddl]
			\arrow[dr ]
			\arrow[from=dr, shift left,dashed] 
			\arrow[rrr,gray,] & & & 
			\BLUE{W_{1}\otimes \zeta^{(1)}_1\zeta^{(1)}_2}
			\arrow[dr,blue] 
			\arrow[from=dr, shift left,dashed,blue] 
			\arrow[rrr,gray] & & & 
			\GREEN{W_{0}\otimes \zeta^{(0)}_1\zeta^{(0)}_2}	
			\arrow[from=dr,dashed,gray] & 
			\\
			& &
			W_{2}\otimes \zeta^{(2)}_2	
			\arrow[ldd]
			\arrow[from=ldd, shift left,dashed] 
			\arrow[dr,] 
			\arrow[rrr,gray] & & & 
			\BLUE{W_{1}\otimes \zeta^{(1)}_2}	
			\arrow[from=ldd,dashed,gray] 
			\arrow[dr,blue] 
			\arrow[rrr,blue]  & & & 
			\BLUE{W_{0}\otimes \zeta^{(0)}_2}	
			\arrow[from=ldd,dashed,gray]
			\\
			W_{2}\otimes \zeta^{(2)}_1	
			\arrow[dr]
			\arrow[from=dr, shift left,dashed] 
			\arrow[rrr,crossing over] 
			\arrow[uur,shift left,dashed] & & & 
			W_{1}\otimes \zeta^{(1)}_1	
			\arrow[dr]
			\arrow[from=dr, shift left,dashed] 
			\arrow[rrr,crossing over,gray] 
			\arrow[uur,crossing over,dashed,gray] & & & 
			\BLUE{W_{0}\otimes \zeta^{(0)}_1}	
			\arrow[from=dr,dashed,gray]
			\arrow[uur,crossing over,dashed,gray] & & 
			\\
			&
			 W_{2}\otimes 1	 \arrow[rrr]& & & 
			 W_{1}\otimes 1	 \arrow[rrr]& & & 
			 W_{0}\otimes 1	 &
		\end{tikzcd}
			};
		\end{tikzpicture}
\end{equation}
The subquotients with respect to the filtration indicated by the dotted
lines (and colored \textbf{black}, {\color{blue} \textbf{blue}}, and {\color{green} \textbf{green}})
are homotopy equivalent to the indicated complexes that appear 
on the left-hand side of \eqref{eq:conv}. 
(Compare to the corresponding figure in \cite[Section 1]{HRW1}, 
where the dashed components of the differential do not appear.)
Additional details can be found in Example~\ref{exa:KMCS}.
\end{exa}

In \S\ref{sec:splitting} we use the curved colored skein relation \eqref{eq:conv}
to study splitting properties of the deformed, colored, triply-graded link homology. 
In particular, we obtain an explicit model for the \emph{colored link splitting map} 
from the colored full twist on two strands to the identity braid in \S\ref{ss:explicit splitting map}. 
In Theorem~\ref{thm:skein split} we obtain a 
\emph{simultaneous splitting} of the complex of threaded digons.

In \S \ref{s:colored Hilb}, we begin the study of the colored full twist. 
Paralleling the uncolored case, we conjecture that the \emph{parity} 
property enjoyed by positive torus links 
\cite{MR3880028, mellit2017homology, hogancamp2019torus} 
holds in the colored setting as well; see Conjecture \ref{conj:FT is parity}. 
Under this assumption, we show that, in lowest Hochschild degree,
the deformed, colored homology of the $\brc$-colored $(m,m)$-torus link (the closure of the full twist) 
embeds as an ideal $J_{\brc}$ in the algebra
\[
E_{\brc} := \bigotimes_{i=1}^m \Sym(\X^{b_i})[v_{i,1},\ldots,v_{i,b_i}] \, .
\]
In the uncolored case, the ideal $J_{1^m}$ is precisely the ideal $I_m$ 
mentioned above, 
which has an explicit description as
\[
I_m := \k[\X,\Y] \cdot \left\{ f(\X,\Y) \mid f\in \k[\X,\Y] \text{ is antisymmetric for } \symg_{m} \right\} \, .
\]
Generalizing this, we posit the following in Conjecture \ref{conj:FT ideal}:

\begin{conj}\label{conj:FT}
Let $\brc=(b_1,\ldots,b_m)$ and $N=\sum_{i=1}^m b_i$.
Let $T(m,m;\brc)$ denote the $\brc$-colored $(m,m)$-torus link, 
then the colored splitting map identifies
the lowest Hochschild degree summand of 
$\YS H_{\KR}\big(T(m,m;\brc)\big)$ with the ideal
\[
I_{\brc}:= E_{\brc}\cdot \left\{\frac{f(\X,\Y)}{\Delta(\X^{b_1})\cdots \Delta(\X^{b_m})} 
\mid f\in \k[\X,\Y] \text{ is antisymmetric for } \symg_{N} \right\} \, .
\]
Here, $\X= \bigcup_{i=1}^m \X^{b_i} = \{x_1,\ldots,x_N \}$,
$\Y=\{y_1,\ldots,y_N \}$, 
$\Delta(\X^{b_i})$ denotes the Vandermonde determinant, 
and we include $\k[\X,\Y] \hookrightarrow E_{\brc}$ using the appropriate 
analogue of \eqref{eq:firstytov} (precisely recorded in \eqref{eq:y to v 2}).
\end{conj}

In \S \ref{s:hopf link}, we prove this conjecture in the $2$-strand case, 
i.e. for the $(a,b)$-colored Hopf link. 
First, in Proposition \ref{prop:Hopfparityandgens} we confirm that the colored Hopf link 
is indeed parity. We then embark on a complicated inductive journey, 
using the simultaneous colored skein splitting referenced above, 
which culminates in Theorem \ref{thm:Jab} with
the verification of Conjecture \ref{conj:FT} when $m=2$. 
Along the way, we encounter specific \emph{Haiman determinants} 
(reviewed in \S \ref{ss:haiman dets}) and in Corollary \ref{cor:JHaiman} 
we give an explicit set of generators for $I_{a,b}$ using these elements.

Finally, in \S \ref{sec:linksplit}, we use Theorem \ref{thm:Jab} to extend our 
link splitting results from \S \ref{sec:splitting} to the case of arbitrary colored links.
For certain specializations, 
this generalizes results obtained in \cite{BS,GH} to our setting, 
and recovers the colored link splitting result from \cite{MR4178751}.
We also speculate on the interpretation of more-interesting specializations, 
making contact with Conjecture \ref{conj:E}, which is stated in the following section.

\subsection{Further conjectures}

To conclude this (extended) introduction, we collect two particularly enticing
conjectures that are motivated by the following $\sim100$ pages of this paper.

\subsubsection{Homology of cables}
We propose a precise relation between deformed, colored, triply-graded
homology and the deformed (uncolored) triply-graded homology of cables, 
focusing on the case of cabled knots, for ease of exposition. 
Recall the following result in the undeformed setting.

\begin{thm}[{\cite[Theorem 6.1 and Corollary 6.5]{2019arXiv190404481G}}]
	\label{thm:cables}
Let $\KB$ be a framed, oriented knot, and let $\KB^b$ denote the $b$-cable of $\KB$ (a $b$-component link).  
Both the triply-graded Khovanov--Rozansky complex $C_{\KR}(\KB^b)$ and the $\glN$
Khovanov--Rozansky complex $C_{\KR_N}(\KB^b)$ carry an action of the symmetric
group $\symg_b$, up to homotopy, induced by braiding components of the cable. In
the $\glN$ case, the isotypic component corresponding to the trivial
representation is equivalent to the $b$-colored $\glN$ Khovanov--Rozansky complex
$C_{\KR_N}(\KB;b)$.
\end{thm}

It is natural to consider the extension of this result for deformed triply-graded link homology. 
Interestingly, the most na\"{i}ve extension of this result is false.

\begin{exa}
Let $\UB$ be the 0-framed unknot, and $\UB^b$ its $b$-cable. 
Then, 
\[
\YS C_{\KR}^{\text{low}}(\UB^b) = \k[x_1,\ldots,x_b,y_1,\ldots,y_b] 
\]
with the standard $\symg_b$-action. 
The $\symg_b$-invariant part of $\YS C_{\KR}^{\text{low}}(\UB^b)$ does not equal the $b$-colored invariant 
\[
\YS C_{\KR}(\UB;b)=\k[x_1,\ldots,x_b]^{\symg_b}\otimes \k[v_1,\ldots,v_b] 
\]
from \eqref{eq:bcoloredU}, since $\wt(v_k) = \qdeg^{-2k}\tdeg^2$ 
while $\wt(y_i)=\qdeg^{-2}\tdeg^2$ for all $i$.
Nonetheless, there is a natural algebra map
\[
\YS C_{\KR}(\UB^b)^{\symg_b} \rightarrow  \YS C_{\KR}(\UB;b) \, , \quad 
y_i\mapsto \sum_{k=1}^b x_i^{k-1} v_k.
\]
\end{exa}

Based on our investigations in the present paper, we formulate the following.

\begin{conj}\label{conj:cables} Let $\KB$ be a framed, oriented knot, and let
$\chi\KB$ denote $\KB$ with framing increased by one. For each integer $b\geq 0$
we consider a directed system of complexes
\begin{equation}\label{eq:intro directed system}
\YS C_{\KR}(\KB^b)^{\text{triv}} \rightarrow \YS C_{\KR}((\chi\KB)^b)^{\text{sgn}} 
\rightarrow \YS C_{\KR}((\chi^2\KB)^b)^{\text{triv}} \rightarrow \cdots 
\end{equation}
in which the maps are inherited from the ``bottom eigenmap'' for the 
curved Rickard complex $\YS C(\FT_b)$ associated to the $b$-strand full-twist $\FT_b$. 
The (homotopy) colimit of this directed system is quasi-isomorphic to the $b$-colored 
curved Rickard complex $\YS C_{\KR}(\KB,b)$, as complexes of modules over
\[
\k[x_1,\ldots,x_b]^{\symg_b}\otimes \k[v_1,\ldots,v_b] \cong 
\operatorname{colim}\left(\YS C_{\KR}(\UB^b)^{\text{triv}} 
\to \YS C_{\KR}((\chi\UB)^b)^{\text{sgn}} \to \cdots \right) \, .
\]
\end{conj}

Let us comment on this conjecture. 
First, note that the $b$-cable of $\chi \KB$ and the $b$-cable of $\KB$ differ by the insertion of a full twist braid $\FT_b$. 
Thus, the directed system \eqref{eq:intro directed system} is inherited from a directed system
\[
\oone_b \rightarrow \FT_b \rightarrow \FT_b^2 \rightarrow \cdots
\]
in the category $\YS\SBim_b$ defined in \cite{GH}. 
There are many choices one can make for the connecting maps. 
Indeed, results in \cite{GH} show that
\[
\Hom_{\YS\SBim_b}(\FT_b^k,\FT_b^{k+1}) \simeq \Hom_{\YS\SBim_b}(\oone_b,\FT_b) 
\simeq \left\langle \k[\X,\Y]^{\text{sgn}}\right\rangle \subset \k[\X,\Y] \, .
\]
Among all such morphisms, 
the one corresponding to the Vandermonde determinant $\Delta(\X)$ 
is distinguished as the generator of cohomological degree zero.  
The associated morphism $\oone_b \rightarrow \FT_b$ in $\YS\SBim_b$ 
(or in the undeformed category of complexes or Soergel bimodules) is referred to as the
``bottom eigenmap'', adopting terminology from \cite{EH}.

Conjecture \ref{conj:cables} is supported by our Theorem \ref{thm:Jab}, 
i.e. our verification of Conjecture \ref{conj:FT} in the case of colored Hopf links, 
since taking colimits of directed systems of the form \eqref{eq:intro directed system} 
is akin to inverting the Vandermonde $\Delta(\X)$. 
We save explorations along these lines for future work.

\subsubsection{Threaded digons and the Hilbert scheme}

Finally, in a different direction, we propose that a generalization of the complex $\TD_b(a)$ of 
threaded digons from Theorem \ref{thm:mainthm} constructs a family of complexes desired in 
the Gorsky--Negu\c{t}--Rasmussen (GNR) conjecture \cite{GNR}.

\begin{conj}\label{conj:E}
Let $a_1,\ldots,a_m,b \geq 0$ and set $\aa=(a_1,\ldots,a_m)$, 
then there exists a one-sided twisted \emph{complex of threaded digons}:
\[
\TD_b(\aa)=
\left( \
\left\llbracket
	\begin{tikzpicture}[scale=.4,smallnodes,rotate=90,anchorbase]
		\draw[very thick] (1,-2) node[right=-2pt]{$b$} to (1,-1) to [out=150,in=270] (-1.75,0); 
		\draw[line width=5pt,color=white] (0,-2) to [out=90,in=270] (.5,0);
		\draw[very thick] (0,-2) node[right=-2pt]{$a_m$}to [out=90,in=270] (.5,0) 
			to [out=90,in=270] (0,2) node[left=-2pt]{$a_m$};
		\draw[line width=5pt,color=white] (-1.5,-2) to [out=90,in=270] (-1,0);
		\draw[very thick] (-1.5,-2) node[right=-2pt]{$a_1$}to [out=90,in=270] (-1,0) 
			to [out=90,in=270] (-1.5,2) node[left=-2pt]{$a_1$};
		\begin{scope}
			\clip (-1.75,.1) rectangle (.875,1.05);
			\draw[line width=5pt,color=white] (-1.75,0) to [out=90,in=210] (1,1); 
		\end{scope}
		\draw[very thick] (-1.75,0) to [out=90,in=210] (1,1) to (1,2) node[left=-2pt]{$b$}; 
		\draw[dotted] (1,-1) to [out=30,in=330] node[above,yshift=-2pt]{$0$} (1,1); 
		\node at (-.5,0) {$\vdots$}; \node at (-.925,-1.5) {$\vdots$}; \node at (-.925,1.5) {$\vdots$};
	\end{tikzpicture}
\right\rrbracket_{\YS}
\to
\left\llbracket
	\begin{tikzpicture}[scale=.4,smallnodes,rotate=90,anchorbase]
		\draw[very thick] (1,-2) node[right=-2pt]{$b$} to (1,-1) to [out=150,in=270] (-1.75,0); 
		\draw[line width=5pt,color=white] (0,-2) to [out=90,in=270] (.5,0);
		\draw[very thick] (0,-2) node[right=-2pt]{$a_m$}to [out=90,in=270] (.5,0) 
			to [out=90,in=270] (0,2) node[left=-2pt]{$a_m$};
		\draw[line width=5pt,color=white] (-1.5,-2) to [out=90,in=270] (-1,0);
		\draw[very thick] (-1.5,-2) node[right=-2pt]{$a_1$}to [out=90,in=270] (-1,0) 
			to [out=90,in=270] (-1.5,2) node[left=-2pt]{$a_1$};
		\begin{scope}
			\clip (-1.75,.1) rectangle (.875,1.05);
			\draw[line width=5pt,color=white] (-1.75,0) to [out=90,in=210] (1,1); 
		\end{scope}
		\draw[very thick] (-1.75,0) to [out=90,in=210] (1,1) to (1,2) node[left=-2pt]{$b$}; 
		\draw[very thick] (1,-1) to [out=30,in=330] node[above,yshift=-2pt]{$1$} (1,1); 
		\node at (-.5,0) {$\vdots$}; \node at (-.925,-1.5) {$\vdots$}; \node at (-.925,1.5) {$\vdots$};
	\end{tikzpicture}
\right\rrbracket_{\YS}
\to \cdots \to
\left\llbracket
	\begin{tikzpicture}[scale=.4,smallnodes,rotate=90,anchorbase]
		\draw[dotted] (1,-1) to [out=150,in=270] (-1.75,0); 
		\draw[very thick] (0,-2) node[right=-2pt]{$a_m$}to [out=90,in=270] (.5,0) 
			to [out=90,in=270] (0,2) node[left=-2pt]{$a_m$};
		\draw[very thick] (-1.5,-2) node[right=-2pt]{$a_1$}to [out=90,in=270] (-1,0) 
			to [out=90,in=270] (-1.5,2) node[left=-2pt]{$a_1$};
		\draw[dotted] (-1.75,0) to [out=90,in=210] (1,1); 
		\draw[very thick] (1,-2) node[right=-2pt]{$b$} to (1,-1) to [out=30,in=330] node[above,yshift=-2pt]{$b$} (1,1)
			to (1,2) node[left=-2pt]{$b$}; 
		\node at (-.5,0) {$\vdots$}; \node at (-.925,-1.5) {$\vdots$}; \node at (-.925,1.5) {$\vdots$};
	\end{tikzpicture}
\right\rrbracket_{\YS} \
\right) 
\]
(potentially with longer arrows pointing to the right) such that:
\begin{enumerate}[(i)]
	\item $\TD_b(\aa)\simeq 0$ if $b>a_1+\cdots+a_m$, and \label{conj:E1}
	\item the partial trace $\Tr^{b}(\TD_b(1,\dots, 1) \otimes \k[v_1,\ldots,v_b])$ 
	categorifies the $b^{th}$ elementary symmetric function in the Jucys-Murphy braids on $m$ strands
	(here, $\{v_1,\ldots,v_b\}$ are deformation parameters for the $b$-labeled strand).
	Further, this is the complex $\cal{E}_b \in \YS(\SBim_m)$ corresponding to the $b^{th}$ exterior power of the 
	tautological bundle on the flag Hilbert scheme $\FHilb_m(\C^2)$ under the GNR conjecture. \label{conj:E2}
\end{enumerate}
\end{conj}

The $m=1$ case of Conjecture \ref{conj:E} follows from Theorem \ref{thm:mainthm}.
Indeed, we noted above that $\TD_b(a) \simeq 0$ when $a < b$, 
and $\Tr^{1}(\TD_1(1) \otimes \k[v_1])$ can be explicitly identified with the identity braid 
on one strand using the right-hand side of \eqref{eq:conv}.
As further evidence for Conjecture \ref{conj:E}, note that item \eqref{conj:E1} is a natural extension of the $m=1$ case.
For item \eqref{conj:E2}, observe that the elementary symmetric function $e_b(-)$ has similar behavior 
to the (conjectured) behavior of $\TD_b(-)$: it vanishes when the ``input'' is larger than $b$.
Less-heuristically, results of Morton \cite{Morton} show that the complex 
$\Tr^{1}(\TD_1(1,\dots, 1) \otimes \k[v_1])$ would categorify the sum 
(i.e. $1^{st}$ elementary symmetric function) of the Jucys-Murphy elements. 
The extension to general $b$ is suggested by this, and Conjecture \ref{conj:cables}.
Lastly, we mention that work of Elias \cite{EliasG} proposes a different approach to constructing $\cal{E}_1$.

\subsection{Coefficient conventions}\label{ss:coeff}

Throughout, we work over the field $\k$ of rational numbers for simplicity. 
All of our results remain true over an arbitrary field of characteristic zero. 
We believe that all results should hold over the integers; 
however, the proof we present for Markov invariance of $\YS H_{\KR}(\LB)$ uses
the power-sum symmetric functions $\frac{1}{k}p_k(\X)$, 
thus requires working over a field of characteristic zero.
Nonetheless, our $2$-category $\YS(\SSBim)$ is defined using $\Delta e$-curvature 
(equivalently, $h \Delta$-curvature), which allows for an integral version of this $2$-category.

\subsection*{Acknowledgements}
This project was conceived during the conference 
``Categorification and Higher Representation Theory'' at the Institute Mittag-Leffler, 
and began in earnest during the workshop ``Categorified Hecke algebras, link
homology, and Hilbert schemes'' at the American Institute for Mathematics. We
thank the organizers and hosts for a productive working atmosphere. We would
also thank Eugene Gorsky and Lev Rozansky for many useful discussions.

\subsection*{Funding}

M.H.  was supported by NSF grant DMS-2034516.
D.R. and P.W. were supported in part by the National Science Foundation under
Grant No. NSF PHY-1748958 during a visit to the program ``Quantum Knot
Invariants and Supersymmetric Gauge Theories'' at the Kavli Institute for
Theoretical Physics. 
D.R. was partially supported by Simons Collaboration Grant 523992: 
``Research on knot invariants, representation theory, and categorification.''
P.W. was partially supported by the Australian Research
Council grants `Braid groups and higher representation theory' DP140103821 and
`Low dimensional categories' DP160103479 while at the Australian National
University during early stages of this project. P.W. was also supported by the
National Science Foundation under Grant No. DMS-1440140, while in residence at
the Mathematical Sciences Research Institute in Berkeley, California, during the
Spring 2020 semester.

\section{Symmetric functions}
\label{sec:symfns}
In this section, we collect assorted background material on symmetric functions.
The reader can safely skip \S\ref{ss:symfctns}, \S\ref{ss:derivatives}, and \S\ref{ss:h reduction} and
return whenever results or formulas from here are used in the later parts of the
paper. We do, however, recommend a look at \S\ref{ss:haiman dets} for readers
unfamiliar with Haiman determinants.

\subsection{Symmetric functions}
\label{ss:symfctns}
Symmetric functions play an important role throughout this paper. A more
detailed exposition appears in~\cite[\HRWsym]{HRW1}. 

\emph{Alphabets} are finite or countably infinite sets that we denote by blackboard
letters, such as $\X$, $\Y$ etc. 
Given an alphabet $\X$, the $\k$-algebra of symmetric functions on
$\X$ will be denoted $\Sym(\X)$. Symmetric functions on a finite alphabet will
also be called symmetric polynomials.
For pairwise disjoint alphabets $\X_1, \dots, \X_r$, we write
\[
\Sym(\X_1|\cdots|\X_r) \cong \Sym(\X_1)\otimes \cdots \otimes \Sym(\X_r)
\]
for the ring of functions on $\X_1\cup\cdots\cup\X_r$ that are separately
symmetric in each of the alphabets $\X_i$. 

\begin{defi}\label{def:sym}
The \emph{elementary} symmetric functions $e_j(\X)$,
	\emph{complete} symmetric functions $h_j(\X)$, and \emph{power sum}
	symmetric functions $p_j(\X)$ are each defined via
	their generating functions as follows:
	\[
	\begin{aligned}
	E(\X,t) &:= \prod_{x\in \X} (1+x t) =: \sum_{j \geq 0} e_j(\X) t^j \\
	H(\X,t) &:= \prod_{x\in \X} (1-x t)\inv =: \sum_{j\geq 0} h_j(\X) t^j \\
	P(\X,t) & := \sum_{x \in \X} \frac{x t}{1-x t} =: \sum_{j\geq 1} p_j(\X) t^j
	\end{aligned}
	\]
	By convention $e_0(\X) = h_0(\X) = 1$ and $p_0(\X)$ is undefined. In the
	case of countably infinite alphabets, we sometimes drop the alphabet from
	the notation and write $E(t)$, $H(t)$, and $P(t)$, and
	$e_j$, $h_j$, and $p_j$, for the functions introduced above.
\end{defi} 
	
The elementary and complete symmetric functions are related by the identity
\begin{equation}\label{eqn:HE}
H(\X,t)E(\X,-t) = 1 \, , \quad \text{i.e.}\quad
\sum_{i+j=k} (-1)^j h_i(\X) e_j(\X) = 0 \quad \forall k \geq 1 \, ,
\end{equation}
and each are related to the power sum symmetric functions by the Newton identity:
\[
\frac{t\frac{d}{dt}H(\X,t)}{H(\X,t)} = P(\X,t) \, , \quad \text{i.e.}\quad
H(\X,t) = \mathrm{exp}  \int P(\X,t) \frac{dt}{t} \, .
\]

We will work with the highly useful formalism of linear combinations of alphabets, 
see \cite[\HRWDefFLA]{HRW1}. 
In particular, for the generating functions in Definition \ref{def:sym}, we have
\begin{align*}
H(a_1\X_1+a_2\X_2,t) &= H(\X_1,t)^{a_1} H(\X_2,t)^{a_2} \\
E(a_1\X_1+a_2\X_2,t) &= E(\X_1,t)^{a_1} E(\X_2,t)^{a_2} \\
P(a_1\X_1+a_2\X_2,t) &= a_1 P(\X_1,t) +  a_2 P(\X_2,t)
\end{align*}
for any $a_1,a_2\in \k$.   
A good illustration of how this formalism works is given by the following computation:
\[
H(\leftX - \rightX,t) = H(\X,t)H(\rightX,t)\inv = H(\X,t) E(\rightX,-t) = E(\rightX,-t) E(\X,-t)\inv,
\]
from which we obtain the following
\begin{equation}\label{eq:HE2}
h_k(\leftX - \rightX) = \sum_{i+j=k} (-1)^j h_i(\X)e_j(\rightX)
\end{equation}
and its well-known consequence:
\begin{equation}\label{eq:wellknown}
\sum_{i+j=k} (-1)^j h_i(\X)e_j(\X) =
\begin{cases}
1 & \text{if } k=0 \\
0 & \text{else} \, .
\end{cases} 
\end{equation}

\subsection{Kernel of multiplication}
\label{ss:derivatives}

If $A$ is a commutative algebra, we let $N\lhd A\otimes A$ be the kernel of the
multiplication map $A\otimes A \rightarrow A$. Equivalently, $N$ is the ideal
inside $A\otimes A$ generated by differences $a\otimes 1 - 1\otimes a$. 

\begin{rem}
The Hochschild homology $\HH_\bullet(A)$ of a commutative algebra $A$ is itself a
graded-commutative algebra. We have $\HH_0(A)\cong A$ and $\HH_1(A)\cong N/ N^2$.
More generally, each $\HH_k(A)$ is a module over $\HH_0(A)\cong A$. 
\end{rem}

We will focus on the case $A=\Sym(\X)$ for a finite alphabet $\X$ and identify
$\Sym(\X)^{\otimes 2} = \Sym(\X|\rightX)$. Let $N(\X,\rightX)\lhd \Sym(\X|\rightX)$ be the
ideal generated by elements of the form $f(\X)-f(\rightX)$. In this section, we will
record various relationships between generating sets of $N(\X,\rightX)$.

\begin{prop} If $|\X|=|\rightX|=a$, then $N(\X,\rightX)\lhd \Sym(\X|\rightX)$ is generated
	by any of the following:
	\begin{gather*}
	\{e_k(\X)-e_k(\rightX)\}_{1\leq k\leq a}
	 \, ,\quad
	\{h_k(\X)-h_k(\rightX)\}_{1\leq k\leq a}
	 \, ,\quad
	\{p_k(\X)-p_k(\rightX)=p_k(\leftX - \rightX)\}_{1\leq k\leq a} \\
	\{h_k(\leftX - \rightX)\}_{1\leq k\leq a}
	\, ,\quad 
	\{e_k(\leftX - \rightX)\}_{1\leq k\leq a}\, .
	\end{gather*}
	Furthermore, the element
	$h_k(\leftX - \rightX)$ equals $\frac{1}{k}p_k(\leftX - \rightX)$ modulo $N(\X,\rightX)^2$.
\end{prop}
\begin{proof} 
 Lemma~\ref{lem:somerelations} below shows that the families of elements
$e_k(\X)-e_k(\rightX)$, $h_k(\leftX - \rightX)$, $h_k(\X)-h_k(\rightX)$, and $e_k(\leftX - \rightX)$ for
$1\leq k\leq a$ generate the same ideal. Since any symmetric function $f$ can be
written as a polynomial in the $e_k$, it is straightforward to check that this ideal
is $N(\X,\rightX)$. Next, observe that $p_k(\X)-p_k(\rightX) = p_k(\leftX - \rightX)$, so each of
the above symmetric functions involves the virtual alphabet $\leftX - \rightX$. The
relations \eqref{eq:somerelations2} and \eqref{eq:somerelations3} then imply the
second and third statement. 
\end{proof}

\begin{lemma}\label{lem:somerelations} Let $\X,\rightX$ be alphabets. 
For $k\geq 1$, we have
	\begin{align}
	\label{eq:somerelations2}
		p_k(\leftX - \rightX) &= \sum_{j=1}^k (-1)^{k-j} j h_j(\leftX - \rightX) e_{k-j}(\leftX - \rightX)\, ,\\
				\label{eq:somerelations1}
		h_k(\leftX - \rightX) &= \sum_{j=1}^k (-1)^{j-1} h_{k-j}(\X) (e_j(\X)-e_j(\rightX))\, ,
\end{align}
and their inverses
\begin{align}
		\label{eq:somerelations0}
		e_k(\X)-e_k(\rightX) &= \sum_{j=1}^k (-1)^{j-1} e_{k-j}(\X) h_j(\leftX - \rightX)\, ,\\
		\label{eq:somerelations3}
		h_k(\leftX - \rightX) &= \frac{1}{k} \sum_{j=1}^k h_{k-j}(\leftX - \rightX)p_j(\leftX - \rightX)\, .
\end{align}
Moreover,
\begin{align}
\label{eq:somerelations4}
	(-1)^k e_k(\leftX - \rightX) = h_k(\rightX-\X) &= - \sum_{j=1}^{k} h_{k-j}(\rightX-\X) h_j(\leftX - \rightX)\, .
\end{align}
\end{lemma}
Note that we obtain another collection of identities by applying
the algebra involution $p_k\mapsto -p_k$, $h_k\leftrightarrow (-1)^k e_k$, or by
swapping the roles of $\X$ and $\rightX$.
\begin{proof}
These equations are efficiently proved by the following manipulations of
generating functions, i.e.
\begin{align*}
P(\leftX - \rightX,t) &= \big(t \frac{d}{dt} H(\leftX - \rightX,t)\big) E(\leftX - \rightX,-t)\\
H(\X,t) H(\rightX,t)^{-1} - 1 &= -H(\X,t) \big(E(\X,-t) - E(\rightX,-t)\big)
\end{align*}
and
\begin{align*}
\big(E(\X,-t) - E(\rightX,-t)\big)&= -E(\X,-t)(H(\X,t) H(\rightX,t)^{-1} - 1) \\
t\frac{d}{dt}H(\leftX - \rightX,t) &= H(\leftX - \rightX,t)P(\leftX - \rightX,t)
\end{align*}
establish \eqref{eq:somerelations2} -- \eqref{eq:somerelations3}.
Equation \eqref{eq:somerelations4} is proved by the computation
	\begin{align*}
		E(\leftX - \rightX,-t) - 1 = H(\rightX-\X,t) - 1 &= - H(\rightX-\X,t)(H(\leftX - \rightX,t) -1 )\, . \qedhere
	\end{align*}
\end{proof}

\subsection{Hook Schur functions and h-reduction}
\label{ss:h reduction}
Let $\X$ be an alphabet (or a formal linear combination of alphabets).  If $\Y$
is an alphabet with $|\Y|\leq c$, then complete symmetric functions $h_N(\X)$
can be expressed as $\Sym(\X+\Y)$-linear combinations of complete symmetric
functions $h_n(\X)$ for $n\leq c$. We refer to this process as
$h$\emph{-reduction} and describe it explicitly in Lemma~\ref{lemma:h
reduction}.

\begin{example}
Consider the following identity in the polynomial ring $\k[x_1,\ldots,x_a]$ (for
each $1\leq i\leq a$)
\[
0 = \prod_{j=1}^a (x_i-x_j) = \sum_{k=0}^a (-1)^{a-k} e_{a-k}(\X) x_i^k\, .
\]
This allows us to write $x_i^a$ as
a $\k[\X]^{\symg_a}$-linear combination of monomials $x_i^k$ with $0\leq k\leq
a-1$. 
\[
x_i^a = \sum_{k=0}^{a-1} (-1)^{a-k-1} e_{a-k}(\X) x_i^k\, .
\]
\end{example}

The general case of $h$-reduction requires Schur functions associated to hook shapes.

\begin{definition}\label{def:hookschur} For $i,j\geq 0$ we write
$\mathfrak{s}_{(i|j)}:=\mathfrak{s}_{(i+1,1^j)}$ for the hook Schur functions,
which can be described as the family of symmetric functions satisfying:
\[
\mathfrak{s}_{(i-1|0)}= h_i \, , \quad \mathfrak{s}_{(0|j-1)}= e_j \, , \quad  h_i e_j = \mathfrak{s}_{(i|j-1)}+\mathfrak{s}_{(i-1|j)}\, .
\]
By convention $\mathfrak{s}_{(i|j)}=0$ if $i<0$ or $j<0$.  We denote the
two-parameter generating function of the hook Schur functions in an alphabet
$\X$ by
\[
S(t,u):=\sum_{i,j\geq 0} \mathfrak{s}_{(i|j)}t^iu^j\, .
\]
\end{definition}

\begin{lemma}\label{lemma:hook gen fun}
The two-parameter generating function of the hook Schur functions satisfies
\[
S(t,u) = \frac{H(t)E(u)-1}{t+u}\, .
\]
\end{lemma}
\begin{proof}
The rearrangement $H(t)E(u) = 1+ (t+u)S(t,u)$ is a generating function
restatement of the characterizing identity
\[
h_i e_j = \mathfrak{s}_{(i|j-1)}+\mathfrak{s}_{(i-1|j)} \, , \quad
(i,j) \in \Z_{\geq 0}\times \Z_{\geq 0}\smallsetminus \{(0,0)\} \, ,
\]
provided we interpret $\mathfrak{s}_{(i|-1)}$ and $\mathfrak{s}_{(-1|j)}$ as zero.
\end{proof}

In particular the algebra automorphism $\Sym(\X)\rightarrow \Sym(\X)$ sending 
$h_k\leftrightarrow e_k$ also sends $S(t,u)\mapsto S(u,t)$, hence $\Schur_{(i|j)}\mapsto \Schur_{(j|i)}$.
The following description of hook Schur functions is also useful.

\begin{lem}\label{lemma:hook rewrite}
For $i,j\geq 0$ we have
\begin{align*}
(-1)^j\Schur_{(i|j)}= \sum_{k+l=j} (-1)^l  h_{i+k+1} e_l= \sum_{k+l=i} (-1)^{l+j}h_{k}  e_{j+l+1}\, .
\end{align*}
\end{lem}
\begin{proof}
We will prove the first identity; the second follows by symmetry (apply the
involution $h_k\leftrightarrow e_k$).  First, note that we have the generating
function identity
\[
\frac{H(t)-H(u)}{t-u} = \sum_{k\geq 0} \frac{t^k-u^k}{t-u} h_{k} = \sum_{i,j\geq 0} h_{i+j+1} t^i u^j.
\]
Then we rewrite the hook Schur generating function as follows
\[
S(t,-u) = \frac{H(t)E(-u)-1}{t-u} = \frac{H(t)-H(u)}{t-u} E(-u) =
\sum_{i,j,k\geq 0} (-1)^k h_{i+j+1}e_k t^{i} u^{j+k},
\]
which gives rise to the identity in the statement.
\end{proof}

\begin{lemma}[$h$-reduction]\label{lemma:h reduction}
If $\Y$ has cardinality $|\Y|\leq c$, then for any $\X$ and $r\geq 1$ we have
\begin{equation}
\label{eq:specialrel2}
h_{c+r}(\X) =  \sum_{0\leq i\leq c} (-1)^{c-i}\Schur_{(r-1|c-i)}(\X+\Y) h_i(\X)\, .
\end{equation}
\end{lemma}
Before proving this lemma, we note the following special cases.

\begin{cor}[Reducing monomials]\label{cor:monomial reduction}
For all $m\geq a \geq 1$ and $1\leq i \leq a$, we have
\[
x_i^{m} =  \sum_{1\leq j\leq a} (-1)^{a-j}\Schur_{(m-a|a-j)}(x_1,\ldots,x_a) x_i^{j-1}.
\]
\end{cor}
\begin{proof}
Take $\X = \{x_i\}$, $\Y=\{x_1,\ldots,\widehat{x_i},\ldots,x_a\}$, and $c=a-1$ in Lemma \ref{lemma:h reduction}.
\end{proof}

\begin{cor}\label{cor:h of difference reduction}
If $\X$ and $\rightX$ are alphabets of cardinality $c$, then we have
\[
h_{c+r}(\leftX - \rightX) =  \sum_{1\leq i\leq c} (-1)^{c-i}\Schur_{(r-1|c-i)}(\X) h_i(\leftX - \rightX)\, .
\]
\end{cor}
\begin{proof}
Take 
$\X\mapsto\leftX - \rightX$ and $\Y\mapsto \rightX$ in Lemma \ref{lemma:h reduction}.  
Note that the $i=0$ summand in \eqref{eq:specialrel2}
is zero in this case since the Young diagram for the hook $(r-1|c)$ has $c+1$ rows, 
which exceeds the cardinality of $\X$.
\end{proof}

\begin{proof}[Proof of Lemma \ref{lemma:h reduction}]
We begin by proving the $\Y= \varnothing$ case. 
Lemma \ref{lemma:hook rewrite} gives us the identity
\[
(-1)^i \Schur_{(r-1|i)}(\X) h_j(\X) = \sum_{k+l=i} (-1)^l h_{r+k}(\X) e_l(\X) h_j(\X)\, .
\]
Summing up over all indices $i,j\geq 0$ with $i+j=c$ yields
\[
\sum_{i+j=c} (-1)^i \Schur_{(r-1|i)}(\X) h_j(\X)  
= \sum_{j+k+l=c} (-1)^l h_{k+r}(\X) e_l(\X) h_j(\X) \stackrel{\eqref{eq:wellknown}}{=} h_{c+r}(\X)\, ,
\]
which is the $|\Y|=0$ case of the lemma.

Now, we deduce the general identity.
First, we compute
\[
h_{c+r}(\X) = h_{c+r}((\X+\Y) - \Y)
=  \sum_{i=0}^c (-1)^i h_{c-i+r}(\X+\Y)  e_i(\Y)\, .
\]
We have used the fact that $e_i(\Y)=0$ for $i > c \geq |\Y|$. 
Next, we apply the $\Y= \varnothing$ case of the lemma to rewrite $h_{c-i+r}(\X+\Y)$. 
The resulting identity is
\[
h_{c+r}(\X) = \sum_{i+j+k=c} (-1)^{k+i} \Schur_{(r-1|k)}(\X+\Y) h_j(\X+\Y) e_i(\Y)\, .
\]
Finally, we fix $k$ and sum over all indices $i,j$ with $i+j=c-k$ to obtain
\[
h_{c+r}(\X) = \sum_{0\leq k\leq c} (-1)^{k} \Schur_{(r-1|k)}(\X+\Y) h_{c-k}(\X)\, .\qedhere
\]
\end{proof}

\subsection{Haiman determinants}
\label{ss:haiman dets}
Our description of the deformed, colored homology of $(m,m)$-torus knots 
(conjectural, when $m>2$) in \S \ref{s:colored Hilb} 
and \S \ref{s:hopf link} relies on certain 
(anti)symmetric polynomials constructed using determinants.

Let $R$ be a commutative $\k$-algebra and consider a tuple $(f_1,\dots,f_N)$ of elements
$f_i\in R$.   Let $f_{i,j}$ denote the element of $R^{\otimes N}$ given by
$
f_{i,j} = 1\otimes \cdots \otimes f_i\otimes\cdots \otimes 1,
$
where $f_i$ occurs in the $j$-th position,
and set 
\begin{equation}\label{eq:Haiman}
\hdet(f_1,\dots,f_N):= \det(f_{i,j})_{1\leq i,j\leq N} :=
	\begin{vmatrix}
	f_{1,1} & \cdots & f_{1,N} \\ 
	\vdots & \ddots & \vdots \\  
	f_{N,1} & \cdots &  f_{N,N}
	\end{vmatrix}
	\in R^{\otimes N}
\end{equation}
Equivalently, $\hdet(f_1,\dots,f_N)$ is the anti-symmetrization of 
$f_1\otimes \cdots \otimes f_N$ inside $R^{\otimes N}$.

\begin{remark}
We will consider this construction in the special cases $R=\k[x]$ and $R=\k[x,y]$.
In such cases, we will identify $R^{\otimes N}$ with the polynomial ring 
$\k[\X]$ or $\k[\X,\Y]$, respectively, where
$\X=\{x_1,\ldots,x_N\}$ and $\Y=\{y_1,\ldots,y_N\}$. 
When the $f_i \in R$ are monic monomials, we refer to the elements constructed in 
\eqref{eq:Haiman} as \emph{Haiman determinants}, 
due to their appearance in \cite[Section 2.2]{Haiman}.
\end{remark}

It will be useful to introduce the following short-hand.

\begin{definition}\label{def:monomial list} 
Let $\l_1\geq \cdots \geq \l_N \geq 0$ be a weakly decreasing sequence of 
non-negative integers of length $N$, then we associate to it 
an $N$-tuple of monomials in $x$ as follows:
\[
\mathcal{M}_N(\l):=(x^{\l_1+N-1},\ldots,x^{\l_{N-1}+1},x^{\l_N}).
\]
\end{definition}

\begin{conv}
Given a positive integers $N \geq l$ and a partition 
$\lambda=(\lambda_1,\ldots,\lambda_l)$ with $l$ parts,  
we will sometimes view $\lambda$ as a weakly decreasing sequence of length $N$ 
by appending to it $N-l$ zeros.
\end{conv}

\begin{example}
We have $\mathcal{M}_N(\emptyset) = (x^{N-1},\ldots,x,1)$ and thus
$\hdet(\mathcal{M}_N(\emptyset)) = \Delta(\X)$, 
where $\Delta(\X) = \prod_{1\leq i<j\leq N} (x_i-x_j)$ is the usual 
\emph{Vandermonde determinant}.  
More generally, for a partition $\lambda=(\lambda_1,\dots, \lambda_l)$ 
with $l \leq N$, Jacobi's bialternant formula implies
\begin{equation}\label{eq:Schur}
\hdet(\mathcal{M}_N(\l)) =  \Delta(\X)	\Schur_{\lambda}(\X),
\end{equation}
where $\Schur_{\lambda}(\X)$ denotes the Schur polynomial associated to
$\lambda$ in the alphabet $\X$ of cardinality $N$.
\end{example}

\begin{definition}\label{def:shapes and dets} 
Let $S$ be a finite set of monic monomials in $R=\k[x,y]$ and set $N:=|S|$. 
We identify $R^{\otimes N}=\k[\X,\Y]$ where $|\X|=N=\Y$
and consider the associated \emph{Haiman determinant}
\[
\Delta_S(\X,\Y) 
: = \hdet(S) \in \k[\X,\Y].
\]
Our convention is to order the monomials in $S$ by writing
$
S = S^0 \cup S^1 y \cup \cdots \cup S^r y^r
$
where each $S^k$ is a tuple of monomials in $x$, written in decreasing order.
\end{definition}

\begin{example}
For $S=\{x^2,x,1,y\}$, the Haiman determinant is
\[
	\Delta_{S}(\X,\Y) = \begin{vmatrix}
	x_1^2 & x_2^2 & x_3^2 & x_4^2 \\ 
	x_1 & x_2 & x_3 & x_4 \\ 
	1 & 1 & 1 & 1 \\ 
	y_1 & y_2 & y_3 & y_4
	\end{vmatrix} .
\]
\end{example}

\section{Categorical background}
In this section we review background on homological algebra and 
singular Soergel bimodules.

\subsection{Categories of coefficients}
\label{ss:notation and cats}
We will assume familiarity with the notion of a category \emph{enriched} in a
 symmetric monoidal category $\RS$. In particular, if $\BS$ is $\RS$-enriched,
 then $\Hom_{\BS}(X,Y)$ is an object of $\RS$ for all $X,Y\in \BS$, and the
 composition of morphisms in $\BS$ is given by morphisms in $\RS$
\[
\Hom_{\BS}(Y,Z)\otimes \Hom_\BS(X,Y)\rightarrow \Hom_\BS(X,Z)
\]
where $\otimes$ is the monoidal structure in $\RS$. 
We view $\RS$ as the \emph{category of coefficients} for $\BS$. 
In this section, we discuss the
various categories of coefficients that will appear in this paper. 

We let $\KS$ denote the category of
finite-dimensional $\k$-vector spaces and let $\overline{\KS}$ denote the
category of all $\k$-vector spaces. 
A category is \emph{$\k$-linear} if it is enriched in $\overline{\KS}$.  

Let $\Gamma$ be an abelian group and let $\KS[\Gamma]$ 
denote the category of finite-dimensional $\Gamma$-graded $\k$-vector spaces. 
An object of this category is a $\k$-vector space of the form 
$\bigoplus_{\gamma\in \Gamma} M_\gamma$ 
where each $M_\gamma$ is finite-dimensional, 
and $M_\gamma=0$ for all but finitely many $\gamma\in \Gamma$.  
Morphism spaces in this category are the $\Gamma$-graded $\k$-vector spaces
\[
\Hom_{\KS[\Gamma]}(M,N) = \bigoplus_{\gamma\in \Gamma} \Hom_{\KS[\Gamma]}^\gamma(M,N) 
\]
where
\[
\Hom_{\KS[\Gamma]}^\gamma(M,N) 
= \prod_{\gamma'\in \Gamma} \Hom_\KS(M_{\gamma'},N_{\gamma'+\gamma}) \, .
\]

Now, suppose $\Gamma$ is equipped with a symmetric bilinear form 
$\langle - , - \rangle \colon \Gamma\times \Gamma \rightarrow \Z/2\Z$.  
This determines a monoidal structure on $\KS[\Gamma]$, 
given on objects by
\[
(M\otimes N)_\gamma = \bigoplus_{\gamma_1+\gamma_2= \gamma} M_{\gamma_1}\otimes N_{\gamma_2}
\]
and morphisms by
\[
(f\otimes g)(m\otimes n) = (-1)^{\langle \deg(g), \deg(m)\rangle} f(m)\otimes g(n).
\]
The monoidal structure just defined is symmetric, with braiding given by
\[
M\otimes N \rightarrow N\otimes M 
\, , \quad 
m\otimes n \mapsto (-1)^{\langle \deg(m),\deg(n) \rangle} n\otimes m \, .
\]

At times, we will relax the finiteness conditions that define $\KS[\Gamma]$, 
and let $\overline{\KS}\llbracket \Gamma \rrbracket$ denote
the category of $\Gamma$-graded $\k$-vector spaces.
This category is again symmetric monoidal, via the same formulae.

Now, if $\Gamma$ is additionally equipped with an element 
$\diffdeg\in \Gamma$ satisfying $\langle\diffdeg,\diffdeg\rangle = 1\in \Z/2\Z$, 
then we can consider the category $\KS[\Gamma]_{\dg}$
of $\Gamma$-graded complexes with differentials of degree $\diffdeg$. 
(Note that we suppress $\langle - , - \rangle$ and $\diffdeg$ from this notation, 
as they will be clear in context.)
Objects in $\KS[\Gamma]_{\dg}$ are pairs $(X,\d)$ where $X\in \KS[\Gamma]$ and $\d\in
\End^{\diffdeg}_{\KS[\Gamma]}(X)$ satisfies $\d^2=0$, 
and morphism spaces are the complexes
\[
\Hom_{\KS[\Gamma]_{\dg}}\Big((X,\d_X),(Y,\d_Y)\Big) = \Hom_{\KS[\Gamma]}(X,Y)
\]
with differential $f\mapsto \d_Y\circ f - (-1)^{|f|}f\circ \d_X$. 
This category is the prototypical differential $\Gamma$-graded category. 
It comes equipped with a tensor product
\[
(X,\d_X)\otimes (Y,\d_Y) = \Big( X\otimes Y\:,\: \d_X\otimes \id_Y + \id_X\otimes \d_Y\Big) \, .
\]
The category $\overline{\KS}\llbracket \Gamma \rrbracket_{\dg}$ is defined similarly. 

\begin{definition}
A $\k$-linear category $\BS$ is \emph{$\Gamma$-graded} if it is enriched in
$\overline{\KS}\llbracket\Gamma\rrbracket$ and 
\emph{differential $\Gamma$-graded} if it is enriched in
$\overline{\KS}\llbracket\Gamma\rrbracket_{\dg}$.
\end{definition}

The most important $\Gamma$ for our uses 
are introduced in the following examples.

\begin{example}\label{ex:Zq}
Let $\Gamma=\Z$ with trivial bilinear form $\langle i,j\rangle =0$.  
We will identify the group algebra $\k[\Z]$ with the algebra of Laurent polynomials $\k[\Z]  = \k[\qdeg^{\pm}]$, 
and similarly indicate the choice of coordinate $\qdeg$ in $\Z$ by writing $\Z=\Z_\qdeg$.  
Paralleling this notation, we denote
\[
\KS[\qdeg^{\pm}]:= \KS[\Z_\qdeg]
\, , \quad
\overline{\KS}\llbracket \qdeg^{\pm} \rrbracket := \overline{\KS}\llbracket\Z_\qdeg\rrbracket.
\]
In both cases, we also let $\qdeg$ denote the grading shift functor, 
defined by $(\qdeg M)_i = M_{i-1}$. 
These category appear when we consider graded rings, modules, and bimodules.
Note that the dg categories $\KS[\Gamma]_{\dg}$ and
$\overline{\KS}\llbracket\Z_\qdeg\rrbracket_{\dg}$ are trivial for this choice of $\langle - , - \rangle$, 
since there is no element $\diffdeg\in \Z$ with $\langle \diffdeg ,\diffdeg \rangle=1$.
\end{example}

\begin{example}\label{ex:Zt}
Let $\Gamma=\Z_\tdeg$ with $\langle j,j'\rangle = jj'$ and $\diffdeg = 1$, 
then $\KS[\Gamma]_{\dg}$ is the usual category of bounded complexes
of finite-dimensional $\k$-vector spaces
(with the cohomological convention for complexes, i.e.~differentials have degree +1).
\end{example}

\begin{example}\label{ex:ZqxZt}
Combining Examples \ref{ex:Zq} and \ref{ex:Zt},
let $\Gamma = \Z_\qdeg\times \Z_\tdeg$ with $\langle (i,j),(i',j')\rangle = jj'$.
On the level of group algebras, we have 
$\k[\Z_\qdeg\times \Z_\tdeg] = \k[\qdeg^{\pm},\tdeg^{\pm}]$, 
and we denote the categories of graded $\k$-vector spaces similarly:
\[
\KS[\qdeg^\pm,\tdeg^\pm]:= \KS[\Z_\qdeg\times \Z_\tdeg]
\, , \quad 
\overline{\KS}\llbracket \qdeg^{\pm},\tdeg^\pm \rrbracket 
	:= \overline{\KS}\llbracket\Z_\qdeg\times \Z_\tdeg\rrbracket \, .
\] 
As above, we regard monomials $\qdeg^i \tdeg^j$ as a grading shift functors, 
via $(\qdeg^i \tdeg^j M)_{k,l} = M_{k-i,l-j}$.
Taking $\diffdeg = (0,1)$ gives the dg categories
\[
\KS[\qdeg^\pm,\tdeg^\pm]_{\dg} := \KS[\Z_\qdeg\times \Z_\tdeg]_{\dg}
\, , \quad 
\overline{\KS}\llbracket \qdeg^{\pm},\tdeg^\pm \rrbracket_{\dg} 
	:= \overline{\KS}\llbracket\Z_\qdeg\times \Z_\tdeg\rrbracket_{\dg}
\]
which appear when we consider complexes of graded modules or 
bimodules over $\Z_\qdeg$-graded rings.
\end{example}

Finally, triply-graded Khovanov--Rozansky homology takes values in the following
symmetric monoidal category of triply-graded $\k$-vector spaces.

\begin{example}\label{ex:ZaxZqxZt}
Let $\Gamma = \Z_\adeg\times \Z_\qdeg\times \Z_\tdeg$. 
As in Examples \ref{ex:Zt} and \ref{ex:ZqxZt}, 
the $\tdeg$-grading is cohomological, so we take $\diffdeg = (0,0,1)$. 
The $\adeg$-grading also has a cohomological flavor, but is \emph{independent} from $\tdeg$. 
This is reflected in our choice of symmetric bilinear form:
\[
\langle (i,j,k), (i',j',k')\rangle = ii' + kk' \, .
\]
The resulting categories of triply-graded $\k$-vector spaces will be denoted
$\KS[\adeg^\pm, \qdeg^\pm, \tdeg^\pm]$ and $\overline{\KS}\llbracket \adeg^{\pm},\qdeg^{\pm},\tdeg^\pm
\rrbracket$, and complexes therein by $\KS[\adeg^\pm, \qdeg^\pm, \tdeg^\pm]_{\dg}$ and
$\overline{\KS}\llbracket \adeg^{\pm},\qdeg^{\pm},\tdeg^\pm \rrbracket_{\dg}$.
These categories occur when we consider (co)homological functors, 
such as Hochschild (co)homology, applied to complexes of graded bimodules.
\end{example}

\begin{conv}\label{conv:wt}
In (differential) $\Gamma$-graded categories for $\Gamma$ as in 
Examples \ref{ex:Zq}, \ref{ex:ZqxZt}, and \ref{ex:ZaxZqxZt}, we will typically 
indicate the degree of a morphism multiplicatively, by indicating its \emph{weight}.
For example, $\wt(f) = \adeg^i \qdeg^j \tdeg^k$ means that the 
$\Z_\adeg\times \Z_\qdeg\times \Z_\tdeg$-degree of $f$ is $(i,j,k)$.
\end{conv}

\subsection{Complexes and curved complexes}
\label{ss:cat setup}

Retaining notation from the previous section, 
the most important differential $\Gamma$-graded categories 
are categories of (curved) complexes.
To begin, suppose $\AS$ is a $\k$-linear category.
Let $\AS[\tdeg^\pm]$ denote the category whose objects are sequences 
$(X_i)_{i\in \Z}$ with $X_i\in \AS$ and $X_i=0$ for all but finitely many $i$, 
and morphisms
\[
\Hom_{\AS[\tdeg^\pm]}(X,Y) := \bigoplus_{k\in \Z}\Hom_{\AS[\tdeg^\pm]}^k(X,Y)
\, , \quad
\Hom_{\AS[\tdeg^\pm]}^k(X,Y) := \prod_{i\in \Z} \Hom_{\AS}(X_i,Y_{i+k}) \, .
\]
In other words $\Hom_{\AS[\tdeg^\pm]}(X,Y)$ is the $\Z$-graded $\k$-vector space spanned
by homogeneous \emph{multimaps} $(f_i)_{i\in\Z}$ with $f_i \in \Hom_{\AS}(X_i,Y_{i+k})$.

The dg category of bounded chain complexes $\CS(\AS)$ can be built from
$\AS[\tdeg^\pm]$ in a standard way. Objects of $\CS(\AS)$ are complexes: pairs
$(X,\d)$ where $X\in \AS[\tdeg^\pm]$ and $\d\in \End_{\AS[\tdeg^\pm]}^{1}(X)$
with $\d^2=0$.  
The morphism spaces in $\CS(\AS)$ are the complexes
\[
\Hom_{\CS(\AS)}(X,Y)=\Hom_{\AS[\tdeg^\pm]}(X,Y)
\, , \quad 
d \colon f\mapsto \d_Y\circ f  - (-1)^{|f|} f\circ \d_X \, .
\]

\begin{remark}
If $\AS$ is $\Gamma$-graded, 
then $\AS[\tdeg^\pm]$ is $\Gamma\times\Z_\tdeg$ graded.
Here, we equip the latter with the $\Z/2\Z$-valued symmetric bilinear form
\[
\langle (\gamma,j),(\gamma',j')\rangle_{\Gamma\times \Z} 
= \langle \gamma,\gamma'\rangle + jj'
\]
and let $\diffdeg = (0,1)$.
In other words, when forming categories of complexes over $\Gamma$-graded categories, 
our differentials have degree $(0,1)\in \Gamma\times \Z_\tdeg$.
In this way, $\CS(\AS)$ is a differential $\Gamma\times \Z_\tdeg$-graded category.
\end{remark}

We will define the category of curved complexes in a similar fashion.  
Informally, the basic idea is to replace the equation $\d^2=0$ with the equation $\d^2=F$, 
where $F$ is an element of the center of $\AS[\tdeg^\pm]$.  
(We will see below that this does not work \emph{sensu stricto}, 
but that there is any easy fix.)

\begin{definition}
The \emph{center} of a $\Z_\tdeg$-graded category $\BS$ is the $\Z$-graded algebra $\cal{Z}
(\BS)$ of endonatural transformations of the identity functor of $\BS$.  
Precisely, a degree $k$ element of $\cal{Z}(\BS)$ is an assignment 
$X\mapsto F|_X\in \End^k_{\BS}(X)$ satisfying 
\emph{super-naturality}: $f\circ F|_X = (-1)^{k|f|} F|_Y\circ f$ for all 
$f\in \Hom_{\BS}(X,Y)$.
\end{definition}

\begin{lemma}
We have $\cal{Z}(\AS[\tdeg^\pm])=\cal{Z}(\AS)$.
\end{lemma}
\begin{proof}
Each object $X\in \AS[\tdeg^\pm]$ comes equipped with a family of idempotent
endomorphisms $e_i$ which project on to the $i$-th object $X_i$. 
Any central element must commute with these $e_i$, hence must be degree zero.  
Now, a degree zero central element $F$ acting on $X$ is completely 
determined by how it acts on $X_0$, 
since such an $F$ must commute with the degree $k$ map relating 
$X$ and its shift $\tdeg^k X$.
\end{proof}

We now run into a slight snag: $\delta^2$ has degree two, yet the only degree two
element of $\cal{Z}(\AS)$ is zero.  In order to remedy this, the two standard
approaches are to collapse the grading from $\Z$ to $\Z/2\Z$ or formally extend
scalars from $\k$ to an appropriate graded ring.  We prefer to preserve the
$\Z$-grading, hence extend scalars as follows.

\begin{definition}
	\label{def:Adjoin} 
If $\BS$ is $\Z_\tdeg$-graded category and $R$ is a $\Z_\tdeg$-graded ring, 
then we let $\BS\otimes R$ denote the $\Z_\tdeg$-graded category with 
the same objects as $\BS$, and morphisms
\[
\Hom_{\BS\otimes R}(X,Y):=\Hom_{\BS}(X,Y)\otimes R \, .
\]
\end{definition}

\begin{remark}
The center of $\BS\otimes R$ is isomorphic to $\cal{Z}(\BS)\otimes R$.
In particular, the center of $\AS[\tdeg^\pm]\otimes R$ is isomorphic to
$\cal{Z}(\AS)\otimes R$.
\end{remark}

\begin{definition}[Curved complexes]\label{def:curved cx} Let $\AS$ be a
$\k$-linear category, $R$ a $\Z_\tdeg$-graded ring, and $F$ a degree two element of
$\cal{Z}(\AS)\otimes R$.   Let $\CS_F(\AS;R)$ denote the dg category whose objects
are pairs $(X,\d)$ with $X\in \AS[\tdeg^\pm]$ and 
$\d\in \End_{\AS[\tdeg^\pm]\otimes R}^{1}(X)$ satisfying $\d^2 = F|_X$. 
Morphism spaces in $\CS_F(\AS;R)$ are the complexes
\[
\Hom_{\CS_F(\AS)}(X,Y) := \Hom_{\AS[\tdeg^\pm] \otimes R}(X,Y)
\, , \quad
d \colon f\mapsto \d_Y\circ f  - (-1)^{|f|} f\circ \d_X.
\]
\end{definition}
When the ring $R$ is clear from context, we will omit it from the notation and 
denote the category $\CS_F(\AS;R)$ simply by $\CS_F(\AS)$.

\begin{rem}
If $\AS$ is $\Gamma$-graded, then both $\CS(\AS) = \CS_{0}(\AS)$ and
$\CS_F(\AS)$ are dg $\Gamma \times \Z_\tdeg$-graded categories.
\end{rem}

\subsection{Frobenius extensions}
\label{sec:Frobenius}
Frobenius extensions between rings of partially symmetric polynomial rings will
play an important role in defining morphisms between singular Soergel bimodules.
\begin{definition}
	A \emph{Frobenius extension} is an inclusion of commutative rings
	$\iota\colon A\hookrightarrow B$ such that $B$ is free and finitely
	generated as an $A$-module, together with a non-degenerate $A$-linear map
	$\partial\colon B\to A$, called the \emph{trace}. 
	Here, \emph{non-degeneracy} asserts the existence of $A$-linear 
	\emph{dual bases} $\{x_\alpha\}$ and $\{x'_\alpha\}$ for $B$ such that 
	$\partial(x_\alpha x'_\beta)=\delta_{\alpha, \beta}$. 
	For a \emph{graded Frobenius extension} between graded rings, 
	we require $\iota$ to be grading preserving and
	$\partial$ and the dual bases to be homogeneous.
   \end{definition}

Fix $N>0$, and let $R:=\k[x_1,\ldots,x_N]$ be the polynomial ring,
$\Z_\qdeg$-graded by declaring $\deg_\qdeg(x_i)=2$. Given a parabolic
subgroup $\symg_{\aa} = \symg_{a_1}\times \cdots \times
\symg_{a_m}$ of the symmetric group $\symg_N$, we let
$R^{\aa} \subseteq R$ denote the ring of polynomials invariant under the
action of $\symg_{\aa}$. Note that $R^{\bb}\subset R^{\aa}$ if
and only if $\symg_{\bb} \supset \symg_{\aa}$.

\begin{lem} 
If $\symg_{\bb} \supset \symg_{\aa}$, then
$R^{\bb}\hookrightarrow R^{\aa}$ is a graded Frobenius extension
of rank $|\symg_{\bb}/\symg_{\aa}|$. To describe the
trace, let $w_{\bb}$ denote the longest element of
$\symg_{\bb}$ and $w_{\bb/\aa}=s_{i_1}\cdots
s_{i_k}$ a minimal length coset representative for $w_{\bb}
\symg_{\aa}$. Then we have:
\[
\partial\colon R^{\aa} \to R^{\bb} 
\, , \quad 
f\mapsto \partial_{i_1}\cdots \partial_{i_k}f
\, , \quad \text{where}\quad
\partial_{i}(g)=\frac{g-s_i(g)}{x_{i}-x_{i+1}} \, .
\] 
Here, the transposition $s_i$
acts by swapping variables $x_i$ and $x_{i+1}$ in $g$.
\end{lem}
\begin{proof} See \cite[Theorem 3.1.1 and Corollary 2.14]{Williamson-thesis}.
\end{proof}

More specifically, the following two examples will be used throughout.

\begin{exa}\label{exa:Demazure}
	Let $\X=\{x_1,\dots, x_N\}$ be an alphabet with $\deg_\qdeg(x_i) = 2$.
Then $\Sym(\X)\hookrightarrow \k[\X]$ is a graded Frobenius extension of rank
$N!$ with non-degenerate trace given by:
\[
\k[\X] \ni f \mapsto 
(\partial_{1}\cdots \partial_{N-1}) \cdots (\partial_{1}\partial_{2})\partial_{1} f
= \frac{\mathrm{Alt}(f)}{\Delta(\X)} \in \Sym(\X) \, .
\] 
Here, $\mathrm{Alt}$ denotes the antisymmetrizer 
$\sum_{\sigma\in \symg_a} (-1)^\sigma \sigma \in \k[\symg_N]$ 
acting on $\k[\X]$ by the permutation of variables and 
$\Delta(\X)=\prod_{1\leq i<j\leq N}(x_i-x_j)$ denotes the Vandermonde determinant. 
A $\Sym(\X)$-linear basis of $\k[\X]$ is
given by the monomials $x_1^{n_1}x_2^{n_2}\cdots x_{N-1}^{n_{N-1}}$ where 
$0\leq n_i\leq N-i$.
We have 
\[
\partial(x_1^{n_1}x_2^{n_2}\cdots x_{N-1}^{n_{N-1}}) =
\begin{cases} 1 &\text{ if } n_i=N-i \text{ for } 1\leq i \leq N,\\
0 & \text{ otherwise.}
\end{cases}
\]
The basis dual to the monomial basis above has elements 
$\prod_{k=1}^{N-1} (-1)^{b_k}e_{b_k}(x_{N+1-k},\dots, x_{N})$ where
$b_k=k-n_{N-k}$ for $1\leq k\leq N-1$.
\end{exa}

\begin{exa}\label{exa:Sylvester}
Let $\X_1=\{x_1,\dots,x_a\}$ and $\X_2=\{x_{a+1},\dots, x_{a+b}\}$ 
be alphabets with $|\X_1|=a$ and $|\X_2|=b$. 
Then $\Sym(\X_1+\X_2)\hookrightarrow \Sym(\X_1|\X_2)$ is a graded Frobenius extension
of rank $ab$ with trace given by the \emph{Sylvester operator}:
\[
\partial_{a,b}\colon \Sym(\X_1|\X_2) \ni f \mapsto 
(\partial_{b}\cdots \partial_{1}) \cdots (\partial_{a+b-1} \cdots \partial_a) f \in \Sym(\X_1+\X_2) \, .
\]
A $\Sym(\X_1+\X_2)$-linear basis of $\Sym(\X_1|\X_2)$ is given by the Schur
functions $\Schur_{\lambda}(\X_1)$ indexed by partitions $\lambda$ with
$\lambda_1\leq b$ having at most $a$ parts 
(i.e. the Young diagram for $\lambda$ fits inside the $a\times b$ box). 
We denote the set of such partitions by $P(a,b)$. 
The dual basis is then given by the signed Schur functions
$(-1)^{|\hat{\lambda}|}\Schur_{\hat{\lambda}}(\X_2)$ where 
$\hat{\lambda}\in P(b,a)$ denotes the dual complementary partition.
\end{exa}

\subsection{Singular Soergel bimodules and webs} 
\label{ss:ssbim}
The $\k$-linear monoidal 2-category of singular Soergel bimodules can be
assembled from certain full subcategories of the categories of graded bimodules between
rings of partially symmetric polynomials over $\k$ in variables of degree two.
These subcategories are generated by certain induction and restriction bimodules under
bimodule composition $\hComp$, direct sum, taking direct summands, and grading
shifts. We now unpack this description, using the notation established in
\S\ref{sec:Frobenius}.

Consider the 2-category $\mathrm{Bim}_N$ wherein
\begin{itemize}
	\item Objects are tuples $\aa = (a_1,\ldots,a_m)$ 
		with $a_i \geq 1$ and $\sum_{i=1}^m a_i = N$.
	\item $1$-morphisms $\aa \to \bb$ are graded
		$(R^{\bb},R^{\aa})$-bimodules.
	\item $2$-morphisms are homomorphisms of graded bimodules.
\end{itemize}
Horizontal composition is given by tensor product over the rings
$R^{\aa}$, and will be denoted by $\hComp$. 
The composition of 2-morphisms is the usual composition of bimodule homomorphisms. 
We will write $\oone_{\aa} :=R^{\aa}$ for the identity bimodule, 
saving the notation $R^{\aa}$ for the rings themselves. 
We denote the number of entries in $\aa$ by $\#(\aa)$, 
i.e. for $\aa=(a_1,\ldots,a_m)$ we have $\#(\aa)=m$.
The collection $\Bim := \{\Bim_N\}_{N \geq 0}$ 
is a monoidal $2$-category, with (external) tensor product
\[
\boxtimes\colon  \Bim_{N_1} \times \Bim_{N_2} \to \Bim_{N_1+N_2}
\]
given on objects by concatenation of tuples:
$
(a_1,\ldots,a_{m_1})\boxtimes (b_1,\ldots,b_{m_2}) \ := \ (a_1,\ldots,a_{m_1}, b_1,\ldots,b_{m_2}).
$
and on $1$- and $2$-morphisms by tensor product over $\k$. 

A \emph{singular Bott-Samelson bimodule} is, 
by definition, 
any $(R^{\aa_0},R^{\aa_r})$-bimodule of the form
\[
B \ = \ R^{\aa_0}\otimes_{R^{\bb_1}} R^{\aa_1} \otimes_{R^{\bb_2}}
\cdots  \otimes_{R^{\bb_r}} R^{\aa_r}
\]
for some sequence of rings and subrings $R^{\aa_0}\supset R^{\bb_1}\subset
\cdots \supset R^{\bb_r} \subset R^{\aa_r}$, 
or a grading shift thereof.
In particular, we have the \emph{merge} and \emph{split bimodules} 
(terminology explained below) given by
\begin{equation}\label{eq:MergeSplit}
{}_{\bb}M_{\aa} := 
\qdeg^{\ell(\aa) - \ell(\bb)}
R^{\bb}  \otimes_{R^{\bb}} R^{\aa}
\, , \quad
{}_{\aa}S_{\bb} := R^{\aa} \otimes_{R^{\bb}} R^{\bb},
\end{equation}
whenever $R^{\bb}\subset R^{\aa}$ 
(equivalently $\symg_{\bb} \supset \symg_{\aa}$).  
Here, $\qdeg^k$ denotes a shift up in degree by $k$, 
and $\ell(\aa)$ denotes the length of the longest element in $\symg_{\aa}$.
Any singular Bott-Samelson bimodule is
isomorphic to a shift of the horizontal composition 
of some sequence of merges and splits.

\begin{defi}
The 2-category $\SSBim_N$ of singular Soergel bimodules is the smallest full
$2$-subcategory of $\mathrm{Bim}_N$ that contains 
the singular Bott-Samelson bimodules 
and is closed under taking shifts, direct sums, and direct summands.
The collection $\{\SSBim_N \}_{N \geq 0}$ assembles to form a full monoidal 2-subcategory of 
$\Bim$ that we denote $\SSBim$. 
We will denote the $1$-morphism category in $\SSBim$ from $\aa$ to $\bb$ 
by ${}_{\bb}\SSBim_{\aa}$.
\end{defi}

If $\aa = (a,b)$ and $\aa' = (a+b)$, 
then we represent the corresponding 
merge and split bimodules by diagrams:
\begin{equation}\label{eq:GenWeb}
_{\aa'}M_{\aa} = 
\begin{tikzpicture}[scale =.5, smallnodes, rotate=90,anchorbase]
	\draw[very thick] (0,0) node[right,xshift=-2pt]{$a$} to [out=90,in=210] (.5,.75);
	\draw[very thick] (1,0) node[right,xshift=-2pt]{$b$} to [out=90,in=330] (.5,.75);
	\draw[very thick] (.5,.75) to (.5,1.5) node[left,xshift=2pt]{$a{+}b$};
\end{tikzpicture}
\;\; \text{ and } \;\;
_{\aa}S_{\aa'} =
\begin{tikzpicture}[scale =.5, smallnodes,rotate=270,anchorbase]
	\draw[very thick] (0,0) node[left,xshift=2pt]{$b$} to [out=90,in=210] (.5,.75);
	\draw[very thick] (1,0) node[left,xshift=2pt]{$a$} to [out=90,in=330] (.5,.75);
	\draw[very thick] (.5,.75) to (.5,1.5) node[right,xshift=-2pt]{$a{+}b$};
\end{tikzpicture}.
\end{equation}
All other singular Bott-Samelson bimodules can be obtained from these 
merge and split bimodules using grading shift, 
horizontal composition $\hComp$, and tensor product $\boxtimes$.  
Graphically, $\hComp$ corresponds to the
horizontal glueing of diagrams along a common boundary, and $\boxtimes$
corresponds to vertical stacking of diagrams. We will refer to the graphs built
from the diagrams in \eqref{eq:GenWeb} via $\hComp$ and $\boxtimes$ as
\emph{webs}, which we always understand as mapping from the labels at their
right endpoints to those at their left.

All maps between singular Bott-Samelson bimodules can be built using $\hComp$ and
$\boxtimes$ from the following elemental maps\footnote{This follows e.g. from the results in \cite{WebSchur}, 
which combine with \cite{QR} to show that the $n \to \infty$ limit of 
the $\gln$ foam $2$-category defined in the latter is equivalent to the 
$2$-category of singular Bott-Samelson bimodules.} (which encode the Frobenius
extension structures discussed in \S\ref{sec:Frobenius}):
\begin{enumerate}
\item \emph{Decoration endomorphisms} 
\[
R^{\aa} =
\begin{tikzpicture}[scale=.35,smallnodes,rotate=90,anchorbase]
	\draw[very thick] (1,-1) node[right=-2pt]{$a$} to (1,1);
\end{tikzpicture}
\xrightarrow{
\begin{tikzpicture}[scale=.35,smallnodes,rotate=90,anchorbase]
	\draw[very thick] (1,-1) to node{$\bullet$} node[above=0pt]{$f$} (1,1);
\end{tikzpicture}}
\begin{tikzpicture}[scale=.35,smallnodes,rotate=90,anchorbase]
	\draw[very thick] (1,-1) node[right=-2pt]{$a$} to (1,1);
\end{tikzpicture}
= R^{\aa}
\, , \quad
1 \mapsto f
\]
for $f \in R^{\aa} = \Sym(\X)$.

\item \emph{Digon creation} morphisms
\[
R^{a+b} =
\begin{tikzpicture}[scale=.35,smallnodes,rotate=90,anchorbase]
	\draw[very thick] (1,-1) node[right=-2pt]{$a{+}b$} to (1,1);
\end{tikzpicture}
\xrightarrow{ \ \cre \ }
\begin{tikzpicture}[scale=.35,smallnodes,rotate=90,anchorbase]
	\draw[very thick] (1,1) to (1,2);
	\draw[very thick] (1,-2) to (1,-1); 
	\draw[very thick] (1,-1) to [out=150,in=210] node[below=-2pt]{$a$} (1,1);
	\draw[very thick] (1,-1) to [out=30,in=330] node[above=-2pt]{$b$} (1,1); 
\end{tikzpicture}
= R^{(a,b)}
\, , \quad
1 \mapsto 1
\]
of weight $\qdeg^{-ab}$.

\item \emph{Digon collapse} morphisms:
\[
R^{(a,b)} =
\begin{tikzpicture}[scale=.35,smallnodes,rotate=90,anchorbase]
	\draw[very thick] (1,1) to (1,2);
	\draw[very thick] (1,-2) to (1,-1); 
	\draw[very thick] (1,-1) to [out=150,in=210] node[below=-2pt]{$a$} (1,1);
	\draw[very thick] (1,-1) to [out=30,in=330] node[above=-2pt]{$b$} (1,1); 
\end{tikzpicture}
\xrightarrow{ \ \col \ }
\begin{tikzpicture}[scale=.35,smallnodes,rotate=90,anchorbase]
	\draw[very thick] (1,-1) node[right=-2pt]{$a{+}b$} to (1,1);
\end{tikzpicture}
= R^{a+b}
\, , \quad
f \mapsto \partial_{a,b}(f)
\]
of weight $\qdeg^{-ab}$.
Here $\partial_{a,b}$ is the Sylvester operator from
Example~\ref{exa:Sylvester}.

\item \emph{Zip} morphisms:
\[
R^{(a,b)}=
\begin{tikzpicture}[scale=.35,smallnodes,anchorbase,rotate=90]
	\draw[very thick] (0,-1.25) node[right=-2pt]{$a$} to (0,1.25);
	\draw[very thick] (1,-1.25) node[right=-2pt]{$b$} to (1,1.25);
\end{tikzpicture}
\xrightarrow{\zip}
\begin{tikzpicture}[scale=.35,smallnodes,anchorbase,rotate=90]
	\draw[very thick] (0,-1.5) node[right=-2pt]{$a$} to [out=90,in=210] (.5,-.5);
	\draw[very thick] (1,-1.5) node[right=-2pt]{$b$} to [out=90,in=330] (.5,-.5);
	\draw[very thick] (.5,-.5) to (.5,.5);
	\draw[very thick] (.5,.5) to [out=150,in=270] (0,1.5) node[left=-2pt]{$a$};
	\draw[very thick] (.5,.5) to [out=30,in=270] (1,1.5) node[left=-2pt]{$b$};
\end{tikzpicture}
\, , \quad
1 \mapsto \Schur_{b^a}(\leftX_1 - \rightX_2)
\]
of weight $\qdeg^{ab}$. Here, the latter is viewed as an element in 
$\Sym(\leftX_1 | \leftX_2) \otimes_{R^{a+b}} \Sym(\rightX_1 | \rightX_2)$.

\item \emph{Un-zip} morphisms:
\[
\begin{tikzpicture}[scale=.35,smallnodes,anchorbase,rotate=90]
	\draw[very thick] (0,-1.5) node[right=-2pt]{$a$} to [out=90,in=210] (.5,-.5);
	\draw[very thick] (1,-1.5) node[right=-2pt]{$b$} to [out=90,in=330] (.5,-.5);
	\draw[very thick] (.5,-.5) to (.5,.5);
	\draw[very thick] (.5,.5) to [out=150,in=270] (0,1.5) node[left=-2pt]{$a$};
	\draw[very thick] (.5,.5) to [out=30,in=270] (1,1.5) node[left=-2pt]{$b$};
\end{tikzpicture}
\xrightarrow{\un}
\begin{tikzpicture}[scale=.35,smallnodes,anchorbase,rotate=90]
	\draw[very thick] (0,-1.25) node[right=-2pt]{$a$} to (0,1.25);
	\draw[very thick] (1,-1.25) node[right=-2pt]{$b$} to (1,1.25);
\end{tikzpicture}
= R^{(a,b)}
\, , \quad
f \otimes g \mapsto fg
\]
of weight $\qdeg^{ab}$. 
\end{enumerate}
In the cases (2)-(5), the degree/weight of the morphism is determined by the
shift present in the definition of the merge bimodule in \eqref{eq:MergeSplit}.

The following webs will play an important role in the following:
\begin{equation}
	\label{eq:Rickardweb}
	\F^{(l)} \E^{(k)}\oone_{a,b}
:=
\begin{tikzpicture}[rotate=90,anchorbase,smallnodes]
	\draw[very thick] (0,.25) to [out=150,in=270] (-.25,1) 
		node[left,xshift=2pt]{$a{+}k{-}l$};
	\draw[very thick] (.5,.5) to (.5,1) node[left,xshift=2pt]{$b{-}k{+}l$};
	\draw[very thick] (0,.25) to node[left,yshift=-1pt]{$l$} (.5,.5);
	\draw[very thick] (0,-.25) to node[below,yshift=-1pt]{$a{+}k$} (0,.25);
	\draw[very thick] (.5,-.5) to [out=30,in=330] 
		node[above,yshift=-2pt]{$b{-}k$} (.5,.5);
	\draw[very thick] (0,-.25) to node[right,yshift=-1pt,xshift=-1pt]{$k$} (.5,-.5);
	\draw[very thick] (.5,-1) node[right,xshift=-2pt]{$b$} to (.5,-.5);
	\draw[very thick] (-.25,-1)node[right,xshift=-2pt]{$a$} 
		to [out=90,in=210] (0,-.25);
\end{tikzpicture}
=
\begin{tikzpicture}[rotate=90, anchorbase, smallnodes]
	\draw[very thick] (0,.25) to [out=150,in=270] (-.25,1) node[left=-2pt]{$\leftX_1$};
	\draw[very thick] (.5,.5) to (.5,1) node[left=-2pt]{$\leftX_2$};
	\draw[very thick] (0,.25) to node[left=-1pt,yshift=-2pt]{$\M$} (.5,.5);
	\draw[very thick] (0,-.25) to node[below,yshift=0pt]{$\Fr$} (0,.25);
	\draw[very thick] (.5,-.5) to [out=30,in=330] node[above,yshift=-2pt]{$\B$} (.5,.5);
	\draw[very thick] (0,-.25) to node[right,yshift=-2pt]{$\M'$} (.5,-.5);
	\draw[very thick] (.5,-1) node[right=-2pt]{$\rightX_2$} to (.5,-.5);
	\draw[very thick] (-.25,-1)node[right=-2pt]{$\rightX_1$} to [out=90,in=210] (0,-.25);
\end{tikzpicture}
\end{equation}
In the second diagram, we establish conventions for the alphabets associated with each edge; 
they have cardinalities as given by the corresponding labels in the first diagram.
This will aid in specifying decoration endomorphisms of the corresponding bimodules.
The notation $\F^{(l)} \E^{(k)}\oone_{a,b}$ is borrowed from the theory of the
categorified quantum group for $\slnn{2}$, whose \emph{extended graphical calculus}
\cite{KLMS} can also be used to encode 2-morphisms between certain singular
Bott--Samelson bimodules, see \cite[\HRWCQGSSBim]{HRW1}. 

For example, we will use morphisms:  
\begin{equation}\label{eq:chi plus}
	\chi_r^+ :=
	\CQGsgn{(-1)^{b-k}}
	\begin{tikzpicture}[baseline=0,smallnodes]
	\draw[CQG,ultra thick,<-] (0,-.5) node[below]{\scriptsize$l$}  to (0,.7);
	\draw[CQG,ultra thick,->] (.75,-.5) node[below]{\scriptsize$k$} to (.75,.7);
	\draw[CQG,thick, directed=.75] (.75,0) to [out=90,in=90] 
		node[black,yshift=-.5pt]{$\bullet$} node[below,black]{\scriptsize$r$} (0,0);
	\end{tikzpicture}
	\colon
	\begin{tikzpicture}[rotate=90,anchorbase,smallnodes]
		\draw[very thick] (0,.25) to [out=150,in=270] (-.25,1) 
			node[left,xshift=2pt]{$a{+}k{-}l$};
		\draw[very thick] (.5,.5) to (.5,1) node[left,xshift=2pt]{$b{-}k{+}l$};
		\draw[very thick] (0,.25) to node[left,yshift=-1pt]{$l$} (.5,.5);
		\draw[very thick] (0,-.25) to node[below,yshift=-1pt]{$a{+}k$} (0,.25);
		\draw[very thick] (.5,-.5) to [out=30,in=330] 
			node[above,yshift=-2pt]{$b{-}k$} (.5,.5);
		\draw[very thick] (0,-.25) to node[right,yshift=-1pt,xshift=-1pt]{$k$} (.5,-.5);
		\draw[very thick] (.5,-1) node[right,xshift=-2pt]{$b$} to (.5,-.5);
		\draw[very thick] (-.25,-1)node[right,xshift=-2pt]{$a$} 
			to [out=90,in=210] (0,-.25);
	\end{tikzpicture}
	\longrightarrow
	\begin{tikzpicture}[rotate=90,anchorbase,smallnodes]
		\draw[very thick] (0,.25) to [out=150,in=270] (-.25,1) 
			node[left,xshift=2pt]{$a{+}k{-}l$};
		\draw[very thick] (.5,.5) to (.5,1) node[left,xshift=2pt]{$b{-}k{+}l$};
		\draw[very thick] (0,.25) to node[left,yshift=-1pt]{$l{-}1$} (.5,.5);
		\draw[very thick] (0,-.25) to node[below,yshift=-3pt]{$a{+}k{-}1$} (0,.25);
		\draw[very thick] (.5,-.5) to [out=30,in=330] 
			node[above,yshift=-2pt]{$b{-}k{+}1$} (.5,.5);
		\draw[very thick] (0,-.25) to node[right,yshift=-1pt,xshift=-1pt]{$k{-}1$} (.5,-.5);
		\draw[very thick] (.5,-1) node[right,xshift=-2pt]{$b$} to (.5,-.5);
		\draw[very thick] (-.25,-1)node[right,xshift=-2pt]{$a$} 
			to [out=90,in=210] (0,-.25);
	\end{tikzpicture}
	\end{equation}
	\begin{equation}\label{eq:chi minus}
	\chi_r^- :=
	\CQGsgn{(-1)^{a+b+k+l-1}}
	\begin{tikzpicture}[baseline=0,smallnodes,yscale=-1]
	\draw[CQG,ultra thick,->] (0,-.5) to (0,.7) node[below]{\scriptsize$l$};
	\draw[CQG,ultra thick,<-] (.75,-.5) to (.75,.7) node[below]{\scriptsize$k$};
	\draw[CQG,thick, rdirected=.3] (.75,0) to [out=90,in=90] 
		node[black,yshift=-.5pt]{$\bullet$} node[below,black]{\scriptsize$r$} (0,0);
	\end{tikzpicture}
	\colon
	\begin{tikzpicture}[rotate=90,anchorbase,smallnodes]
		\draw[very thick] (0,.25) to [out=150,in=270] (-.25,1) 
			node[left,xshift=2pt]{$a{+}k{-}l$};
		\draw[very thick] (.5,.5) to (.5,1) node[left,xshift=2pt]{$b{-}k{+}l$};
		\draw[very thick] (0,.25) to node[left,yshift=-1pt]{$l$} (.5,.5);
		\draw[very thick] (0,-.25) to node[below,yshift=-1pt]{$a{+}k$} (0,.25);
		\draw[very thick] (.5,-.5) to [out=30,in=330] 
			node[above,yshift=-2pt]{$b{-}k$} (.5,.5);
		\draw[very thick] (0,-.25) to node[right,yshift=-1pt,xshift=-1pt]{$k$} (.5,-.5);
		\draw[very thick] (.5,-1) node[right,xshift=-2pt]{$b$} to (.5,-.5);
		\draw[very thick] (-.25,-1)node[right,xshift=-2pt]{$a$} 
			to [out=90,in=210] (0,-.25);
	\end{tikzpicture}
	\longrightarrow
	\begin{tikzpicture}[rotate=90,anchorbase,smallnodes]
		\draw[very thick] (0,.25) to [out=150,in=270] (-.25,1) 
			node[left,xshift=2pt]{$a{+}k{-}l$};
		\draw[very thick] (.5,.5) to (.5,1) node[left,xshift=2pt]{$b{-}k{+}l$};
		\draw[very thick] (0,.25) to node[left,yshift=-1pt]{$l{+}1$} (.5,.5);
		\draw[very thick] (0,-.25) to node[below,yshift=-3pt]{$a{+}k{+}1$} (0,.25);
		\draw[very thick] (.5,-.5) to [out=30,in=330] 
			node[above,yshift=-2pt]{$b{-}k{-}1$} (.5,.5);
		\draw[very thick] (0,-.25) to node[right,yshift=-1pt,xshift=-1pt]{$k{+}1$} (.5,-.5);
		\draw[very thick] (.5,-1) node[right,xshift=-2pt]{$b$} to (.5,-.5);
		\draw[very thick] (-.25,-1)node[right,xshift=-2pt]{$a$} 
			to [out=90,in=210] (0,-.25);
	\end{tikzpicture}
	\end{equation}
each of which has degree $a-b+k-l+1+2r$.
Here, the signed\footnote{The indicated signs are required to give a well-defined $2$-functor 
from the categorified quantum group in \cite{KLMS} to the 2-category of singular Soergel bimodules. 
We will always depict the signs in \CQGsgn{green} for signed diagrams that 
are sent to the ``na\"{i}ve'' (unsigned) movie of webs in the image.} 
extended graphical calculus diagrams can be interpreted as encoding 
a \emph{horizontal} (dashed) slice through a movie of webs that describes the morphism of singular Soergel bimodules:
\[
\begin{tikzpicture}[rotate=90,anchorbase,tinynodes]
	\draw[very thick] (0,.25) to [out=150,in=270] (-.25,.875);
	\draw[very thick] (.5,.5) to (.5,.875);
	\draw[very thick] (0,.25) to (.5,.5);
	\draw[very thick] (0,-.25) to (0,.25);
	\draw[very thick] (.5,-.5) to [out=30,in=330] (.5,.5);
	\draw[very thick] (0,-.25) to (.5,-.5);
	\draw[very thick] (.5,-.875) to (.5,-.5);
	\draw[very thick] (-.25,-.875) to [out=90,in=210] (0,-.25);
	\draw[densely dashed, CQG] (.25,-.875) to (.25,.875);
\end{tikzpicture}
\ \xleftrightarrow[\col \hComp \col]{\cre \hComp \cre} \
\begin{tikzpicture}[rotate=90,anchorbase,tinynodes]
	\draw[very thick] (0,.25) to [out=150,in=270] (-.25,.875);
	\draw[very thick] (.5,.5) to (.5,.875);
	\draw[very thick] (0,.25) to (.5,.5);
	\draw[very thick] (.125,.3125) to [out=300,in=225] (.375,.25) 
		node[above,yshift=-2pt]{$1$} to [out=45,in=300] (.375,.4375);
	\draw[very thick] (0,-.25) to (0,.25);
	\draw[very thick] (.5,-.5) to [out=30,in=330] (.5,.5);
	\draw[very thick] (.125,-.3125) to [out=60,in=135] (.375,-.25) 
		node[above,yshift=-2pt]{$1$} to [out=315,in=60] (.375,-.4375);	
	\draw[very thick] (0,-.25) to (.5,-.5);
	\draw[very thick] (.5,-.875) to (.5,-.5);
	\draw[very thick] (-.25,-.875) to [out=90,in=210] (0,-.25);
	\draw[densely dashed, CQG] (.25,-.875) to (.25,.875);
\end{tikzpicture}
\ \cong \
\begin{tikzpicture}[rotate=90,anchorbase,tinynodes]
	\draw[very thick] (0,.25) to [out=150,in=270] (-.25,.875);
	\draw[very thick] (.5,.5) to (.5,.875);
	\draw[very thick] (0,.25) to (.5,.5);	
	\draw[very thick] (0,.125) to node[above,yshift=-2pt,xshift=2pt]{$1$} (.625,.375);
	\draw[very thick] (0,-.25) to (0,.25);
	\draw[very thick] (.5,-.5) to [out=30,in=330] (.5,.5);
	\draw[very thick] (0,-.125) to node[above,yshift=-2pt,xshift=-2pt]{$1$} (.625,-.375);
	\draw[very thick] (0,-.25) to (.5,-.5);
	\draw[very thick] (.5,-.875) to (.5,-.5);
	\draw[very thick] (-.25,-.875) to [out=90,in=210] (0,-.25);
	\draw[densely dashed, CQG] (.25,-.875) to (.25,.875);
\end{tikzpicture}
\ \xleftrightarrow[\zip]{\un} \
\begin{tikzpicture}[rotate=90,anchorbase,tinynodes]
	\draw[very thick] (0,.25) to [out=150,in=270] (-.25,.875);
	\draw[very thick] (.5,.5) to (.5,.875);
	\draw[very thick] (0,.25) to (.5,.5);	
	\draw[very thick] (0,-.25) to (0,.25);
	\draw[very thick] (.5,-.5) to [out=30,in=330] (.5,.5);
	\draw[very thick] (.625,.375) to [out=210,in=90] 
		node[above]{$1$} (.125,0) to [out=270,in=150] (.625,-.375);
	\draw[very thick] (0,-.25) to (.5,-.5);
	\draw[very thick] (.5,-.875) to (.5,-.5);
	\draw[very thick] (-.25,-.875) to [out=90,in=210] (0,-.25);
	\draw[densely dashed, CQG] (.25,-.875) to (.25,.875);
\end{tikzpicture}
\ \xleftrightarrow[x^r]{x^r} \
\begin{tikzpicture}[rotate=90,anchorbase,tinynodes]
	\draw[very thick] (0,.25) to [out=150,in=270] (-.25,.875);
	\draw[very thick] (.5,.5) to (.5,.875);
	\draw[very thick] (0,.25) to (.5,.5);	
	\draw[very thick] (0,-.25) to (0,.25);
	\draw[very thick] (.5,-.5) to [out=30,in=330] (.5,.5);
	\draw[very thick] (.625,.375) to [out=210,in=90] 
		node[above]{$1$} (.125,0) to [out=270,in=150] (.625,-.375);
	\draw[very thick] (0,-.25) to (.5,-.5);
	\draw[very thick] (.5,-.875) to (.5,-.5);
	\draw[very thick] (-.25,-.875) to [out=90,in=210] (0,-.25);
	\draw[densely dashed, CQG] (.25,-.875) to (.25,.875);
\end{tikzpicture}
\ \xleftrightarrow[\;\cre\;]{\col} \
\begin{tikzpicture}[rotate=90,anchorbase,tinynodes]
	\draw[very thick] (0,.25) to [out=150,in=270] (-.25,.875);
	\draw[very thick] (.5,.5) to (.5,.875);
	\draw[very thick] (0,.25) to (.5,.5);
	\draw[very thick] (0,-.25) to (0,.25);
	\draw[very thick] (.5,-.5) to [out=30,in=330] (.5,.5);
	\draw[very thick] (0,-.25) to (.5,-.5);
	\draw[very thick] (.5,-.875) to (.5,-.5);
	\draw[very thick] (-.25,-.875) to [out=90,in=210] (0,-.25);
	\draw[densely dashed, CQG] (.25,-.875) to (.25,.875);
\end{tikzpicture}
\]
The morphism \eqref{eq:chi plus} is given by reading left-to-right with the top arrow labels 
and \eqref{eq:chi minus} is given by reading right-to-left with the bottom arrow labels.
The variable $x$ is associated with the $1$-labeled edge, and 
the extended graphical calculus diagram encodes the intersection with the dashed slice.

We will occasionally wish to encode such morphisms using 
\emph{perpendicular graphical calculus}, 
which corresponds instead to taking a \emph{vertical} slice, \eg 
\[
\begin{tikzpicture}[rotate=90,anchorbase,tinynodes]
	\draw[very thick] (0,.25) to [out=150,in=270] (-.25,.875);
	\draw[very thick] (.5,.5) to (.5,.875);
	\draw[very thick] (0,.25) to (.5,.5);
	\draw[very thick] (0,-.25) to (0,.25);
	\draw[very thick] (.5,-.5) to [out=30,in=330] (.5,.5);
	\draw[very thick] (0,-.25) to (.5,-.5);
	\draw[very thick] (.5,-.875) to (.5,-.5);
	\draw[very thick] (-.25,-.875) to [out=90,in=210] (0,-.25);
	\draw[densely dotted, FS] (-.125,0) to (.875,0);
\end{tikzpicture}
\ \xleftrightarrow[\col \hComp \col]{\cre \hComp \cre} \
\begin{tikzpicture}[rotate=90,anchorbase,tinynodes]
	\draw[very thick] (0,.25) to [out=150,in=270] (-.25,.875);
	\draw[very thick] (.5,.5) to (.5,.875);
	\draw[very thick] (0,.25) to (.5,.5);
	\draw[very thick] (.125,.3125) to [out=300,in=225] (.375,.25) 
		node[above,yshift=-2pt]{$1$} to [out=45,in=300] (.375,.4375);
	\draw[very thick] (0,-.25) to (0,.25);
	\draw[very thick] (.5,-.5) to [out=30,in=330] (.5,.5);
	\draw[very thick] (.125,-.3125) to [out=60,in=135] (.375,-.25) 
		node[above,yshift=-2pt]{$1$} to [out=315,in=60] (.375,-.4375);	
	\draw[very thick] (0,-.25) to (.5,-.5);
	\draw[very thick] (.5,-.875) to (.5,-.5);
	\draw[very thick] (-.25,-.875) to [out=90,in=210] (0,-.25);
	\draw[densely dotted, FS] (-.125,0) to (.875,0);
\end{tikzpicture}
\ \cong \
\begin{tikzpicture}[rotate=90,anchorbase,tinynodes]
	\draw[very thick] (0,.25) to [out=150,in=270] (-.25,.875);
	\draw[very thick] (.5,.5) to (.5,.875);
	\draw[very thick] (0,.25) to (.5,.5);	
	\draw[very thick] (0,.125) to node[above,yshift=-2pt,xshift=2pt]{$1$} (.625,.375);
	\draw[very thick] (0,-.25) to (0,.25);
	\draw[very thick] (.5,-.5) to [out=30,in=330] (.5,.5);
	\draw[very thick] (0,-.125) to node[above,yshift=-2pt,xshift=-2pt]{$1$} (.625,-.375);
	\draw[very thick] (0,-.25) to (.5,-.5);
	\draw[very thick] (.5,-.875) to (.5,-.5);
	\draw[very thick] (-.25,-.875) to [out=90,in=210] (0,-.25);
	\draw[densely dotted, FS] (-.125,0) to (.875,0);
\end{tikzpicture}
\ \xleftrightarrow[\zip]{\un} \
\begin{tikzpicture}[rotate=90,anchorbase,tinynodes]
	\draw[very thick] (0,.25) to [out=150,in=270] (-.25,.875);
	\draw[very thick] (.5,.5) to (.5,.875);
	\draw[very thick] (0,.25) to (.5,.5);	
	\draw[very thick] (0,-.25) to (0,.25);
	\draw[very thick] (.5,-.5) to [out=30,in=330] (.5,.5);
	\draw[very thick] (.625,.375) to [out=210,in=90] 
		node[above]{$1$} (.125,0) to [out=270,in=150] (.625,-.375);
	\draw[very thick] (0,-.25) to (.5,-.5);
	\draw[very thick] (.5,-.875) to (.5,-.5);
	\draw[very thick] (-.25,-.875) to [out=90,in=210] (0,-.25);
	\draw[densely dotted, FS] (-.125,0) to (.875,0);
\end{tikzpicture}
\ \xleftrightarrow[x^r]{x^r} \
\begin{tikzpicture}[rotate=90,anchorbase,tinynodes]
	\draw[very thick] (0,.25) to [out=150,in=270] (-.25,.875);
	\draw[very thick] (.5,.5) to (.5,.875);
	\draw[very thick] (0,.25) to (.5,.5);	
	\draw[very thick] (0,-.25) to (0,.25);
	\draw[very thick] (.5,-.5) to [out=30,in=330] (.5,.5);
	\draw[very thick] (.625,.375) to [out=210,in=90] 
		node[above]{$1$} (.125,0) to [out=270,in=150] (.625,-.375);
	\draw[very thick] (0,-.25) to (.5,-.5);
	\draw[very thick] (.5,-.875) to (.5,-.5);
	\draw[very thick] (-.25,-.875) to [out=90,in=210] (0,-.25);
	\draw[densely dotted, FS] (-.125,0) to (.875,0);
\end{tikzpicture}
\ \xleftrightarrow[\;\cre\;]{\col} \
\begin{tikzpicture}[rotate=90,anchorbase,tinynodes]
	\draw[very thick] (0,.25) to [out=150,in=270] (-.25,.875);
	\draw[very thick] (.5,.5) to (.5,.875);
	\draw[very thick] (0,.25) to (.5,.5);
	\draw[very thick] (0,-.25) to (0,.25);
	\draw[very thick] (.5,-.5) to [out=30,in=330] (.5,.5);
	\draw[very thick] (0,-.25) to (.5,-.5);
	\draw[very thick] (.5,-.875) to (.5,-.5);
	\draw[very thick] (-.25,-.875) to [out=90,in=210] (0,-.25);
	\draw[densely dotted, FS] (-.125,0) to (.875,0);
\end{tikzpicture}
\]
and in this calculus the morphisms \eqref{eq:chi plus} and \eqref{eq:chi minus}
are given by
\begin{equation}\label{eq:perpchi}
\chi_r^+ :=
\begin{tikzpicture}[anchorbase,smallnodes]
	\draw[FS,ultra thick,->] (0,-.5) node[below=-1pt]{$a{+}k$}  to (0,1);
	\draw[FS,ultra thick,->] (.75,-.5) node[below=-1pt]{$b{-}k$} to (.75,1);
	\draw[FS,thick, directed=.75] (0,0) to [out=90,in=270] (.75,.5);
	\node[FS] at (0.25,0) {$1$};
	\node at (.375,.25) {$\bullet$};
	\node at (.375,.45) {$r$};
\end{tikzpicture}
\quad \text{and} \quad
\chi_r^- :=
\begin{tikzpicture}[anchorbase,smallnodes]
	\draw[FS,ultra thick,->] (0,-.5) node[below=-1pt]{$a{+}k$}  to (0,1);
	\draw[FS,ultra thick,->] (.75,-.5) node[below=-2pt]{$b{-}k$} to (.75,1);
	\draw[FS,thick, directed=.75] (.75,0) to  [out=90,in=270]  (0,.5) ;
	\node[FS] at (0.5,0) {$1$};
	\node at (.375,.25) {$\bullet$};
	\node at (.375,.45) {$r$};
\end{tikzpicture} \, .
\end{equation}
Note that some of the web edges are not visible in this calculus. We will denote
decoration endomorphisms on such edges by drawing the endomorphism in an
appropriate region, e.g. using the conventions in \eqref{eq:Rickardweb} we have
\[
\begin{tikzpicture}[anchorbase,smallnodes]
	\draw[FS,ultra thick,->] (0,-.5) node[below=-1pt]{$a{+}k$}  to (0,1);
	\draw[FS,ultra thick,->] (1,-.5) node[below=-1pt]{$b{-}k$} to (1,1);
	\node at (.5,.25) {\tiny $\CQGbox{f(\leftM)}$};
\end{tikzpicture}
=
\begin{tikzpicture}[rotate=90,anchorbase,smallnodes]
	\draw[very thick] (0,.25) to [out=150,in=270] (-.25,1);
	\draw[very thick] (.5,.5) to (.5,1);
	\draw[very thick] (0,.25) to node{$\bullet$} node[right=-1pt,yshift=2pt]{$f$} (.5,.5);
	\draw[very thick] (0,-.25) to (0,.25);
	\draw[very thick] (.5,-.5) to [out=30,in=330] (.5,.5);
	\draw[very thick] (0,-.25) to (.5,-.5);
	\draw[very thick] (.5,-1) to (.5,-.5);
	\draw[very thick] (-.25,-1) to [out=90,in=210] (0,-.25);
\end{tikzpicture}
=
\begin{tikzpicture}[anchorbase,smallnodes]
	\draw[CQG,ultra thick,<-] (0,-.5) node[below]{\scriptsize$l$}  to node[black]{$\bullet$} node[black,right]{$f$} (0,.7);
	\draw[CQG,ultra thick,->] (.75,-.5) node[below]{\scriptsize$k$} to (.75,.7);
\end{tikzpicture}
\]

All of the relations used in the sequel between perpendicular graphical calculus
diagrams can be deduced either from the corresponding relations in extended
graphical calculus, or from relations in the $\gln$ foam 2-category defined in \cite{QR}. 
As mentioned above, the latter is known to describe the 2-category of singular
Bott-Samelson bimodules in the $n \to \infty$ limit 
(see e.g. \cite[Section 5.2]{QRS}, \cite[Proposition 3.4]{Wed3}, or \cite[Appendix A]{HRW1}).
See Figure \ref{diffslices} for a graphical depiction of such a foam, 
together with the slices giving \eqref{eq:chi plus} and \eqref{eq:perpchi}.

\begin{figure}[h]
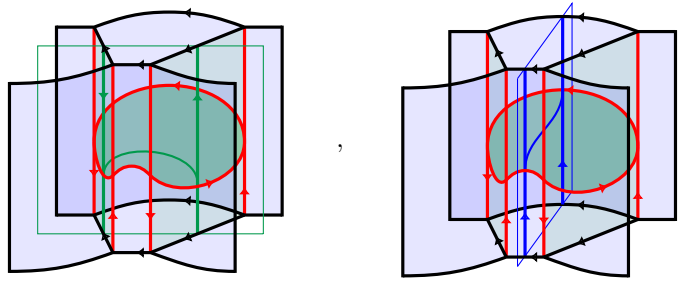

	\[\diffslices\]
\caption{The foam corresponding to $\chi_0^+$ and its slices that yield the
corresponding extended and perpendicular graphical calculus diagrams, 
respectively.}
\label{diffslices}
\end{figure}

There exist two $\hComp$- and $\circ$-contravariant duality functors on $\SSBim$:
\[
{}^{\vee}(-) \text{ and } (-)^\vee \colon {}_{\aa}\SSBim_{\bb} \to {}_{\bb}\SSBim_{\aa}
\] 
defined by
\[
{}^{\vee}X := \Hom_{R^{\aa}}(X,R^{\aa})
\, , \quad
{X}^{\vee}:= \Hom_{R^{\bb}}(X,R^{\bb})
\]
which satisfy the adjunctions
\[
\Hom_{\SSBim}(X \hComp Y, Z) \cong \Hom_{\SSBim}(Y, {}^{\vee}X \hComp Z)
\, , \quad
\Hom_{\SSBim}(X \hComp Y, Z) \cong \Hom_{\SSBim}(X, Z \hComp Y^\vee) \, .
\]
Since the bimodules ${}_{a+b}M_{a,b}$ and ${}_{a,b}S_{a+b}$ generate $\SSBim$ as 
a monoidal 2-category, this duality can be succinctly recorded as follows:

\begin{prop}\label{prop:dualityonMS}
Let $a,b \geq 0$, then
\[
{}_{a+b}M_{a,b}^\vee \cong \qdeg^{ab} {}_{a,b}S_{a+b}
\, , \quad
{}^{\vee}_{a+b}M_{a,b} \cong \qdeg^{-ab} {}_{a,b}S_{a+b}
\]
and
\[
{}_{a,b}S_{a+b}^\vee \cong \qdeg^{-ab} {}_{a+b}M_{a,b}
\, , \quad
{}^\vee_{a,b}S_{a+b} \cong \qdeg^{ab} {}_{a+b}M_{a,b} \, .
\]
Further, the relevant (co)unit morphisms are given by the 
digon creation/collapse and (un)zip morphisms. \qed
\end{prop}

We will be interested in complexes of singular Soergel bimodules. The natural
setting for their study is the dg 2-category of singular Soergel bimodules, which is
obtained by taking the dg category of complexes is each $\Hom$-category of
$\SSBim$.

\begin{defi}
Let $\CS(\SSBim)$ be the monoidal dg 2-category with the same objects as $\SSBim$, 
and wherein the 1-morphism category $\aa\rightarrow \bb$ 
equals $\CS({}_{\bb}\SSBim {}_{\aa})$.
\end{defi}

In other words, 1-morphisms in $\CS(\SSBim)$ are complexes of singular Soergel bimodules
and 2-morphism spaces in $\CS(\SSBim)$ are $\Hom$-complexes of bimodule maps.
Horizontal composition and external tensor product of $1$-morphisms is defined as usual, 
e.g.
\begin{equation}\label{eq:HCompC}
(X\hComp Y)^k = \bigoplus_{i+j=k} X^i\hComp Y^j
\, , \quad
\d_{X\hComp Y} = \d_{X}\hComp \Id_Y + \Id_X\hComp \d_Y \, .
\end{equation}
The components of the horizontal composition and external tensor product 
of $2$-morphisms are defined using the Koszul sign rule. 
For example, if $f\in \Hom_{\CS(\SSBim)}(X,X')$ and 
$g\in \Hom_{\CS(\SSBim)}(Y,Y')$ are given, 
then $f\hComp g$ is defined component-wise by:
\[
(f\hComp g)|_{X^i\hComp Y^j} = (-1)^{i |g|} f|_{X^i}\hComp g|_{Y^j} \, .
\]
Note that the (graded) middle interchange law:
\begin{equation}\label{eq:MidInt}
(f_1\hComp g_1) \circ (f_2\hComp g_2) 
= (-1)^{|g_1| |f_2|}(f_1 \circ f_2) \hComp (g_1 \circ g_2)
\end{equation}
holds in $\CS(\SSBim)$.

\begin{conv}
Since the 1-morphism categories of $\SSBim$ are $\Z_\qdeg$-graded,
the 1-morphism category $\CS({}_{\bb}\SSBim {}_{\aa})$ is enriched in 
$\overline{\KS}\llbracket \qdeg^{\pm},\tdeg^\pm \rrbracket_{\dg}$.
We will use the convention that $\deg(f)=(i,j)$ means $f$ has $\qdeg$-degree 
(or ``Soergel degree'') $i$ and cohomological degree $j$. 
Further, the singly-indexed $\Hom$-space 
$\Hom_{\CS(\SSBim)}^k(X,Y)$ always refers to cohomological degree, 
while the doubly-indexed $\Hom_{\CS(\SSBim)}^{i,j}(X,Y)$ 
consists of $f$ with $\deg(f)=(i,j)$.
For example, if $X$ is a $1$-morphism in $\CS(\SSBim)$, 
then its differential satisfies 
\[
\d_X \in \End_{\CS(\SSBim)}^{0,1}(X) 
:= \Hom_{\CS(\SSBim)}^{0,1}(X,X) \subseteq \Hom_{\CS(\SSBim)}^{1}(X,X) 
=: \End_{\CS(\SSBim)}^{1}(X) \, .
\]
As in Convention \ref{conv:wt}, 
we will typically indicate these degrees multiplicatively by writing $\wt(f) = \qdeg^i \tdeg^j$, 
and will also use the variables $\qdeg,\tdeg$ to denote the corresponding shift functors. 
Thus, for example, $\wt(\d_X) = \qdeg^0 \tdeg^1 = \tdeg$.
\end{conv}

\subsection{Colored braids and Rickard complexes}
\label{sec:colbraid}
We next recall the complexes of singular Soergel bimodules assigned to colored
braids. In this paper, the set $S$ of colors we will be $\Z_{\geq 1}$.
Let $\Br_m$ denote the $m$-strand braid group, which acts on $S^m$ by permuting
coordinates (this action factors through the symmetric group $\symg_m$).

\begin{defi}\label{def:CBG}
The \emph{$S$-colored braid groupoid} $\mathfrak{Br}(S)$ is the category wherein
\begin{itemize}
\item objects are sequences $(a_1,\ldots,a_m)$ with $a_i\in S$, $m\geq 1$, and
\item morphisms are given by
\[
\Hom_{\mathfrak{Br}(S)}(\aa,\bb)=\left\{\b \in \Br_m\:|\: a_i=b_{\b(i)} 
\text{ for } 1 \leq i \leq m \right\}
\]
with $\aa=(a_1,\ldots,a_m)$ and $\bb=(b_1,\ldots,b_m)$.
\end{itemize}
\end{defi}

Morphisms in $\mathfrak{Br}(S)$ are called \emph{$S$-colored braids}
and elements in $\Hom_{\mathfrak{Br}}(\aa,\bb)$ 
will be denoted by ${}_{\bb} \b_{\aa}$, 
or occasionally by ${}_{\bb} \b$ or $ \b_{\aa}$ since the domain/codomain
determine one another.
We will write $\Br_m(S)$ for the full subcategory of $\mathfrak{Br}(S)$ with objects 
having exactly $m$ entries. 

The colored braid groupoid is generated by the colored Artin generators
\[
\agen_i \colon (a_1,\ldots,a_i,a_{i+1},\ldots,a_m) \rightarrow 
(a_1,\ldots,a_{i+1},a_i,\ldots,a_m)
\]
which, when composable, satisfy relations analogous to the usual 
(type $A$) braid relations.
A \emph{colored braid word} is a sequence of colored Artin generators 
and their inverses. 
We say that a colored braid word $(\underline{\b})_{\aa}$ \emph{represents} 
the corresponding product of colored Artin generators in $\mathfrak{Br}(S)$.   

We now use the colored Artin generators to associate complexes
$C({}_{\bb}\b_{\aa})$ in $\SSBim$ to $\Z_{\geq 1}$-colored braid words ${}_{\bb}
{\b}_{\aa}$. Here, it is convenient to abuse notation by writing:
\[
C({}_{\bb} \b_{\aa}) = \oone_{\bb}  C(\b)  \oone_{\aa} 
= \oone_{\bb}  C(\b) = C(\b)  \oone_{\aa}
\]
(Note that $C(\b)$ alone does not denote a well-defined complex.)

\begin{defi}\label{def:Rickardcx}
	Let $a,b \geq 0$. 
	The \emph{$2$-strand Rickard complex} $C_{a,b}$ is the (bounded) complex
	\[
	\begin{aligned}
	C_{a,b} :=& 
	\left\llbracket
	\begin{tikzpicture}[rotate=90,scale=.5,smallnodes,anchorbase]
		\draw[very thick] (1,-1) node[right,xshift=-2pt]{$b$} to [out=90,in=270] (0,1);
		\draw[line width=5pt,color=white] (0,-1) to [out=90,in=270] (1,1);
		\draw[very thick] (0,-1) node[right,xshift=-2pt]{$a$} to [out=90,in=270] (1,1);
	\end{tikzpicture}
	\right\rrbracket :=
	\cdots 
	\xrightarrow{\;\; \chi_0^+ \;}
	\qdeg^{-k} \tdeg^k
	\begin{tikzpicture}[smallnodes,rotate=90,anchorbase,scale=.75]
		\draw[very thick] (0,.25) to [out=150,in=270] (-.25,1) node[left,xshift=2pt]{$b$};
		\draw[very thick] (.5,.5) to (.5,1) node[left,xshift=2pt]{$a$};
		\draw[very thick] (0,.25) to (.5,.5);
		\draw[very thick] (0,-.25) to (0,.25);
		\draw[very thick] (.5,-.5) to [out=30,in=330] node[above,yshift=-2pt]{$k$} (.5,.5);
		\draw[very thick] (0,-.25) to (.5,-.5);
		\draw[very thick] (.5,-1) node[right,xshift=-2pt]{$b$} to (.5,-.5);
		\draw[very thick] (-.25,-1)node[right,xshift=-2pt]{$a$} to [out=90,in=210] (0,-.25);
	\end{tikzpicture}
	\xrightarrow{\;\; \chi_0^+ \;}
	\qdeg^{-k-1}\tdeg^{k+1}
	\begin{tikzpicture}[smallnodes,rotate=90,anchorbase,scale=.75]
		\draw[very thick] (0,.25) to [out=150,in=270] (-.25,1) node[left,xshift=2pt]{$b$};
		\draw[very thick] (.5,.5) to (.5,1) node[left,xshift=2pt]{$a$};
		\draw[very thick] (0,.25) to (.5,.5);
		\draw[very thick] (0,-.25) to (0,.25);
		\draw[very thick] (.5,-.5) to [out=30,in=330] 
			node[above,yshift=-2pt]{$k{+}1$} (.5,.5);
		\draw[very thick] (0,-.25) to (.5,-.5);
		\draw[very thick] (.5,-1) node[right,xshift=-2pt]{$b$} to (.5,-.5);
		\draw[very thick] (-.25,-1)node[right,xshift=-2pt]{$a$} to [out=90,in=210] (0,-.25);
	\end{tikzpicture}
	\xrightarrow{\;\; \chi_0^+ \;}
	\cdots
	\end{aligned}
	\]
	of singular Soergel bimodules. The rightmost non-zero term is either
	$\qdeg^{-b}\tdeg^b \F^{(a-b)}\oone_{a,b}$ or $\qdeg^{-a}\tdeg^a \E^{(b-a)}\oone_{a,b}$
	depending on whether $a\geq b$ or $a\leq b$, respectively. Analogously, we
	also have:
	\[
	\begin{aligned}
	C^\vee_{a,b} :=& 
	\left\llbracket
	\begin{tikzpicture}[rotate=90,yscale=.5,xscale=-.5,smallnodes,anchorbase]
		\draw[very thick] (1,-1) node[right,xshift=-2pt]{$a$} to [out=90,in=270] (0,1);
		\draw[line width=5pt,color=white] (0,-1) to [out=90,in=270] (1,1);
		\draw[very thick] (0,-1) node[right,xshift=-2pt]{$b$} to [out=90,in=270] (1,1);
	\end{tikzpicture}
	\right\rrbracket :=
	\cdots 
	\xrightarrow{\;\; \chi_0^- \;}
	\qdeg^{k+1} \tdeg^{-k-1}
	\begin{tikzpicture}[smallnodes,rotate=90,anchorbase,scale=.75]
		\draw[very thick] (0,.25) to [out=150,in=270] (-.25,1) node[left,xshift=2pt]{$b$};
		\draw[very thick] (.5,.5) to (.5,1) node[left,xshift=2pt]{$a$};
		\draw[very thick] (0,.25) to (.5,.5);
		\draw[very thick] (0,-.25) to (0,.25);
		\draw[very thick] (.5,-.5) to [out=30,in=330] node[above,yshift=-2pt]{$k{+}1$} (.5,.5);
		\draw[very thick] (0,-.25) to (.5,-.5);
		\draw[very thick] (.5,-1) node[right,xshift=-2pt]{$b$} to (.5,-.5);
		\draw[very thick] (-.25,-1)node[right,xshift=-2pt]{$a$} to [out=90,in=210] (0,-.25);
	\end{tikzpicture}
	\xrightarrow{\;\; \chi_0^- \;}
	\qdeg^{k}\tdeg^{-k}
	\begin{tikzpicture}[smallnodes,rotate=90,anchorbase,scale=.75]
		\draw[very thick] (0,.25) to [out=150,in=270] (-.25,1) node[left,xshift=2pt]{$b$};
		\draw[very thick] (.5,.5) to (.5,1) node[left,xshift=2pt]{$a$};
		\draw[very thick] (0,.25) to (.5,.5);
		\draw[very thick] (0,-.25) to (0,.25);
		\draw[very thick] (.5,-.5) to [out=30,in=330] 
			node[above,yshift=-2pt]{$k$} (.5,.5);
		\draw[very thick] (0,-.25) to (.5,-.5);
		\draw[very thick] (.5,-1) node[right,xshift=-2pt]{$b$} to (.5,-.5);
		\draw[very thick] (-.25,-1)node[right,xshift=-2pt]{$a$} to [out=90,in=210] (0,-.25);
	\end{tikzpicture}
	\xrightarrow{\;\; \chi_0^- \;}
	\cdots
	\end{aligned}
	\]
\end{defi}
As graded objects, we identify 
\[
C_{a,b}=\bigoplus_{k=0}^{\min(a,b)} \qdeg^{-k} \tdeg^k C_{a,b}^k
\, , \quad
C^\vee_{a,b}=\bigoplus_{k=0}^{\min(a,b)} \qdeg^{k} \tdeg^{-k} C_{a,b}^k
\]
where $C_{a,b}^k := \F^{(a-k)}\E^{(b-k)}\oone_{a,b}$. 
As the notation suggests, $(C_{a,b})^\vee = C_{b,a}^\vee$.

\begin{defi}\label{def:RickardGeneral}
For the Artin generator $\agen_i$ of the braid group $\Br_m$ and
$\aa=(a_1,\ldots,a_m)$, we set:
\[
\begin{aligned}
C(\agen_i) \oone_{\aa} &:= \oone_{(a_1,\ldots,a_{i-1})}\boxtimes 
C_{a_i,a_{i+1}}\boxtimes \oone_{(a_{i+2} \ldots,a_m)} \\
C(\agen_i\inv) \oone_{\aa} &:= \oone_{(a_1,\ldots,a_{i-1})}\boxtimes 
C^\vee_{a_{i},a_{i+1}}\boxtimes \oone_{(a_{i+2},\ldots,a_m)} \, .
\end{aligned}
\]
This assignment extends to arbitrary colored braid words using horizontal
composition. Given a braid word $\b = \agen_{i_r}^{\e_r}\cdots
\agen_{i_1}^{\e_1}$, we call 
\begin{equation}\label{eq:Rick-hComp}
	C(\beta) \oone_{\aa} = C(\agen_{i_r}^{\e_r}\cdots \agen_{i_1}^{\e_1})\oone_{\aa}
:= C(\agen_{i_r}^{\e_r}) \hComp \cdots \hComp C(\agen_{i_1}^{\e_1})\oone_{\aa}
\end{equation}
the \emph{Rickard complex} assigned to the colored braid $\b_{\aa}$.
\end{defi}

This terminology is justified by the following proposition.

\begin{prop}[{\cite[\HRWPropRick]{HRW1}}]
	\label{prop:rickard invariance}
The complexes $C(\agen_{i_1}^{\e_1}\cdots \agen_{i_r}^{\e_r})\oone_{\aa}$
satisfy the (colored) braid relations, up to canonical homotopy equivalence. \qed
\end{prop}

Rickard complexes of colored braids extend to invariants
of braided webs (using horizontal composition and external tensor product),
since they satisfy the following \emph{fork-slide} and
\emph{twist-zipper} relations.

\begin{prop}[{\cite[\HRWPropWeb]{HRW1}}] \label{prop:forkslide} We have homotopy
equivalences
\begin{equation}\label{eq:forkslide}
\left\llbracket
\begin{tikzpicture}[rotate=90,scale=.5,smallnodes,anchorbase]
	\draw[very thick] (.5,-1) node[right,xshift=-2pt]{$c$} to [out=90,in=270] 
		(-1,2) node[left,xshift=2pt]{$c$};
	\draw[line width=5pt,color=white] (-.5,-1) to [out=90,in=270] (.5,1);
	\draw[very thick] (-.5,-1) node[right,xshift=-2pt]{$a{+}b$} to [out=90,in=270] (.5,1);
	\draw[very thick] (.5,1) to [out=30,in=270] (1,2) node[left,xshift=2pt]{$b$};
	\draw[very thick] (.5,1) to [out=150,in=270] (0,2) node[left,xshift=2pt]{$a$};
\end{tikzpicture}
\right\rrbracket	
\simeq
\left\llbracket
\begin{tikzpicture}[rotate=90,scale=.5,smallnodes,anchorbase]
	\draw[very thick] (.5,-1) node[right,xshift=-2pt]{$c$} to [out=90,in=270] 
		(-1,2) node[left,xshift=2pt]{$c$};
	\draw[line width=5pt,color=white] (-.5,-.25) to [out=30,in=270] (1,2);
	\draw[line width=5pt,color=white] (-.5,-.25) to [out=150,in=270] (0,2);
	\draw[very thick] (-.5,-1) node[right,xshift=-2pt]{$a{+}b$} to (-.5,-.25);
	\draw[very thick] (-.5,-.25) to [out=30,in=270] (1,2) node[left,xshift=2pt]{$b$};
	\draw[very thick] (-.5,-.25) to [out=150,in=270] (0,2) node[left,xshift=2pt]{$a$};
\end{tikzpicture}
\right\rrbracket
\, , \quad
\left\llbracket
\begin{tikzpicture}[rotate=90,scale=.5,smallnodes,anchorbase,yscale=-1]
	\draw[very thick] (-.5,-1) node[left,xshift=2pt]{$b{+}c$} to [out=90,in=270] (.5,1);
	\draw[very thick] (.5,1) to [out=30,in=270] (1,2) node[right,xshift=-2pt]{$c$};
	\draw[very thick] (.5,1) to [out=150,in=270] (0,2) node[right,xshift=-2pt]{$b$};
	\draw[line width=5pt,color=white] (.5,-1) to [out=90,in=270] (-1,2);
	\draw[very thick] (.5,-1) node[left,xshift=2pt]{$a$} to [out=90,in=270] 
		(-1,2) node[right,xshift=-2pt]{$a$};
\end{tikzpicture}
\right\rrbracket	
\simeq
\left\llbracket
\begin{tikzpicture}[rotate=90,scale=.5,smallnodes,anchorbase,yscale=-1]
	\draw[very thick] (-.5,-1) node[left,xshift=2pt]{$b{+}c$} to (-.5,-.25);
	\draw[very thick] (-.5,-.25) to [out=30,in=270] (1,2) node[right,xshift=-2pt]{$c$};
	\draw[very thick] (-.5,-.25) to [out=150,in=270] (0,2) node[right,xshift=-2pt]{$b$};
	\draw[line width=5pt,color=white] (.5,-1) to [out=90,in=270] (-1,2);
	\draw[very thick] (.5,-1) node[left,xshift=2pt]{$a$} to [out=90,in=270] 
		(-1,2) node[right,xshift=-2pt]{$a$};
\end{tikzpicture}
\right\rrbracket,
\end{equation}
\begin{equation}\label{eq:twistzipper}
	\left\llbracket
\begin{tikzpicture}[rotate=90,scale=.5,smallnodes,anchorbase]
	\draw[very thick] (1,-1) node[right,xshift=-2pt]{$b$} to [out=90,in=270] (0,1)
		to [out=90,in=210] (.5,2);
	\draw[line width=5pt,color=white] (0,-1) to [out=90,in=270] (1,1);
	\draw[very thick] (0,-1) node[right,xshift=-2pt]{$a$} to [out=90,in=270] (1,1)
		to [out=90,in=330] (.5,2);
	\draw[very thick] (.5,2) to (.5,2.75);
\end{tikzpicture}
\right\rrbracket		
\simeq
\qdeg^{a b}
\left\llbracket
\begin{tikzpicture}[rotate=90,scale=.5,smallnodes,anchorbase]
	\draw[very thick] (0,1) node[right,xshift=-2pt]{$b$} to [out=90,in=210] (.5,2);
	\draw[very thick] (1,1) node[right,xshift=-2pt]{$a$} to [out=90,in=330] (.5,2);
	\draw[very thick] (.5,2) to (.5,2.75);
\end{tikzpicture}
\right\rrbracket	
\end{equation}
as well as reflections thereof. \qed
\end{prop}

\section{Curved Rickard complexes and interpolation coordinates}
\label{s:curved rickard etc}
In this section, we introduce a dg 2-category of curved complexes of singular
Soergel bimodules, and define \emph{curved Rickard complexes} 
as certain special 1-morphisms.

\subsection{Perturbation theory for (curved) complexes}\label{ss:pert}
In the following, we will need the notion of a twist.  
Suppose that $X\in \AS[\tdeg^\pm]$ comes equipped with two endomorphisms 
$\d,\a\in \End_{\AS[\tdeg^\pm]}(X)\otimes R$ such that
\[
\d^2 = F_1\quad \text{and} \quad (\d+\a)^2 = F_1+F_2
\]
for $F_1,F_2 \in \cal{Z}(\AS)\otimes R$ of (cohomological) degree two.
(As a special case, we could have $F_1 =0=F_2$.)
It follows that $(X,\d)$ is an object of $\CS_{F_1}(\AS)$, and the object
$(X,\d+\a)\in \CS_{F_1+F_2}(\AS)$ is said to be a \emph{twist} of $(X,\d)$. 
We set
\[
\tw_\a((X,\d)) := (X,\d+\a)
\]
and will often simply write the former as 
$\tw_{\alpha}(X)$ when the differential $\d$ on $X$ is understood.
Note that the element $\a$
satisfies the \emph{Maurer--Cartan equation} with curvature:
\[
[\d,\a] + \a^2 = F_2 \, .
\]
We will refer to $\a$ as a \emph{(curved) Maurer--Cartan element}, 
or, by abuse of terminology, as a \emph{twist}.

In various places, we will need to promote a homotopy equivalence $X\simeq Y$ to
a homotopy equivalence between twists $\tw_\a(X)\simeq \tw_\b(Y)$. This is the
subject of homological perturbation theory. For our purposes, the following
result suffices; see \eg \cite{Markl,Hog3} for a more-thorough discussion.

\begin{prop}\label{prop:HPT} Let $X,Y \in \CS_{F_1}(\AS)$ and let $f \in
\Hom^0_{\CS_{F_1}(\AS)}(X,Y)$ and $g \in \Hom^0_{\CS_{F_1}(\AS)}(Y,X)$ determine
a homotopy equivalence $X \simeq Y$ with associated homotopies $k_X \in
\End^{-1}_{\CS_{F_1}(\AS)}(X)$ and $k_Y \in \End^{-1}_{\CS_{F_1}(\AS)}(Y)$.
Suppose that $\alpha \in \End^{1}_{\CS_{F_1}(\AS)}(X)$ is a Maurer--Cartan
element with curvature $F_2$ such that $\Id_X + \alpha \circ k_X \in
\End^{0}_{\CS(\AS)}(X)$ is invertible, 
then
\[
\begin{aligned}
\tilde{f} &:= f \circ (\Id_X + \alpha \circ k_X)\inv  \in 
\Hom^0_{\CS_{F_1+F_2}(\AS)}(\tw_{\alpha}(X),\tw_{\beta}(Y)) \\
\tilde{g} &:= (\Id_X + k_X \circ \alpha)\inv  \circ g \in 
\Hom^0_{\CS_{F_1+F_2}(\AS)}(\tw_{\beta}(Y),\tw_{\alpha}(X))
\end{aligned}
\]
determine a homotopy equivalence $\tw_{\alpha}(X) \simeq \tw_{\beta}(Y)$  in 
$\CS_{F_1+F_2}(\AS)$,
where $\beta := f \circ \alpha \circ \tilde{g}$. \qed
\end{prop}

\begin{remark}\label{rem:onesided}
In all of our applications of Proposition
\ref{prop:HPT}, invertibility of $\Id_X + \alpha \circ k_X$ will follow since
$\alpha \circ k_X$ is nilpotent. For example, this holds when $k_X$ acts
summand-wise on a \emph{finite one-sided twisted complex} $\tw_{\alpha}(X)$.
Recall that the latter means that
\[
(X,\delta) = \bigoplus_{i \in \Z} (X_i , \d_i)
\]
with $X_i=0$ for all but finitely many $i \in \Z$, 
and the components $\alpha_{i,j} \colon X_j \to X_i$ of the twist $\alpha$ 
satisfy $\alpha_{i,j} = 0$ for $i \leq j$.

Further, invertibility of $\Id_X + \alpha \circ
k_X$ implies that $\Id_X + k_X \circ \alpha$ is also invertible, so no further
assumptions are necessary to define $\tilde{g}$. Although the homotopy $k_Y$
does not appear in the definition of $\tilde{f}$ or $\tilde{g}$, it would appear
in the formula for the perturbed homotopy $\tilde{k_Y} \in
\End^{-1}_{\CS_{F_1+F_2}(\AS)}(Y)$.
\end{remark}

\subsection{A preliminary discussion on curvature}
\label{ss:curvature discussion}
Our curved complexes of singular Soergel bimodules will have curvatures modeled
on \emph{strand-wise curvature}\footnote{Informally, $\leftX$ and $\rightX$
should be thought of as the alphabets on the left and right of an $a$-colored
strand in a braid.} of the form
\begin{equation}\label{eq:h(X-X)}
\sum_{k=1}^a h_k(\leftX-\rightX) v_k \, ,
\end{equation}
which we refer to as $h\Delta$\emph{-curvature}. Here $\leftX,\rightX$ are
alphabets of cardinality $a$ and $v_1,\ldots,v_a$ are deformation parameters
with $\wt(v_k) = \qdeg^{-2k}\tdeg^2$. However, in order to define horizontal
composition in our 2-category of curved complexes of singular Soergel bimodules,
it will be auspicious to work with a different curvature that is modeled on
strand-wise $\Delta e$\emph{-curvature}, which is of the form
\begin{equation}\label{eq:e(X)-e(X)}
\sum_{k=1}^a  (e_k(\leftX)-e_k(\rightX)) u_k \, .
\end{equation}
As it turns out, we can regard these curvatures as equivalent after an
appropriate change of variables.

\begin{definition}\label{def:U and V} Let $\X$ be an alphabet of cardinality $a$ 
and consider collections of deformation parameters 
$\U=\{u_1,\ldots,u_a\}$ and $\V=\{v_1,\ldots,v_a\}$ 
with $\wt(v_k) = \qdeg^{-2k}\tdeg^2 = \wt(u_k)$. 
Let us identify the algebras $\k[\X,\U]$ and $\k[\X,\V]$ by declaring
\begin{equation}\label{eq:VvsU}
v_k = (-1)^{k-1}\sum_{k\leq l\leq a} e_{l-k}(\X) u_l
\, , \quad 
u_k  =(-1)^{k-1}\sum_{k\leq l\leq a} h_{l-k}(\X) v_l \, .
\end{equation}
\end{definition}

It is an easy exercise using \eqref{eq:HE2}
to verify that the formulae in \eqref{eq:VvsU} are mutually inverse.  

\begin{lemma}\label{lemma:sliding v} 
Identify $\k[\X,\U]\otimes_{\k[\U]} \k[\X,\U] \cong \k[\leftX,\rightX,\U]$ 
via the $\k[\U]$-linear map sending
$f(\X)\otimes g(\X)\mapsto f(\leftX)g(\rightX)$. 
Then, inside $\k[\X,\U]\otimes_{\k[\U]} \k[\X,\U]$, we have
\[
1\otimes v_k = \sum_{k \leq l \leq a} h_{l-k}(\leftX-\rightX) \cdot (v_l\otimes 1)
\]
for all $1\leq k\leq a$.
\end{lemma}
\begin{proof}
We compute
\begin{align*}
1\otimes v_k
&=(-1)^{k-1}\sum_{k \leq l \leq a} 1\otimes e_{l-k}(\X)u_l
=(-1)^{k-1}\sum_{k \leq l \leq a} u_l\otimes e_{l-k}(\X)\\
&=(-1)^{k-1}\sum_{k \leq l \leq m \leq a} (-1)^{l-1} h_{m-l}(\X) v_m\otimes e_{l-k}(\X)
=\sum_{k \leq m \leq a} h_{m-k}(\leftX-\rightX) \cdot (v_m\otimes 1)\, . \qedhere
\end{align*}
\end{proof}

\begin{cor}\label{cor:curvature description aux}
Under the identification of $\k[\leftX,\rightX,\U] \cong \k[\leftX,\rightX,\V]$ 
given via \eqref{eq:VvsU}, we have
\[
\sum_{k< l\leq a} (e_{l-k}(\leftX)-e_{l-k}(\rightX))u_l = \sum_{k<l\leq a} h_{l-k}(\leftX-\rightX)v_l
\]
for all $1\leq k\leq a$.
\end{cor}
\begin{proof}
Identify $\k[\leftX,\rightX,\U] \cong \k[\X,\U]\otimes_{\k[\U]} \k[\X,\U]$, 
which thus identifies $v_k$ with $v_k \otimes 1$.
We then compute
\[
1\otimes v_k - v_k\otimes 1 
= \sum_{k\leq l\leq a} \Big(e_{l-k}(\X)u_l\otimes 1 - 1\otimes e_{l-k}(\X)u_l\Big) 
= \sum_{k\leq l\leq a}(e_{l-k}(\leftX)-e_{l-k}(\rightX))u_l \, .
\]
Note that the $k=l$ term in this sum is zero. 
On the other hand, Lemma \ref{lemma:sliding v} gives
\[
1\otimes v_k - v_k\otimes 1 
= \left( \sum_{k \leq l \leq a} h_{l-k}(\leftX-\rightX) \cdot (v_l\otimes 1) \right) - v_k\otimes 1 
=  \sum_{k< l\leq a} h_{l-k}(\leftX-\rightX)v_l \, .\qedhere
\]
\end{proof}

Note that Lemma \ref{lemma:sliding v} and 
Corollary \ref{cor:curvature description aux} remain true (with the same proof) 
if we extend $\U$ and $\V$ to include deformation parameters $u_0,v_0$, 
with weights $\qdeg^0\tdeg^2$,
which we assume are related via the obvious extension of \eqref{eq:VvsU}
to the case $k=0$. 
Thus we also have the following identity, which is the 
``$k=0$ case'' of Corollary \ref{cor:curvature description aux}.
(Alternatively, this could be established by a straightforward computation.)

\begin{cor}\label{cor:curvature description}
We have
\[
\sum_{1\leq l\leq a} (e_{l}(\leftX)-e_{l}(\rightX))u_l 
	= \sum_{1\leq l\leq a} h_{l}(\leftX-\rightX)v_l\, . \qedhere
\]
\end{cor}

\subsection{Curved complexes over \texorpdfstring{$\SSBim$}{SSBim}}
\label{ss:curved cxs}

We now introduce the monoidal dg 2-category of curved complexes of singular
Soergel bimodules. Informally, this 2-category is formed via the following
procedure. First, we consider a $2$-subcategory consisting of $1$-morphism
$X\colon \aa \to \bb$ in $\CS(\SSBim)$ where $\bb$ is obtained by permuting the
indices of $\aa$. In particular, all Rickard complexes give $1$-morphisms in
this 2-category. Next, in each $\Hom$-category we adjoin $\#(\aa)$ alphabets
$\{\U_i\}_{i=1}^{\#(\aa)}$ of formal variables via Definition \ref{def:Adjoin}.
Finally, we pass to a certain category of curved complexes in each
$\Hom$-category, for a choice of curvature that \eg encodes the homotopies that
``slide'' the action of symmetric polynomials on the left boundary along a strand of 
a Rickard complex to the right.

We now make this informal description precise. 

\begin{defi}\label{def:Y} Fix an integer $m\geq 0$ and let
$\YS(\SSBim,m)$ be the 2-category wherein:
\begin{itemize}
\item objects are pairs $(\aa,\sigma)$ where $\aa$ is an object in $\SSBim$ with
$\#(\aa)=m$ and $\sigma \in \symg_{m}$,
\item
the $\Hom$-category ${}_{\bb,\tau}\YS(\SSBim)_{\aa,\sigma}$ from 
$(\aa,\sigma)$ to $(\bb,\tau)$
is empty unless $a_{\sigma(i)} = b_{\tau(i)}$ for all $1 \leq i \leq m$. In this case,
introduce an alphabet $\U_i = \{\u_{i,r}\}_{r=1}^{a_{\sigma(i)}}$ for each $1\leq
i\leq m$ with $\wt(\u_{i,r}) = \qdeg^{-2r} \tdeg^2$ and set
\[
{}_{\bb,\tau}\YS(\SSBim)_{\aa,\sigma} :=
\CS_F\big( {}_{\bb}\SSBim_{\aa} ; \k[\U_1,\ldots,\U_{m}]\big)
\]
where the curvature element is
\begin{equation}\label{eq:Ucurv}
F= \sum_{i=1}^{m} \sum_{r=1}^{a_{\sigma(i)}} 
\big(e_r(\leftX_{\tau(i)}) - e_r(\rightX_{\sigma(i)}) \big) \u_{i,r}\, .
\end{equation}
Here, $\leftX_{\tau(i)}$ and $\rightX_{\sigma(i)}$ are the relevant 
alphabets acting on the left and right (boundary), respectively.
In such expressions, we will sometimes suppress the summation 
limits depending on colors $a_{\sigma(i)}$ by summing over 
all $r \geq 1$ and declaring $\u_{i,r}:=0$ for $r>a_{\sigma(i)}$.
\end{itemize}

We let $\YS(\SSBim) := \bigsqcup_{m\geq 0} \YS(\SSBim,m)$.
The horizontal composition of 1-morphisms in $\YS(\SSBim)$ is defined by the usual rule
\begin{equation}\label{eq:HCompDef0}
(X,\d^{\tot}_X)\hComp (Y, \d^{\tot}_Y) 
:= (X\hComp Y\, , \ \d^{\tot}_X\hComp \id_Y + \id_X\hComp \d^{\tot}_Y)\, .
\end{equation}
and horizontal and vertical composition of 2-morphisms is inherited from 
the $2$-morphism composition in $\CS(\SSBim)$.
\end{defi}

\begin{rem}
The appearance of permutations in the objects of $\YS(\SSBim)$ may appear surprising. 
Below, when we assign $1$-morphisms in $\YS(\SSBim)$ to braids, 
the permutations in the (co)domain objects will encode a numbering of the strands in the braid. 
E.g. if the domain object is $(\aa,\sigma)$, then $\sigma(i) = j$ tells us that 
the strand meeting the $j^{th}$ boundary point on the right of the braid is the $i^{th}$ strand 
in this numbering.
\end{rem}

When $m$ is understood in Definition \ref{def:Y}, 
we write $\U = \U_1\cup \cdots \cup \U_m$, hence \eg $\k[\U]=\k[\U_1,\ldots,\U_m]$.
We have written the differential on a 1-morphism $(X,\d^{\tot}_X)$ in $\YS(\SSBim)$ 
using the superscript ``tot'' in order to emphasize the fact that $\d^{\tot}_X$ 
decomposes canonically (and uniquely) into a sum of terms:
\[
\d^{\tot}_X = \d_X+\Delta_X
\]
where $\d_X$ lives in $\End_{\CS(\SSBim)}(X)$ and $\Delta_X$ lives in 
the ideal $\End_{\CS(\SSBim)}(X)\otimes \k[\U]_{>0}$ generated by polynomials in $\U$ 
with zero constant term.  
Hence, $(X,\d_X)$ defines a complex of singular Soergel bimodules, and
$(X,\d_X+\Delta_X) = \tw_{\Delta_X}((X,\d_X))$.  It will frequently be useful to
decompose 2-morphisms in $\YS(\SSBim)$ according to their $\U$-degree zero parts
and their ``strictly positive $\U$-degree'' parts, according to the following
definition.

\begin{definition}
A 2-morphism $f$ in $\YS(\SSBim)$ is \emph{$\U$-irrelevant} if $f$ is zero after
setting all $\U$-variables equal to zero.  In other words, $f$ is
$\U$-irrelevant if it is an element of $\Hom_{\CS(\SSBim)}(X,Y)\otimes\k[\U]_{>0}$
for appropriate $X,Y$.
\end{definition}

\begin{conv}\label{conv:curved twists} Henceforth, we will write
1-morphisms in $\YS(\SSBim)$ in the form $\tw_{\Delta_X}(X)$ where $X$ is a
1-morphism in $\CS(\SSBim)$ and $\Delta_X$ is a curved Maurer--Cartan element in
$\End_{\CS(\SSBim)}(X)\otimes \k[\U]_{>0}$ 
(that is to say, $\Delta$ is $\U$-irrelevant).  
In this language, the composition of $1$-morphisms takes the form
\begin{equation}\label{eq:HCompDef}
\tw_{\Delta_X}(X) \hComp \tw_{\Delta_Y}(Y) := 
\tw_{\Delta_{X\hComp Y}}(X\hComp Y)
\, , \quad
\Delta_{X\hComp Y} :=\Delta_X\hComp \Id_Y + \Id_X\hComp \Delta_Y.
\end{equation}
Further, we will sometime embellish this notation as
$\tw_{\Delta_X}({}_{\bb,\tau}X_{\aa,\sigma})$ or $\tw_{\Delta_X}({}_{\bb}X_{\aa})$ 
when we wish to emphasize the data specifying the objects. 
\end{conv}

\begin{rem}\label{rmk:strict} 
If $\tw_{\Delta_X}({}_{\bb,\tau}X_{\aa,\sigma})$ is a 1-morphism in $\YS(\SSBim)$,
then the linear part of $\Delta_X$ is a sum of terms of the form 
$\Psi_{i,r} \uvar_{i,r}$, 
where $\Psi_{i,r}\in \End^{2r,-1}(X)$ satisfies
\[
[\d_X, \Psi_{i,r}] =   e_r(\leftX_{\tau(i)}) - \ e_r(\rightX_{\sigma(i)})\, .
\]
If $\Delta_X$ is linear in the variables $\uvar_{i,r}$, 
then $\tw_{\Delta_X}(X)$ is called a \emph{strict 1-morphism} in $\YS(\SSBim)$. 
This implies that the endomorphisms $\Psi_{i,r}$ 
necessarily square to zero and pairwise anti-commute, so
\[
[\d_X+\Delta_X, \Psi_{i,r}] =  e_r(\leftX_{\tau(i)}) - \ e_r(\rightX_{\sigma(i)}) \, .
\]
Hence, $e_r(\leftX_{\tau(i)}) - \ e_r(\rightX_{\sigma(i)})$ 
is null-homotopic on $\tw_{\Delta_X}(X)$.
This conclusion holds for non-strict morphisms as well, 
as can be seen by differentiating the equation 
$(\d_X+\Delta_X)^2=F$ with respect to the variable $u_{i,r}$.
\end{rem}

Note that a typical 2-morphism 
$f\in \Hom_{\YS(\SSBim)}(X,Y)$ 
is a formal sum
\begin{equation}\label{eq:f components}
f = \sum_{\ii,\rr,\kk} f_{\ii,\rr,\kk} \otimes \uvar_{\ii,\rr}^{\kk}
\end{equation}
over finitely many triples $(\ii,\rr,\kk)$
where
$
\u_{\ii,\rr}^{\kk} :=
\u_{i_{1}, r_{1}}^{k_{1}} \cdots 
\u_{i_{\ell}, r_{\ell}}^{k_{\ell}}
$
and $f_{\ii,\rr,\kk} \in \Hom_{\CS(\SSBim)}(X,Y)$.
If $f$ is homogeneous of weight $\qdeg^{l_1} \tdeg^{l_2} $, 
then we have
\begin{equation}\label{eq:ComponentWt}
\wt(f_{\ii,\rr,\kk}) = \qdeg^{l_1+2\sum_j k_j r_j} \tdeg^{l_2-2\sum_j k_j},
\end{equation}
since $\u_{i,r}^k$ has weight $\qdeg^{-2kr} \tdeg^{2k}$.
In particular, since the curved complexes $\tw_{\Delta_X}(X)$ 
in $\YS(\SSBim)$ are bounded, 
this implies that the curved Maurer--Cartan element $\Delta_X$ 
is nilpotent.

We now elaborate on the $2$-categorical structure on $\YS(\SSBim)$. 
First, we check the following.

\begin{lem}\label{lem:hComp}
Horizontal composition is well-defined on $\YS(\SSBim)$.
\end{lem}
\begin{proof}
It suffices to show that 
given composable $1$-morphisms 
$\tw_{\Delta_X}({}_{\aa,\sigma}X_{\aa',\sigma'})$ 
and $\tw_{\Delta_X}({}_{\aa',\sigma'}X_{\aa'',\sigma''})$,
the twist
\[
\Delta_{X\hComp Y} = \Delta_X\hComp \Id_Y + \Id_X\hComp \Delta_Y
\]
satisfies the conditions in Definition \ref{def:Y}. 
The only non-immediate check is to compute its curvature.
Let $m = \#(\aa) = \#(\aa') = \#(\aa'')$, and note that 
$a_{\sigma(i)} = a'_{\sigma'(i)} = a''_{\sigma''(i)}$ for all $1 \leq i \leq m$.
We then have
\[
(\d_X+\Delta_X)^2 = 
\sum_{i=1}^{m} \sum_{r=1}^{a'_{\sigma'(i)}} 
 \big( e_r(\leftX_{\sigma(i)}) - e_r(\rightX_{\sigma'(i)}) \big) \u_{i,r}
\]
and
\[
(\d_Y+\Delta_Y)^2 = 
\sum_{i=1}^{m} \sum_{r=1}^{a''_{\sigma''(i)}} 
\big( e_r(\rightX_{\sigma'(i)}) - e_r(\rightrightX_{\sigma''(i)}) \big) \u_{i,r}\, ,
\]
thus
\[
\begin{aligned}
\Big( (\d_X+\Delta_X)\hComp \Id_Y + \Id_X \hComp(\d_Y+\Delta_Y) \Big)^2 
&= (\d_X+\Delta_X)^2 \hComp \Id_Y + \Id_X \hComp (\d_Y+\Delta_Y)^2 \\
	& \qquad \qquad + [(\d_X+\Delta_X)\hComp \Id_Y , \Id_X \hComp(\d_Y+\Delta_Y)] \\
&= \sum_{i=1}^{m} \sum_{r=1}^{a''_{\sigma''(i)}} 
	\big( e_r(\leftX_{\sigma(i)}) - e_r(\rightX_{\sigma'(i)}) + 
		e_r(\rightX_{\sigma'(i)}) - e_r(\rightrightX_{\sigma''(i)}) \big) \u_{i,r} \\
&= \sum_{i=1}^{m} \sum_{r=1}^{a''_{\sigma''(i)}} 
	 \big( e_r(\leftX_{\sigma(i)}) - e_r(\rightrightX_{\sigma''(i)}) \big) \u_{i,r}
\end{aligned}
\]
as desired. Here, we use that 
$[(\d_X+\Delta_X)\hComp \Id_Y , \Id_X \hComp(\d_Y+\Delta_Y)]=0$ by \eqref{eq:MidInt}.
\end{proof}

Next, we note that the external tensor product $\boxtimes$ can be 
extended from $\CS(\SSBim)$ to $\YS(\SSBim)$ in a 
straightforward manner.  
On the level of objects, we have
\[
(\aa_1,\sigma_1) \boxtimes (\aa_2,\sigma_2) = 
(\aa_1 \boxtimes \aa_2, \sigma_1 \boxtimes \sigma_2),
\]
where $\sigma_1 \boxtimes \sigma_2 \in \symg_{\#(\aa_1)+\#(\aa_2)}$ 
is defined by
\[
\sigma_1 \boxtimes \sigma_2 \colon i \mapsto 
\begin{cases} 
\sigma_1(i) & \text{if } 1\leq i\leq \#(\aa_1) \\ 
\#(\aa_1)+\sigma_2(i-\#(\aa_1)) & \text{if } \#(\aa_1)+1\leq i\leq \#(\aa_1)+\#(\aa_2) \, ,
\end{cases}
\]
i.e. using the standard inclusion 
$\symg_{\#(\aa_1)} \times \symg_{\#(\aa_2)} 
\hookrightarrow \symg_{\#(\aa_1)+\#(\aa_2)}$.

On the level of $1$-morphisms, if $\tw_{\Delta_i}(X_i)$ are 1-morphisms 
$(\aa_i,\sigma_i) \to (\bb_i,\tau_i)$ for $i=1,2$,
then
\begin{equation}\label{eq:ETenDef}
\tw_{\Delta_1}(X_1) \boxtimes \tw_{\Delta_2}(X_2) 
:= \tw_{\Delta_1\boxtimes \Id_{X_2} + \Id_{X_1}\boxtimes \Delta_2}
(X_1\boxtimes X_2)
\end{equation}
where, analogous to \eqref{eq:HCompC}, we have 
$X_1\boxtimes X_2 = \bigoplus_{i,j} X_1^i \boxtimes X_2^j$.
The verification that the twist $\Delta_1\boxtimes \Id_{X_2} + \Id_{X_1}\boxtimes \Delta_2$ 
satisfies the curvature condition in Definition \ref{def:Y} is similar to the 
verification in Lemma \ref{lem:hComp}, thus we omit it.

Finally, the external tensor product of 2-morphisms in $\YS(\SSBim)$ is defined
via the $\k[\U]$-linear extension of the external tensor product in
$\CS(\SSBim)$. Explicitly, for homogeneous $f\colon X_1 \to Y_1$ and $g\colon
X_2 \to Y_2$ it is given by
\[
(f\boxtimes g)|_{X_1^i\boxtimes X_2^j} = (-1)^{i |g|} f|_{X_1^i}\boxtimes g|_{X_2^j}.
\]
Thus defined, $\boxtimes$ endows $\YS(\SSBim)$ with the structure of a monoidal dg 2-category.

\begin{rem} \label{rem:curvedlift}
There is a monoidal dg 2-functor $\YS(\SSBim)\rightarrow \CS(\SSBim)$ 
that forgets the permutations, sets all variables $\uvar_{i,r}$ equal to zero, 
and sends $\tw_{\Delta_X}(X) \mapsto X$. 
We refer to $\tw_{\Delta_X}(X)$ as a \emph{curved lift} or 
\emph{curved deformation} of the complex $X$, 
and similarly for $2$-morphisms in $\YS(\SSBim)$.
\end{rem}

Our next result allows us to upgrade homotopy equivalences between 
$1$-morphisms in $\CS(\SSBim)$ to equivalences in $\YS(\SSBim)$, 
provided we are given a curved lift of one of the $1$-morphisms.

\begin{prop}\label{prop:conservation} 
Suppose that $\tw_{\Delta_X}(X)$ is a 1-morphism in $\YS(\SSBim)$, 
and $f\colon X \to Y$ is a homotopy equivalence in $\CS(\SSBim)$, 
then there exists a curved lift $\tw_{\Delta_Y}(Y)$ of $Y$ 
and an induced homotopy equivalence 
$\tilde{f}\colon \tw_{\Delta_X}(X) \to \tw_{\Delta_Y}(Y)$.
\end{prop}
\begin{proof}
We use Proposition \ref{prop:HPT}, and its notation. 
The present result is an immediate consequence, 
applied to $X \in \CS(\SSBim)$ with $\alpha=\Delta_X$, 
once we have confirmed that 
$\Delta_X \circ k_X \in \End_{\CS(\SSBim)[\U]}(X)$ is nilpotent.
To see the latter, 
note that since $\Delta_X = 0 \mod \langle u_{i,r} \rangle$, 
the same holds for $\Delta_X \circ k_X$. 
This implies that $(\Delta_X \circ k_X)^\ell = 0 \mod \langle u_{i,r} \rangle^\ell$.
Writing $(\Delta_X \circ k_X)^\ell$ in components as in \eqref{eq:f components}, 
this in turn implies that there exists $\ell \geq 0$ so that $(\Delta_X \circ k_X)^\ell = 0$ 
by equation \eqref{eq:ComponentWt}, since $X$ is bounded.
\end{proof}

\begin{lem}\label{lem:curvinginvertibles}
Suppose that $(\aa,\sigma)$ and $(\bb,\tau)$ are objects in $\YS(\SSBim)$ 
such that $\#(\aa) = \#(\bb)$ and $a_{\sigma(i)} = b_{\tau(i)}$ for all $1 \leq i \leq \#(\aa)$.
Let $\oone_{\bb} \hComp L \hComp \oone_{\aa}$ be an invertible $1$-morphism in $\CS(\SSBim)$
such that $f(\leftX_{\tau(i)}) \hComp \Id_L \simeq \Id_L\hComp f(\rightX_{\sigma(i)})$ 
for all $1 \leq i \leq \#(\aa)$ and all symmetric functions $f$,
then there exists a curved lift 
\[
\tw_{\Delta_L}({}_{\bb,\tau}L_{\aa,\sigma})
\]
of $L$ in $\YS(\SSBim)$ which is unique up to homotopy equivalence.
\end{lem}
\begin{proof}
The existence statement follows via obstruction-theoretic 
arguments analogous to those in \cite[Section 2.10]{GH}. 
The crucial point is that the homotopy equivalence
\[
\End_{\CS(\SSBim)}(L) \simeq \End_{\CS(\SSBim)}(\oone_{\aa})
\]
implies that there are no obstructions to constructing the twist $\Delta_L$ 
using the given homotopies.

To show uniqueness of the lift up to homotopy equivalence, let $X =
\tw_{\Delta}(L)$ and $Y = \tw_{\Delta'}(L)$ be two curved lifts of $L$ and
$X^\vee = \tw_{\Delta''}(L^\vee)$ a curved lift of an (up to homotopy) inverse
of $L$. It follows that $X^{\vee}\hComp Y$ is a curved lift of a complex that is
homotopy equivalent to $\oone_{\aa}$ in $\CS(\SSBim)$. Hence,
Proposition~\ref{prop:conservation} implies the existence of a homotopy
equivalence
\[
X^{\vee}\hComp Y \simeq \tw_{\Delta'''}(\oone_{\aa})
\]
for some twist $\Delta'''$. 
However, since $\End_{\CS(\SSBim)}^k(\oone_{\aa})=0$ for $k < 0$, 
\eqref{eq:ComponentWt} implies that $\Delta'''=0$.
Thus $X^{\vee}\hComp Y \simeq \oone_{\aa,\sigma}$, 
which shows that $Y$ is a two-sided inverse to $X^\vee$ up to homotopy.  
The same argument shows that $X$ is a two-sided inverse to $X^\vee$ up 
to homotopy.  Thus $X\simeq Y$ by uniqueness of two-sided inverses.
\end{proof}

\begin{rem}\label{rem:StrongUniqueness}
Using obstruction-theoretic arguments, 
it is possible to strengthen the uniqueness statement in Lemma \ref{lem:curvinginvertibles} as follows.
Suppose, that $X \in \CS(\SSBim)$ is invertible and that $\tw_{\Delta}(X)$ and $\tw_{\Delta'}(X)$ are 
two curved lifts of $X$ in $\YS(\SSBim)$. 
Then, in fact, there is a (closed) \emph{isomorphism} $\phi \colon \tw_{\Delta}(X) \to \tw_{\Delta'}(X)$ 
of curved complexes of singular Soergel bimodules
such that $\phi = \id_X + \phi_{>0}$ with $\phi_{>0} \in \End_{\CS(\SSBim)}(X)\otimes\k[\U]_{>0}$.
\end{rem}

\subsection{Curved Rickard complexes}
\label{ss:curved rickard}

Our first goal is to define a lift of the two-strand Rickard complex $C_{a,b} \in \CS(\SSBim)$
from Definition~\ref{def:Rickardcx} to a \emph{curved Rickard complex}
$\YS C_{a,b} \in \YS(\SSBim)$. More precisely, writing $\aa:=(a,b)$, $\aa':=(b,a)$, 
and $\trans$ for the transposition in $\symg_2$, we define lifts of $C_{a,b}$ to
1-morphisms in 
${}_{\bb,\trans \circ \sigma}\YS(\SSBim)_{\aa,\sigma}$
for both possible permutations $\sigma\in \symg_2$. 
Pictorially, we denote the $2$-strand Rickard complex and its curved analogue by
\[
C_{a,b} = \left\llbracket
\begin{tikzpicture}[rotate=90,scale=.5,smallnodes,anchorbase]
	\draw[very thick,->] (1,-1) node[right,xshift=-2pt]{$b$} to [out=90,in=270] (0,1);
	\draw[line width=5pt,color=white] (0,-1) to [out=90,in=270] (1,1);
	\draw[very thick,->] (0,-1) node[right,xshift=-2pt]{$a$} to [out=90,in=270] (1,1);
\end{tikzpicture}
\right\rrbracket
\, , \quad
\YS C_{a,b} = \left\llbracket
\begin{tikzpicture}[rotate=90,scale=.5,smallnodes,anchorbase]
	\draw[very thick,->] (1,-1) node[right,xshift=-2pt]{$b$} to [out=90,in=270] (0,1);
	\draw[line width=5pt,color=white] (0,-1) to [out=90,in=270] (1,1);
	\draw[very thick,->] (0,-1) node[right,xshift=-2pt]{$a$} to [out=90,in=270] (1,1);
\end{tikzpicture}
\right\rrbracket_\YS
\]
respectively.

As in Definition \ref{def:Rickardcx}, we will abbreviate by writing
$C^k_{a,b}:=\F^{(a-k)}\E^{(b-k)}\oone_{a,b}$. 
This is the singular Soergel bimodule corresponding to the web
\[
C^k_{a,b} = 
\begin{tikzpicture}[rotate=90, anchorbase, smallnodes]
	\draw[very thick] (0,.25) to [out=150,in=270] (-.25,1) node[left,xshift=2pt]{$\leftX_1$};
	\draw[very thick] (.5,.5) to (.5,1) node[left,xshift=2pt]{$\leftX_2$};
	\draw[very thick] (0,.25) to node[left,yshift=-2pt,xshift=1pt]{$\leftM$} (.5,.5);
	\draw[very thick] (0,-.25) to node[below,yshift=0pt]{$\Fr$} (0,.25);
	\draw[very thick] (.5,-.5) to [out=30,in=330] node[above,yshift=-2pt]{$\B$} (.5,.5);
	\draw[very thick] (0,-.25) to node[right,yshift=-2pt,xshift=0pt]{$\rightM$} (.5,-.5);
	\draw[very thick] (.5,-1) node[right,xshift=-2pt]{$\rightX_2$} to (.5,-.5);
	\draw[very thick] (-.25,-1)node[right,xshift=-2pt]{$\rightX_1$} to [out=90,in=210] (0,-.25);
\end{tikzpicture}
\] 
where the labels on the edges are given by the sizes of the alphabets
$|\X_1|=|\rightX_2|=b$, $|\X_2|=|\rightX_1|=a$, $|\M|=a-k$, and $|\M'|=b-k$. 
This implies $|\B|=k$ and $|\Fr|=a+b-k$. 
Recall the morphisms $\chi_r^{\pm}\colon C^k_{a,b}\rightarrow C^{k\pm 1}_{a,b}$ 
from equations \eqref{eq:chi plus} and \eqref{eq:chi minus}.

\begin{prop}\label{prop:curved crossing}
Let $a,b\geq 0$, $\sigma\in \symg_2$, and
define $\aa:=(a_1,a_2):=(a,b)$, $\aa':=(b,a)$. 
The following diagram define a $1$-morphism $\YS C_{a,b}\in \YS(\SSBim)$ 
from $(\aa,\sigma)$ to $(\aa', \trans \circ\sigma)$, 
which is a curved lift of the Rickard complex $C_{a,b}$:
\begin{equation}\label{eq:curved crossing}
\begin{tikzpicture}[anchorbase]
\tikzstyle{every node}=[font=\small]
\node (a) at (0,0) {$\cdots $};
\node (b) at (4,0) {$\qdeg^{-k}\tdeg^k\begin{tikzpicture}[smallnodes,rotate=90,anchorbase,scale=.75]
	\draw[very thick] (0,.25) to [out=150,in=270] (-.25,1) node[left,xshift=2pt]{$b$};
	\draw[very thick] (.5,.5) to (.5,1) node[left,xshift=2pt]{$a$};
	\draw[very thick] (0,.25) to (.5,.5);
	\draw[very thick] (0,-.25) to (0,.25);
	\draw[very thick] (.5,-.5) to [out=30,in=330] node[above,yshift=-2pt]{$k$} (.5,.5);
	\draw[very thick] (0,-.25) to (.5,-.5);
	\draw[very thick] (.5,-1) node[right,xshift=-2pt]{$b$} to (.5,-.5);
	\draw[very thick] (-.25,-1)node[right,xshift=-2pt]{$a$} to [out=90,in=210] (0,-.25);
\end{tikzpicture}$};
\node (c) at (9,0) {$\qdeg^{-k-1}\tdeg^{k+1}\begin{tikzpicture}[smallnodes,rotate=90,anchorbase,scale=.75]
	\draw[very thick] (0,.25) to [out=150,in=270] (-.25,1) node[left,xshift=2pt]{$b$};
	\draw[very thick] (.5,.5) to (.5,1) node[left,xshift=2pt]{$a$};
	\draw[very thick] (0,.25) to (.5,.5);
	\draw[very thick] (0,-.25) to (0,.25);
	\draw[very thick] (.5,-.5) to [out=30,in=330] node[above,yshift=-2pt]{$k{+}1$} (.5,.5);
	\draw[very thick] (0,-.25) to (.5,-.5);
	\draw[very thick] (.5,-1) node[right,xshift=-2pt]{$b$} to (.5,-.5);
	\draw[very thick] (-.25,-1)node[right,xshift=-2pt]{$a$} to [out=90,in=210] (0,-.25);
\end{tikzpicture}$};
\node (d) at (13,0) {$\cdots$};
\path[->,>=stealth,shorten >=1pt,auto,node distance=1.8cm,
  thick]
([yshift=3pt] a.east) edge node {$\d$}	([yshift=3pt] b.west)
([yshift=-3pt] b.west) edge node {$\Delta$}	([yshift=-3pt] a.east)
([yshift=3pt] b.east) edge node {$\d$}	([yshift=3pt] c.west)
([yshift=-3pt] c.west) edge node {$\Delta$} ([yshift=-3pt] b.east)
([yshift=3pt] c.east) edge node {$\d$}	([yshift=3pt] d.west)
([yshift=-3pt] d.west) edge node {$\Delta$}	([yshift=-3pt] c.east);
\end{tikzpicture} \, .
\end{equation}
Here $\d,\Delta$ are given componentwise by $\d=\chi_0^+$ and
\[
\Delta|_{C_{a,b}^k} = 
\sum_{1\leq r\leq m<\infty}(-1)^{r+k-1}
\Big(e_{m-r}(\leftX_2)\u_{\sigma(1),m} - e_{m-r}(\rightX_2)\u_{\sigma(2),m}\Big) \chi_{r-1}^-
\]
(following Definition \ref{def:Y}, $\u_{\sigma(i),m}=0$ for $m>a_{i}$). 
An analogous construction defines a 1-morphism 
$\YS C^\vee_{a,b}\in \YS(\SSBim)$ from $(\aa,\sigma)$ to $(\aa', \trans \circ\sigma)$, 
which is a curved lift of the inverse Rickard complex $C^\vee_{a,b}$.
\end{prop}

The proof of this proposition is found below, 
but first we make some observations.
The rightward differential $\d$ is the usual differential on the Rickard complex
$C_{a,b}$.  Meanwhile the leftward differential $\Delta$ is linear in the
variables $\u_{\sigma(1),m},\u_{\sigma(2),m}$.  
Thus, \eqref{eq:curved crossing} defines 
a \emph{strict} curved deformation of $C_{a,b}$, 
in the sense of Remarks~\ref{rmk:strict} and \ref{rem:curvedlift}.

\begin{definition}\label{def:crossing Theta} 
For each integer $r\geq 1$, 
let $\Theta_r$ denote the degree $\qdeg^{2r} \tdeg\inv$ endomorphism of $C_{a,b}$
which is given component-wise on $C_{a,b}^k$ by the morphisms
$(-1)^{k}\chi^-_{r-1}$, where $\chi_{r-1}^-$ is the morphism from \eqref{eq:chi minus}.
\end{definition}

\begin{lem}\label{lem:crossing Theta} 
The endomorphisms $\Theta_r\in
\End_{\CS(\SSBim)}(C_{a,b})$ satisfy
\[
\Theta_i^2 =0 \, , \quad 
\Theta_i \Theta_j+\Theta_j \Theta_i=0 \, , \quad
[\d,\Theta_r] = h_r(\X_2-\X_1') \, .
\]
\end{lem}

\begin{proof} 
The first two relations follow from the definition of $\Theta_r$ and an easy
computation in the $\USd(\glnn{2})$ thick calculus \cite{KLMS}. The proof of the
third relation is a computation analogous to the one appearing in the proof of
\cite[Proposition 5.7]{RW}. We omit the details here (they appear in
\cite[\HRWDotslide]{HRW1}, where a more-general result is established).
\end{proof}

We now pass from the homotopies $\Theta_r$, which are related to
$h\Delta$-curvature \eqref{eq:h(X-X)}, to homotopies related to the 
$\Delta e$-curvature \eqref{eq:e(X)-e(X)}.

\begin{lem}\label{lemma:crossing Psi} 
For each $m\geq 1$, let $\Psi_m,\Psi_m'\in \End_{\CS(\SSBim)}(C_{a,b})$ 
be given by
\[
\Psi_m := \sum_{r=1}^m (-1)^{r-1}e_{m-r}(\leftX_2)\Theta_r
\, , \quad
\Psi_m' := \sum_{r=1}^m (-1)^{r-1}e_{m-r}(\rightX_2)\Theta_r \, .
\]
These endomorphisms satisfy
\begin{gather*}
\Psi_i^2 =0={\Psi_i'}^2 
\, , \quad 
\Psi_i \Psi_j+\Psi_j \Psi_i=0
\, , \quad 
\Psi_i \Psi_j'+\Psi_j' \Psi_i=0
\, , \quad 
\Psi_i' \Psi_j'+\Psi_j' \Psi_i'=0 \, , \\
[\d,\Psi_m] = e_m(\leftX_2)-e_m(\rightX_1)
\, , \quad 
[\d,\Psi_m'] = e_m(\rightX_2)-e_m(\leftX_1) \, .
\end{gather*}
\end{lem}
\begin{proof}
The first four relations are immediate from Lemma \ref{lem:crossing Theta}.
For the penultimate relation, we compute
\[
[\d,\Psi_m] = \sum_{r=1}^m (-1)^{r-1}e_{m-r}(\leftX_2)[\d,\Theta_r] \\
	= \sum_{r=1}^m (-1)^{r-1}e_{m-r}(\X_2)h_r(\leftX_2-\rightX_1) 
	\stackrel{\eqref{eq:somerelations0}}{=} e_m(\leftX_2)-e_m(\rightX_1) \, .
\]
For the last relation, 
we first note that
\[
e_m(\X_2')-e_m(\leftX_1) 
\stackrel{\eqref{eq:somerelations0}}{=} \sum_{r=1}^m (-1)^{r-1} e_{m-r}(\rightX_2) h_r(\rightX_2-\leftX_1) 
= \sum_{r=1}^m (-1)^{r-1} e_{m-r}(\rightX_2) h_r(\leftX_2-\rightX_1) \, .
\]
Here, we use that $\rightX_2 - \leftX_1 = \leftX_2 - \rightX_1$ 
when acting on ${}_{b,a}\SSBim_{a,b}$.
We then have
\[
[\d,\Psi_m'] = \sum_{r=1}^m (-1)^{r-1}e_{m-r}(\rightX_2)[\d,\Theta_r]
= \sum_{r=1}^m (-1)^{r-1}e_{m-r}(\rightX_2)h_r(\leftX_2-\rightX_1)
= e_m(\rightX_2)-e_m(\leftX_1) \, . \qedhere
\]
\end{proof}

\begin{rem}\label{rem:Psiou}
The homotopies $\Psi_k^{\mathsf{o}}$ and $\Psi_k^{\mathsf{u}}$ appearing 
in \S \ref{ss:curvRick} are $\Psi_k$ and $-\Psi'_k$ from 
Lemma \ref{lemma:crossing Psi}, respectively.
\end{rem}

Lemma \ref{lemma:crossing Psi} implies that $\Psi_m$ is closed for $m>a$, since
$\leftX_2$ and $\rightX_1$ have cardinality $a$. Similarly, $\Psi_m'$ is closed
for $m>b$. The following says that these closed endomorphisms are in fact zero.
Analogously, it shows that $\Theta_m$ is ``redundant'' for $m>a$ or $m>b$.

\begin{proposition}\label{prop:Redundancy}
For $r>0$, we have
\[
\Psi_{a+r} = 0 \, ,\quad \Psi_{b+r} = 0
\]
and
\[
\Theta_{a+r} = \sum_{k=1}^{a} (-1)^{a-k} \mathfrak{s}_{(r-1|{a-k})}(\leftX_2) \Theta_k
\, , \quad
\Theta_{b+r} = \sum_{k=1}^{b} (-1)^{b-k} \mathfrak{s}_{(r-1|{b-k})}(\rightX_2) \Theta_k \, .
\]
\end{proposition}
\begin{proof}
Let us distinguish the alphabets $\B^{(k)}$ and $\B^{(k+1)}$ 
(of cardinality $k$ and $k+1$), living on webs $C_{a,b}^k$ and $C_{a,b}^{k+1}$.  
Note that we may regard $\Hom_{\SSBim}(C_{a,b}^{k+1},C_{a,b}^k)$ 
as a module over $\Sym(\leftX_2|\rightX_2|\B^{(k)}|\B^{(k+1)})$.
We also have actions of symmetric functions in the alphabets
\[
\leftM^{(k)}:=\leftX_2-\B^{(k)}
\, , \quad \rightM^{(k)}:=\rightX_2-\B^{(k)}
\, , \quad \leftM^{(k+1)}=\leftX_2-\B^{(k+1)}
\, , \quad \rightM^{(k+1)}:=\rightX_2-\B^{(k+1)} \, .
\]
Let $M \subset \Hom_{\SSBim}(C_{a,b}^{k+1},C_{a,b}^k)$ be the
$\Sym(\leftX_2|\rightX_2|\B^{(k)}|\B^{(k+1)})$-submodule generated by the
elements $\{\chi_r^-\}_{r\geq 0}$ and let $\D=\{x\}$ denote the alphabet on the
cup of $\chi^-_0$. Since $\chi^-_{r-1}=x^{r-1}\cdot \chi^-_{0} =
h_{r-1}(\D)\cdot \chi^-_{0}$, we may regard $M$ as a module over
$\Sym(\leftX_2|\rightX_2|\B^{(k)}|\B^{(k+1)}|\D)$. It follows (e.g. from the
``movie'' following \eqref{eq:chi minus}, or the corresponding foam) that
\[
\D = \leftM^{(k)} - \leftM^{(k+1)} = \B^{(k+1)} - \B^{(k)} =  {\rightM}^{(k)} - {\rightM}^{(k+1)}
\]
when acting on $\chi^-_0$. 
Observe that $\Theta_m|_{C_{a,b}^k} = (-1)^{k} h_{m-1}(\D)\cdot \chi^-_{0}$, so
\begin{align*}
\Psi_m|_{C_{a,b}^k} 
= \sum_{r=1}^m(-1)^{r-1} e_{m-r}(\leftX_2)\Theta_r
&=(-1)^{k} \sum_{r=1}^m(-1)^{r-1} e_{m-r}(\leftX_2)h_{r-1}(\D)\cdot \chi^-_{0}\\
&=(-1)^{k} e_{m-1}(\leftX_2-\D)\cdot \chi^-_{0}
\end{align*}
Now, $\leftX_2-\D = \M^{(k+1)} + \B^{(k)}$ has cardinality $a-1$, 
which implies that $\Psi_m=0$ for $m>a$. 
A similar argument shows that
$\Psi_m'|_{C_{a,b}^k} = (-1)^{a+b-k} e_{m-1}(\rightX_2-\D)\cdot \chi^-_{0}$
which is zero for $m>b$ since $\rightX_2-\D$ has cardinality $b-1$.
Finally, the identities for $\Theta_{a+r}$ and $\Theta_{b+r}$ follow from $h$-reduction,
i.e. Lemma \ref{lemma:h reduction}. 
For instance:
\begin{align*}
\Theta_{a+r}|_{C_{a,b}^k} = (-1)^{k} h_{a+r-1}(\D)\cdot \chi^-_{0}
&= \sum_{i=1}^{a} (-1)^{a-i} \mathfrak{s}_{(r-1|{a-i})}(\leftX_2) (-1)^{k} h_{i-1}(\D)\cdot \chi^-_{0}\\
&= \sum_{i=1}^{a} (-1)^{a-i} \mathfrak{s}_{(r-1|{a-i})}(\leftX_2) \Theta_{i}|_{C_{a,b}^k}
\end{align*}
where we have used \eqref{eq:specialrel2} with $\X=\D$, $\Y=\X_2-\D$, and $c=a-1$
on the first line. 
The relation for $\Theta_{b+r}$ is proven similarly.
\end{proof}

\begin{remark}
Proposition \ref{prop:Redundancy} shows that, in $\End_{\CS(\SSBim)}(C_{a,b})$, 
the elements $\Theta_r$ and $\Psi_r$ can all be written as 
$\Sym(\leftX_2|\rightX_2)$-linear combinations of $\Theta_1,\ldots,\Theta_{\min\{a,b\}}$.
\end{remark}

\begin{proof}[Proof of Proposition \ref{prop:curved crossing}]
The fact that $\d\circ \d=0$ is well-known (the Rickard complex is a complex).
We write 
$\Delta = \sum_{m\geq 1}(\Psi_m u_{\sigma(1),m} - \Psi'_m u_{\sigma(2),m})$, 
where we again set $u_{\sigma(i),m}=0$ for $m>a_{i}$. Lemma
\ref{lemma:crossing Psi} implies that $\Delta\circ \Delta = 0$ and that
\[
[\d,\Delta] = \sum_{m\geq 1}u_{\sigma(1),m}\Big(e_m(\leftX_2)-e_m(\rightX_1)\Big)
	-\sum_{m\geq 1} u_{\sigma(2),m}\Big(e_m(\rightX_2)-e_m(\leftX_1)\Big) \, ,
\]
which is as desired.
\end{proof}

We now use Proposition \ref{prop:curved crossing} to assign a curved 
Rickard complex to any colored braid.
The input for this construction is a $\Z_{\geq 1}$-colored braid word $\beta_{\aa}$ 
(in the sense of \S\ref{sec:colbraid}) and a numbering of the 
strands\footnote{Thus, one can think of such braids as being 
$\Z_{\geq 1}\times \Z_{\geq 1}$-colored.} in the braid. 
If $\beta \in \Br_m$, then we require that $\beta_{\aa}$ 
has strands numbered from $1$ to $m$. 
Reading the sequence of strand numbers on the incoming and outgoing boundaries 
(respectively) defines two permutations $\sigma,\tau\in \symg_m$ 
such that $\tau=\beta\circ \sigma$. 
(Here, and in the following, we abuse notation and write $\beta$ for the permutation induced
by the braid $\beta$ under the canonical homomorphism $\Br_k \to \symg_k$.)
We will denote the data of a $\Z_{\geq 1}$-colored braid
$\beta_{\aa}={}_{\bb}\beta_{\aa}$ with numbered strands by
$_{\bb,\tau}\beta_{\aa,\sigma}$ or by $\beta_{\aa,\sigma}$ (since $\tau = \beta \circ \sigma$). 
Further, we will occasionally omit the numbering from our notation 
when it is not locally relevant.

\begin{thm}
	\label{thm:braidinvariant} 
For each braid $\b\in \Br_m$, each $\aa\in \Z_{\geq 1}^m$, and each 
$\sigma\in \symg_m$, 
the Rickard complex ${C}(\b_{\aa})\in \CS(\SSBim)$ has a curved lift to a 1-morphism 
$\YS C(\b_{\aa,\sigma})$ from $(\aa,\sigma)$ to $(\b(\aa),\b\circ\sigma)$ in $\YS(\SSBim)$
that is unique up to homotopy equivalence.
In particular, the assignment $\b_{\aa,\sigma}\mapsto \YS C (\b_{\aa,\sigma})$ 
satisfies the (colored) braid relations up to homotopy.
\end{thm}

Before the proof, we remark that the curvature equation for 
$\YS C(\b_{\aa,\sigma})$ is
\[
(\d+\Delta)^2=\sum_{i=1}^m \sum_{r=1}^{\infty}  (e_r(\leftX_{\tau(i)})-e_r(\rightX_{\sigma(i)})) \uvar_{i, r}
\]
for $\tau = \beta \circ \sigma$.

\begin{proof}
Since ${C}(\b_{\aa})\in \CS(\SSBim)$ is an invertible 1-morphism and the Rickard
complexes assigned to colored braids satisfy the (colored) braid relations up to
homotopy, this is an immediate consequence of
Lemma~\ref{lem:curvinginvertibles}. However, since we only sketched the
existence portion of that proof, we give an explicit construction here.

First, we prove the theorem for the identity braid $\oone_{\aa}$ 
and any chosen $\sigma\in \symg_m$. 
In this case,
\[
\End_{\CS(\SSBim)[\U]}(\oone_{\aa}) = \Sym(\X_1|\cdots|\X_m)\otimes \k[\U]
\]
is supported in even cohomological degrees, 
so any (degree $1$) differential on $\oone_{\aa}$ must be zero for degree reasons. 
In particular, $\d_{\oone_{\aa}} = 0$, so the curvature equation for $\oone_{\aa}$ becomes
\[
(\Delta_{\oone_{\aa}})^2 =
\sum_{i=1}^m \sum_{r=1}^{\infty} (e_r(\X_{\sigma(i)})-e_r(\X_{\sigma(i)}')) u_{i, r} = 0 \, .
\]
Here, we use that $f(\X_{\sigma(i)})-f(\X_{\sigma(i)}')$ acts by zero on the
identity bimodule $\oone_{\aa}$ for any symmetric function $f$. 
This equation has the unique solution $\Delta_{\oone_{\aa}}=0$, 
so the curved Rickard complex associated to the identity braid is just $\oone_{\aa}$ 
with $\d^{\tot} = \d + \Delta = 0$, regardless of our choice of $\sigma$.

We next give a constructive proof of the existence of the curved Rickard complex
associated to a non-trivial braid (word).
For the Artin generator $\agen_i$ and the identity permutation 
$\id \in \symg_m$, we define
\[
\YS C((\agen_i)_{\aa,\id}) :=\oone_{(a_1,\ldots a_{i-1}),\id} 
\boxtimes \YS C_{(a_i,a_{i+1}),\id}\boxtimes \oone_{(a_{i+2},\ldots a_{m}),\id}  
\]
where the middle tensor factor is the $2$-strand curved Rickard complex from 
Proposition \ref{prop:curved crossing}.
For any other $\sigma\in \symg_m$, 
we obtain $\YS C((\agen_i)_{\aa,\sigma})$ from $\YS C((\agen_i)_{\aa,\id})$ 
by the substitution $u_{i,r} \mapsto u_{\sigma^{-1}(i),r}$. 
Analogously, we define $\YS C((\agen_i\inv)_{\aa,\sigma})$ 
using $\YS C^\vee_{(a_i,a_{i+1}),\id}$.
For a general braid word $\beta = \agen_{i_r}^{\e_r} \cdots \agen_{i_1}^{\e_1}$, 
we define $\YS C(\b_{\aa,\sigma})$ by taking the horizontal composition of
the curved Rickard complexes assigned to its constituent Artin generators 
$\YS C((\agen_{i_j}^{\e_j})_{\tau_j(\aa),\tau_j\circ \sigma})$, 
in analogy to \eqref{eq:Rick-hComp}.
Here, $\tau_j$ is the permutation associated with the braid 
$\agen_{i_{j-1}}^{\e_{j-1}} \cdots \agen_{i_r}^{\e_r}$.
Lemma \ref{lem:hComp} guarantees that $\YS C(\b_{\aa,\sigma})$ is a
well-defined $1$-morphism in $\YS(\SSBim)$.
It is the unique curved lift of $C(\b_{\aa})$ associated with the permutation $\sigma$ 
by Lemma~\ref{lem:curvinginvertibles}, since Proposition~\ref{prop:rickard invariance}
gives that $C(\b_{\aa}) \in \CS(\SSBim)$ is an invertible $1$-morphism.

Finally, we show that the assignment 
$\b_{\aa,\sigma}\mapsto \YS C(\b_{\aa,\sigma})$ satisfies the braid relations, 
up to homotopy equivalence. 
Suppose that $\b_{\aa,\sigma}$ and $\gamma_{\aa,\sigma}$ 
are two braid words representing the same colored braid. 
By Proposition~\ref{prop:rickard invariance}, we have a homotopy equivalence
$C(\b_{\aa,\sigma})\simeq C(\gamma_{\aa,\sigma})$.
Proposition~\ref{prop:conservation} provides a lift to a 
homotopy equivalence between $\YS C(\b_{\aa,\sigma})$ and some curved lift of
$C(\gamma_{\aa,\sigma})$. However, by Lemma~\ref{lem:curvinginvertibles}, 
such lifts are unique up to homotopy equivalence, and so
$\YS C(\b_{\aa,\sigma}) \simeq \YS C(\gamma_{\aa,\sigma})$.
\end{proof}

We pause to record a useful observation, which will help to establish 
a well-defined module structure on our deformed colored link homology.
By construction, the complex $\YS C(\b_{\aa,\sigma})$ built in 
Theorem \ref{thm:braidinvariant} is a strict $1$-morphism, 
in the sense of Remark \ref{rmk:strict}. 
As such, the linear part of the curved Maurer--Cartan element gives homotopies 
$\Psi_{i,r} \in \End_{\YS(\SSBim)}\big( \YS C(\b_{\aa,\sigma}) \big)$ 
for $1 \leq i \leq m$ and $1 \leq r \leq a_{\sigma(i)}$
that square to zero, pairwise anti-commute, and satisfy
\[
[\d_{\YS C(\b)} + \Delta_{\YS C(\b)} , \Psi_{i,r}] 
= e_r(\leftX_{\tau(i)})-e_r(\rightX_{\sigma(i)}) \, .
\]
It follows that, for each $1 \leq i \leq m$, 
$\YS C(\b_{\aa,\sigma})$ is a well-defined dg module over the ($i^{th}$) 
\emph{two point} dg algebra
\begin{equation}\label{eq:manytwopointalg}
A^i_{L,R} :=
\Sym(\leftX_{\tau(i)} | \rightX_{\sigma(i)}) 
	\otimes \largewedge(\Psi_{i,1},\ldots,\Psi_{i,a_{\sigma(i)}})
\, , \quad d(\Psi_{i,r}) = e_r(\leftX_{\tau(i)})-e_r(\rightX_{\sigma(i)}) \, ,
\end{equation}
and further that these actions assemble to give a dg 
$(\bigotimes_i A^i_{L,R})$-module structure on $\YS C(\b_{\aa,\sigma})$.
Let $\W_i$ be an alphabet with $|\W_i| = |\leftX_{\tau(i)}| = |\rightX_{\sigma(i)}|$.
By restricting the action of \eqref{eq:manytwopointalg} to the left and right alphabets,
we obtain two (possibly) distinct actions of the algebra $\Sym(\W_i)$
using the identifications
\[
\Sym(\W_i) \cong \Sym(\leftX_{\tau(i)}) \, , \quad 
	\Sym(\W_i) \cong \Sym(\rightX_{\sigma(i)}) \, .
\]
Our next result shows that these actions on $\YS C(\b_{\aa,\sigma})$
give quasi-isomorphic dg $\Sym(\W_i)$-modules.
We state the result in slightly greater generality.

\begin{prop}\label{prop:qiso}
Let $X$ be a dg $(\bigotimes_{i=1^m} A^i_{L,R})$-module 
(e.g. a strict $1$-morphism in $\YS(\SSBim)$),
then the induced left and right actions of $\Sym(\W_i)$ give
quasi-isomorphic dg $\Sym(\W_i)$-modules
(thus quasi-isomorphic dg $(\bigotimes_i \Sym(\W_i))$-modules).
\end{prop}

\begin{proof}
It suffices to show that, given a dg module over
\[
\k[e^L_1,\ldots,e^L_a,e^R_1,\ldots,e^R_a]
	\otimes \largewedge(\Psi_1 ,\ldots,\Psi_a) \, , \quad
d(\Psi_i) = e^L_i - e^R_i \, ,
\]
the corresponding left and right $\k[e_1,\ldots,e_a]$-modules are quasi-isomorphic. 
Correspondingly, this immediately reduces to the $1$-variable case.

Thus, let $C$ be a dg module over
\[
A=\k[x_L,x_R] \otimes \largewedge(\xi) \, , \quad d(\xi) = x_L - x_R
\]
and consider $\cal{M} := C \otimes_A  \cal{T}$ where 
$\cal{T} := \k[z_L,z_R,z_M,T] \otimes \largewedge(\xi_{LR},\xi_{RM},\xi_{LM})$
is a dg $A$-algebra via the inclusion $A \hookrightarrow \cal{T}$ sending
$x_L \mapsto z_L $, $x_R \mapsto z_R$, and $\xi \mapsto \xi_{LR}$.
The remaining differentials on $\cal{T}$ are given by
\[
d(\xi_{RM}) = z_R - z_M \, , \quad d(\xi_{LM}) = z_L - z_M \, , \quad d(T) = \xi_{RM} - \xi_{LM} + \xi_{LR} \, .
\]
Now, view $\cal{M}$ as a dg $\k[x]$-module, where $x$ acts via $z_M$, 
and let $C_i$ denote the dg $\k[x]$-module $C$ wherein $x$ acts via $x_i$, for $i=L,R$.
The deformation retract $\cal{M} \to C_L$ determined by 
\[
z_M \mapsto x_L \, , \quad \xi_{LM} \mapsto 0 \, , \quad \xi_{RM} \mapsto -\xi \, , \quad T \mapsto 0
\]
is $\k[x]$-linear; however, its homotopy inverse is not. 
Thus it (only) gives a quasi-isomorphism of $\k[x]$-modules $\cal{M} \qiso C_L$. 
Similarly, there is a deformation retract $\cal{M} \to C_R$ determined by 
\[
z_M \mapsto x_R \, , \quad \xi_{LM} \mapsto \xi \, , \quad \xi_{RM} \mapsto 0 \, , \quad T \mapsto 0
\]
that gives a quasi-isomorphism of $\k[x]$-modules $\cal{M} \qiso C_R$. 
We thus have $C_L \qiso \cal{M}\qiso C_R$.
\end{proof}

\subsection{Alphabet soup I: from \texorpdfstring{$u$'s to $v$'s}{u's to v's}}
\label{ss:v-curved cxs}
As discussed in \S \ref{ss:curvature discussion} (and implicitly used in Lemma
\ref{lemma:crossing Psi}), we will find it convenient to translate
back-and-forth between $\Delta e$- and $h\Delta$-curvatures, which are modeled
on $\sum (e_k(\leftX) - e_k(\rightX))u_k$ and $\sum h_k(\leftX-\rightX)v_k$
respectively. Indeed, equation \eqref{eq:HCompDef0} and Lemma \ref{lem:hComp}
show that complexes with $\Delta e$-curvature possess a straightforward
horizontal composition, while Lemma \ref{lem:crossing Theta} suggests that
complexes with $h \Delta$-curvature appear more regularly ``in the wild.''

\begin{defi}\label{def:Vs}
Let $\aa=(a_1,\ldots,a_m)$ and $\bb=(b_1,\ldots,b_m)$ be objects in $\SSBim$,
and let $\sigma,\tau \in \symg_m$ be such that $a_{\sigma(i)} = b_{\tau(i)}$.
For each $1\leq i\leq m$, introduce deformation parameters
$\V_i = \{v_{i,r}\}_{r=1}^{a_{\sigma(i)}}$ with $\wt(v_{i,r}) = \qdeg^{-2r} \tdeg^2$.
Let
\[
{}_{\bb,\tau}\VS(\SSBim)_{\aa,\sigma}
:=
\CS_Z\big( {}_{\bb}\SSBim_{\aa} ; \k[\V_1,\ldots,\V_{m}]\big)
\]
where the curvature element is
\begin{equation}\label{eq:Vcurv}
Z = \sum_{i=1}^{m} \sum_{r=1}^{a_{\sigma(i)}} 
h_r(\leftX_{\tau(i)} - \rightX_{\sigma(i)}) v_{i,r} \, .
\end{equation}
\end{defi}

We will abbreviate by writing
$\V = \V_1 \cup \cdots \cup \V_m$ when $m$ is understood, 
thus $\k[\V] = \k[\V_1 , \ldots ,\V_m]$. 
Similarly, we will write
\[
\VS(\SSBim) := \bigsqcup_{(\bb,\tau),(\aa,\sigma)} 
{}_{\bb,\tau}\VS(\SSBim)_{\aa,\sigma}
\]
which we understand simply as a disjoint union of categories.

\begin{prop}\label{prop:YSSBimToVSSBim}
Retaining the setup from Definition \ref{def:Vs}, 
there is an isomorphism
\[
{}_{\bb,\tau}\VS(\SSBim)_{\aa,\sigma} \cong
{}_{\bb,\tau}\YS(\SSBim)_{\aa,\sigma}
\]
of dg categories determined by the mutually inverse assignments
\begin{equation}\label{eq:generalVtoU}
v_{i,k} \mapsto
(-1)^{k-1}\sum_{l=k}^{a_{\sigma(i)}} e_{l-k}(\X_{\tau(i)}) u_{i,l}
\, , \quad
u_{i,k} \mapsto
(-1)^{k-1}\sum_{l=k}^{a_{\sigma(i)}} h_{l-k}(\X_{\tau(i)}) v_{i,l} \, ,
\end{equation}
(cf. Definition \ref{def:U and V}).
\end{prop}
\begin{proof}
The functor ${}_{\bb,\tau}\VS(\SSBim)_{\aa,\sigma} \to
{}_{\bb,\tau}\YS(\SSBim)_{\aa,\sigma}$
is defined on curved complexes 
$\tw_{\Delta_X}(X) = (X,\d_X + \Delta_X)$ by sending $X$ to itself, 
and defined on morphisms
\[
\Hom_{\SSBim[\tdeg^\pm]}(X,Y) \otimes \k[\V] := \Hom_{\VS(\SSBim)}(X,Y)
\to
\Hom_{\YS(\SSBim)}(X,Y) =: \Hom_{\SSBim[\tdeg^\pm]}(X,Y) \otimes \k[\U]
\]
using the first substitution rule in \eqref{eq:generalVtoU}. In particular, this
determines the image of the curved Maurer--Cartan element $\Delta_X$. This is
well-defined since Corollary \ref{cor:curvature description} implies that it
takes curved complexes with $h\Delta$-curvature \eqref{eq:Vcurv} to those with
$\Delta e$-curvature \eqref{eq:Ucurv}. The functor in the other direction is
defined analogously using the second substitution rule in \eqref{eq:generalVtoU}. As for
\eqref{eq:VvsU}, a computation using \eqref{eq:HE2} shows that the substitutions
\eqref{eq:generalVtoU}, and thus the associated functors, are mutually inverse.
\end{proof}

\begin{remark}\label{rem:subtleV}
By pulling back structure from $\YS(\SSBim)$, the categories 
${}_{\bb,\tau} \VS(\SSBim)_{\aa,\sigma}$ assemble to give a monoidal dg 2-category.  
We will not record the precise formulae for the various operations, 
but do wish to point out a subtlety regarding the action of the deformation parameters $\V$. 
Since $\End_{\VS(\SSBim)}(\oone_{\bb}) \cong \Sym(\leftX_1 | \cdots | \leftX_m)[\V]$ 
is a module over $\k[\V]$, any curved complex ${}_{\bb}X \in \VS(\SSBim)$ inherits an action of the latter 
by acting \textbf{on the left}. Specifically, given $g \in \k[\V]$, we let $g \cdot \id_X$ denote the endomorphism
\[
X \cong \oone_{\bb} \hComp X \xrightarrow{g \hComp \id_X} \oone_{\bb} \hComp X \cong X \, .
\]
However, we could also consider the action of such $g$ act on the right, via
\[
X \cong X \hComp \oone_{\aa} \xrightarrow{\id_X \hComp g} X \hComp \oone_{\aa} \cong X \, .
\]
These actions \textbf{do not} necessarily agree. 
Rather, they are related by an appropriate analogue of Lemma \ref{lemma:sliding v}. 
Specifically, on ${}_{\bb,\tau}X_{\aa,\sigma}$ we have that 
\[
\id_X \hComp v_{i,r} = 
\sum_{r \leq l \leq a_{\sigma(i)}} h_{l-r}(\leftX_{\tau(i)} - \rightX_{\sigma(i)}) \cdot (v_{i,l} \hComp \id_X) \, .
\]
However, the action of $h_{l-r}(\leftX_{\tau(i)} - \rightX_{\sigma(i)})$ 
on any $1$-morphism in $\VS(\SSBim)$ is null-homotopic (see e.g. Remark \ref{rmk:strict}), 
thus the left and right actions of $\V$ are always homotopic.
\end{remark}

\begin{conv}
Henceforth, we will not distinguish between the dg categories 
${}_{\bb,\tau} \VS(\SSBim)_{\aa,\sigma}$ and ${}_{\bb,\tau}\YS(\SSBim)_{\aa,\sigma}$, 
and in each instance will use the notation coinciding with the relevant deformation 
parameters $\U$ or $\V$.
\end{conv}

In \S \ref{s:curved skein rel} -- \S \ref{s:hopf link}, 
we will be particularly interested in the $\sigma=\id=\tau$ case of Definition \ref{def:Vs}.
To simplify notation, we will use the shorthand
\begin{equation}\label{eq:Vnotation}
\VS_{\aa} := {}_{\aa,\id}\VS(\SSBim)_{\aa,\id}
= \CS_Z\big( {}_{\aa}\SSBim_{\aa} ; \k[\V_1,\ldots,\V_{m}]\big)
\end{equation}
for the category of curved complexes of singular Soergel bimodules
with curvature
\[
Z = \sum_{i=1}^{m} \sum_{r=1}^{a_i} h_r(\leftX_i - \rightX_i) v_{i,r}\, .
\]

\subsection{Alphabet soup II: from \texorpdfstring{$y$'s to $u$'s and $v$'s}{y's to u's and v's}}
\label{ss:ytouv}

In the special (uncolored) case of $\aa = 1^m = \bb$, the categories
${}_{1^{m},\tau}\YS(\SSBim)_{1^{m},\sigma}$ recover (a version of) the category
$\YS(\SBim)$ of curved complexes of Soergel bimodules from \cite{GH}. Setting
$\leftX = \{x_1,\ldots,x_m\}$ and $\rightX = \{x'_1,\ldots,x'_m\}$, this
incarnation of $\YS(\SBim)$ is the category of curved complexes of
$(\k[\leftX],\k[\rightX])$-bimodules (equivalently,
$\k[\leftX,\rightX]$-modules) with ``thin'' curvature
\begin{equation}\label{eq:x-x}
\sum_{i=1}^m (x_{\tau(i)} - x_{\sigma(i)}) y_i
\end{equation}
for deformation parameters $\Y = \{y_1,\ldots,y_m \}$.
These parameters can be understood as equaling either the $u$'s from 
Definition \ref{def:Y} or the $v$'s from Definition \ref{def:Vs}
(since $h_1(x_i - x'_i) = x_i-x'_i = e_1(x_i) - e_1(x'_i)$).
Note that $\k[\leftX,\rightX]$-modules can be viewed as 
$\Sym(\leftX | \rightX)$-modules, by restricting the left and right actions.
In terms of $\SSBim$, this corresponds to horizontally composing Soergel bimodules $X$ 
by appropriate merge and split bimodules which, on each side of $X$, 
merge the boundaries to a single $m$-colored strand, e.g. 
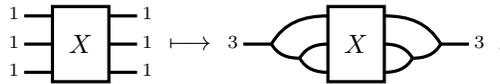
\[
\begin{tikzpicture}[scale=.5,smallnodes,anchorbase]
	\draw[very thick] (-1.5, .75) node[left=-2pt]{$1$} to (-.75, .75);
	\draw[very thick] (-1.5, 0) node[left=-2pt]{$1$} to (-.75,0);
	\draw[very thick] (-1.5, -.75) node[left=-2pt]{$1$} to (-.75, -.75);
	\node[yshift=-2pt] at (0,0) {\normalsize$X$};
	\draw[very thick] (1.5, .75) node[right=-2pt]{$1$} to (.75, .75);
	\draw[very thick] (1.5, 0) node[right=-2pt]{$1$} to (.75,0);
	\draw[very thick] (1.5, -.75) node[right=-2pt]{$1$} to (.75, -.75);	
		\draw[very thick] (.75,1) rectangle (-.75,-1);
\end{tikzpicture}
\longmapsto
\begin{tikzpicture}[scale=.5,smallnodes,anchorbase,yscale=-1]
	\draw[very thick] (-3,0) node[left=-2pt]{$3$} to (-2.25,0);
	\draw[very thick] (-2.25,0) to [out=60,in=180] (-1.5, .375);
	\draw[very thick] (-2.25,0) to [out=300,in=180] (-.75, -.75);
	\draw[very thick] (-1.5, .375) to [out=60,in=180] (-.75, .75);
	\draw[very thick] (-1.5, .375) to [out=300,in=180] (-.75,0);
	\node[yshift=-2pt] at (0,0) {\normalsize$X$};
	\draw[very thick] (3,0) node[right=-2pt]{$3$} to (2.25,0);
	\draw[very thick] (2.25,0) to [out=120,in=0] (1.5, .375);
	\draw[very thick] (2.25,0) to [out=240,in=0] (.75, -.75);
	\draw[very thick] (1.5, .375) to [out=120,in=0] (.75, .75);
	\draw[very thick] (1.5, .375) to [out=240,in=0] (.75,0);
	\draw[very thick] (.75,1) rectangle (-.75,-1);
\end{tikzpicture} \, .
\]
Since colored links can be interpreted as the ``(anti)symmetric part'' of
appropriate cables of links in a similar manner 
(recall Theorem \ref{thm:cables} and Conjecture \ref{conj:cables}), 
we will find it fortuitous to relate the thin
curvature in \eqref{eq:x-x} to the $h \Delta$-curvature from \eqref{eq:h(X-X)}
and the $\Delta e$-curvature from \eqref{eq:e(X)-e(X)}. Indeed, certain
endomorphisms that appear naturally in this story (see \S \ref{ss:interpolation}
below) are crucial in our investigation of the colored link splitting map in \S
\ref{sec:splitting}.

To this end, let $a \geq 1$ and fix alphabets
\[
\leftX = \{x_1,\ldots, x_a\}
\, , \quad
\rightX= \{x'_1,\ldots,x'_a\}
\, , \quad
\Y=\{y_1,\ldots,y_a\}
\, , \quad
\U=\{u_1,\ldots,u_a\}
\, , \quad
\V=\{v_1,\ldots,v_a\} \, .
\]
Let $\AS$ denote the category of $\k[\leftX,\rightX]$-modules
and consider the following categories of curved complexes
\begin{gather*}
\VS\AS:=\CS_{\sum_{k=1}^a h_k(\leftX - \rightX) v_k}(\AS,\k[\V])
\, , \quad
\YS\AS:=\CS_{\sum_{k=1}^a (e_k(\leftX) - e_k(\rightX))u_k}(\AS,\k[\U]) \\
\mathsf{Y} \! \AS:= \CS_{\sum_{i=1}^a(x_i-x_i')y_i}(\AS, \k[\Y]) \, .
\end{gather*}
Writing $h_k(\leftX - \rightX)$ and $e_k(\leftX) - e_k(\rightX)$ 
as $\k[\leftX,\rightX]$-linear combinations of the elements $x_i-x'_i$
produces dg functors from $\mathsf{Y} \! \AS$ to $\VS\AS$ and $\YS\AS$, respectively.
Precisely, we have the following.

\begin{proposition}\label{prop:y to u and v}
The dg categories $\VS\AS$ and $\YS\AS$ are isomorphic via the mutually inverse
substitutions
\[
v_k \mapsto (-1)^{k-1}\sum_{k \leq l \leq a} e_{l-k}(\X)u_l
\, , \quad
u_k \mapsto (-1)^{k-1}\sum_{k \leq l \leq a} h_{l-k}(\X)v_l \, .
\]
Moreover, the substitutions
\begin{equation}\label{eq:y to v}
y_i \mapsto \sum_{l=1}^a h_{l-1} \big( \{x_i,x_{i+1},\ldots,x_a\} - \{x'_{i+1},\ldots,x'_a\} \big) v_l
\end{equation}
and
\begin{equation}\label{eq:y to u}
y_i\mapsto \sum_{l=1}^a e_{l-1}(x_1,\ldots,x_{i-1},x'_{i+1},\ldots,x'_a) u_l
\end{equation}
determine dg functors 
$\mathsf{Y} \! \AS \to \VS\AS$ and $\mathsf{Y} \! \AS \to \YS\AS$
that are compatible with the isomorphism $\VS\AS\cong \YS\AS$.
\end{proposition}
\begin{proof}
The isomorphism between $\VS\AS$ and $\US\AS$ follows from 
the discussion in \S \ref{ss:curvature discussion}, 
which is the $1$-strand case of Proposition \ref{prop:YSSBimToVSSBim}.

Similarly, the substitutions \eqref{eq:y to v} and \eqref{eq:y to u} define algebra homomorphisms
$\k[\leftX,\rightX,\Y] \to \k[\leftX,\rightX,\V]$ and $\k[\leftX,\rightX,\Y] \to \k[\leftX,\rightX,\U]$
that determine dg functors which are the identity on objects
and are given on morphism complexes by extension of scalars 
$\k[\leftX,\rightX,\V]\otimes_{\k[\leftX,\rightX,\Y]}(-)$ and 
$\k[\leftX,\rightX,\U]\otimes_{\k[\leftX,\rightX,\Y]}(-)$.

To confirm that these functors are indeed well-defined, 
it suffices to show that the algebra maps preserve curvature. 
We first confirm this for the functor $\mathsf{Y} \! \AS \to \YS\AS$. 
To begin, we explicitly write $e_k(\leftX)-e_k(\rightX)$ as an element of 
the ideal generated by $x_i-x'_i \in \k[\leftX,\rightX]$. 
Consider the difference of monomials:
\[
x_{i_1}\cdots x_{i_k} - x'_{i_1} \cdots x'_{i_k} = 
\sum_{j=1}^k x_{i_1}\cdots x_{i_{j-1}} (x_{i_j}-x'_{i_j}) x'_{i_{j+1}} \cdots x'_{i_k} \, .
\]
Summing over all sequences with $1\leq i_1< \cdots < i_k\leq a$ gives the identity
\[
e_k(\leftX)-e_k(\rightX) = 
\sum_{i=1}^a e_{k-1}(x_1,\ldots,x_{i-1},x'_{i+1},\ldots,x'_a) \cdot (x_i-x'_i) \, .
\]
Hence, we compute
\begin{align*}
\sum_{k=1}^a \big( e_k(\leftX) - e_k(\rightX) \big) u_k 
&= \sum_{k=1}^a \sum_{i=1}^a e_{k-1}(x_1,\ldots,x_{i-1},x'_{i+1},\ldots,x'_a) \cdot (x_i-x'_i) \cdot u_k \\
&= \sum_{i=1}^a (x_i-x'_i) \left( \sum_{k=1}^a  e_{k-1}(x_1,\ldots,x_{i-1},x'_{i+1},\ldots,x'_a) u_k \right) \, ,
\end{align*}
which proves that the algebra map $\k[\leftX,\rightX,\Y]\rightarrow \k[\leftX,\rightX,\U]$
defined by \eqref{eq:y to u} preserves curvature.

Since $\YS\AS \cong \VS\AS$, the proof is completed by showing that
the algebra maps $\k[\leftX,\rightX,\Y] \to \k[\leftX,\rightX,\U]$ 
and $\k[\leftX,\rightX,\Y] \to \k[\leftX,\rightX,\V]$ 
are intertwined by the isomorphism $\k[\leftX,\rightX,\U] \cong \k[\leftX,\rightX,\V]$.
Indeed:
\begin{align*}
y_i 
&\mapsto \sum_{k=1}^a  e_{k-1}(x_1,\ldots,x_{i-1},x'_{i+1},\ldots,x'_a) u_k \\
&\mapsto \sum_{1 \leq k \leq l \leq a} (-1)^{k-1} e_{k-1}(x_1,\ldots,x_{i-1},x'_{i+1},\ldots,x'_a) h_{l-k}(x_1,\ldots,x_a) v_l \\
&= \sum_{l=1}^a h_{l-1}\big( \{x_1,\ldots,x_a \} - \{x_1,\ldots,x_{i-1},x'_{i+1},\ldots,x'_a \}\big) v_l \\
&= \sum_{l=1}^a h_{l-1}\big( \{x_i,\ldots,x_a \} - \{x'_{i+1},\ldots,x'_a \}\big) v_l \, . \qedhere
\end{align*}
\end{proof}

\subsection{Alphabet soup III: interpolation coordinates}
\label{ss:interpolation}
Retain the notation from \S \ref{ss:ytouv} and consider the identity bimodule
$\oone_{(1,\ldots,1)} \in \mathsf{Y} \! \AS$. In $\End_{\mathsf{Y} \!
\AS}(\oone_{(1,\ldots,1)})$, we have that $x_i = x'_i$, thus we find that the dg
functor $\mathsf{Y} \! \AS \to \VS\AS$ induces a map
$
\k[\leftX,\Y]\cong \End_{\mathsf{Y} \! \AS}(\oone_{(1,\ldots,1)})
 \to \End_{\VS\AS}(\oone_{(1,\ldots,1)}) \cong \k[\leftX,\V]$ sending
\[y_i \mapsto \sum_{r=1}^a h_{r-1}(x_i) v_r \, .\]
This motivates the following.
\begin{defi}\label{def:ic} 
Set
\begin{equation}\label{eq:y and v}
y_i := \sum_{r=1}^a x_i^{r-1} v_r \in \k[\X,\V] \, .
\end{equation}
In other words, $y_i$ is defined to be the polynomial of degree $|\X|-1$ in
$x_i$ with coefficients $v_r$. We call these coefficients \emph{interpolation
coordinates} and highlight that they are independent of $i$.
\end{defi}
Using \eqref{eq:y and v}, we can view $\k[\X,\Y]$ as a subalgebra of
$\k[\X,\V]$. Note, however, that this is a special case of \eqref{eq:y to v},
which is only compatible with Proposition \ref{prop:y to u and v} when $\leftX =
\rightX$. Nevertheless, symmetric functions in the alphabet $\Y$ give
well-defined elements of $\Sym(\X)[\V] \cong \End_{\VS(\SSBim)}(\oone_{(a)})$,
thus we can consider them as operators acting (on the left or right) on suitable
1-morphisms in $\VS(\SSBim)$. For the duration of the paper, the variables $y_i$
will always be understood in this context. (See Remark \ref{rem:subtleV} above,
which addresses a subtle point concerning these actions.)

Our terminology in Definition \ref{def:ic} is chosen since we 
can express the $v_i$ in terms of $x_i,y_i$ by formulae familiar from interpolation theory.
We now make this precise,
and establish further identities involving the interpolation coordinates.
We will make use of identities from \S \ref{ss:h reduction} 
and the Haiman determinants from \S \ref{ss:haiman dets}.

\begin{example}\label{ex:2var}
When $a=2$, the elements $y_i\in \k[x_1,x_2,v_1,v_2]$ are given by 
$y_i =v_1 + x_i v_2$ and
\[
v_2= \frac{y_1-y_2}{x_1-x_2} =  \frac{\hdet(y,1)}{x_1-x_2} 
	= -\frac{\Delta_{\mathcal{M}_1}(\X,\Y)}{\Delta(\X)} \, , \quad 
v_1 = \frac{x_1y_2-x_2y_1}{x_1-x_2} 
	= \frac{\hdet(x,y)}{x_1-x_2} = \frac{\Delta_{\mathcal{M}_2}(\X,\Y)}{\Delta(\X)}
\]
for $\mathcal{M}_1=\{1,y\}$ and $\mathcal{M}_2=\{x,y\}$.
\end{example}

Generalizing Example \ref{ex:2var},
the following gives the general rule to recover the interpolation coordinates
from the $x_i$ and $y_i$.
(See also Lemma \ref{lemma:v from y2} below for another formulation.)

\begin{lemma}\label{lemma:v from y}
Let $\X=\{x_1,\ldots,x_a\}$ and $\Y=\{y_1,\ldots,y_a\}$, then we have
\[
v_{a-k+1} = \frac{\hdet(x^{a-1},\dots,x^{a-k+1},y,x^{a-k-1},\dots x^0)}{\hdet(x^{a-1},\dots, x^0)} 
	=  (-1)^{a-k}\frac{\Delta_{\mathcal{M}_k}(\X,\Y)}{\Delta(\X)}
\]
where $\mathcal{M}_k=\{x^{a-1},\ldots,\widehat{x^{a-k}},\ldots,1, y\}$.
\end{lemma}

\begin{proof}
Consider the matrix defining the determinant 
$\hdet(x^{a-1},\dots,x^{a-k+1},y,x^{a-k-1},\dots x^0)$ 
and rewrite its $k$th row as:
\[
\sqmatrix{y_1&  \cdots  & y_a} 
= \sum_{r=1}^a v_r\sqmatrix{x_1^{r-1} & \cdots & x_a^{r-1}} . \qedhere
\]
\end{proof}

In Definition~\ref{def:ic} we have defined the variables $y_i$ in terms of $x_i$
and a family of interpolation coordinates $v_r$ which depends on the cardinality
of the alphabet $\X$. Sometimes, however, it is useful to take the opposite
viewpoint and start with alphabets $\X$ and $\Y$ and the assumption that the
variables $y_i$ can be interpolated by polynomials in $x_i$ (with coefficients
then determined by Lemma~\ref{lemma:v from y}). 
For example, in later parts of the paper, we will need to understand the behavior of
the interpolation coordinates $v_r$ under inclusions of 
the alphabets $\X$ and $\Y$.

\begin{proposition}\label{prop:va and vb} 
Consider integers $1 \leq c \leq d$ and alphabets
\[
\X^{(c)} := \{x_1,\ldots,x_c\} \subset \X^{(d)} := \{x_1,\ldots,x_d\}
\, , \quad 
\Y^{(c)} := \{y_1,\ldots,y_c\} \subset \Y^{(d)} := \{y_1,\ldots,y_d\}
\]
with associated collections of interpolation coordinates 
$\V^{(c)} = \{v_1^{(c)},\ldots,v_c^{(c)}\}$ and $\V^{(d)}=\{v_1^{(d)},\ldots,v_d^{(d)}\}$.
The standard inclusion $\k[\X^{(c)},\Y^{(c)}] \hookrightarrow \k[\X^{(d)},\Y^{(d)}]$ 
extends uniquely to an algebra map 
\[
\k[\X^{(c)},\V^{(c)}] \hookrightarrow \k[\X^{(d)},\V^{(d)}]
\]
sending
\begin{equation}\label{eq:va and vb}
v_k^{(c)} \mapsto  v_k^{(d)} + (-1)^{c-k}\sum_{l=c+1}^d \Schur_{(l-c-1|c-k)}(\X^{(c)}) v_l^{(d)}.
\end{equation}
\end{proposition}
\begin{proof}
We first prove uniqueness.  
Suppose $\phi,\psi\colon \k[\X^{(c)},\V^{(c)}] \hookrightarrow \k[\X^{(d)},\V^{(d)}]$ 
are two algebra maps extending the inclusion
$\k[\X^{(c)},\Y^{(c)}] \hookrightarrow \k[\X^{(d)},\Y^{(d)}]$.
Since $\Delta(\X^{(c)}) v_k^{(c)} \in \k[\X^{(c)} , \Y^{(c)}]$, 
it follows that $\phi(v_k^{(c)}) - \psi(v_k^{(c)})$ is an element of $\k[\X^{(d)},\Y^{(d)}]$
that is annihilated by $\Delta(\X^{(c)})$, hence zero. 

Now, define $\phi\colon \k[\X^{(c)},\V^{(c)}] \hookrightarrow \k[\X^{(d)},\V^{(d)}]$
to be the algebra map defined by $\phi(x_i) = x_i$ 
and the substitution rule \eqref{eq:va and vb}.  
It suffices to verify that $\phi(y_i)=y_i$ for $1 \leq i \leq c$.
Indeed:
\begin{align*}
\phi(y_i)
&= \sum_{k=1}^c x_i^{k-1} \phi(v_k^{(c)})
= \sum_{k=1}^c x_i^{k-1} \left(  v_k^{(d)}  
	+ \sum_{l=c+1}^d (-1)^{c-k}\Schur_{(l-c-1|c-k)}(\X^{(c)}) v_l^{(d)}\right)\\
&= \sum_{k=1}^c x_i^{k-1} v_k^{(d)}  + \sum_{l=c+1}^d v_l^{(d)} 
	\sum_{k=1}^c (-1)^{c-k}\Schur_{(l-c-1|c-k)}(\X^{(c)}) x_i^{k-1} \\
&\stackrel{\text{Cor. } \ref{cor:monomial reduction}}{=} 
	\sum_{k=1}^c x_i^{k-1} v_k^{(d)}  + \sum_{l=c+1}^d x_i^{l-1} v_l^{(b)}
= \sum_{k=1}^d x_i^{k-1} v_k^{(d)}
=y_i . \qedhere
\end{align*}
\end{proof}

There is an interesting generalization of Proposition \ref{prop:va and vb} that
we will use (in various forms) throughout this paper.

\begin{lemma}\label{lemma:curvature stability}
\label{lem:OverV}
Let $\leftX^{(c)} = \{x_1,\ldots,x_c\}$ and $\rightX^{(c)} = \{x_1',\ldots,x_c'\}$.
For each subset $S\subset \{1,\ldots,c\}$, 
consider the element of $\k[\leftX^{(c)},\rightX^{(c)},\V^{(c)}]$ defined by
\[
 Z_S^{(c)} := \sum_{k=1}^c h_k(\leftX_S - \rightX_S) v_k^{(c)},
\]
where $\leftX_S \subset \X^{(c)}$ and $\rightX_S \subset \rightX^{(c)}$ 
denote the corresponding subalphabets.
For $c\leq d$, the algebra map
\begin{equation}\label{eq:stabilitymap}
\phi\colon \k[\leftX^{(c)},\rightX^{(c)},\V^{(c)}] \hookrightarrow \k[\leftX^{(d)},\rightX^{(d)},\V^{(d)}]
\end{equation}
determined by $x_i\mapsto x_i$, $x_i'\mapsto x_i'$, and the rule \eqref{eq:va and vb}
sends $Z_S^{(c)}\mapsto Z_S^{(d)} := \sum_{k=1}^d h_k(\leftX_S - \rightX_S) v_k^{(d)}$.
\end{lemma}
\begin{proof}
We compute:
\begin{align*}
\phi(Z_S^{(c)})
&= \sum_{k=1}^c h_k(\X_S-\X_S') \phi(v_k^{(c)})
= \sum_{k=1}^c h_k(\X_S-\X_S') \left( v_k^{(d)}  
	+ \sum_{l=c+1}^d  (-1)^{c-k}\Schur_{(l-c-1|c-k)}(\X^{(c)}) v_l^{(d)}\right)\\
&= \sum_{k=1}^c h_k(\X_S-\X_S')  v_k^{(d)}  + \sum_{l=c+1}^d v_l^{(d)} 
	\sum_{k=1}^c (-1)^{c-k}\Schur_{(l-c-1|c-k)}(\X^{(c)}) h_k(\X_S-\X_S')  \\
&= \sum_{k=1}^c h_k(\X_S-\X_S') v_k^{(d)}  + \sum_{l=c+1}^d h_l(\X_S-\X_S')  v_l^{(d)}
= \sum_{k=1}^d h_k(\X_S-\X_S') v_k^{(d)} = Z_S^{(d)} .
\end{align*}
In passing to the last line, 
we have used Lemma \ref{lemma:h reduction} 
with $\X = \X_S-\X_S'$ and $\Y = (\X^{(c)} - \X_S) + \X_S'$
(which has cardinality $c$).
\end{proof}

\begin{remark}
Lemma \ref{lemma:curvature stability} establishes a crucial ``stability'' property for 
the expressions $Z_S^{(c)}$ under the inclusion \eqref{eq:stabilitymap}.
The fact that $y_i\mapsto y_i$ in Proposition \ref{prop:va and vb} is essentially the
special case of Lemma \ref{lemma:curvature stability} corresponding to $S=\{i\}$
and $\rightX^{(c)}=0$ (hence $Z_S = x_iy_i$).
\end{remark}

In this section, we have discussed the variables $y_i$ associated with a single
$a$-colored strand. Indeed, as mentioned above, any symmetric polynomial in the
alphabet $\Y$ gives a well-defined element in $\Sym(\X)[\V] =
\End_{\VS(\SSBim)}(\oone_{(a)})$. Paralleling the passage from \S
\ref{ss:curvature discussion} to \S \ref{ss:v-curved cxs}, it is possible to
pass from the one-strand case to the general case, since curvature in
$\YS(\SSBim)$ is modeled on the strand-wise $h \Delta$- and $\Delta
e$-curvatures. 
See \S \ref{ss:two strand V cats} and \S \ref{ss:FT setup}
for aspects of the (pure) $2$-strand and $m$-strand cases, 
respectively.

\begin{remark}
The reader familiar with the Hilbert scheme $\Hilb_a(\C^2)$ should note that interpolation 
coordinates arise naturally in its study.
Compare the generators $\{e_1(\X),\ldots,e_a(\X),v_1,\ldots,v_a\}$ 
of $\Sym(\X)[\V] = \End_{\VS(\SSBim)}(\oone_{(a)})$ to the coordinates on 
the open affine set $U_x \subset \Hilb_a(\C^2)$ in \cite[Proposition 3.6.3]{Haiman}.
\end{remark}

\subsection{Alphabet soup IV: the two-strand categories}\label{ss:two strand V cats}

In this section we set up some framework and notation for working
with the categories $\VS_{a,b}$, i.e. the $m=2$ case of \eqref{eq:Vnotation}. 
This section can be skipped on first reading, and referred to as needed.

First, we will use the abbreviation
\[
\CS_{a,b} := \CS({}_{a,b}\SSBim_{a,b})
\]
for the relevant dg category of (uncurved) complexes of singular Soergel bimodules.
We fix alphabets as follows:
\begin{gather*}
\leftX = \{x_1,\ldots,x_{a+b}\} \, , \quad \leftX_1=\{x_1,\ldots,x_{a}\} \, , \quad \leftX_2=\{x_{a+1},\ldots,x_{a+b}\} \\
\rightX = \{x'_1,\ldots,x'_{a+b}\} \, ,\quad \rightX_1 =\{x'_1,\ldots,x'_{a}\} \, , \quad \rightX_2=\{x'_{a+1},\ldots,x'_{a+b}\} \\
\Y = \{y_1,\ldots,y_{a+b}\} \, ,\quad \Y_1=\{y_1,\ldots,y_{a}\} \, ,\quad \Y_2=\{y_{a+1},\ldots,y_{a+b}\}.
\end{gather*}
We denote\footnote{Our notation here indicates that $\V_L^{(a)}$ and $\V_R^{(b)}$ should be viewed as associated 
to the ``left'' $a$-colored and ``right'' $b$-colored strands of a $2$-strand pure braid, when drawn vertically. 
They would be called $\V_1$ and $\V_2$ in the language of \S \ref{ss:v-curved cxs}.} 
the deformation parameters associated with the first and second entries of $\aa = (a,b)$ by
\begin{equation}\label{eq:2strandVvars}
\V_L^{(a)} := \{v_{L,1}^{(a)},\ldots,v_{L,a}^{(a)}\}
\, , \quad
\V_R^{(b)} := \{v_{R,1}^{(b)},\ldots,v_{R,b}^{(b)}\}
\end{equation}
respectively, and let $\V:=\V_L^{(a)}\cup \V_R^{(b)}$.  
We regard $\k[\X,\Y]$ as a subalgebra of $\k[\X,\V]$ by the
identification
\begin{equation}
	\label{eq:def-yi-2strand}
y_i = \begin{cases} \sum_{k=1}^a x_i^{k-1} v_{L,k}^{(a)} & \text{if } 1\leq i\leq a\\
 \sum_{k=1}^b x_i^{k-1} v_{R,k}^{(b)} & \text{if } a< i\leq a+b \, ,
\end{cases}
\end{equation}
which is the (pure) $2$-strand analogue of \eqref{eq:y and v}.

We now observe that it is possible to specify objects in $\VS_{a,b}$ using a
reduced collection of deformation parameters.
\begin{definition}\label{def:V prime}
Let $\VSred_{a,b} := \CS_{\bar{Z}}({}_{a,b}\SSBim_{a,b}, \k[\bar{v}_1,\ldots,\bar{v}_b])$
be the category of curved complexes
where $\wt(\bar{v}_r) = \qdeg^{-2r} \tdeg^2$ and with curvature element
\[
\bar{Z} = \sum_{r=1}^b h_r(\leftX_2 - \rightX_2) \bar{v}_r\, .
\]
\end{definition}

In other words, $\VSred_{a,b}$ is the category of curved complexes over 
${}_{a,b}\SSBim_{a,b}$ where we have 
only deformed ``on the $b$-labeled strand.''

We now aim to introduce functors $\VSred_{a,b}\leftrightarrow \VS_{a,b}$ which
lift the identity on $\CS_{a,b}$. 
The first functor $\VS_{a,b}\rightarrow \VSred_{a,b}$ is just the specialization 
$v^{(a)}_{L,i}=0$ and $v^{(b)}_{R,j}=\bar{v}_j$ for all $i,j$.  
The functor in the other direction is more interesting. 
The starting point for its constructing is the observation that since
\[
f(\leftX_1 + \leftX_2- \rightX_1 - \rightX_2)|_X = 0
\]
for any $X \in {}_{a,b}\SSBim_{a,b}$ and any (positive degree) symmetric function $f$, 
we have that
\[
h_r(\leftX_1 - \rightX_1)  
= h_r\big( (\leftX_1 + \leftX_2  - \rightX_1 - \rightX_2) + (\rightX_2 - \leftX_2) \big)
= h_r(\rightX_2 - \leftX_2) \, .
\]
When $a \geq b$, an application of Lemma \ref{lemma:curvature stability} shows that if we set
\begin{equation}\label{eq:VLb}
v_{L,j}^{(b)} := 
v_{L,j}^{(a)}+(-1)^{b-j} \sum_{i=1}^{a-b} \Schur_{(i-1|b-j)}(\rightX_2) v_{L,b+i}^{(a)}
\end{equation}
for $1 \leq j \leq b$, then
\[
\sum_{k=1}^a h_k(\leftX_1 - \rightX_1) v_{L,k}^{(a)} = \sum_{k=1}^a h_k(\X_2'-\X_2) v_{L,k}^{(a)} 
= \sum_{k=1}^b h_k(\X_2'-\X_2) v_{L,k}^{(b)} \, .
\]
Further, this remains true (trivially) when $a<b$, 
provided we let $v_{L,j}^{(a)} = 0$ for $j>0$, 
since in this case \eqref{eq:VLb} gives
\[
v_{L,j}^{(b)} := 
\begin{cases}
v_{L,j}^{(a)} & \text{if } 1 \leq j \leq a \\
0 & \text{if } a < j \leq b \, .
\end{cases}
\]
It follows that
\begin{equation}\label{eq:Zred}
\begin{aligned}
Z &= \sum_{1\leq i\leq a} h_i(\leftX_1 - \rightX_1) v_{L,i}^{(a)} 
	+ \sum_{1\leq j\leq b} h_j(\leftX_2 - \rightX_2) v_{R,j}^{(b)} \\
&= \sum_{1\leq k\leq b} h_k(\rightX_2 - \leftX_2) v_{L,k}^{(b)} 
	+ \sum_{1\leq j\leq b} h_j(\leftX_2 - \rightX_2) v_{R,j}^{(b)} \\
&\stackrel{\eqref{eq:somerelations4}}{=}  -\sum_{1\leq j\leq k\leq b} 
	h_j(\leftX_2 - \rightX_2) h_{k-j}(\rightX_2-\leftX_2)v_{L,k}^{(b)} 
+  \sum_{1\leq j\leq b} h_j(\leftX_2 - \rightX_2) v_{R,j}^{(b)}\\
&=\sum_{1\leq j\leq b} h_j(\leftX_2 - \rightX_2) 
	\left(v_{R,j}^{(b)} - \sum_{j\leq k\leq b} h_{k-j}(\rightX_2 - \leftX_2) v_{L,k}^{(b)}\right) \, .
\end{aligned}
\end{equation}
This suggests the following.

\begin{prop}\label{prop:red functor} There is a functor $\VSred_{a,b}\rightarrow
\VS_{a,b}$  which is the identity on objects and morphisms in $\CS_{a,b}$, and
which sends
\begin{equation}\label{eq:vred}
\bar{v}_{j} \mapsto v_{R,j}^{(b)} 
	- \sum_{k=j}^b h_{k-j}(\rightX_2 - \leftX_2) v_{L,k}^{(b)} \, .
\end{equation}
\end{prop}
\begin{proof}
It need only be checked that \eqref{eq:vred} is compatible with curvatures, i.e. that
\begin{equation}\label{eq:2Zto1Z}
\bar{Z} = \sum_{r=1}^b h_r(\X_2-\X_2')\bar{v}_{r}  
\mapsto \sum_{r=1}^a h_r(\X_1-\X_1') v_{L,r}^{(a)} 
	+ \sum_{r=1}^b h_r(\X_2-\X_2') v_{R,r}^{(b)} = Z \, .
\end{equation}
This is true by the computation \eqref{eq:Zred} preceding this proposition.
\end{proof}

\begin{definition}
	\label{def:reduction}
Let $\pi\colon \VS_{a,b}\rightarrow \VS_{a,b}$ denote the functor which is the
identity on objects and morphisms in $\CS_{a,b}$, and which sends
$v_{L,i}^{(a)} \mapsto 0$ and sends $v_{R,i}^{(b)}$ to the expression on the right-hand side of
\eqref{eq:vred}.  We call $\pi$ the \emph{reduction functor}. We say that an
object $X\in \VS_{a,b}$ is \emph{reduced} if $\pi(X)=X$, and that a morphism
$f\colon X \to Y$ between reduced objects is reduced if $\pi(f) = f$.
\end{definition}

\begin{remark}\label{rem:Vprime} It is clear that $\pi^2=\pi$ is idempotent, and
has image equal (or canonically isomorphic to) $\VSred_{a,b}$.  Thus, we may
identify $\VSred_{a,b}$ as a (non-full) subcategory of $\VS_{a,b}$ consisting of
reduced objects and morphisms in $\VS_{a,b}$. The notion of reduction saves us
some work when constructing curved lifts of complexes in
$\CS_{a,b}$. Indeed, it suffices to construct a curved lift in
$\VSred_{a,b}$, which then immediately gives a curved lift in $\VS_{a,b}$.
\end{remark}

We conclude this section by establishing further notation for working with the
reduced categories $\VSred_{a,b}$. 
At times, we will need to consider the relation between such categories as we allow $b$ to vary. 
For each $\ell\geq 0$, introduce interpolation coordinates 
$\overline{\V}^{(\ell)} = \{\bar{v}^{(\ell)}_1,\ldots,\bar{v}^{(\ell)}_\ell\}$ and elements
\begin{equation}\label{eq:ybar and vbar}
\yred_i := \sum_{k=1}^\ell x_i^{k-1} \bar{v}^{(\ell)}_k
\end{equation}
for all $a+1\leq i\leq a+\ell$. 
A priori, the definition of $\yred_i$ depends on $\ell$.
However, if we let 
\begin{equation}\label{eq:Xinterval}
\leftX_{[1,\ldots,c]} = \{x_1,\ldots,x_c\} 
\, , \quad
\rightX_{[1,\ldots,c]} = \{x'_1,\ldots,x'_c\} 
\end{equation}
then for each pair of integers $\ell\leq b$ we have an inclusion of algebras
\begin{equation}\label{eq:ellbigger}
\k[\leftX_{[1,a+\ell]},\rightX_{[1,a+\ell]},\overline{\V}^{(\ell)}] \hookrightarrow 
\k[\leftX_{[1,a+b]},\rightX_{[1,a+b]},\overline{\V}^{(b)}] \defeq \k[\leftX,\rightX,\overline{\V}^{(b)}] 
\end{equation}
sending $x_i\mapsto x_i$, $x_i'\mapsto x_i'$, and 
\begin{equation}
\label{eq:vbar l and vbar b}
\bar{v}_k^{(\ell)} \mapsto  \bar{v}_k^{(b)} + (-1)^{\ell-k}\sum_{r=\ell+1}^b 
	\Schur_{(r-\ell-1|\ell-k)}(\leftX_{[a+1,a+\ell]}) \bar{v}_r^{(b)} \, .
\end{equation}
Crucially, Lemma \ref{prop:va and vb} implies that this sends $\yred_i \mapsto \yred_i$.
Furthermore, Lemma \ref{lemma:curvature stability} implies that 
the inclusion \eqref{eq:ellbigger} is compatible with curvature, in the sense that
\begin{equation}\label{eq:reduced curvature stability}
\sum_{i=1}^\ell h_i(\leftX_{[a+1,a+\ell]}-\rightX_{[a+1,a+\ell]}) \bar{v}_i^{(\ell)} 
\mapsto
\sum_{i=1}^b h_i(\leftX_{[a+1,a+b]}-\rightX_{[a+1,a+b]}) \bar{v}_i^{(b)} \, .
\end{equation}

\section{Deformed colored link homology}

In this section, we use curved complexes of singular Soergel bimodules to
construct our deformed, colored, triply-graded link homology. We begin by
recalling the construction of colored, triply-graded Khovanov--Rozansky
homology, which is an invariant of framed, oriented, colored links taking values
in the symmetric monoidal category $\overline{\KS}\llbracket \adeg^\pm,
\qdeg^\pm,\tdeg^\pm \rrbracket$ from \S\ref{ss:notation and cats}. 
Khovanov--Rozansky homology is defined by applying the 
Hochschild homology functor to the Rickard complex of a
braid representative of a link $\LB$, and our deformation is defined by
replacing the Rickard complex with its curved analogue from 
\S \ref{ss:curved rickard}.
In this section (and those following) we will typically denote our (co)domain 
objects in $\SSBim$ by $\brc$ and $\brcc$, 
since our notation for the object $\aa$ is easily confused with the 
Hochschild degree $\adeg$.

\subsection{Hochschild (co)homology}\label{sec:HH}

We begin by recalling the basics of Hochschild homology and cohomology.
Suppose that $R$ and $S$ are $\Z_\qdeg$-graded algebras and let 
${}_R\BS_{S}$ denote the category of $\Z_\qdeg$-graded $(R,S)$-bimodules.
When $R=S$, a graded $(R,R)$-bimodule $M \in {}_{R}\BS_{R}$ may be viewed as a module over 
the enveloping algebra $R^e=R\otimes R^\op$, 
and the \emph{$i^{th}$ Hochschild homology and cohomology} of $M$ are defined as:
\begin{equation}\label{eq:defHH}
\HH_i(M) := \Tor_i^{R^e}(R,M)
\, , \quad
\HH^i(M) := \Ext^i_{R^e}(R,M)
\end{equation}
for $i \geq 0$, respectively. 
Note that both inherit a $\Z_\qdeg$-grading from ${}_{R}\BS_{R}$, 
thus are objects in $\overline{\KS}\llbracket \qdeg^\pm\rrbracket$.
The \emph{total Hochschild (co)homology functors} are defined by
\begin{equation}\label{eq:totalHH}
\HH_\bullet(M) := \bigoplus_{i\geq 0} \adeg^{-i} \HH_i(M)
\, , \quad
\HH^\bullet(M) := \bigoplus_{i\geq 0} \adeg^i \HH^i(M)
\end{equation}
which we therefore view as functors 
${}_{R}\BS_{R} \to \overline{\KS}\llbracket \adeg^\pm,\qdeg^\pm\rrbracket$ 
(see \S\ref{ss:notation and cats}).

\begin{remark}
We view Hochschild homology as supported in negative $\adeg$-degree, while
Hochschild cohomology is supported in positive $\adeg$-degrees.
\end{remark}

The following is a standard fact about Hochschild homology; 
see e.g. \cite{Rou2Braid}.
\begin{proposition}\label{prop:HH tracelike} 
Let $R$ and $S$ be $\Z_\qdeg$-graded algebras. Suppose that
${}_{R}M_{S} \in {}_{R}\BS_{S}$ and ${}_{S}N_{R} \in {}_{S}\BS_{R}$ are $\Z_\qdeg$-graded bimodules
that are projective as $R$-modules and $S$-modules, 
then there is an isomorphism of $\Z_\adeg \times \Z_\qdeg$-graded $\k$-vector spaces
\[
\HH_\bullet(M\otimes_S N) \cong \HH_\bullet(N\otimes_R M)
\]
that is natural in $M$ and $N$. \qed
\end{proposition}

The total Hochschild (co)homology functors can be extended from bimodules 
to complexes of bimodules term-wise. 
Explicitly, if $X \in \CS({}_{R}\BS_{R})$ is a complex of graded $(R,R)$-bimodules, 
then we may write $X = \tw_\d (\bigoplus_k \tdeg^k X^k)$ 
for graded $(R,R)$-bimodules $X^k \in {}_{R}\BS_{R}$.
Define the complex
\begin{equation}\label{eq:HHcomplex}
\HH_\bullet(X) := \tw_{\HH_\bullet(\d)} \big( \bigoplus_k \tdeg^k \HH_{\bullet}(X^k) \big)
\end{equation}
which is an object in 
$\overline{\KS}\llbracket \adeg^{\pm},\qdeg^{\pm},\tdeg^\pm \rrbracket_{\dg}$.
In other words, $\HH_\bullet(X)$ equals the $\Z_\adeg\times \Z_\qdeg\times \Z_\tdeg$-graded
$\k$-vector space $\bigoplus_{i,k}\adeg^{-i}\tdeg^k \HH_i(X^k)$, with appropriate differential.
Proposition \ref{prop:HH tracelike} extends to this setting as follows.

\begin{proposition}\label{prop:HH tracelike 2} 
Let $R$ and $S$ be $\Z_\qdeg$-graded algebras.
Suppose that ${}_R X_S \in \CS({}_{R}\BS_{S})$ and ${}_S Y_R \in \CS({}_{S}\BS_{R})$ 
are complexes of $\Z_\qdeg$-graded bimodules, 
that are projective as $R$-modules or $S$-modules, 
then there is an isomorphism of
$\Z_\adeg\times \Z_\qdeg\times \Z_\tdeg$-graded complexes
\[
\HH_\bullet(X\otimes_S Y) \cong \HH_\bullet(Y\otimes_R X),
\]
that is natural in $X,Y$ (in the dg sense).
\end{proposition}
\begin{proof}
Denote the natural isomorphism from Proposition \ref{prop:HH tracelike} by
\[
\tau_{M,N}\colon \HH_\bullet(M \otimes_S N) \to \HH_\bullet(N \otimes_R M) \, .
\]
This induces an isomorphism on the
level of $\Z_\adeg\times \Z_\qdeg\times \Z_\tdeg$-graded $\k$-vector spaces:
\[
\HH_\bullet(X\otimes_S Y) \cong \bigoplus_{k,l\in \Z} \tdeg^{k+l} \HH_\bullet(X^k\otimes_S Y^l) 
\xrightarrow{\bigoplus_{k,l}(-1)^{kl}\tau_{X^k, Y^l}}  
\bigoplus_{k,l\in \Z} \tdeg^{k+l} \HH_\bullet(Y^l \otimes_R X^k) \cong \HH_\bullet(Y\otimes_R X)
\]
that we denote by $\tau_{X,Y}$. 

We now observe that $\tau_{X,Y}$ is natural.
Let $f\in \Hom_{\CS({}_{R}\BS_{S})}(X_1,X_2)$ and $g\in\Hom_{\CS({}_{S}\BS_{R})}(Y_1,Y_2)$
be morphisms of complexes of bimodules. 
Let $z \in \HH_\bullet(X_1^k\otimes_S Y_1^l)\subset \HH_\bullet(X_1 \otimes_S Y_1)$, 
then the Koszul sign rule tells us that
\[
\HH_{\bullet}(f\otimes g) (z) = (-1)^{|g|k} \HH_\bullet(f|_{X_1^k}\otimes g|_{Y_1^l})(z) \, .
\]
After applying $\tau_{X_2,Y_2}$, we obtain
\[
(-1)^{(k+|f|)(l+|g|)+|g|k} \HH_\bullet(g|_{Y_1^l}\otimes f|_{X_1^k})(\tau_{X_1^k,Y_1^l}(z))
\]
using naturality of the isomorphism from Proposition \ref{prop:HH tracelike}.
The sign here is equal to $(-1)^{|f||g|+ kl+|f|l}$. 
On the other hand, we have
\[
\left(\HH_\bullet(g\otimes f) \circ \tau_{X_1,Y_1}\right) (z) =
(-1)^{|f|l+kl} \HH_\bullet(g|_{Y_1^l} \otimes f|_{X_1^k})(\tau_{X_1^k,Y_1^l}(z)) \, ,
\]
thus
\[
\tau_{X_2,Y_2}\circ \HH_\bullet(f\otimes g) = (-1)^{|f||g|}\HH_\bullet(g\otimes f)\circ \tau_{X_1,Y_1}
\]
as desired.

Indeed, this is the appropriate (signed) version of naturality in this context.
For example, it guarantees that the isomorphism 
$\tau_{X,Y}:\HH_\bullet(X\otimes_S Y)\cong \HH_\bullet(Y\otimes_R X)$ 
intertwines the differentials
$\HH_\bullet(\d_X\otimes \id_Y+\id_X\otimes \d_Y)$ 
and $\HH_\bullet(\d_Y\otimes \id_X+\id_Y\otimes \d_X)$.
\end{proof}

\begin{rem}
The hypotheses of Propositions \ref{prop:HH tracelike} and \ref{prop:HH tracelike 2} 
hold for (complexes of) singular Soergel bimodules. Indeed, any 
$X \in {}_{\brc}\SSBim_{\brcc}$ is free as either a left $R^{\brc}$-module 
or right $R^{\brcc}$-module.
\end{rem}

From this point forward, we focus on the case when $R$ is a polynomial ring,
since this is the relevant setting for singular Soergel bimodules. 
In this case, the Koszul resolution provides a direct means for computing 
Hochschild (co)homology and gives an explicit relation between these two invariants. 
Let\footnote{Here, we call our variables $\{z_i\}$, since in general they will not be 
the variables $\X=\{x_i\}$, but rather symmetric functions in subalphabets of the alphabet $\X$.} 
$R=\k[z_1,\ldots,z_N]$ with $\deg_{\qdeg}(z_i)=d_i$ for $1\leq i\leq N$. 
The Koszul resolution of $R$ is the complex of graded $(R,R)$-bimodules given by
\begin{equation}\label{eq:Koszul}
\mathbf{K} := {\bigotimes_{i=1}^N}
\left( \adeg\inv \qdeg^{d_i}  R\otimes R 
\xrightarrow{z_i\otimes 1 - 1 \otimes z_i} R\otimes R \right) \, .
\end{equation}
It follows from \eqref{eq:defHH} that $\HH_\bullet(M)$ is the homology of $\mathbf{K}\otimes_{R\otimes R} M$,
while $\HH^\bullet(M)$ is the homology of $\Hom_{R\otimes R}(\mathbf{K}, M)$.
This implies that $\HH_i(M) \cong \qdeg^{d_1+\cdots+d_N} \HH^{N-i}(M)$, 
hence
\begin{equation}\label{eq:HHvHH tot}
\HH_\bullet(M) \cong \adeg^{-N} \qdeg^{d_1+\cdots+d_N} \HH^{\bullet}(M)
\end{equation}
by \eqref{eq:totalHH}.
Moreover,
\[
\HH_\bullet(R) \cong \largewedge[\eta_1,\ldots,\eta_N]
\, , \quad \wt(\eta_i) = \adeg\inv \qdeg^{d_i}
\]
and
\[
\HH^\bullet(R) \cong \largewedge[\eta_1^\ast,\ldots,\eta_N^\ast]
\, , \quad \wt(\eta_i^\ast) = \adeg \qdeg^{-d_i} \, .
\]
The latter acts on the former by identifying $\eta_i^\ast$ with 
the derivation sending $\eta_i\mapsto 1$ and $\eta_j\mapsto 0$ for $j\neq i$.

\subsection{Partial traces}
\label{sec:pHHH}

We now recall the partial Hochschild (co)homology functors from \cite{RTub2},
which generalize the (uncolored) partial trace functors first introduced in
\cite{Hog2}. These functors refine the Hochschild (co)homology functors from
\S\ref{sec:HH}, and allow them to be applied to the complex ${C}(\b_{\brc})$
``one strand at a time.'' As such, they are useful in proving the invariance of
colored, triply-graded link homology (and its deformation defined in 
\S \ref{sec:deformed} below) under the second Markov move. We will refer to these
functors collectively as the \emph{colored partial trace} functors.

Recall that $\SSBim$ is a full $2$-subcategory 
of the $2$-category $\Bim$ from \S \ref{ss:ssbim},
and that in these $2$-categories the $\Hom$-categories are enriched in the
symmetric monoidal category $\overline{\KS}\llbracket \qdeg^{\pm} \rrbracket$ of
$\Z_\qdeg$-graded $\k$-vector spaces. We now consider the bounded derived category of
$\Bim$, denoted $\DS(\Bim)$, which is enriched in the category
$\overline{\KS}\llbracket \adeg^{\pm},\qdeg^{\pm} \rrbracket$ of $\Z_\adeg\times
\Z_\qdeg$-graded $\k$-vector spaces. This $2$-category $\DS(\Bim)$ is the natural
setting for Hochschild (co)homology of singular Soergel bimodules, via the
inclusion
\begin{equation}\label{eq:SSBimIncl}
\SSBim \hookrightarrow \Bim \hookrightarrow \DS(\Bim) \, .
\end{equation}
We emphasize to the reader that the cohomological grading in $\DS(\Bim)$ is the 
$\adeg$-degree, which is independent from the 
cohomological grading used in the dg category of (curved) complexes $\CS(\SSBim)$, 
which is the $\tdeg$-degree.

\begin{definition}
Let $\DS(\Bim)$ be the monoidal 2-category wherein:
\begin{itemize}
\item objects of $\DS(\Bim)$ are the same as in $\Bim$.
\item the $\Hom$-category ${}_{\brc}\DS(\Bim)_{\brcc}$ from $\brcc$ to $\brc$ 
is the bounded derived category
$D^b({}_{\brc}\Bim_{\brcc})$ of graded $(R^{\brc},R^{\brcc})$-bimodules
(equivalently, this is the bounded derived category of graded $R^{\brc}\otimes R^{\brcc}$-modules).
Horizontal composition of 1-morphisms is given by derived tensor product over 
the intermediate rings $R^{\brc}$. 
\item The external tensor product $\boxtimes$ is defined as in \S\ref{ss:ssbim}.
\end{itemize}
\end{definition}

Since singular Soergel bimodules in ${}_{\brc}\SSBim_{\brcc}$ are free as either left 
$R^{\brc}$-modules or right $R^{\brcc}$-modules, horizontal composition of such bimodules is exact.
Thus, the inclusion in \eqref{eq:SSBimIncl} is indeed a $2$-functor 
(i.e. derived tensor product over the rings $R^{\brc}$ equals the usual tensor product for such bimodules).
Since $\boxtimes$ is given in both settings by tensor product over $\k$, 
\eqref{eq:SSBimIncl} is a monoidal $2$-functor.

For $X,Y\in {}_\brc\SSBim_{\brcc}$, we have
\[
\Hom_{\DS(\Bim)} (X,Y) = \Ext_{R^\brc \otimes R^\brcc}(X,Y) \supset \Hom_{\SSBim}(X,Y) \, .
\]
In particular, for $X \in {}_\brc\SSBim_{\brc}$ this gives
\[
\Hom_{\DS(\Bim)} (R^{\brc},X) = \HH^\bullet(X) \, .
\]

\begin{definition}[Partial trace]\label{def:PT}
Let $X\in {}_{\brc}\DS(\Bim)_{\brcc}$ be a $1$-morphism between objects 
of the form $\brc=\brc'\boxtimes c$ and $\brcc=\brcc'\boxtimes c$,
then the \emph{partial Hochschild homology} of $X$ is:
\[
\Tr_c(X)  := \tw_\d\Big(X\otimes\largewedge[\eta_1,\ldots,\eta_c]\Big)
\, , \quad \wt(\eta_i) = \adeg\inv \qdeg^{2i}
\]
with twist $\d = \sum_{1\leq i\leq c}(e_i(\leftX)-e_i(\rightX))\otimes \eta_i^\ast$.
Here, $\X$ and $\rightX$ denote the alphabets corresponding to the last ($c$-labeled) 
boundary points.
The complex $\Tr_c(X)$ is regarded as an object in the 1-morphism category 
${}_{\brc'}\DS(\Bim)_{\brcc'}$; 
in particular, its cohomological degree is the $\adeg$-degree.
Paralleling the relation in \eqref{eq:HHvHH tot}, 
we also define the \emph{partial Hochschild cohomology} for such $X$ to be:
\[
\Tr^c(X) := \adeg^c \qdeg^{-c(c+1)} \Tr_c(X)
\]
\end{definition}

\begin{remark}[Relation to Hochschild (co)homology]\label{rem:HHvHH}
Given a 1-morphism $X\in {}_{\brc}\DS(\Bim)_{\brc}$ with $\brc=(\bre_1,\dots,\bre_m)$, 
it is immediate from \eqref{eq:Koszul} and Definition \ref{def:PT} that:
\[
(\Tr_ {\bre_1} \circ \cdots \circ \Tr_ {\bre_m})(X) \cong \HH_\bullet(X)
\, , \quad
(\Tr^ {\bre_1} \circ \cdots \circ \Tr^ {\bre_m})(X) \cong \HH^\bullet(X) \, .
\]
\end{remark}

We now recall various properties of the colored partial trace functors. 
Further details are provided in \cite[Section 4.C]{RTub2} 
and \cite[Sections 3.2 and 3.3]{Hog2}; 
hence, our treatment is concise.

We first record an adjunction that will be used to compute various
$\Hom$-spaces.

\begin{prop}\label{prop:TrAdj}
Let $\brc=\brc'\boxtimes c$ and $\brcc=\brcc'\boxtimes c$ be objects in $\Bim$, 
then partial Hochschild cohomology gives a functor
\[
\Tr^c \colon {}_{\brc}\DS(\Bim)_{\brcc} \to {}_{\brc'}\DS(\Bim)_{\brcc'}
\]
that is right adjoint to the functor 
${}_{\brc'}\DS(\Bim)_{\brcc'} \to {}_{\brc}\DS(\Bim)_{\brcc}$
 sending
$Y \mapsto Y \boxtimes \oone_c$.\qed
\end{prop}

Using the Koszul resolution of $\Sym(\X)$ as a bimodule over itself, it is easy
to see the following.
\begin{prop}\label{prop:trace of one}
Let $X \in \DS(\Bim)$ and $c \geq 1$, then
\[
\Tr_c(X\boxtimes \oone_c) \cong 
X\otimes \HH_\bullet(\Sym(\X)) \cong X\otimes \k[e_1(\X),\ldots,e_c(\X)]\otimes \largewedge[\eta_1,\ldots,\eta_c]
\]
where $|\X|=c$, $\wt(e_i(\X))=\qdeg^{2i}$, and $\wt(\eta_i) = \adeg\inv \qdeg^{2i}$. \qed
\end{prop}

The next result establishes a certain ``bilinearity'' of the colored partial trace.

\begin{prop}\label{prop:TrBi}
Let $\brc=\brc'\boxtimes c$ and $\brcc=\brcc'\boxtimes c$ be objects in $\Bim$, 
and let $Y_1,Y_2\in {}_{\brc'}\DS(\Bim)_{\brcc'}$ and $X\in {}_{\brc}\DS(\Bim)_{\brcc}$ 
be $1$-morphisms. There is an isomorphism
\begin{equation}\label{eq:trace linearity}
\Tr^c \big( (Y_1\boxtimes \oone_c) \hComp X \hComp  (Y_2\boxtimes \oone_c) \big)
\cong
Y_1 \hComp \Tr^c(X)\hComp Y_2
\end{equation}
in ${}_{\brc'}\DS(\Bim)_{\brcc'}$ that is natural in $X, Y_1$, and $Y_2$.
In particular, if $Z_1 \boxtimes Z_2 \in {}_{\brc}\DS(\Bim)_{\brcc}$, then
\[
\Tr^c(Z_1 \boxtimes Z_2) \cong Z_1 \boxtimes \Tr^c(Z_2)
\]
naturally in $Z_1$ and $Z_2$. \qed
\end{prop}

Our next result describes the behavior of Rickard complexes under the
second Markov move. In order to state it, we first need the following.

\begin{definition}\label{def:CDBim} Let $\CS(\DS(\Bim))$ denote the monoidal
2-category with the same objects as $\DS(\Bim)$ (and $\Bim$), 
and 1-morphism categories 
\[
{}_\brc\CS(\DS(\Bim))_\brcc:=\CS({}_\brc\DS(\Bim)_\brcc) \, .
\]  
\end{definition}
Note that $\CS(\DS(\Bim))$ is enriched in 
$\overline{\KS}\llbracket \adeg^\pm,\qdeg^\pm,\tdeg^\pm\rrbracket_{\dg}$, 
and that the functors $\Tr_c$ and $\Tr^c$ induce functors
${}_{\brc}\CS(\DS(\Bim))_{\brcc} \to {}_{\brc'}\CS(\DS(\Bim))_{\brcc'}$
by applying $\Tr_c$ and $\Tr^c$ to complexes term-wise as in \eqref{eq:HHcomplex}. 
Further, the inclusion \eqref{eq:SSBimIncl} induces an 
inclusion $\CS(\SSBim) \hookrightarrow \CS(\DS(\Bim))$.

\begin{lem}\label{lem:Markov2} 
Let $b \geq 1$, then we have homotopy equivalences
\[
\Tr_{b}(\oone_{\brcc}\boxtimes C_{b,b}) \simeq  \adeg^{-b} \qdeg^{b^2}  \tdeg^b \oone_{\brcc}\boxtimes \oone_b
\quad \text{and} \quad
\Tr_{b}(\oone_{\brcc}\boxtimes  C_{b,b}^\vee) \simeq  \qdeg^{-b^2} \oone_{\brcc}\boxtimes \oone_b
\]
inside ${}_{\brcc\boxtimes b}\CS(\DS(\Bim))_{\brcc\boxtimes b}$.
\end{lem}
\begin{proof}
This appears in \cite[Lemma 4.12]{RTub2} in terms of partial
Hochschild cohomology, namely:
\begin{equation}\label{eq:R1HochCo}
\Tr^{b}(\oone_{\brcc}\boxtimes C_{b,b}) \simeq  \qdeg^{-b}  \tdeg^b \oone_{\brcc}\boxtimes \oone_b
\quad \text{and} \quad
\Tr^{b}(\oone_{\brcc}\boxtimes C_{b,b}^\vee) \simeq  \adeg^b \qdeg^{- 2b^2-b} \oone_{\brcc}\boxtimes \oone_k .
\end{equation}
After shifting the grading by $\adeg^{b}\qdeg^{-b(b+1)}$ according to
Remark~\ref{rem:HHvHH}, we recover the shown expressions.
\end{proof}

\begin{proposition}\label{prop:markov ii}
For any 1-morphism $X$ in ${}_{\brcc\boxtimes b}\CS(\SSBim)_{\brcc\boxtimes b}$, 
we have equivalences
\[
\HH_\bullet\Big((X\boxtimes \oone_b)\hComp (\oone_{\brcc}\boxtimes C_{b,b})\Big) 
\simeq \adeg^{-b} \qdeg^{b^2}  \tdeg^b \HH_\bullet(X)
\, , \quad 
\HH_\bullet\Big((X\boxtimes \oone_b)\hComp (\oone_{\brcc}\boxtimes C_{b,b}^\vee)\Big) 
\simeq  \qdeg^{-b^2}  \HH_\bullet(X)
\]
that are natural in $X$.
\end{proposition}
\begin{proof}
This is an immediate consequence of Remark \ref{rem:HHvHH}, Proposition \ref{prop:TrBi}, 
and Lemma \ref{lem:Markov2}.
\end{proof}

\subsection{Colored, triply-graded homology}
\label{sec:homology}
In \cite{MR2339573}, 
Khovanov showed that triply-graded Khovanov--Rozansky homology \cite{MR2421131}
can be reformulated using the Hochschild homology of Soergel bimodules.
We now recall the extension of this result to colored, triply-graded link homology via 
singular Soergel bimodules, as described in \cite{MR2746676,Wed3,Cautis3}.
See also \cite{MR3687104} for a geometric analogue of this construction.
We say that a $\Z_{\geq 1}$-colored braid ${}_{\brc} \b_{\brcc}$ is \emph{balanced}
if it is an endomorphism in $\mathfrak{Br}(\Z_{\geq 1})$, i.e. if $\brcc=\brc$. In
this case, the (standard) braid closure $\widehat{\b_{\brc}}$ represents a
colored link. The classical Alexander theorem implies that every colored link in
$S^3$ arises in this way, and Markov's theorem classifies the redundancy in
presenting links by braid closures.

Following the references above, we will define the colored, triply-graded link
homology as the homology of a properly normalized version of the Hochschild
homology of Rickard complexes. The result will be an invariant of colored,
framed, oriented links. First, we need some setup.

\begin{definition}\label{def:cycles} If $\perm \in \symg_m$ is a permutation then we put
an equivalence relation on $\{1,\ldots,m\}$ by declaring $i\sim j$ if $i=\perm^k(j)$ 
for some $k>0$.
The set of equivalence classes $\Omega(\perm):=\{1,\ldots,m\}/\sim$ is called the set of
\emph{cycles} of $\perm$.  We denote by $[i]$ the cycle containing 
$i\in \{1,\ldots, m\}$.  If $\brc=(\bre_1,\ldots,\bre_m)$ satisfies 
$\perm(\brc)=\brc$, then $\bre_{i}$ depends
only on the cycle containing $i$, so we may write $\bre_{[i]}:= \bre_i$.
\end{definition}

Note that if $\perm$ is the permutation represented by a braid $\b$, 
then $\Omega(\perm)$ is in canonical bijection with the set of link components of $\hat{\b}$. 
We now introduce some numerical invariants of our colored braid $\b_{\brc}$.

\begin{definition}\label{def:braid numerical invts} Suppose $\b$ is an
$m$-strand braid and $\perm$ is the permutation represented by $\b$, and
let $\brc=(\bre_1,\ldots,\bre_m)$ be a set of colors.  The \emph{colored writhe}
$\e(\beta_\brc) \in \Z$ 
of a colored braid $\beta_\brc \in \mathfrak{Br}(\Z_{\geq 1})$ is defined by
\[
	\e(\beta_\brc):=\sum_{\cross} s(\cross) a(\cross) b(\cross)
\]
where $\cross$ ranges over all crossings in the braid $\beta$, 
$\{a(\cross) , b(\cross)\}$ is the multiset of colors meeting $\cross$, 
and $s(\cross)\in \{\pm 1\}$ is the sign of $\cross$. 
Furthermore, when $\beta_\brc$ is balanced we set:
\[
	n(\b_\brc) := \sum_{[i]\in \Omega(\perm)} \bre_{[i]} \, , \quad
	N(\b_\brc):= \sum_{1\leq i\leq m} \bre_i \, , \quad 
	Q(\b_\brc):= \sum_{1\leq i\leq m} \bre_i(\bre_i+1) \, .
\]
These quantities give the sum of the colors of the components 
of the closure of $\b_\brc$, the colors on the strands of $\b_\brc$, 
and the sum of the $\qdeg$-degrees of all the generators
of the partially symmetric polynomial ring $R^{\brc}$.
\end{definition}

\begin{definition}\label{def:HHH} Let $\LB$ be a colored link, 
presented as the closure of a balanced, colored braid $\b_{\brc}$.  
The \emph{colored, triply-graded Khovanov--Rozansky complex} of $\b_{\brc}$ is
\[
C_{\KR}(\b_{\brc}) 
:= (\adeg\tdeg\inv)^{\frac{1}{2}(\e(\beta_\brc)+N(\b_\brc)-n(\b_\brc))} 
\qdeg^{-\e(\beta_\brc)} \HH_{\bullet}(C(\b_\brc))
\]
This is an object\footnote{In fact, the homogeneous components 
of $C_{\KR}(\b_{\brc})$ are finite-dimensional, 
the $\adeg$- and $\tdeg$-grading are bounded 
and the $\qdeg$-grading is bounded from below. 
The same holds for $H_{\KR}(\b_{\brc})$.} 
of $\overline{\KS}\llbracket \adeg^\pm,
\qdeg^\pm,\tdeg^\pm\rrbracket_{\dg}$ (see \S\ref{ss:notation and cats}). 
The \emph{colored, triply-graded Khovanov--Rozansky homology} of $\b_{\brc}$ 
is defined by $H_{\KR}(\b_{\brc}) := H(C_{\KR}(\b_{\brc}))$, 
which is an object of 
$\overline{\KS}\llbracket \adeg^\pm,\qdeg^\pm,\tdeg^\pm\rrbracket$.
\end{definition}

\begin{remark}
It can be seen that $\e(\b_\brc)+N(\b_\brc) - n(\b_\brc)$ is even, so the shift
$(\adeg\tdeg\inv)^{\frac{1}{2}(\e(\beta_\brc)+N(\b_\brc)-n(\b_\brc))}$ is well-defined.  
We leave verification of this as an exercise.
	\end{remark}
	
\begin{remark}\label{rem:HHc}
One can also express $C_{\KR}(\b_{\brc})$ in terms of Hochschild cohomology:
\[
C_{\KR}(\b_{\brc}) 
= \adeg^{\frac{1}{2}(\e(\b_\brc)-N(\b_\brc)-n(\b_\brc))} 
\qdeg^{Q(\b_\brc)-\e(\b_\brc)} \tdeg^{\frac{1}{2}(-\e(\b_\brc)-N(\b_\brc)+n(\b_\brc))}
\HH^{\bullet}(C(\b_\brc))
\]
via \eqref{eq:HHvHH tot}.
	\end{remark}

\begin{thm} \label{thm:HHHinvariance} A change of braid representative
$\b_{\brc}$ for a framed, oriented, colored link $\LB$ induces a homotopy equivalence
between the corresponding complexes $C_{\KR}(\b_{\brc})$. 
Consequently, 
$H_{\KR}(\LB):= H_{\KR}(\b_{\brc}) \in 
	\overline{\KS}\llbracket \adeg^\pm,\qdeg^\pm,\tdeg^\pm\rrbracket$ 
is a well-defined
invariant of the framed, oriented, colored link $\LB$, up to isomorphism.
Furthermore, both $C_{\KR}(\b_{\brc})$ and $H_{\KR}(\LB)$ are monoidal under split
(disjoint) union $\sqcup$:
\[
C_{\KR}(\b_{\brc} \sqcup \b'_{\brc'}) \cong C_{\KR}(\b_{\brc})\otimes_{\k} C_{\KR}(\b'_{\brc'})
\, , \quad
H_{\KR}(\LB\sqcup \LB') \cong H_{\KR}(\LB)\otimes_{\k} H_{\KR}(\LB') \, .
\]
Changing framing by $\pm1$ on a $b$-labeled component shifts $H_{\KR}(\LB)$ by
$(\adeg\tdeg\inv)^{\pm\frac{1}{2}b(b-1)}$. 

A choice of a point $\point$ on a $b$-labeled component 
equips $C_{\KR}(\b_{\brc})$ and $H_{\KR}(\LB)$ with an action of
the symmetric polynomial ring $\Sym(\X_\point)$ with $|\X_\point|=b$. 
Different choices of points on a component give quasi-isomorphic 
$\Sym(\X_\point)$-module structures on 
$C_{\KR}(\b_{\brc})$ and equal actions on $H_{\KR}(\LB)$.
\end{thm}
\begin{proof} 
	Following the classical Markov theorem, the invariance statement is a
consequence of Proposition \ref{prop:rickard invariance}, conjugacy invariance
of Hochschild homology (Proposition~\ref{prop:HH tracelike 2}), and the
behavior of Hochschild homology of Rickard complexes under the second Markov
move (Lemma \ref{lem:Markov2}). See also \cite[Theorem 1.1]{MR3687104} or
\cite[Theorem 4.1]{Cautis3}. The monoidality follows since split links can be
represented by split braids, the fact that the Hochschild homology of Rickard
complexes is monoidal under the external tensor product $\boxtimes$, and the
observation that the numerical invariants $\e(-)$, $N(-)$, and $n(-)$ are
additive under $\boxtimes$. The framing behavior follows from
Lemma~\ref{lem:Markov2}. 
The module structure on $C_{\KR}(\b_{\brc})$ is inherited from the
module structure of singular Soergel bimodules. 
Homotopies relating the module structures specified by points on a 
single strand on two sides of a crossing were studied in \S\ref{ss:curved rickard}, 
and the last statement therefore follows from Proposition \ref{prop:qiso}.
\end{proof}

\begin{example}
	\label{exa:unknot}
Let $\UB(b)$ be the 0-framed $b$-colored unknot, presented as the closure of a
single $b$-labeled strand. 
In this case, no differentials or grading shifts enter into Definition~\ref{def:HHH}. 
To compute the unknot invariant, let $\X$ be the size
$b$ alphabet associated to the $b$-labeled strand. 
Then, we have:
\[
H_{\KR}(\UB(b)) = C_{\KR}(\UB(b)) 
= \HH_{\bullet}(\Sym(\X)) 
\cong \k[e_1(\X),\ldots,e_b(\X)]\otimes \largewedge[\eta_1,\ldots,\eta_b]
\]
where $\wt(e_i(\X))=\qdeg^{2i}$ and $\wt(\eta_i) = \adeg\inv \qdeg^{2i}$. 
The Poincar\'e series of this homology
is:
\[
\left(\frac{1+\adeg\inv \qdeg^{2}}{1-\qdeg^2}\right)\cdots 
\left(\frac{1+\adeg\inv \qdeg^{2b}}{1-\qdeg^{2b}}\right) \, .
\]

As is typical in link homology theory, what is usually called ``the unknot
invariant'' actually serves two distinct roles: first literally as the invariant
of the unknot, and second as the algebra that acts on the invariant of any link
upon specifying a chosen point. 
The latter is the \emph{derived sheet algebra} from Example \ref{ex:colored sheet alg},
so-named because it computes endomorphisms of a single strand in $\DS(\Bim)$. 
In the colored, triply-graded Khovanov--Rozansky theory, the derived sheet algebra for a
$b$-labeled strand is given by the dg algebra $\HH^{\bullet}(\Sym(\X))$, whose
Poincar\'e series is:
\[
\left(\frac{1+ \adeg \qdeg^{-2}}{1-\qdeg^2}\right)\cdots 
\left(\frac{1+ \adeg \qdeg^{-2b}}{1-\qdeg^{2b}}\right) \, .
\]
The action of the derived sheet algebra on the unknot invariant coincides with the
classical action of Hochschild cohomology on Hochschild homology. 
By \eqref{eq:HHvHH tot}, 
the underlying triply-graded vector spaces of the unknot invariant and the derived
sheet algebra only differ by a grading shift.
\end{example}

\begin{rem}
	\label{rem:dgmodule}
	We expect that Theorem~\ref{thm:HHHinvariance} can be strengthened as
	follows. If the labels on the components of $\LB$ are $\bre_1,\dots,\bre_r$,
	then we expect that $C_{\KR}(\LB)$ is well-defined up to quasi-isomorphism as a dg-module
	over the dg algebra given as the tensor product of the \emph{derived} sheet algebras of all
	$\bre_i$. We will not need this stronger statement in the present paper.
\end{rem}

\begin{rem}
	It is sometimes desirable to use a renormalized version of
	Definition~\ref{def:HHH} that favors Hochschild cohomology over Hochschild
	homology, and in which the unknot invariant is identified with the derived
	sheet algebra. For this one sets: 
\begin{align*}
C'_{\KR}(\LB)
&=
\adeg^{\frac{1}{2}(\e(\beta_\brc)+N(\b_\brc)+ n(\b_\brc))}
\tdeg^{\frac{1}{2}(-\e(\beta_\brc)-N(\b_\brc)+ n(\b_\brc))} \qdeg^{-\e(\beta_\brc) -n_2(\b_\brc)} 
\HH_{\bullet}(C(\b_\brc)) \\
&=
\adeg^{\frac{1}{2}(\e(\b_\brc)-N(\b_\brc)+n(\b_\brc))}
\tdeg^{\frac{1}{2}(-\e(\b_\brc)-N(\b_\brc)+n(\b_\brc))}
\qdeg^{-\e(\b_\brc)-n_2(\b_\brc)+Q(\b_\brc)}
\HH^{\bullet}(C(\b_\brc)),
\end{align*}
where $n_2(\b_\brc):= \sum_{[i]\in \Omega(\perm)} \bre_{[i]}(\bre_{[i]}+1)$.
\end{rem}

\begin{rem}
The invariant $H_{\KR}(\LB)$ decategorifies to a version of the HOMFLYPT
invariant $P(\LB)$ of links colored by one-column Young diagrams by specializing
the three-variable Poincar\'e series at $\tdeg=-1$. More specifically, the
\emph{uncolored} part of the invariant (i.e. where all colors are single box
Young diagrams) is determined by the unknot invariant that we read off
Example~\ref{exa:unknot} and the following skein relation:
\[
\qdeg P\left( 
	\begin{tikzpicture}[rotate=90,scale=.5,smallnodes,anchorbase]
		\draw[very thick,->] (1,-1) node[right,xshift=-2pt]{$1$} to [out=90,in=270] (0,1);
		\draw[line width=5pt,color=white] (0,-1) to [out=90,in=270] (1,1);
		\draw[very thick,->] (0,-1) node[right,xshift=-2pt]{$1$} to [out=90,in=270] (1,1);
	\end{tikzpicture}
\right)	+\adeg \qdeg\inv 
P\left( 
	\begin{tikzpicture}[rotate=90,scale=.5,smallnodes,anchorbase]
		\draw[very thick,->] (0,-1) node[right,xshift=-2pt]{$1$} to [out=90,in=270] (1,1);
		\draw[line width=5pt,color=white] (1,-1) to [out=90,in=270] (0,1);
		\draw[very thick,->] (1,-1) node[right,xshift=-2pt]{$1$} to [out=90,in=270] (0,1);
	\end{tikzpicture}
\right)
 = (-\adeg)^r (\qdeg-\qdeg\inv) 
 P\left( 
	\begin{tikzpicture}[rotate=90,scale=.5,smallnodes,anchorbase]
		\draw[very thick,->] (0,-1) node[right,xshift=-2pt]{$1$} to [out=90,in=270] (0,1);
		\draw[very thick,->] (1,-1) node[right,xshift=-2pt]{$1$} to [out=90,in=270] (1,1);
	\end{tikzpicture}
\right)
\]
where $r=0$ if all shown strands on the left-hand side belong to the same link
component, and $r=1$ otherwise.
\end{rem}

\subsection{The deformed, colored, triply-graded homology}
\label{sec:deformed}
We now define our deformed, colored, triply-graded homology. 
This proceeds in parallel to \S\ref{sec:homology} by replacing the 
Rickard complex $C(\b_\brc)$ with an appropriate curved analogue 
constructed from $\YS C(\b_{\brc,\perm\inv})$ and the set $\Omega(\perm)$ 
of cycles of $\beta$.
We will use the $\V$-variables description of the $\Hom$-categories in $\YS(\SSBim)$, 
as it is more convenient for our present considerations.

Suppose we are given a 1-morphism 
$\tw_\Delta(X) \in {}_{\brc,\id}\YS(\SSBim)_{\brc,\perm\inv}$, 
where $\brc=(\bre_1,\ldots,\bre_m)$ and $\perm \in \symg_m$.  
If $\perm=\id$, then the $h\Delta$-curvature on $\tw_{\Delta}(X)$ 
(given by \eqref{eq:Vcurv}) takes the form 
$\sum_{i,k} h_k(\leftX_i-\rightX_i)v_{i,k}$, 
which becomes zero after applying $\HH_\bullet$.
Thus,
\begin{equation}\label{eq:YHH1}
\tw_{\HH_\bullet(\Delta)}\Big(\HH_\bullet(X)\otimes \k[\V]\Big)
\end{equation}
is a well-defined complex (with zero curvature), 
and we obtain our sought after deformation of 
\eqref{def:HHH} by replacing $\HH_\bullet$ with \eqref{eq:YHH1}.
If $\perm \neq \id$, 
then we will need to ``prepare'' $\tw_{\Delta}(X)$ before taking Hochschild homology.

Hence, we introduce $\X$-alphabets parametrized by
elements of $\Omega(\perm)$ (i.e. the cycles of $\perm$)
by setting
\[
\leftX_{[i]} := \sum_{j\in[i]} \leftX_j 
\, , \quad 
\rightX_{[i]}:= \sum_{j\in[i]} \rightX_j,
\]
and we introduce deformation parameters indexed by $\Omega(\perm)$, 
denoted
\begin{equation}\label{eq:VOmega}
\V_{[i]} := \{v_{[i],1},\ldots,  v_{[i],b_i}\}
\, ,\quad 
\V^{\Omega(\perm)} := \bigcup_{[i]\in \Omega(\perm)} \V_{[i]}.
\end{equation}
We now change variables from $\V$ to $\V^{\Omega(\perm)}$ 
so that the resulting complex has curvature
\begin{equation}\label{eq:bundled curvature}
\sum_{[i]\in \Omega(\perm)}\sum_{k=1}^{b_i} h_k(\leftX_{[i]}-\rightX_{[i]}) v_{[i],k}.
\end{equation}
Note that if $\tw_{\Delta}(X) = \YS C(\beta_{\brc})$ 
is the curved complex associated to a braid, 
then \eqref{eq:bundled curvature} has one alphabet 
$\V_{[i]} = \{v_{[i],1},\ldots,  v_{[i],b_i}\}$ of deformation parameters 
for each link component in $\widehat{\beta_{\brc}}$. 
We refer to curvature of the form \eqref{eq:bundled curvature} as \emph{bundled curvature}, 
since the deformation parameters on various braid strands are bundled according to 
the corresponding link components.

The following gives a method for obtaining bundled curvature.

\begin{definition}\label{def:bundling v's}
Define a surjective algebra map 
$\k[\leftX,\rightX,\V] \to \k[\leftX,\rightX,\V^{\Omega(\perm)}]$ 
by declaring
\begin{equation}\label{eq:bundled vs}
v_{i,k} \mapsto 
\sum_{k\leq l\leq b_i} h_{l-k}\left(\sum_{\substack{j<i\\ j\sim i}}(\leftX_j - \rightX_{\perm\inv(j)})\right) v_{[i],l} \, .
\end{equation}
\end{definition}

\begin{lemma}\label{lemma:bundled curvature}
Suppose $\tw_\Delta(X) \in {}_{\brc,1}\YS(\SSBim)_{\brc,\perm\inv}$.
The substitution \eqref{eq:bundled vs} yields a complex 
\[
\tw_{\bDelta}\left(X\otimes_{\k[\leftX,\rightX,\V]} \k[\leftX,\rightX,\V^{\Omega(\perm)}]\right)
\]
with curvature \eqref{eq:bundled curvature}.
\end{lemma}
\begin{proof}
For simplicity, we consider the case when $\perm$ has only a single cycle
$[1]=[2]=\cdots=[m]$, 
so that $\leftX_{[1]}=\X_1+\cdots+\X_m$ and similarly for $\rightX_{[1]}$.
Observe that
\begin{align*}
h_l(\leftX_{[1]}-\rightX_{[1]})
&= h_l \Big((\leftX_1-\rightX_{\perm\inv(1)}) + \cdots+(\leftX_m-\rightX_{\perm\inv(m)})\Big) \\
&=\sum_{i=1}^m \sum_{k=1}^l h_k(\leftX_i-\rightX_{\perm\inv(i)}) 
	h_{l-k}(\leftX_1+\cdots+\leftX_{i-1}-\rightX_{\perm\inv(1)}-\cdots-\rightX_{\perm\inv(i-1)}) \, ,
\end{align*}
where the second line is obtained by iterating identities of the form
\[
h_l(\Z_1+\Z_2) = h_l(\Z_1) + \sum_{k=1}^l h_{l-k}(\Z_1)h_k(\Z_2).
\]
It thus follows that
\[
\sum_{l=1}^b h_l(\leftX_1+\cdots+\leftX_m - \rightX_1-\cdots- \rightX_m)v_{[1],l} 
= \sum_{i=1}^m \sum_{k=1}^l h_k(\X_i-\X_{\perm\inv(i)}') v_{i,k}
\]
under the substitution
$v_{i,k} = \sum_{l=k}^b h_{l-k}(\leftX_1+\cdots+\leftX_{i-1}-\rightX_{\perm\inv(1)}-\cdots-\rightX_{\perm\inv(i-1)}) v_{[1],l} $.  
This completes the proof when $\perm$ has one cycle. 
The proof for general $\perm$ is accomplished by applying the above computation 
to each cycle of $\perm$. It differs only in more-tedious bookkeeping, 
so we omit the details.
\end{proof}

There may be other changes of variables that obtain the curvature \eqref{eq:bundled curvature} 
from the curvature $\sum_{i,k}h_k(\X_{i}-\X_{\perm\inv(i)}') v_{i,k}$. 
The following says that, for curved Rickard complexes $\YS C(\b_{\brc})$,
any two choices are equivalent.

\begin{lemma}\label{lemma:bundled uniqueness}
Let $\b_{\brc} \in \Br_m(\Z_{\geq 1})$ be balanced
and let $\perm \in \symg_m$ be the permutation represented by $\b$.
If $\bDelta,\bDelta' \in \End_{\CS(\SSBim)}(C(\beta_{\brc})) \otimes \k[\V^{\Omega(\perm)}]$ 
are two Maurer--Cartan elements with the same curvature \eqref{eq:bundled curvature}, 
then $\tw_{\bDelta}(C(\beta_{\brc})) \simeq \tw_{\bDelta'}(C(\beta_{\brc}))$.
\end{lemma}
\begin{proof}
Similar to Lemma \ref{lem:curvinginvertibles}.
\end{proof}

Note that the curvature \eqref{eq:bundled curvature} vanishes upon identifying 
$\leftX_{[i]}$ and $\rightX_{[i]}$. Since Hochschild homology factors through the 
quotient $\leftX_i = \rightX_i$ (which implies $\leftX_{[i]} = \rightX_{[i]}$), 
we now arrive at the definition of our link invariant.

\begin{definition}
	\label{def:YHHH}
Let $\LB$ be a colored link which is presented as the closure of a balanced
colored $m$-strand braid $\b_{\brc}$ and let $\perm \in \symg_m$ 
be the permutation represented by $\b$. 
Let
\[
\bDelta \in \End_{\CS(\SSBim)}(C(\b_{\brc}))\otimes \k[\V^{\Omega(\perm)}]
\]
be the curved Maurer--Cartan element with curvature \eqref{eq:bundled curvature}
from Lemma \ref{lemma:bundled curvature}.
Let
\begin{equation}\label{eq:YHH}
\YHH_\bullet(\b_{\brc}) := 
\tw_{\HH_\bullet(\bDelta)}\Big(\HH_\bullet(C(\b))\otimes \k[\V^{\Omega(\perm)}] \Big)
\end{equation}
then the \emph{deformed, colored, triply-graded link homology} 
$\YH_{\KR}(\LB)$ is the homology of the chain complex
\begin{equation}\label{eq:YC(b)}
\YS C_{\KR}(\b_{\brc}) := 
(\adeg\tdeg\inv)^{\frac{1}{2}(\e(\beta_\brc)+N(\b_\brc)-n(\b_\brc))} \qdeg^{-\e(\beta_\brc)} 
\YHH_{\bullet}(\b_{\brc}) \, .
\end{equation}
In other words,
\[
\YH_{\KR}(\LB) := H(\YS C_{\KR}(\b_{\brc})) \, .
\]
\end{definition}

\begin{remark}
Although we have used the specific Maurer--Cartan element from Lemma \ref{lemma:bundled curvature} 
to define $\YH_{\KR}(\LB)$, Lemma \ref{lemma:bundled uniqueness} shows that we could have used 
any Maurer--Cartan element with bundled curvature \eqref{eq:bundled curvature}.
\end{remark}

\begin{remark}\label{rem:YC(L)}
We will sometimes abuse notation by writing $\YS C_{\KR}(\LB)$ instead of $\YS C_{\KR}(\b_{\brc})$. 
This is justified by Theorem \ref{thm:YHHH} below. 
As noted above, the set of cycles $\Omega(\perm)$ of the permutation $\perm$ determined by $\b$ 
can be identified with the set $\pi_0(\LB)$ of components of the link $\LB$.
We will thus also denote $\V^{\pi_0(\LB)} := \V^{\Omega(\perm)}$ in this context,
and further write $\V^\comp:= \V^{[i]}$ and $v_{\comp,r} := v_{[i],r}$ 
when $\comp \in \pi_0(\LB)$ corresponds to the cycle $[i] \in \Omega(\perm)$.
\end{remark}

As defined, $\YH_{\KR}(\LB)$ is an object of 
$\overline{\KS}\llbracket \adeg^\pm, \qdeg^\pm,\tdeg^\pm\rrbracket$, 
and (as with $H_{\KR}(\LB)$)
the homogeneous components are finite-dimensional, 
the $\adeg$- and $\tdeg$-grading are bounded, 
and the $\qdeg$-grading is bounded from below.
In fact, the module structure on $\YS C_{\KR}(\b_{\brc})$ 
allows us to endow $\YH_{\KR}(\LB)$ with additional structure, 
that we now describe.

Let $\LB$ be the closure of a colored braid $\b_\brc$.
For each component $\comp \in \pi_0(\LB)$ of color $b(\comp)$, 
introduce an alphabet $\X_{\comp}$ of cardinality $b(\comp)$ and set
\[
A_{\LB} := \bigotimes_{\comp \in \pi_0(\LB)} \Sym(\X_{\comp}) \otimes \k[\V^\comp] \, .
\]
Given a point $\point \in \b_\brc$ (away from a crossing), the $2$-categorical
structure of $\YS(\SSBim)$ endows $\YS C_{\KR}(\b_{\brc})$ with the structure of
a dg module over $\Sym(\X_\comp)$, where $\comp$ is the component of $\LB$
containing $\point$.
If we choose one such point for each component of $\LB$, this endows $\YS
C_{\KR}(\b_{\brc})$ with a dg $A_{\LB}$-module structure. 
We call such a
choice of points a \emph{pointing} of $\LB$.

We can now state precisely in what sense $\YH_{\KR}(\LB)$ is a colored link
invariant.

\begin{thm}\label{thm:YHHH} 
Choose a pointing of $\hat{\b_{\brc}}$, then 
the renormalized complex $\YS C_{\KR}(\b_{\brc})$ from \eqref{eq:YC(b)} 
depends only on the framed, oriented, colored link $\LB:=\hat{\b_{\brc}}$, 
up to quasi-isomorphism of $A_{\LB}$-modules.
Consequently, $\YS H_{\KR}(\LB)$ is a well-defined $A_{\LB}$-module, 
up to isomorphism.
\end{thm}

The proof of Theorem \ref{thm:YHHH} (i.e. the Markov invariance of $\YS C_{\KR}(\b_{\brc})$) 
is established in the following section. 
There, we show that $\YS C_{\KR}(\b_{\brc})$ can 
equivalently be described using curved Rickard complexes with 
strand-wise curvature $\sum_{k=1}^a \frac{1}{k}(p_k(\leftX) - p_k(\rightX)) \pv_k$; 
this curvature is the most-straightforwardly adapted to the Markov moves.

Before doing so, however, we first establish some easy consequences of Definition \ref{def:YHHH}.

\begin{example}
\label{exa:Yunknot}
Let $\UB(b)$ be the 0-framed $b$-colored unknot. 
Since $\UB(b)$ can be presented as the closure of
a single $b$-labeled strand, no differentials and grading shifts enter
into Definition~\ref{def:YHHH}. 
To compute the unknot invariant, let $\X$ and $\V$ be the size $b$ alphabets 
associated to the $b$-labeled strand.
Then, we have:
\[
\YS H_{\KR}(\UB(b)) = \YS C_{\KR}(\UB(b)) = \HH_{\bullet}(\Sym(\X))\otimes \k[\V] 
\cong \k[e_1(\X),\ldots,e_b(\X)]\otimes \largewedge[\eta_1,\ldots,\eta_b] \otimes \k[v_1,\dots,v_b]
\]
where $\wt(e_i(\X))=\qdeg^{2i}$, $\wt(\eta_i) = \adeg\inv \qdeg^{2i}$, $\wt(v_i) = \qdeg^{-2i} \tdeg^2$. 
The Poincar\'e series of the unknot homology is thus:
\[
\frac{(1+\adeg\inv \qdeg^{2})}{(1-\qdeg^2)(1-\qdeg^{-2}\tdeg^2)}
	\cdots \frac{(1+\adeg\inv \qdeg^{2b})}{(1-\qdeg^{2b})(1-\qdeg^{-2b}\tdeg^2) }
\]
while the corresponding derived sheet algebra is 
$\HH^{\bullet}(\Sym(\X))\otimes \k[\V]$ with Poincar\'e series:
\[
\frac{(1+\adeg \qdeg^{-2})}{(1-\qdeg^2)(1-\qdeg^{-2}\tdeg^2)}
	\cdots \frac{(1+\adeg \qdeg^{-2b})}{(1-\qdeg^{2b})(1-\qdeg^{-2b}\tdeg^2) } \, .
\]
\end{example}
 
 \begin{rem}
As for the undeformed invariant, we expect that
Theorem~\ref{thm:YHHH} can be strengthened to exhibit $\YS
C_{\KR}(\LB)$ as a dg-module over the tensor product of derived sheet algebras,
well defined up to quasi-isomorphism; see Remark~\ref{rem:dgmodule}.
\end{rem}

Note that $\YS H_{\KR}(\UB(b))$ is a free module over $\k[v_1,\ldots,v_b]$. 
In fact, this behavior persists for all $b$-colored \emph{knots}.

\begin{prop}
\label{prop:knots are boring}
Consider a framed oriented $b$-colored knot $\KB$. Then we have 
a $\Sym(\X^b) \otimes \k[v_1,\ldots,v_b]$-linear homotopy equivalence
\[
\YS C_{\KR}(\KB) \simeq C_{\KR}(\KB) \otimes \k[v_1,\ldots,v_b]
\]
and therefore an isomorphism
\[
\YS H_{\KR}(\KB) \cong H_{\KR}(\KB) \otimes \k[v_1,\ldots,v_b]
\]
of triply-graded $\Sym(\X^b) \otimes \k[v_1,\ldots,v_b]$-modules.
\end{prop}
\begin{proof}
Let $\b_{b^m}$ be a braid representative of $\KB$.  Since $\KB$ is a knot, the associated 
permutation $\perm\in \symg_m$ is an $m$-cycle, so $v_{[1],k}=v_{[2],k}=\dots=v_{[m],k}$ 
for all $1\leq k\leq b$ and
the bundled curvature element \eqref{eq:bundled curvature} equals
\[
\sum_{k=1}^m h_k(\X_1+\cdots + \X_m- (\X_1'+\cdots +\X_m'))v_{[1],k} = 0 \, .
\]
The uniqueness of curved lifts with bundled curvature 
(i.e. Lemma \ref{lemma:bundled uniqueness}) implies that 
we may take $\bar{\Delta}=0$ when forming $\YS C_{\KR}(\KB)$. 
This implies the first statement (after identifying $v_k=v_{[1],k}$), 
and taking homology gives the second.
\end{proof}

\subsection{Alphabet soup V: power sums}
\label{ss:power sum}

In our considerations thus far, we have worked with strand-wise curvature 
modeled on $h \Delta$-curvature (or, equivalently, $\Delta e$-curvature).
In order to most easily establish invariance of the complex $\YS C_{\KR}(\b_{\brc})$ 
under the Markov moves, we find it beneficial to also consider strand-wise curvature 
modeled on $\Delta p$-curvature:
\[
\sum_{k=1}^a \frac{1}{k}(p_k(\leftX)-p_k(\rightX))\pv_k \, .
\]
The following lemma will allow us to translate between such curvature and those 
previously considered.

\begin{lemma}\label{lemma:h and p curvatures}
We have 
$\sum_{k=1}^a \frac{1}{k}(p_k(\leftX) - p_k(\rightX)) \pv_k
=\sum_{k=1}^a h_k(\leftX-\rightX) v_k$ 
under the following mutually inverse substitutions
\begin{equation}\label{eq:vvdot}
\pv_k = \sum_{l=k}^a \frac{k}{l} h_{l-k}(\leftX-\rightX) v_l
\, , \quad 
v_k = \sum_{l=k}^a \frac{k}{l} (-1)^{l-k}e_{l-k}(\leftX-\rightX)\pv_l \, .
\end{equation}
\end{lemma}
\begin{proof}
This is a straightforward application of Newton's identity relating the power
sum and complete symmetric functions, as manifest in \eqref{eq:somerelations2}
and \eqref{eq:somerelations3}. For example, after the stated change of
variables, we have
\[
\sum_{1\leq k\leq a} \frac{1}{k}p_k(\leftX-\rightX) \pv_k = \sum_{1\leq k\leq l\leq a} 
\frac{1}{k}p_k(\leftX-\rightX)  \frac{k}{l} h_{l-k}(\leftX-\rightX) v_l 
\stackrel{\eqref{eq:somerelations3}}{=} \sum_{1\leq l\leq a} h_{l}(\leftX-\rightX) v_l \, .
\]
Since $p_k(\leftX-\rightX) = p_k(\leftX) - p_k(\rightX)$ we obtain the identity in the statement.
\end{proof}

\begin{remark}\label{rmk:pv equiv v} The above substitutions send 
$\pv_k \leftrightarrow v_k$ modulo $N(\leftX,\rightX)$.
Ultimately, this will justify our temporary transition to the curvature modeled on power
sums.
\end{remark}

For each $\brc=(b_1,\ldots,b_m)$ of $\SSBim$, 
introduce alphabets of deformation parameters $\pV_1,\ldots,\pV_m$ 
where $\pV_i=\{\pv_{i,1},\ldots,\pv_{i,b_i}\}$ and $\wt(\pv_{i,r}) = \qdeg^{-2r} \tdeg^2$.  
Set $\pV_\brc :=\pV_1\cup \cdots \cup \pV_m$. 
Given $\brc$ and a permutation $\perm\in \symg_m$, let
\begin{equation}\label{eq:defofS}
(S_\perm)_\brc :=  
\Big(\k[\pV_{\perm(\brc)}]\otimes \k[\pV_\brc]\Big) \Big/ 
\langle 1\otimes \pv_{i,k} - \pv_{\perm(i),k}\otimes 1 \mid 1\leq i\leq m, 1\leq k\leq b_i\rangle \, .
\end{equation}
We regard $(S_\perm)_\brc$ as both 
a $(\k[\pV_{\perm(\brc)}], \k[\pV_\brc])$-bimodule and a commutative algebra with multiplication
\[
(f_1\otimes g_1)\cdot (f_2\otimes g_2) = f_1f_2\otimes g_1g_2 \, .
\]
It will help us encode the action of the deformation parameters $\pv_{i,r}$ 
on the left and right ends of a strand.
Paralleling our notation for the alphabets $\leftX_i$ and $\rightX_i$, 
we will write $\pv_{i,k}:=\pv_{i,k}\otimes 1$ and $\pv_{i,k}':=1\otimes v_{i,k}$.
We will sometimes denote $(S_\perm)_\brc$ by either 
${}_{\perm(\brc)}(S_\perm)_\brc$ or ${}_{\perm(\brc)}(S_\perm)$.
If we let $\hComp=\otimes_{\k[\pV]}$ (in this context), then
\[
(S_{\perm_1})_{\perm_2(\brc)}\hComp (S_{\perm_2})_\brc \cong (S_{\perm_1\perm_2})_\brc \, .
\]

Finally, if $S$ is any $\k$-algebra and $B$ is an $(S,S)$-bimodule, 
then we will let $[B]$ denote the \emph{$S$-coinvariants}, i.e.
\[
[B] := B \big/  \big( \k \cdot \{sb-bs \mid b\in B, s\in S \} \big)
\]
Note that $[B]=\HH_0(B)$, 
but we wish to not confuse the reader with this occurrence of Hochschild homology 
and the functor $\HH_\bullet$, 
which (in this paper) we apply exclusively to singular Soergel bimodules. 
Observe that
\begin{equation}\label{eq:VdotOmega}
[(S_\perm)_\brc] \cong  \k[\pV_1,\ldots,\pV_m] / (\pv_{i,k}\sim \pv_{\perm(i),k}) \cong 
\k[\pv_{[i],k}]_{[i]\in \Omega(\perm), k\in \{1,\ldots,b_i\}} 
=: \k[\dot{\V}^{\Omega(\perm)}] \, .
\end{equation}

\begin{lemma}\label{lemma:exists power sum MC}
Let $\b_\brc$ be a balanced, colored braid, 
and let $\perm\in \symg_m$ be the permutation represented by $\b$.
There exists a Maurer--Cartan element\footnote{Since $\b_\brc$ is balanced, 
$\perm\inv(\brc) = \brc$, but we wish to emphasize the formal similarity to the complex $\YS C(\b_{\brc,\perm\inv})$.} 
$\Delta\in \End_{\CS(\SSBim)} \big( C(\b_{\brc}) \big) \otimes \big({}_\brc(S_\perm)_{\perm\inv(\brc)}\big)$
with curvature
\begin{equation}\label{eq:curvature with S}
\sum_{i=1}^m \sum_{k=1}^{b_i} \frac{1}{k} 
\Big((p_k(\leftX_i)\otimes \pv_{i,k}) - (p_k(\rightX_i)\otimes 
\pv_{i,k}')\Big).
\end{equation}
\end{lemma}
This twist is unique up to homotopy equivalence in the sense of Lemma \ref{lemma:bundled uniqueness}.
\begin{proof}
By \eqref{eq:defofS}, the curvature element \eqref{eq:curvature with S} can also be written as
\[
\sum_{i=1}^m \sum_{k=1}^{b_i} \frac{1}{k} 
\Big(p_k(\leftX_i)-p_k(\rightX_{\perm\inv(i)})\Big)\otimes \pv_{i,k}.
\]
Thus a Maurer--Cartan element with curvature \eqref{eq:curvature with S}
can be constructed from the Maurer--Cartan element on the curved complex $\YS C(\b_{\brc,\perm\inv})$, 
which has curvature $\sum_{i=1}^m \sum_{k=1}^{b_i} h_k(\leftX_i-\rightX_{\perm\inv(i)}) v_{i,k}$,
using Lemma \ref{lemma:h and p curvatures}.
\end{proof}

We next show how to obtain bundled $\Delta p$-curvature, 
establishing the analogue of Lemma \ref{lemma:bundled curvature} in this context.
In the following two results, the Maurer--Cartan element $\Delta$ is understood to be
the one constructed in the proof of Lemma \ref{lemma:exists power sum MC}.

\begin{lemma}\label{lemma:exists bundled power sum MC} 
Let $\b_\brc$ be a balanced, colored braid, 
and let $\perm\in \symg_m$ be the permutation represented by $\b$.
There exists a Maurer--Cartan element 
$[\Delta]\in \End_{\CS(\SSBim)} \big( C(\b_\brc) \big) \otimes [{}_\brc (S_\perm)_{\perm\inv(\brc)}]$ 
with curvature
\begin{equation}\label{eq:bundled curvature with S}
\sum_{[i]\in \Omega(\perm)} \sum_{k=1}^{b_i} \frac{1}{k}
\Big(p_k(\leftX_{[i]})-p_k(\rightX_{[i]})\Big)\otimes \pv_{[i],k} \, .
\end{equation}
\end{lemma}
\begin{proof}
The quotient map 
${}_\brc (S_\perm)_{\perm\inv(\brc)} \twoheadrightarrow [{}_\brc (S_\perm)_{\perm\inv(\brc)}]$ 
gives us an algebra map
\[
\End_{\CS(\SSBim)} \big( C(\b_{\brc}) \big) \otimes \big({}_\brc(S_\perm)_{\perm\inv(\brc)}\big)
\to \End_{\CS(\SSBim)} \big( C(\b_\brc) \big) \otimes [{}_\brc (S_\perm)_{\perm\inv(\brc)}] \, .
\]
Taking $[\Delta]$ to be the image of $\Delta$ under this algebra map produces a
Maurer--Cartan element with curvature \eqref{eq:bundled curvature with S}.
\end{proof}

The twist in Lemma \ref{lemma:exists bundled power sum MC} is unique, 
in the sense of Lemma \ref{lemma:bundled uniqueness}.
Using this $[\Delta]$, we can recover the complex $\YHH_\bullet(\b_\brc)$ 
from \eqref{eq:YHH}, up to homotopy equivalence.

\begin{lemma}\label{lemma:alternate YHH}
Let $\b_\brc$ be a balanced, colored braid, 
and let $\perm\in \symg_m$ be the permutation represented by $\b$.
After identifying $\pv_{[i],k}=v_{[i],k}$, 
we have
\begin{equation}\label{eq:alternate YHH}
\tw_{\HH_\bullet([\Delta])}\Big(\HH_\bullet \big( C(\b_\brc) \big) 
	\otimes [{}_\brc (S_\perm)_{\perm\inv(\brc)}]\Big) 
\simeq \YHH_\bullet(\b_{\brc})
\end{equation}
where $\HH_\bullet([\Delta])$ denotes the image of $[\Delta]$ from 
Lemma \ref{lemma:exists bundled power sum MC} under the algebra map
\[
\End_{\CS(\SSBim)} \big( C(\b_\brc ) \big) \otimes [{}_\brc (S_\perm)_{\perm\inv(\brc)}]
\rightarrow \End_{\k}\big( \HH_\bullet (C(\b_\brc)) \big) \otimes [{}_\brc (S_\perm)_{\perm\inv(\brc)}] \, .
\]
\end{lemma}
\begin{proof}
Note that a substitution as in Lemma \ref{lemma:h and p curvatures} will 
convert the curvature \eqref{eq:bundled curvature with S} into \eqref{eq:bundled 
curvature}.  
Further, since we work with bundled curvature and Hochschild homology identifies 
the alphabets $\leftX_{[i]}$ and $\rightX_{[i]}$, 
Remark \ref{rmk:pv equiv v} shows that the relevant substitution simply sets 
$\pv_{[i],k}=v_{[i],k}$.
Our uniqueness statement (Lemma \ref{lemma:bundled uniqueness}) 
then establishes the lemma.
\end{proof}

\begin{rem}\label{rem:Monodromies}
The Maurer--Cartan element $\Delta$ from Lemma \ref{lemma:exists power sum MC} 
is strict, in the sense of Remark \ref{rmk:strict}, thus its linear part determines 
null-homotopies $\tilde{\Xi}_{[i],k}$ for the action of 
$\frac{1}{k} \Big(p_k(\leftX_{[i]})-p_k(\rightX_{[i]})\Big)$. 
Applying $\HH_{\bullet}$ then produces the monodromy maps $\Xi_{\comp,k}$ 
from Proposition \ref{prop:intro}.
This pairs with the discussion preceding Theorem \ref{thm:YHHH} to establish 
Proposition \ref{prop:intro}.
\end{rem}

Using this ``power sum model'' for $\YHH_\bullet(\b_{\brc})$ established in 
Lemma \ref{lemma:alternate YHH},
we now prove Markov invariance of $\YS C_{\KR}(\b_{\brc})$.

\begin{proof}[Proof of Theorem \ref{thm:YHHH}]
It suffices to show that $\YS C_{\KR}(\b_{\brc})$ is invariant under the 
Markov moves, up to quasi-isomorphism of $A_{\LB}$-modules.
We will establish the Markov moves in turn, 
and then observe that all maps used are 
$A_{\LB}$-module quasi-isomorphisms. \\

\noindent \textbf{Markov I:}
Let $\b_\brc = \b'_\brcc \b''_\brc$ be balanced, where  
$\b'$ and $\b''$ are $m$-strand braids 
with corresponding permutations $\perm'$ and $\perm''$.
Hence, $\brc=\perm'(\brcc)$ and $\brcc=\perm''(\brc)$, 
so ${}_{\brcc}\b'' \b'_{\brcc}$ is a balanced, colored braid as well.
Let 
$\Delta' \in \End_{\CS(\SSBim)} \big( C({}_\brc \b'_\brcc ) \big) 
	\otimes \big({}_\brc(S_{\perm'})_\brcc \big)$ 
and 
$\Delta'' \in \End_{\CS(\SSBim)} \big( C({}_\brcc \b''_\brc ) \big) 
	\otimes \big({}_\brcc(S_{\perm''})_\brc \big)$ 
be Maurer--Cartan elements as in Lemma \ref{lemma:exists power sum MC}.  

A straightforward computation shows that the Maurer--Cartan elements
\begin{equation} \label{eq:power sum 1comp 1}
\Delta' \hComp \id + \id \hComp \Delta'' \in 
\End_{\CS(\SSBim)}\big(C(\b'_\brcc) \hComp C(\b''_\brc)\big)
	\otimes \big({}_\brc(S_{\perm'})_\brcc \hComp {}_\brcc (S_{\perm''})_\brc \big)
\end{equation}
and
\begin{equation}\label{eq:power sum 1comp 2}
\Delta'' \hComp \id + \id \hComp \Delta' \in 
\End_{\CS(\SSBim)}\big(C(\b''_\brc)\hComp C(\b'_\brcc)\big)
	\otimes \big({}_\brcc(S_{\perm''})_\brc \hComp {}_\brc (S_{\perm'})_\brcc \big)
\end{equation}
have curvature as in Lemma \ref{lemma:exists power sum MC}.  
Now, using Proposition \ref{prop:HH tracelike 2}, 
we have an isomorphism of dg algebras
\[
\End_{\k}\Big( \HH_\bullet \big(C(\b'_\brcc) \hComp C(\b''_\brc)\big) \Big)\otimes 
\big[ \big({}_\brc(S_{\perm'}) \hComp (S_{\perm''})_\brc \big) \big]
\cong
\End_{\k}\Big( \HH_\bullet \big(C(\b''_\brc)\hComp C(\b'_\brcc)\big) \Big)\otimes 
\big[ \big({}_\brcc(S_{\perm''}) \hComp (S_{\perm'})_\brcc \big) \big]
\]
which exchanges the Maurer--Cartan elements induced 
by \eqref{eq:power sum 1comp 1} and \eqref{eq:power sum 1comp 2}.
Since $C(\b' \b''_{\brc}) = C(\b'_\brcc) \hComp C(\b''_\brc)$ and 
$C(\b'' \b'_{\brcc}) = C(\b''_\brc) \hComp C(\b'_\brcc)$, 
this implies that 
\begin{multline}\label{eq:Markov1}
\tw_{\HH_\bullet([\Delta' \hComp \id + \id \hComp \Delta''])}
\Big(\HH_\bullet \big( C(\b' \b''_{\brc}) \big) 
	\otimes [{}_\brc (S_{\perm' \perm''})_\brc] \Big) \\
\cong
\tw_{\HH_\bullet([\Delta'' \hComp \id + \id \hComp \Delta'])}
\Big(\HH_\bullet \big( C(\b'' \b'_{\brcc}) \big) 
	\otimes [{}_\brcc (S_{\perm'' \perm'})_\brcc] \Big) \, .
\end{multline}
Thus, by Lemma \ref{lemma:alternate YHH}, 
$\YHH_\bullet(\b' \b''_{\brc}) \simeq \YHH_\bullet(\b'' \b'_{\brcc})$.
Since the numerical invariants $\e(-)$, $n(-)$, and $N(-)$ 
from Definition \ref{def:braid numerical invts} agree for 
agree for $\b' \b''_{\brc}$ and $\b'' \b'_{\brcc}$, 
this establishes invariance of $\YS C_{\KR}(-)$ under the first Markov move. \\

\noindent \textbf{Markov II:}
Let $\brc = \brcc\boxtimes b$ and suppose that $\b_\brc$ is a balanced, colored $m$-strand braid 
that can be written as ${}_\brc \b = {}_\brc(\b'\boxtimes \oone_{b})\agen_{m-1}$.
(Recall that $\agen_{m-1}$ denotes the $(m-1)^{st}$ Artin generator.)
Let $\perm\in \symg_m$ be the permutation represented by $\b$ 
and let $\perm'\in \symg_{m-1}$ be the permutation represented by $\b'$.
Observe that $\b'_\brcc$ is necessarily also balanced.

For $1\leq i\leq m-1$, 
the cycles $[i]_{\perm'}$ and $[i]_{\perm}$ are related by
\[
\begin{cases}
[i]_{\perm'} = [i]_{\perm} \smallsetminus \{m\} & \text{if } [i]_{\perm}= [m]_{\perm} \\
[i]_{\perm'} = [i]_{\perm} & \text{otherwise}
\end{cases}
\]
Thus, there is a canonical bijection between the cycles of $\perm$ and $\perm'$
given by sending $[i]_{\perm'}\mapsto [i]_\perm$ for $1\leq i\leq m-1$. 
(Note that $[m]_\perm = [m-1]_\perm$ by our hypotheses on the braid $\b$.) 
Henceforth we will identify the algebras
\[
[{}_\brc (S_\perm)_\brc ] = [ {}_\brcc (S_{\perm'})_\brcc] 
\]
and we will use the notation $\pv_{[i],k}$ without specifying whether $[i]$ is regarded 
as a cycle of $\perm$ or $\perm'$.

Introduce alphabets $\leftX_i, \rightX_i, \rightrightX_i$
which act by left-, middle-, and right-multiplication (respectively) on
\[
C({}_\brc\b) = \oone_{\brc}\hComp C(\b'\boxtimes \oone_b) 
	\hComp \oone_{\brc} \hComp  C(\agen_{m-1}) \hComp \oone_{\brc} \, .
\]
In particular we have 
$\rightX_i = \rightrightX_i$ for $i=1,\ldots,m-2$ and $\X_m=\rightX_m$ 
when acting on $C({}_\brc\b)$.
The bundled curvature \eqref{eq:bundled curvature with S} equals
\[
\sum_{i=1}^m \sum_{k=1}^{b_i} \frac{1}{k} \Big(p_k(\leftX_i) - p_k(\rightrightX_i)\Big)\pv_{[i],k} 
= \sum_{i=1}^m \sum_{k=1}^{b_i} \frac{1}{k} \Big(p_k(\leftX_i) - p_k(\rightX_i)\Big)\pv_{[i],k} 
+ \sum_{i=1}^m \sum_{k=1}^{b_i}  \frac{1}{k} \Big(p_k(\rightX_i) - p_k(\rightrightX_i)\Big)\pv_{[i],k} \, .
\]
Since $\X_m=\rightX_m$, 
the $i=m$ term of the first summation on the right is zero for all $k$. 
Furthermore, since $\rightX_i = \rightrightX_i$ for $i=1,\ldots,m-2$, 
we have
\[
\sum_{i=1}^m \sum_{k=1}^{b_i} \frac{1}{k} \Big(p_k(\X_i') - p_k(\X_i'')\Big)\pv_{[i],k} 
= \sum_{k=1}^{b_i} \frac{1}{k} 
\Big(p_k(\X_{m-1})+p_k(\X_m)-p_k(\X_{m-1}')-p_k(\X_m')\Big)\pv_{[m]} = 0 \, .
\]
Here, we have also used the fact that 
$\rightX_{m-1}+\rightX_m = \rightrightX_{m-1}+\rightrightX_m$ 
when acting on $C({}_\brc \agen_{m-1})$. 
Therefore, the curvature element \eqref{eq:bundled curvature with S} 
in the present setting reduces to
\[
\sum_{i=1}^m \sum_{k=1}^{b_i} \frac{1}{k} \Big(p_k(\X_i) - p_k(\X_i'')\Big)\pv_{[i],k} 
= \sum_{i=1}^{m-1} \sum_{k=1}^{b_i} \frac{1}{k} \Big(p_k(\X_i) - p_k(\X_i')\Big)\pv_{[i],k},
\]
which coincides with the bundled curvature 
\eqref{eq:bundled curvature with S} for $\b'$.  

Thus, if 
$[\Delta']\in \End_{\CS(\SSBim)} \big( C(\b'_\brcc) \big) \otimes [{}_{\brcc} (S_{\perm'})_\brcc]$ 
satisfies the conditions of Lemma \ref{lemma:exists power sum MC} for $\b'_\brcc$, 
then 
\[
([\Delta']\boxtimes \id)\hComp \id 
\in \End_{\CS(\SSBim)} \big( C(\b_\brc) \big) \otimes [{}_{\brc} (S_{\perm})_\brc]
\]
satisfies the conditions of Lemma \ref{lemma:exists power sum MC} for $\b_\brc$.  
Lemma \ref{lemma:alternate YHH} now gives that
\[
\tw_{\HH_\bullet(([\Delta']\boxtimes \id)\hComp \id)} 
\Big(\HH_\bullet\big(C({}_\brc (\b'\boxtimes \oone_b)\agen_{m-1})\big)\Big) 
\simeq \YHH_\bullet(\b_{\brc}) \, .
\]
Recall that Proposition \ref{prop:markov ii} gives a 
homotopy equivalence
\[
\HH_\bullet\Big(C({}_\brc (\b'\boxtimes \oone_b)\agen_{m-1})\Big) \simeq \adeg^{-b} 
\qdeg^{b^2}  \tdeg^b \HH_\bullet(\b'_\brcc)
\]
of undeformed Rickard complexes. 
The naturality of this equivalence with respect to the action of 
$\End_{\CS(\SSBim)}(C(\b'_\brcc))$ implies that the induced map
\[
\HH_\bullet\Big(C({}_\brc (\b'\boxtimes \oone_b)\agen_{m-1})\Big) 
	\otimes [{}_\brc (S_\perm)_\brc] 
\to
\adeg^{-b} \qdeg^{b^2}  \tdeg^b 
\HH_\bullet(\b'_\brcc)\otimes [{}_\brcc (S_{\perm'})_\brcc]
\]
intertwines the actions of Maurer--Cartan elements $\HH_\bullet(([\Delta']
\boxtimes \id)\hComp \id)$ and $\HH_\bullet([\Delta'])$.
Hence, Lemma \ref{lemma:alternate YHH} gives that
\begin{equation}\label{eq:Markov2}
\YHH_\bullet({}_\brc(\b'\boxtimes \oone_b)\agen_{m-1}) \simeq 
\adeg^{-b} \qdeg^{b^2}  \tdeg^b\, \YHH_\bullet(\b'_\brcc) \, ,
\end{equation}
and a similar argument gives that
\[
\YHH_\bullet({}_\brc(\b'\boxtimes \oone_b)\agen_{m-1}\inv) \simeq 
\qdeg^{-b^2}  \YHH_\bullet(\b'_\brcc) \, .
\]
Comparing the shifts in \eqref{eq:YC(b)}, 
this establishes the requisite behavior of $\YS C_{\KR}(-)$ under the second Markov move. \\

\noindent \textbf{Module structure:}
First, note that all homotopy equivalences used above are $\k[\V^{\pi_0(\LB)}]$-linear, 
so it suffices to show that they are quasi-isomorphisms of 
$\bigotimes_{\comp \in \pi_0(\LB)} \Sym(\X_\comp)$-modules.
This follows from Proposition \ref{prop:qiso}. 
Indeed, therein it is shown that, up to quasi-isomorphism, we can assume that 
the $\Sym(\X_\comp)$ action is given at any point $\point$ on the corresponding component. 
In particular, all homotopy equivalences following from our uniqueness results are 
quasi-isomorphisms, since we can assume $\Sym(\X_\comp)$ is acting via alphabets on the left, 
and these homotopy equivalences are equivalences in categories of curved complexes of bimodules. 
This similarly shows that the maps establishing Markov II invariant are quasi-isomorphisms: we can 
assume that the $\Sym(\X_\comp)$ action is given on the left, and does not act via $\leftX_m$. 
It remains to show that \eqref{eq:Markov1} is a quasi-isomorphism. For this, 
we can use Proposition \ref{prop:qiso} to assume that the $\Sym(\X_\comp)$ actions occur 
in the ``middle'' of the left-hand side (i.e. via the action of $\End_{\SSBim}(\oone_\brcc)$ on 
$C(\b'_{\brcc}) \hComp \oone_\brcc \hComp  C(\b''_{\brc})$), and on the left on the right-hand side. 
This map is a $\Sym(\X_\comp)$-linear isomorphism for these actions, thus a quasi-isomorphism.
\end{proof}

\subsection{Coefficients and spectral sequences}
\label{ss:homology with coefficients}
\label{sec:ss}

We now collect some straightforward results on homology with coefficients 
and spectral sequences that we need for our link splitting results in 
\S \ref{s:colored Hilb} -- \ref{sec:linksplit} below.
First, we will use the following  
common generalization of Definitions~\ref{def:YHHH} and~\ref{def:HHH}. 

\begin{definition}\label{def:YC(b,M)}
Let $\b_\brc$ be a balanced, colored braid. 
We consider two types of homology with coefficients. 
If $M$ is a $\k[\V^{\pi_0(\LB)}]$-module, we define 
\begin{equation}\label{eq:YLM1}
\YS C_{\KR}(\b_\brc,M) := \YS C_{\KR}(\b_\brc) \otimes_{\k[\V^{\pi_0(\LB)}]} M \, .
\end{equation}
If instead $M'$ is an $A_{\LB}$-module, we define
\begin{equation}\label{eq:YLM2}
\YS C_{\KR}(\b_\brc,M') := \YS C_{\KR}(\b_\brc) \stackrel{L}{\otimes}_{A_{\LB}} M' \, .
\end{equation}
If either case,  if $\LB$ is the colored link obtained as the closure of $\b_\brc$, then 
the \emph{deformed, colored, triply-graded Khovanov--Rozansky homology} 
of $\LB$ \emph{with coefficients} in $M$ is defined by
\[
\YS H_{\KR}(\LB,M) :=  H(\YS C_{\KR}(\b_\brc,M)) \, .
\] 
\end{definition}

\begin{remark}
Since $\YS C_{\KR}(\b_\brc)$ is free when regarded as a module over $\k[\V^{\pi_0(\LB)}]$, 
the tensor product in \eqref{eq:YLM1} coincides with the derived tensor product 
$\stackrel{L}{\otimes}_{\k[\V^{\pi_0(\LB)}]} M$.   
So, \eqref{eq:YLM1} can (and often will) be thought of as a special case of \eqref{eq:YLM2}, 
with $M'=A_{\LB} \otimes_{\k[\V^{\pi_0(\LB)}]} M$. 
Moreover, we expect that $\YS C_{\KR}(\b_\brc)$ is free as an $A_{\LB}$-module, 
so that the derived tensor product in \eqref{eq:YLM2} 
may be replaced with the ordinary tensor product. 
This amounts to showing that Hochschild homology of any singular Soergel bimodule 
$\oone_{\brc}B\oone_{\brc}$ is free as a module over the appropriate symmetric 
polynomial ring $\End_{\SSBim}(\oone_{\brc})$.
(Note that this holds in the uncolored case.)
\end{remark}

\begin{remark}
As in Remark \ref{rem:YC(L)}, we will sometimes write $\YS C_{\KR}(\LB,M)$
instead of $\YS C_{\KR}(\b_\brc,M)$. Strictly speaking, this complex depends on
a choice of braid representative of $\LB$, but the resulting complex depends
only on the framed, oriented, colored link $\LB$ up to quasi-isomorphism of $A_{\LB}$-modules. 
Thus,  the homology with coefficients $\YS H_{\KR}(\LB,M)$ is a well-defined module 
over $A_{\LB}/\mathrm{Ann}(M)$, up to isomorphism 
(here $\mathrm{Ann}(M)\subset A_{\LB}$ denotes the annihilator of $M$).
\end{remark}

\begin{remark}
We do not require that $M$ be a doubly graded module over $A_{\LB}$,
so tensoring with $M$ may involve a collapse of gradings.  
\end{remark}

\begin{exa}
\label{exa:coefficients} We consider the following examples:
\begin{enumerate}
\item 
For $M = \k[\V^{\pi_0(\LB)}]$, we have $\YS
C_{\KR}(\LB,M)=\YS C_{\KR}(\LB)$.

\item For the trivial module $M=\k$ 
(on which all variables $v_{[i],k}$ act by zero), 
we have $\YS C_{\KR}(\b_{\brc},M)=C_{\KR}(\b_{\brc})$.

\item \label{exa:coefficients3} Fix a scalar $z_\comp \in \k$ for each link
component $\comp \in \pi_0(\LB)$ and let $u$ be a formal variable of weight
$\wt(u)=\qdeg^{-2}\tdeg^2$. Consider the $\Z_\qdeg\times \Z_\tdeg$-graded
$\k[\V^{\pi_0(\LB)}]$-module $M_{\underline{z}} := \k[u]$ 
on which $v_{\comp,1} \in \V^{\pi_0(\LB)}$ acts as multiplication
by $z_\comp u$, and $v_{\comp,r}$ acts by zero when $r>1$. In this case, $\YS
H_{\KR}(\LB,M_{\underline{z}})$ recovers the deformed triply-graded homology of
Cautis--Lauda--Sussan \cite[Theorem 6.3]{MR4178751}, which satisfies splitting
properties between components labeled by distinct scalars $z_\comp$ after inverting $u$.
\end{enumerate}
\end{exa}

\begin{rem}
\label{rem:glN}
The bigraded Khovanov--Rozansky $\glN$ link homologies (as well as certain
deformations thereof) can be computed from the Rickard complexes of colored
braids by applying a functor that is trace-like up to homotopy, and which
induces homotopy equivalences for the second Markov move similar as in
Lemma~\ref{lem:Markov2}; see \cite[Theorem 3.21]{Wed3} or \cite[Section 6]{QR2}.
This implies the
existence of link splitting deformations of bigraded colored Khovanov--Rozansky
$\glN$ link homologies. Specifically, for coefficients in $M_{\underline{z}}$ as
in Example~\ref{exa:coefficients} \eqref{exa:coefficients3}, one obtains
bigraded colored homologies that satisfy link splitting properties and agree
with the invariants from \cite[Theorem 5.4]{MR4178751} (modulo conventions).
\end{rem}

Finally, we collect various spectral sequences associated with $\YS C_{\KR}(\LB)$. 
Let $\LB$ be a colored link. 
For each $\comp \in \pi_0(\LB)$, introduce an alphabet of (odd) variables 
$\Xi_\comp = \{\xi_{\comp,r}\}_{r=1}^{b(\comp)}$ with $\wt (\xi_{\comp,r}) = \qdeg^{2r}\tdeg^{-1}$ 
and where $b(\comp)$ denotes the color of the component $\comp$. 
Set $\Xi^{\pi_0(\LB)}:= \bigcup_{\comp \in \pi_0(\LB)} \Xi_\comp$.

\begin{proposition}\label{prop:YCvC}
We have
\[
\YS C_{\KR}(\LB) \cong \tw_{D}(C_{\KR}(\LB)\otimes \k[\V^{\pi_0(\LB)}])
\]
where the twist $D$ is $\k[\Xi^{\pi_0(\LB)}]$-linear and
$\V^{\pi_0(\LB)}$-irrelevant. We also have
\[
C_{\KR}(\LB) \simeq \tw_{D'}(\YS C_{\KR}(\LB) \otimes \largewedge[\Xi^{\pi_0(\LB)}])
\]
where the twist $D'$ equals 
$\sum_{\comp \in \pi_0(\LB)} \sum_{r=1}^{b(\comp)} v_{\comp,r} \xi_{\comp,r}^\ast$.
\end{proposition}
\begin{proof}
The first statement is true by construction, and the second is an easy exercise in Koszul 
duality (relating complexes of modules over polynomial and exterior algebras).
\end{proof}

The twist $D$ in Proposition \ref{prop:YCvC} strictly increases $\V$-degree, 
and $D'$ strictly decreases $\Xi$-degree. 
Taking the spectral sequence associated to these filtered complexes 
thus yields the following.

\begin{cor}
There are  spectral sequences
\[
\pushQED{\qed}
C_{\KR}(\LB)\otimes \k[\V^{\pi_0(\LB)}] \implies \YS C_{\KR}(\LB)
\, , \quad
\YS C_{\KR}(\LB) \otimes \largewedge[\Xi^{\pi_0(\LB)}] \implies C_{\KR}(\LB)\, . \qedhere
\popQED
\]
\end{cor}

\begin{remark}
The filtration by $\Xi$-degree on the complex 
$\tw_{D'}(\YS C_{\KR}(\LB) \otimes \largewedge[\Xi^{\pi_0(\LB)}])$ has finitely many steps, 
so the associated spectral sequence converges after finitely many steps.  
On the other hand, the filtration by $\V$-degree on $\tw_{\d}(C_{\KR}(\LB)\otimes \k[\V^{\pi_0(\LB)}])$ is infinite.  
Nonetheless, the associated spectral sequence is a first quadrant spectral sequence 
(bounded below in both cohomological degree and $\V$-degree), 
hence is reasonably well-behaved.
\end{remark}

The relation between $C_{\KR}(\LB)$ and $\YS C_{\KR}(\LB)$ 
can be reformulated in terms of homological perturbation theory as follows.

\begin{lemma}\label{lemma:simplifying YC}
There is a homotopy equivalence
\[
\YS C_{\KR}(\LB) \simeq \tw_{D''}(H_{\KR}(\LB) \otimes \k[\V^{\pi_0(\LB)}])
\]
of differential $\Z_\adeg\times \Z_\qdeg\times \Z_\tdeg$-graded 
$\k[\V^{\pi_0(\LB)}]$-modules, for some $\k[\V^{\pi_0(\LB)}]$-linear twist $D''$.  
Here, $H_{\KR}(\LB) \otimes \k[\V^{\pi_0(\LB)}]$ is viewed as a complex with zero differential.
\end{lemma}
\begin{proof}
Since we are working with field coefficients, we can choose a 
homotopy equivalence in $\overline{\KS}\llbracket
\adeg^{\pm},\qdeg^{\pm},\tdeg^\pm \rrbracket_{\dg}$ relating $C_{\KR}(\LB)$ and
its homology:
\[
C_{\KR}(\LB)\simeq H_{\KR}(\LB).
\]
Homological perturbation (Proposition \ref{prop:HPT}) now gives us a homotopy equivalence
\[
\YS C_{\KR}(\LB) \cong \tw_{D}(C_{\KR}(\LB) \otimes \k[\V^{\pi_0(\LB)}])
\simeq \tw_{D''}(H_{\KR}(\LB) \otimes \k[\V^{\pi_0(\LB)}])
\]
for some twist $D''$. 
Here, we use Remark \ref{rem:onesided} 
(and the analogue of \eqref{eq:ComponentWt} to see the necessary nilpotence).
\end{proof}

Lemma \ref{lemma:simplifying YC} is particularly useful in the following context.

\begin{defi}\label{def:parity} 
A colored link $\LB$ is \emph{parity} if its (undeformed) triply-graded link homology 
$H_{\KR}(\LB)$ is supported in purely even (or purely odd) cohomological degrees.
\end{defi}

\begin{thm}
\label{thm:parityfree}
Suppose that $\LB$ is parity, then 
\[
\YS C_{\KR}(\LB) \simeq H_{\KR}(\LB) \otimes \k[\V^{\pi_0(\LB)}]
\]
as differential $\Z_\adeg\times \Z_\qdeg\times \Z_\tdeg$-graded $\k[\V^{\pi_0(\LB)}]$-modules. 
In particular, parity implies that 
$\YS H_{\KR}(\LB)$ is a free $\k[\V^{\pi_0(\LB)}]$-module.
\end{thm}
\begin{proof}
The twist $D''$ from Lemma \ref{lemma:simplifying YC} is necessarily zero for degree
reasons: $\deg_{\tdeg}(D'')=1$ and 
$H_{\KR}(\LB) \otimes \k[\V^{\pi_0(\LB)}]$ 
is supported in exclusively even (or exclusively odd) cohomological degrees.
\end{proof}

\section{The curved colored skein relation}
\label{s:curved skein rel}

In \cite{HRW1}, we proved the \emph{colored skein relation} for 
singular Soergel bimodules, which takes the form of
a homotopy equivalence:
\begin{equation} \label{eq:convrecall0}
\tw_{D_1}\!\!\left( \bigoplus_{s=0}^b 
\qdeg^{s(b-1)} \tdeg^s \left\llbracket
\begin{tikzpicture}[scale=.4,smallnodes,rotate=90,anchorbase]
   \draw[very thick] (1,-1) to [out=150,in=270] (0,0); 
   \draw[line width=5pt,color=white] (0,-2) to [out=90,in=270] (.5,0) to [out=90,in=270] (0,2);
   \draw[very thick] (0,-2) node[right=-2pt]{$a$} to [out=90,in=270] (.5,0) 
   	to [out=90,in=270] (0,2) node[left=-2pt]{$a$};
   \draw[very thick] (1,1) to (1,2) node[left=-2pt]{$b$};
   \draw[line width=5pt,color=white] (0,0) to [out=90,in=210] (1,1); 
   \draw[very thick] (0,0) to [out=90,in=210] (1,1); 
   \draw[very thick] (1,-2) node[right=-2pt]{$b$} to (1,-1); 
   \draw[very thick] (1,-1) to [out=30,in=330] node[above=-2pt]{$s$} (1,1); 
   \end{tikzpicture}
		\right\rrbracket
	\right) 
	\simeq 
	\qdeg^{b(a-b-1)}\tdeg^b
	K \!\!\left(
		 \left\llbracket
\begin{tikzpicture}[scale=.4,smallnodes,anchorbase,rotate=270]
\draw[very thick] (1,-1) to [out=150,in=270] (0,1) to (0,2) node[right=-2pt]{$b$}; 
\draw[line width=5pt,color=white] (0,-2) to (0,-1) to [out=90,in=210] (1,1);
\draw[very thick] (0,-2) node[left=-2pt]{$b$} to (0,-1) to [out=90,in=210] (1,1);
\draw[very thick] (1,1) to (1,2) node[right=-2pt]{$a$};
\draw[very thick] (1,-2) node[left=-2pt]{$a$} to (1,-1); 
\draw[very thick] (1,-1) to [out=30,in=330] node[below=-1pt]{$a{-}b$} (1,1); 
\end{tikzpicture}
		\right\rrbracket \right),
\end{equation}
where $K(-)$ denotes the Koszul complex associated to the action of
$h_i(\X_2-\X_2')$ for $1\leq i\leq b$.  This is a homotopy equivalence of
filtered complexes, where the filtration on the left-hand side is given by the
index of summation $s$, and the filtration on the right-hand side requires some
preparation to describe.   We will use the shorthand $\qdeg^{s(b-1)} \tdeg^s
\MCCS^s_{a,b}$ for the summands on the left-hand side, because, reading
left-to-right, the associated diagram is the horizontal composition of
``Merge-Crossing-Crossing-Split.'' We denote the term on the right-hand side
(without the shift) by $\KMCS_{a,b}$, since it is a Koszul complex built on the
complex
\begin{equation}
	\label{eq:defMCS}
\MCS_{a,b} :=  
\left\llbracket
\begin{tikzpicture}[scale=.4,smallnodes,anchorbase,rotate=270]
\draw[very thick] (1,-1) to [out=150,in=270] (0,1) to (0,2) node[right=-2pt]{$b$}; 
\draw[line width=5pt,color=white] (0,-2) to (0,-1) to [out=90,in=210] (1,1);
\draw[very thick] (0,-2) node[left=-2pt]{$b$} to (0,-1) to [out=90,in=210] (1,1);
\draw[very thick] (1,1) to (1,2) node[right=-2pt]{$a$};
\draw[very thick] (1,-2) node[left=-2pt]{$a$} to (1,-1); 
\draw[very thick] (1,-1) to [out=30,in=330] node[below=-1pt]{$a{-}b$} (1,1); 
\end{tikzpicture}
\right\rrbracket
\end{equation}
whose notation is again suggestive, namely to be read left-to-right as
``Merge-Crossing-Split.''

The purpose of this section is to lift the colored skein relation, 
which is a homotopy equivalence in $\CS_{a,b}$, to the curved setting.  
We begin in \S \ref{ss:dv} with a quick discussion of (curved) Koszul complexes, 
and then in \S \ref{ss:usual skein rel} use this language to describe 
the homotopy equivalence \eqref{eq:convrecall}. 
In \S \ref{ss:curved RHS} and \ref{ss:curved skein}, 
we curve both sides of the skein relation and promote \eqref{eq:convrecall}
to a homotopy equivalence in $\VS_{a,b}$.

\subsection{Koszul complexes with curvature}
\label{ss:dv}

We will use the following construction of (curved) Koszul complexes.

\begin{definition} \label{def:dv} 
Fix integers $a,b\geq 0$ and let 
$\xi_1,\ldots,\xi_b$ be formal odd variables with $\wt(\xi_i)=\qdeg^{2i}\tdeg\inv$.  
For $X\in \CS_{a,b}$, let $K(X)\in \CS_{a,b}$ denote the twisted complex
\[
K(X):=\tw_{\d}(X\otimes \largewedge[\xi_1,\ldots,\xi_b]) 
\, , \quad \d=\sum_{1\leq i\leq b} h_i(\leftX_2-\rightX_2)\otimes \xi_i^\ast .
\]
Let $\YK(X)\in \VSred_{a,b} \subset \VS_{a,b}$ denote the curved twisted complex
\[
\YK(X) := \tw_{\Delta}(K(X)) = \tw_{\d+\Delta}(X\otimes \largewedge[\xi_1,\ldots,\xi_b]) 
\, , \quad \Delta = \sum_{1\leq i\leq b} \bar{v}_i\otimes \xi_i.
\]
\end{definition}

\begin{remark}\label{rmk:K and chain groups} We can write $K(X)$ and $\YK(X)$ as
one-sided twisted complexes constructed from $K(X^k)$ and $\YK(X^k)$, 
where $X = ( \bigoplus_k X^k, \d_X)$.
More precisely:
\[
K(X) = \Big(\cdots \xrightarrow{K(\d_X)} \tdeg^k K(X^k) \xrightarrow{K(\d_X)} \tdeg^{k+1} K(X^{k+1}) 
	\xrightarrow{K(\d_X)} \cdots\Big),
\]
and similarly for $\YK(X)$.
\end{remark}

\begin{prop}
$K$ and $\YK(-)$ extend to dg functors $\CS_{a,b}\rightarrow \CS_{a,b}$ and
$\CS_{a,b}\rightarrow \VSred_{a,b}\subset \VS_{a,b}$, respectively.
\end{prop}
\begin{proof}
This follows since we may describe
\[
K(X) \cong X \otimes_{\Sym(\leftX_2 | \rightX_2)} 
	\tw_{\d} \big( \Sym(\leftX_2 | \rightX_2) \otimes \largewedge[\xi_1,\ldots,\xi_b] \big)
\]
and
\[
\YK(X) \cong X \otimes_{\Sym(\leftX_2 | \rightX_2)} 
	\tw_{\d+\Delta} \big( \Sym(\leftX_2 | \rightX_2) \otimes \largewedge[\xi_1,\ldots,\xi_b] \big)\, .
\qedhere\]
\end{proof}

The formation of curved Koszul complexes such as $\YK(X)$ is one of the most
basic methods for constructing objects of $\VSred_{a,b}$.

\subsection{The uncurved skein relation}
\label{ss:usual skein rel}
We now recall the uncurved version of the skein relation, which was
established in the companion paper \cite{HRW1}.  

Consider the singular Soergel bimodules $W_k \in {}_{a,b}\SSBim_{a,b}$ indicated by the
following webs (with relevant alphabets specified in the second diagram):
\begin{equation}\label{eq:Wk}
	W_k:=
	\begin{tikzpicture}[rotate=90, anchorbase, smallnodes]
		\draw[very thick] (0,.25) to [out=150,in=270] (-.25,1) node[left,xshift=2pt]{$a$};
		\draw[very thick] (.5,.5) to (.5,1) node[left,xshift=2pt]{$b$};
		\draw[very thick] (0,.25) to node[left,yshift=-2pt,xshift=1pt]{$k$} (.5,.5);
		\draw[very thick] (0,-.25) to node[below,yshift=2pt]{} (0,.25);
		\draw[very thick] (.5,-.5) to [out=30,in=330] node[above,yshift=-2pt]{} (.5,.5);
		\draw[very thick] (0,-.25) to node[right,yshift=-2pt,xshift=-1pt]{$k$} (.5,-.5);
		\draw[very thick] (.5,-1) node[right,xshift=-2pt]{$b$} to (.5,-.5);
		\draw[very thick] (-.25,-1)node[right,xshift=-2pt]{$a$} to [out=90,in=210] (0,-.25);
	\end{tikzpicture}
	=
	\begin{tikzpicture}[rotate=90, anchorbase, smallnodes]
		\draw[very thick] (0,.25) to [out=150,in=270] (-.25,1) node[left,xshift=2pt]{$\leftX_1$};
		\draw[very thick] (.5,.5) to (.5,1) node[left,xshift=2pt]{$\leftX_2$};
		\draw[very thick] (0,.25) to node[left,yshift=-2pt,xshift=1pt]{$\leftM^{(k)}$} (.5,.5);
		\draw[very thick] (0,-.25) to node[below,yshift=2pt]{} (0,.25);
		\draw[very thick] (.5,-.5) to [out=30,in=330] node[above,yshift=-2pt]{} (.5,.5);
		\draw[very thick] (0,-.25) to node[right,yshift=-2pt,xshift=-1pt]{$\rightM^{(k)}$} (.5,-.5);
		\draw[very thick] (.5,-1) node[right,xshift=-2pt]{$\rightX_2$} to (.5,-.5);
		\draw[very thick] (-.25,-1)node[right,xshift=-2pt]{$\rightX_1$} to [out=90,in=210] (0,-.25);
	\end{tikzpicture}
\end{equation}
We identify the alphabets with subalphabets of
$\leftX=\{x_1,\ldots,x_{a+b}\}$ and $\rightX=\{x'_1,\ldots,x'_{a+b}\}$
as follows:
\begin{equation}\label{eq:Walphabets}
\begin{gathered}
\leftX_1 = \{x_1,\ldots,x_a\} \, , \quad \leftX_2 = \{x_{a+1},\ldots,x_{a+b}\} 
\, , \quad \leftM^{(k)} = \{x_{a+1},\ldots,x_{a+k}\} \\
\rightX_1 = \{x_1',\ldots,x_a'\}\, , \quad \rightX_2 = \{x_{a+1}',\ldots,x_{a+b}'\},
\, , \quad \rightM^{(k)} = \{x_{a+1}',\ldots,x_{a+k}'\}.
\end{gathered}
\end{equation}

\begin{definition}\label{def:MCS and KMCS} Recall the complex $\MCS_{a,b}$ from
\eqref{eq:defMCS}.  By convention, $\MCS_{a,b}=0$ if $a<b$.  Also define
\[
\MCSmin_{a,b}:=\left(W_b \xrightarrow{\d^H}  \qdeg^{-(a-b+1)} \tdeg W_{b-1} 
\xrightarrow{\d^H}  \cdots  \xrightarrow{\d^H} \qdeg^{-b(a-b+1)} \tdeg^b W_0\right) 
, \;\;
\d^H\big|_{W_k} := \chi_0^+|_{W_k} 
\]
Denote the corresponding Koszul complexes by $\KMCS_{a,b}:=K(\MCS_{a,b})$ and $\KMCSmin_{a,b}:=K(\MCSmin_{a,b})$. 
\end{definition}

\begin{prop}[{\cite[Equation {\MCSMCSmin}]{HRW1}}]\label{prop:MCS}
We have 
\[
\pushQED{\qed}
\MCSmin_{a,b} \simeq \MCS_{a,b} \quad \text{and} \quad
K(\MCSmin_{a,b}) \simeq K(\MCS_{a,b}). \qedhere
\popQED
\]
\end{prop}
The relevant mnemonic is that $\MCSmin_{a,b}$ is the result of certain Gaussian
eliminations on $\MCS_{a,b}$, i.e. a ``slimmer'' version thereof. We can write
$\KMCSmin_{a,b}$ as a one-sided twisted complex as in Remark \ref{rmk:K and
chain groups}:
\begin{equation}
K(\MCSmin_{a,b}) = \Big(K(W_b)\xrightarrow{\d^H} \qdeg^{a-b+1}\tdeg K(W_{b-1})\xrightarrow{\d^H} \cdots \xrightarrow{\d^H} \qdeg^{b(a-b+1)}\tdeg^b K(W_0)\Big),
\end{equation}
where $\d^H = K(\chi^+_0) \colon K(W_k)\rightarrow K(W_{k-1})$.  
Next, we change basis within each $K(W_k)$ by declaring
\begin{equation}\label{eq:xi and zeta}
\zeta_j^{(k)} := \sum_{i=1}^j (-1)^{i-1} e_{j-i}(\leftM^{(k)}) \otimes \xi_i \, , \quad
\xi_i = \sum_{j=1}^i (-1)^{j-1}h_{i-j}(\leftM^{(k)}) \otimes \zeta_j^{(k)}.
\end{equation}
The effect of this change of basis is captured by the following.

\begin{prop}[{\cite[\Zetadiffs]{HRW1}}]\label{prop:diffKMCSexplicit}
We have
\[
K(W_k) \cong \tw_{\d}(W_k\otimes \largewedge[\zeta^{(k)}_1,\ldots,\zeta^{(k)}_b])
\, , \quad \d = \sum_{i=1}^k (e_{i_j}(\leftM^{(k)})-e_{i_j}({\rightM}^{(k)}))\otimes (\zeta_i^{(k)})^\ast.
\]
With respect to this isomorphism, the differential $\d^H:K(W_k)\rightarrow K(W_{k-1})$ 
has a nonzero component
\[
W_k \otimes \zeta^{(k)}_{i_1}\cdots \zeta^{(k)}_{i_r}
\xrightarrow{\d^H}
W_{k-1}\otimes \zeta^{(k-1)}_{j_1}\cdots \zeta^{(k-1)}_{j_r}
\]
if and only if $i_p - j_p\in \{0,1\}$ for all $1\leq p\leq r$,
in which case it equals $\chi_m^+$ (see \eqref{eq:chi plus})
where $m=\sum_{p=1}^r (i_p-j_p)$. \qed
\end{prop}

A priori, the complex $\qdeg^{b(a-b-1)}\tdeg^b\KMCSmin_{a,b}$ 
(which is homotopy equivalent to the right-hand side of the skein relation)
is graded both by cohomological degree in $\MCSmin_{a,b}$ 
and the exterior algebra degree in $\largewedge[\xi_1,\ldots,\xi_b]$.  
According to \cite{HRW1}, the key to
understanding the colored skein relation is a refinement of the exterior grading
into two independent gradings.  This is accomplished with the following.

\begin{definition}\label{def:Pkls}
Let
\[
P_{k,l,s}  :=  \qdeg^{k(a-b+1)-2b} \tdeg^{2b-k}  \
W_k\otimes \largewedge^l[\zeta^{(k)}_1,\ldots, \zeta^{(k)}_k]\otimes
\largewedge^s[\zeta^{(k)}_{k+1},\ldots, \zeta^{(k)}_b],
\]
with indices constrained by $0\leq s\leq b$ and $0\leq l\leq k\leq b-s$.
\end{definition}

\begin{proposition}[{\cite[\KMCSdiffs]{HRW1}}]\label{prop:KMCSdiffs}
We have
\begin{equation}\label{eq:KMCS decomp}
\qdeg^{b(a-b-1)}\tdeg^b\KMCSmin_{a,b} \ \cong \ \tw_{\d^v+\d^h+\d^c}\left(\bigoplus_{0\leq l\leq k\leq b-s} P_{k,l,s}\right),
\end{equation}
where $\d^v$, $\d^h$, $\d^c$ are pairwise anti-commuting differentials given as
follows:
\begin{itemize}
\item the \emph{vertical differential} $\d^v: P_{k,l,s}\rightarrow P_{k,l-1,s}$
is the direct sum of Koszul differentials, up to the sign $(-1)^k$; its component
\[
W_k \otimes \zeta^{(k)}_{i_1}\cdots \zeta^{(k)}_{i_r} \xrightarrow{\d^v}
W_{k}\otimes \zeta^{(k)}_{i_1}\cdots \widehat{\zeta^{(k)}_{i_j}}\cdots
\zeta^{(k)}_{i_r}
\]
is $(-1)^{-k+j-1} (e_{i_j}(\leftM^{(k)})-e_{i_j}({\rightM}^{(k)}))$ if $1\leq i_j\leq k$
(and all other components are zero). 
\item  the \emph{horizontal differential} $\d^h$ and the connecting differential
$\d^c$ are uniquely characterized by $\d^h+\d^c=\d^H$ from Proposition
\ref{prop:diffKMCSexplicit}, together with
\[
\pushQED{\qed} 
\d^h(P_{k,l,s})\subset P_{k-1,l,s} \, , \quad \d^c(P_{k,l,s})\subset P_{k-1,l-1,s+1}. \qedhere
\popQED
\]
\end{itemize}
\end{proposition}

In other words, $\d^h$ is the part of $\d^H$ which preserves 
$s$-degree and $\d^c$ is the part of $\d^H$ which increases $s$-degree by $1$.
Since the differentials $\d^v$ and $\d^h$ preserve $s$-degree, 
we may reorganize the direct sum \eqref{eq:KMCS decomp} as follows:
\[
\qdeg^{b(a-b-1)}\tdeg^b\KMCSmin_{a,b}
\cong \tw_{\d^c}\left(\bigoplus_{0\leq s\leq b} \tw_{\d^v+\d^h}\left(\bigoplus_{0\leq l\leq k\leq b-s} P_{k,l,s}\right)\right).
\]
The skein relation essentially amounts to a topological interpretation of the $s^{th}$ summand above.  

\begin{proposition}[{\cite[\HRWSkeinrel]{HRW1}}]\label{prop:skein}
The complexes
\begin{equation}\label{eq:MCCSmin}
\MCCSmin^s_{a,b}:= 
\qdeg^{-s(b-1)}\tdeg^{-s} \tw_{\d^v+\d^h}
\left(\bigoplus_{0\leq l\leq k\leq b-s}P_{k,l,s}\right).
\end{equation}
satisfy
\[
\MCCSmin_{a,b}^s\simeq \left\llbracket
\begin{tikzpicture}[scale=.4,smallnodes,rotate=90,anchorbase]
   \draw[very thick] (1,-1) to [out=150,in=270] (0,0); 
   \draw[line width=5pt,color=white] (0,-2) to [out=90,in=270] (.5,0) to [out=90,in=270] (0,2);
   \draw[very thick] (0,-2) node[right=-2pt]{$a$} to [out=90,in=270] (.5,0) 
   	to [out=90,in=270] (0,2) node[left=-2pt]{$a$};
   \draw[very thick] (1,1) to (1,2) node[left=-2pt]{$b$};
   \draw[line width=5pt,color=white] (0,0) to [out=90,in=210] (1,1); 
   \draw[very thick] (0,0) to [out=90,in=210] (1,1); 
   \draw[very thick] (1,-2) node[right=-2pt]{$b$} to (1,-1); 
   \draw[very thick] (1,-1) to [out=30,in=330] node[above=-2pt]{$s$} (1,1); 
   \end{tikzpicture}
\right\rrbracket
\]
by \cite[\MCCSCor]{HRW1}. Consequently, the isomorphism
\[
\tw_{\d^c} \! \left(\bigoplus_{s=0}^b \qdeg^{s(b-1)} \tdeg^s \MCCSmin_{a,b}^s\right)
\cong 
\qdeg^{b(a-b-1)}\tdeg^b \KMCSmin_{a,b} 
\]
yields a homotopy equivalence
\begin{equation} \label{eq:convrecall}
\tw_{D_1}\!\!\left( \bigoplus_{s=0}^b 
	\qdeg^{s(b-1)} \tdeg^s \left\llbracket
\begin{tikzpicture}[scale=.4,smallnodes,rotate=90,anchorbase]
   \draw[very thick] (1,-1) to [out=150,in=270] (0,0); 
   \draw[line width=5pt,color=white] (0,-2) to [out=90,in=270] (.5,0) to [out=90,in=270] (0,2);
   \draw[very thick] (0,-2) node[right=-2pt]{$a$} to [out=90,in=270] (.5,0) 
   	to [out=90,in=270] (0,2) node[left=-2pt]{$a$};
   \draw[very thick] (1,1) to (1,2) node[left=-2pt]{$b$};
   \draw[line width=5pt,color=white] (0,0) to [out=90,in=210] (1,1); 
   \draw[very thick] (0,0) to [out=90,in=210] (1,1); 
   \draw[very thick] (1,-2) node[right=-2pt]{$b$} to (1,-1); 
   \draw[very thick] (1,-1) to [out=30,in=330] node[above=-2pt]{$s$} (1,1); 
   \end{tikzpicture}
		\right\rrbracket
	\right) \simeq 
	\qdeg^{b(a-b-1)}\tdeg^b
	K \!\!\left(
		 \left\llbracket
\begin{tikzpicture}[scale=.4,smallnodes,anchorbase,rotate=270]
\draw[very thick] (1,-1) to [out=150,in=270] (0,1) to (0,2) node[right=-2pt]{$b$}; 
\draw[line width=5pt,color=white] (0,-2) to (0,-1) to [out=90,in=210] (1,1);
\draw[very thick] (0,-2) node[left=-2pt]{$b$} to (0,-1) to [out=90,in=210] (1,1);
\draw[very thick] (1,1) to (1,2) node[right=-2pt]{$a$};
\draw[very thick] (1,-2) node[left=-2pt]{$a$} to (1,-1); 
\draw[very thick] (1,-1) to [out=30,in=330] node[below=-1pt]{$a{-}b$} (1,1); 
\end{tikzpicture}
		\right\rrbracket \right)
\end{equation}
by Propositions \ref{prop:MCS} and \ref{prop:HPT}. \qed
\end{proposition}

\subsection{Curving the right-hand side} \label{ss:curved RHS}

We now aim to promote Proposition \ref{prop:skein} to the curved setting. To
begin, we immediately note that the complex
\begin{equation}
	\label{eq:VKMCS}
\YKMCSmin_{a,b} := \YK(\MCSmin) \simeq  \YK(\MCS)
\end{equation}
provides a curved lift of $\KMCSmin_{a,b}$. We will write
$\YKMCSmin_{a,b}$ explicitly as a one-sided twisted complex
\begin{equation}\label{eq:YKMCSmin}
\YKMCSmin_{a,b} =
\left(
\YK(W_b) \xrightarrow{\d^H} 
\qdeg^{-(a-b+1)} \tdeg \YK(W_{b-1})  \xrightarrow{\d^H} 
\cdots  \xrightarrow{\d^H} \qdeg^{-b(a-b+1)} \tdeg^b \YK(W_0)\right)
\end{equation}
where we slightly abuse notation in writing $\d^H$ for $\YK(\d^H)$. Note that
each of the objects $\YK(W_k)$ is itself a curved complex, whose differentials
we draw as arrows pointing downward and upward in illustrations such as
\eqref{eq:exaYKMCS} and Example~\ref{exa:KMCS} below.

Our goal is to write down  the appropriate curved version of Proposition
\ref{prop:KMCSdiffs}.  First, consider the curved Koszul complex
$\YK(W_k):=\tw_{\d+\Delta}(W_k\otimes \largewedge[\xi_1,\ldots,\xi_b])$ 
from Definition \ref{def:dv}. 
Recall that
$\d= \sum_{i=1}^b h_i(\X_2-\X_2')\otimes \xi_i^\ast$ and $\Delta = \sum_{i=1}^b
\bar{v}_i \otimes \xi_i$ , so $(\d+\Delta)^2 = \sum_i h_i(\leftX_2 - \rightX_2) \bar{v}_i$.

\begin{lemma}\label{lemma:zeta koszul}
In terms of the $\zeta$-basis we have
\[
\YK(W_k)\cong \tw_{\d+\Delta}(W_k\otimes \largewedge[\zeta^{(k)}_1,\ldots,\zeta^{(k)}_b]),
\]
where
\[
\d = \sum_{1\leq i\leq b} (e_i(\M)-e_i(\M'))\otimes (\zeta^{(k)})^\ast
\, , \quad \Delta = \sum_{1\leq j\leq l\leq b} (-1)^{j-1} h_{l-j}(\M)\bar{v}_l\otimes \zeta^{(k)}_j.
\]
\end{lemma}
\begin{proof}
The formula for the uncurved differential $\d$ is given above in Proposition \ref{prop:diffKMCSexplicit}; 
the formula for $\Delta$ is immediate from \eqref{eq:xi and zeta}.
\end{proof}

\begin{proposition}\label{prop:KMCSdiffs curved} We have
\begin{equation}\label{eq:KMCS decomp curved}
\qdeg^{b(a-b-1)}\tdeg^b\YKMCSmin_{a,b} \ \cong \ \tw_{\d^v+\Delta^v+\d^h+\d^c+\Delta^c}\left(\bigoplus_{0\leq l\leq k\leq b-s} P_{k,l,s}\right),
\end{equation}
where $\d^v$, $\d^h$, $\d^c$ are as in Proposition \ref{prop:KMCSdiffs}, and
\begin{equation}\label{eq:twists}
\begin{aligned}
	\Delta^v \colon P_{k,l,s}\to P_{k,l+1,s} \, , \quad & 
	\Delta^v = \sum_{\substack{1\leq j \leq l\leq b\\ j\leq k}} (-1)^{j-1}  h_{l-j}(\M^{(k)}) \otimes \zeta_j^{(k)}\bar{v}_l\, ,\\
	\Delta^c\colon P_{k,l,s}\rightarrow P_{k,l,s+1} \, , \quad  &
	\Delta^c = \sum_{\substack{1\leq k\leq l\leq b \\ j>k}} (-1)^{j-1}  h_{l-j}(\M^{(k)}) \otimes \zeta_j^{(k)}\bar{v}_l \, .
\end{aligned}
\end{equation} 
Moreover, the endomorphisms $\d^v,\Delta^v,\d^h,\d^c,\Delta^c$ satisfy 
the following relations:
\[
(\d^h+\d^v+\Delta^v)^2 = \sum_i h_i(\leftX_2 - \rightX_2) \bar{v}_i \, , \quad
(d^c+\Delta^c)^2 = 0 \, , \quad
[\d^h+\d^v+\Delta^v,d^c+\Delta^c]=0 
\]
%
\end{proposition}
\begin{proof}
The first statement holds by construction, and everything else 
follows by taking components in $(\d+\Delta)^2 = \sum_i h_i(\leftX_2 - \rightX_2) \bar{v}_i$.
\end{proof}

Since $\Delta^v$ preserves $s$, we may define curved lifts of $\MCCSmin_{a,b}^s$
as follows.

\begin{definition}\label{def:curved MCCS}
Let
\begin{equation}\label{def:YMCCSmin}
\YMCCSmin_{a,b}^s:= \qdeg^{-s(b-1)}\tdeg^{-s} \
\tw_{\d^v+\d^h+\Delta^v} 
\left(\bigoplus_{0\leq l \leq k \leq b-s} P_{k,l,s}\right).
\end{equation}
This is a well-defined $1$-morphism in $\VSred_{a,b}$ by 
Proposition \ref{prop:KMCSdiffs curved}.
\end{definition}

The following holds by construction. 
\begin{prop}
\label{prop:reduced-curved-skein-rel}
For all integers $a,b\geq 0$ we have
\begin{equation}
\label{eq:curvedskeinmincx}
\tw_{\d^c+\Delta^c}\left(\bigoplus_{s=0}^b 
\qdeg^{s(b-1)} \tdeg^s \YMCCSmin_{a,b}^s \right)
\cong
\qdeg^{b(a-b-1)}\tdeg^b \ \YKMCSmin_{a,b} 
\end{equation}
in which the Maurer--Cartan element $\d^c+\Delta^c$ increases the index $s$ by
one. \qed
\end{prop}

\begin{exa}\label{exa:KMCS} 
We illustrate the complex $\YKMCSmin_{2,2}$, 
as well as the subquotients 
$P_{\bullet,\bullet,s}=\qdeg^{s} \tdeg^s \ \MCCSmin^s_{2,2}$ 
for $0\leq s\leq 2$. We use the symbol $\cdot$ instead of
$\otimes$ to declutter the diagram. 
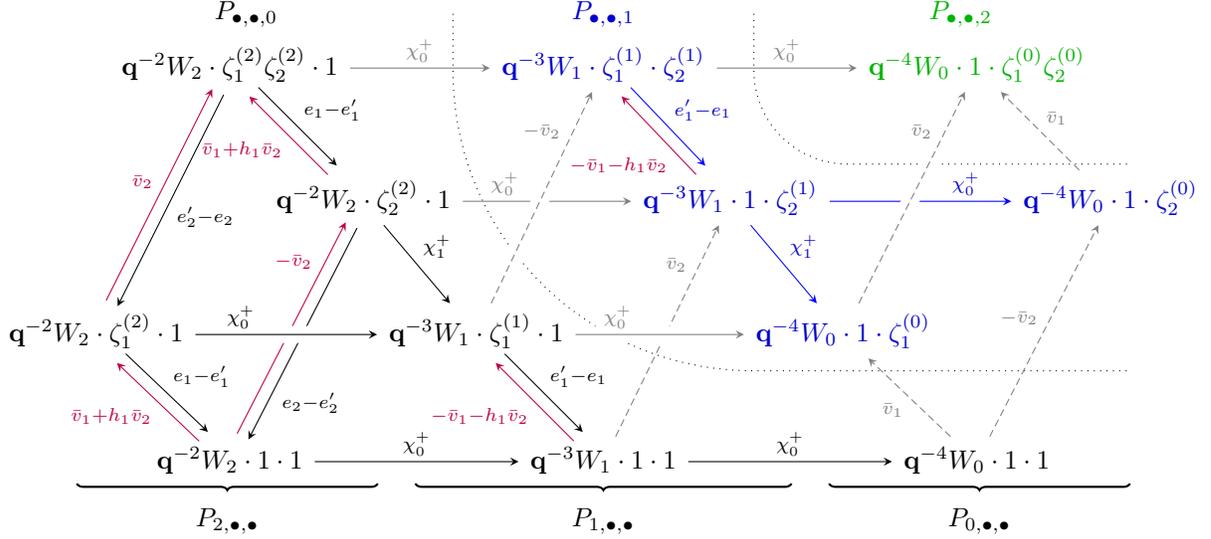
\[
	\begin{tikzpicture}[anchorbase]
		\draw[dotted]  (2,3.75) to (2,3) to [out=270,in=180] (3.25,1.75) to (7,1.75);
		\draw[dotted] (-2,3.75) to (-2,3) to [out=270,in=180] (2,-1) to (7,-1);
		\node[scale=1] at (-4.75,3.75){$P_{\bullet,\bullet,0}$};
		\node[scale=1] at (0,3.75){$\BLUE{P_{\bullet,\bullet,1}}$};
		\node[scale=1] at (4.75,3.75){$\GREEN{P_{\bullet,\bullet,2}}$};
		\node[scale=1] at (0,0.5){
\begin{tikzcd}[row sep=3em,column sep=-3.2em]
& 
\qdeg^{-2}W_2\cdot \zeta^{(2)}_1\zeta^{(2)}_2 \cdot 1 
	\arrow[ddl, "e_2'-e_2", shift left] 
	\arrow[dr,"e_1-e_1'", shift left]
	\arrow[from=dr, shift left,purple,"\bar{v}_1+h_1\bar{v}_2"] 
	\arrow[rrr,gray, "\GRAY{\chi^+_0}"] & & & 
\BLUE{\qdeg^{-3}W_1\cdot\zeta^{(1)}_1 \cdot \zeta^{(1)}_2}
	\arrow[dr,blue,"\BLUE{e_1'-e_1}", shift left] 
	\arrow[from=dr, shift left,purple,"-\bar{v}_1-h_1\bar{v}_2" pos=.3] 
	\arrow[rrr,gray, "\GRAY{\chi^+_0}"] & & & 
\GREEN{\qdeg^{-4}W_0\cdot 1 \cdot \zeta^{(0)}_1\zeta^{(0)}_2}	
	\arrow[from=dr,dashed,gray,swap,"\bar{v}_1"] & 
\\
& &
\qdeg^{-2}W_2\cdot \zeta^{(2)}_2 \cdot  1	
	\arrow[ldd, "e_2-e_2'" near end, shift left]
	\arrow[from=ldd, shift left,purple,"-\bar{v}_2" near end] 
	\arrow[dr,, "\chi^+_1"] 
	\arrow[rrr,gray, "\GRAY{\chi^+_0}" near start] & & & 
\BLUE{\qdeg^{-3}W_1\cdot 1 \cdot \zeta^{(1)}_2}  
	\arrow[from=ldd,dashed, gray,"\bar{v}_2" near end] 
	\arrow[dr,blue,"\BLUE{\chi^+_1}"] 
	\arrow[rrr,blue, "\BLUE{\chi^+_0}" near end]  & & & 
\BLUE{\qdeg^{-4}W_0\cdot 1 \cdot \zeta^{(0)}_2}  
	\arrow[from=ldd,dashed,gray,"-\bar{v}_2"]
\\
\qdeg^{-2}W_2\cdot \zeta^{(2)}_1 \cdot 1	
	\arrow[dr,"e_1-e_1'", shift left]
	\arrow[from=dr, shift left,purple,"\bar{v}_1+h_1\bar{v}_2"] 
	\arrow[rrr,crossing over, "\chi^+_0" near start] 
	\arrow[uur,shift left,purple,"\bar{v}_2"] & & & 
\qdeg^{-3}W_1\cdot \zeta^{(1)}_1 \cdot 1	
	\arrow[dr,"e_1'-e_1", shift left]
	\arrow[from=dr, shift left,purple,"-\bar{v}_1-h_1\bar{v}_2"] 
	\arrow[rrr,crossing over,gray, "\GRAY{\chi^+_0}" near start] 
	\arrow[uur,crossing over,dashed,gray,"-\bar{v}_2" near end] & & & 
\BLUE{\qdeg^{-4}W_0\cdot 1 \cdot \zeta^{(0)}_1}	
	\arrow[from=dr,dashed,gray,"\bar{v}_1"] 
	\arrow[uur,crossing over,dashed,gray,"\bar{v}_2" near end] & & 
\\
&
 \qdeg^{-2}W_2 \cdot 1 \cdot 1 
 	\arrow[rrr, "\chi^+_0"]& & & 
\qdeg^{-3}W_1 \cdot 1\cdot 1  
	\arrow[rrr, "\chi^+_0"]& & & 
\qdeg^{-4} W_0 \cdot 1 \cdot 1	 &
\end{tikzcd}
};
\draw[thick,decoration={brace,mirror,raise=0.5cm},decorate] (-7,-2) to (-3,-2);
\node at (-5,-3) {$P_{2,\bullet,\bullet}$};
\draw[thick,decoration={brace,mirror,raise=0.5cm},decorate] (-2.5,-2) to (2.5,-2);
\node at (0,-3) {$P_{1,\bullet,\bullet}$};
\draw[thick,decoration={brace,mirror,raise=0.5cm},decorate] (3,-2) to (7,-2);
\node at (5,-3) {$P_{0,\bullet,\bullet}$};
\end{tikzpicture}
\]

Black and blue horizontal arrows correspond to components of $\d^h$. All other
black and blue arrows indicate non-zero components of $\d^v$.   The curved twist 
$\Delta^v$ is indicated by red arrows.  Finally the connecting differential $\d^c$ and its 
curved correction $\Delta^c$ are marked using grey horizontal arrows and grey dashed 
arrows, respectively.
\end{exa}

\subsection{The curved colored skein relation}\label{ss:curved skein}
The curved colored skein relation will follow by showing that 
the complex $\YMCCSmin_{a,b}^s$ is homotopy equivalent 
to a curved lifts of the complex
\[
\MCCS_{a,b}^s =
\left\llbracket
\begin{tikzpicture}[scale=.4,smallnodes,rotate=90,anchorbase]
   \draw[very thick] (1,-1) to [out=150,in=270] (0,0); 
   \draw[line width=5pt,color=white] (0,-2) to [out=90,in=270] (.5,0) to [out=90,in=270] (0,2);
   \draw[very thick] (0,-2) node[right=-2pt]{$a$} to [out=90,in=270] (.5,0) 
   	to [out=90,in=270] (0,2) node[left=-2pt]{$a$};
   \draw[very thick] (1,1) to (1,2) node[left=-2pt]{$b$};
   \draw[line width=5pt,color=white] (0,0) to [out=90,in=210] (1,1); 
   \draw[very thick] (0,0) to [out=90,in=210] (1,1); 
   \draw[very thick] (1,-2) node[right=-2pt]{$b$} to (1,-1); 
   \draw[very thick] (1,-1) to [out=30,in=330] node[above=-2pt]{$s$} (1,1); 
   \end{tikzpicture}
\right\rrbracket.
\] 
We begin by defining these curved lifts.
Since $\MCCS_{a,b}^s$ is not invertible for $s\neq 0,b$, 
we cannot simply invoke Lemma \ref{lem:curvinginvertibles} to define the curved lift. 
Instead, we will bootstrap from the $s=0$ case.

\begin{conv}\label{conv:FT}
Since
\[
\MCCS_{a,b}^0 = 
\left\llbracket
 \begin{tikzpicture}[scale=.5,smallnodes,anchorbase,rotate=90]
 	\draw[very thick] (1,0) to [out=90,in=270] (0,1.5) node[left=-2pt]{$a$};
 	\draw[line width=5pt,color=white] (1,-1.5) to [out=90,in=270] (0,0) 
		to [out=90,in=270] (1,1.5);
 	\draw[very thick] (1,-1.5) node[right,xshift=-2pt]{$b$} to [out=90,in=270] (0,0) 
		to [out=90,in=270] (1,1.5) node[left=-2pt]{$b$};
 	\draw[line width=5pt,color=white] (0,-1.5) to [out=90,in=270] (1,0);
 	\draw[very thick] (0,-1.5) node[right,xshift=-2pt]{$a$} to [out=90,in=270] (1,0);
 \end{tikzpicture}
\right\rrbracket
\]
is the Rickard complex assigned to the $(a,b)$-colored 
($2$-strand) \emph{full twist} braid, 
we will denote this complex by $\FT_{a,b} := \MCCS_{a,b}^0$.
Similarly, we let $\FTmin_{a,b} := \MCCSmin_{a,b}^0$.
\end{conv}

We will make use of the following functors in studying 
$\MCCS_{a,b}^s$ when $s>0$.
\begin{definition}
Let $\I^{(s)}: \CS_{a,\ell}\rightarrow \CS_{a,\ell+s}$ denote the functor defined by
\[
\I^{(s)}(X):=
\begin{tikzpicture}[scale=.5,smallnodes,rotate=90,anchorbase]
	\draw[very thick] (1,2.25) to (1,3) node[left=-2pt]{$\ell{+}s$};
	\draw[very thick] (1,-3) node[right=-2pt]{$\ell{+}s$} to (1,-2.25); 
	\draw[very thick] (1,-2.25) to [out=30, in=270] (1.75,-1.6) 
		to (1.75,0) node[below,yshift=1pt]{$s$}to (1.75,1.6)  to  [out=90,in=330] (1,2.25); 
	\node[yshift=-2pt] at (0,0) {\normalsize$X$};
	\draw[very thick] (.75,1.6) rectangle (-.75,-1.6);
	\draw[very thick] (1,-2.25) to [out=150,in=270] 
		node[right,xshift=-1pt,yshift=-1pt]{$\ell$} (.25,-1.6);
	\draw[very thick] (1,2.25) to [out=210,in=90] 
		node[left,xshift=1pt,yshift=-1pt]{$\ell$} (.25,1.6);
	\draw[very thick] (-.25,3) node[left=-2pt]{$a$} to (-.25,1.6);
	\draw[very thick] (-.25,-3) node[right=-2pt]{$a$} to (-.25,-1.6);
\end{tikzpicture} \, .
\]
\end{definition}

More precisely, $\I^{(s)} \colon \CS_{a,\ell}\rightarrow \CS_{a,\ell+s}$ is defined as the composition of first
applying $ ( - ) \boxtimes \oone_{s}$ 
and then horizontal pre- and post-composing 
with $\oone_{a}\boxtimes {}_{(\ell,s)}S_{(\ell+s)}$ and 
$\oone_{a}\boxtimes {}_{(\ell+s)}M_{(\ell,s)}$, respectively.
Here, ${}_{(\ell,s)}S_{(\ell+s)}$ and ${}_{(\ell+s)}M_{(\ell,s)}$ 
denote the split and merge bimodules from \eqref{eq:GenWeb},
and $|\leftL| = \ell = |\rightL|$.

\begin{conv}\label{conv:I alphabets} 
When considering endomorphisms of complexes
of the form $\I^{(s)}(X)$, we will often use the alphabet naming convention
indicated below:
\[
\I^{(s)}(X):=
\begin{tikzpicture}[scale=.5,smallnodes,rotate=90,anchorbase]
	\draw[very thick] (1,2.25) to (1,3) node[left]{$\leftX_2$};
	\draw[very thick] (1,-3) node[right]{$\rightX_2$} to (1,-2.25); 
	\draw[very thick] (1,-2.25) to [out=30, in=270] (1.75,-1.6) 
		to (1.75,0) node[below]{$\B$}to (1.75,1.6)  to  [out=90,in=330] (1,2.25); 
	\node[yshift=-2pt] at (0,0) {\normalsize$X$};
	\draw[very thick] (.75,1.6) rectangle (-.75,-1.6);
	\draw[very thick] (1,-2.25) to [out=150,in=270] 
		node[right,xshift=-1pt,yshift=-1pt]{$\rightL$} (.25,-1.6);
	\draw[very thick] (1,2.25) to [out=210,in=90] 
		node[left,xshift=1pt,yshift=-1pt]{$\leftL$} (.25,1.6);
	\draw[very thick] (-.25,3) node[left]{$\leftX_1$} to (-.25,1.6);
	\draw[very thick] (-.25,-3) node[right]{$\rightX_1$} to (-.25,-1.6);
\end{tikzpicture}.
\]
In other words, the above diagram indicates how
$\Sym(\leftX_1|\rightX_1|\leftX_2|\rightX_2|\leftL|\rightL|\B)$ acts on the functor $\I^{(s)}$ by
natural transformations.
\end{conv}

Set $b=\ell +s$.
We now extend $\I^{(s)}$ to categories of curved complexes 
$\I^{(s)}\colon \VSred_{a,\ell}\rightarrow \VSred_{a,b}$. 
Note that since ${}_{(\ell,s)}S_{(\ell+s)}$ and ${}_{(\ell+s)}M_{(\ell,s)}$
are not $1$-morphisms in $\YS(\SSBim)$, 
we cannot simply invoke the $2$-categorical operations from \S \ref{ss:curved cxs}.

\begin{definition}\label{def:functor Ikb}
Let $I^{(s)} \colon \VSred_{a,\ell}\rightarrow  \VSred_{a,b}$  be the functor defined on objects by
\[
I^{(s)}\big( \tw_{\Delta}(X) \big) = \tw_{I^{(s)}(\Delta)} (I^{(s)}(X))
\]
and on morphisms by the map
\[
\Hom_{\CS_{a,\ell}}(X,Y) \otimes \k[\overline{\V}^{(\ell)}]
\to
\Hom_{\CS_{a,b}} \big( I^{(s)}(X),I^{(s)}(Y)\big) \otimes \k[\overline{\V}^{(b)}]
\]
induced from the map
\[
I^{(s)} \colon \Hom_{\CS_{a,\ell}}(X,Y) \to \Hom_{\CS_{a,b}} \big( I^{(s)}(X),I^{(s)}(Y)\big)
\]
and the assignment \eqref{eq:vbar l and vbar b} (with
$\leftL=\X_{[a+1,\ldots,a+\ell]}$ and $\rightL=\rightX_{[a+1,\ldots,a+\ell]}$).
\end{definition}
Note that $\I^{(s)}$ preserves the curvature elements by \eqref{eq:reduced
curvature stability}, together with the observation that
\[
h_i(\leftL-\rightL) = h_i(\X_2-\X_2').
\]
Thus, this functor is indeed well-defined.

Let $\VFT_{a,b-s} \in \VSred_{a,b-s} \subset \VS_{a,b-s}$ 
denote a chosen curved lift of the full twist Rickard complex $\FT_{a,b-s}$;
this exists since the latter is invertible in $\CS_{a,b-s}$.
Applying the functor $I^{(s)}$ then defines the following curved lift of $\MCCS_{a,b}^s$:
\[
\left\llbracket
\begin{tikzpicture}[scale=.4,smallnodes,rotate=90,anchorbase]
   \draw[very thick] (1,-1) to [out=150,in=270] (0,0); 
   \draw[line width=5pt,color=white] (0,-2) to [out=90,in=270] (.5,0) to [out=90,in=270] (0,2);
   \draw[very thick] (0,-2) node[right=-2pt]{$a$} to [out=90,in=270] (.5,0) 
   	to [out=90,in=270] (0,2) node[left=-2pt]{$a$};
   \draw[very thick] (1,1) to (1,2) node[left=-2pt]{$b$};
   \draw[line width=5pt,color=white] (0,0) to [out=90,in=210] (1,1); 
   \draw[very thick] (0,0) to [out=90,in=210] (1,1); 
   \draw[very thick] (1,-2) node[right=-2pt]{$b$} to (1,-1); 
   \draw[very thick] (1,-1) to [out=30,in=330] node[above=-2pt]{$s$} (1,1); 
   \end{tikzpicture}
\right\rrbracket_\VS
:=
\I^{(s)}(\VFT_{a,b-s})\, .
\]

\begin{proposition}\label{prop:MCCS topologically} 
	We have $\YMCCSmin_{a,b}^s \simeq  \I^{(s)}(\VFT_{a,b-s})$ for all $0\leq s\leq b$.
\end{proposition}

\begin{proof}
Recall that \cite[\QzeroFT]{HRW1} proves $\MCCSmin^0_{a,b{-}s}
\simeq \FT_{a,b-s}$.  Proposition~\ref{prop:conservation} guarantees that there
exists an induced homotopy equivalence between the curved lift
$\YMCCSmin^0_{a,b{-}s}$ of $\MCCSmin^0_{a,b{-}s}$ and some curved lift of
$\FT_{a,b-s}$. However, since $\FT_{a,b-s}$ is invertible in $\CS_{a,b-s}$,
Lemma~\ref{lem:curvinginvertibles} implies such lifts are unique up to homotopy
equivalence, and we get $\YMCCSmin^0_{a,b{-}s} \simeq \VFT_{a,b-s}$.   Applying
the functor $\I^{(s)}$ then gives
\[
\I^{(s)}(\YMCCSmin^0_{a,b{-}s}) \simeq \I^{(s)}(\VFT_{a,b-s})\, .
\]
To complete the proof, we will show that
\begin{equation}\label{eq:MCCS curved}
\YMCCSmin_{a,b}^s \cong \I^{(s)}(\YMCCSmin_{a,b-s}^0) \, .
\end{equation}
To distinguish the webs $W_k$ living in the categories ${}_{a,b-s}\SSBim_{a,b-s}$ 
and those in ${}_{a,b}\SSBim_{a,b}$, we will use the notation $W_k^{(b-s)}$ and $W_k^{(b)}$, 
mimicking our notation for the $v$-variables.
First, note that $\I^{(s)}(\YMCCSmin_{a,b-s}^0)$ is a direct sum (with shifts) of terms of the form 
$\I^{(s)}(W_k^{(b-s)})\otimes \largewedge[\zeta^{(k)}_1,\ldots,\zeta^{(k)}_k]$.  
\cite[\IofWk]{HRW1} establishes an isomorphism
\[
\I^{(s)}(W_k^{(b-s)}) \cong W_k^{(b)}\otimes \largewedge^s[\zeta^{(k)}_{k+1},\ldots,\zeta^{(k)}_b]\, .
\]
Crucially for our present considerations, this isomorphism is $\Sym(\leftM^{(k)})$-linear.

To prove \eqref{eq:MCCS curved}, we only need to verify the commutativity of squares of the form
\[
\begin{tikzcd}[row sep=3em,column sep=5em]
\I^{(s)}(W_k^{(b-s)}\otimes \largewedge[\zeta^{(k)}_1,\ldots,\zeta^{(k)}_k]) 
\arrow[r,"\I^{(s)}(
	\Delta^v)"] 
&\I^{(s)}(W_k^{(b-s)}\otimes \largewedge[\zeta^{(k)}_1,\ldots,\zeta^{(k)}_k]) \\
W_k^{(b)}\otimes \largewedge[\zeta^{(k)}_1,\ldots,\zeta^{(k)}_k]\otimes 
	\largewedge^s[\zeta^{(k)}_{k+1},\ldots,\zeta^{(k)}_b] 
\arrow[u,"\cong"] \arrow[r,"
\Delta^v"] 
& W_k^{(b)}\otimes \largewedge[\zeta^{(k)}_1,\ldots,\zeta^{(k)}_k]\otimes 
	\largewedge^s[\zeta^{(k)}_{k+1},\ldots,\zeta^{(k)}_b] 
\arrow[u,"\cong"]
\end{tikzcd}
\]
since the commutativity of analogous squares for $\d^v$ and $\d^h$ has already
been established in the proof of \cite[\PropQs]{HRW1}.

Equation \eqref{eq:twists} shows that
the action of $\Delta^v$ on $W_k^{(b-s)}\otimes \largewedge[\zeta_1^{(k)},\ldots,\zeta_k^{(k)}]$ is given by
\[
\Delta^v\Big|_{W_k^{(b-s)}\otimes \wedge[\zeta_1^{(k)},\ldots,\zeta_k^{(k)}]}
= \sum_{j=1}^k  (-1)^{j-1}\zeta_j^{(k)}\sum_{i=j}^{b-s}  h_{i-j}(\M^{(k)}) \bar{v}_i^{(b-s)}\, .
\]
(Here, and in the following, we slightly abuse notation in the ordering of our tensor factors; 
since both $h_i(\leftM^{(k)})$ and all $v$-variables have even cohomological degree, 
this does not cause any hidden sign issues.)
Now, we apply $\I^{(s)}$ to this, obtaining
\begin{multline}\label{eq:temp}
\I^{(s)}\left(\Delta^v\Big|_{W_k^{(b-s)}\otimes \wedge[\zeta_1^{(k)},\ldots,\zeta_k^{(k)}]}\right) 
= \sum_{j=1}^k  (-1)^{j-1}\zeta_j^{(k)} \sum_{i=j}^{b-s}  h_{i-j}(\M^{(k)}) \I^{(s)}(\bar{v}_i^{(b-s)}) \\
= \sum_{j=1}^k  (-1)^{j-1}\zeta_j^{(k)} \sum_{i=j}^{b-s}  h_{i-j}(\M^{(k)})
	\left( \bar{v}_i^{(b)} + (-1)^{b-s-i}\!\!\! \sum_{m=b-s+1}^b\!\!\! \Schur_{(m-b+s-1|b-s-i)}(\leftL) \bar{v}_m^{(b)} \right).
\end{multline}
On the other hand, 
\begin{equation}\label{eq:tempt}
\Delta^v\Big|_{W_k^{(b)}\otimes \wedge[\zeta_1^{(k)},\ldots,\zeta_k^{(k)}]} 
= \sum_{j=1}^k (-1)^{j-1} \zeta_j^{(k)} \sum_{i=j}^{b}  h_{i-j}(\M^{(k)}) \bar{v}_i^{(b)}
\end{equation}
and it suffices to show that \eqref{eq:tempt} equals \eqref{eq:temp}.
By comparing coefficients of $\zeta_j^{(k)}$, this follows from
\[
\sum_{i=j}^b h_{i-j}(\leftM^{(k)}) \bar{v}_i^{(b)}
= 
\sum_{i=j}^{b-s}  h_{i-j}(\M^{(k)}) \bar{v}_i^{(b)} 
+ \sum_{i=j}^{b-s} (-1)^{b-s-i} h_{i-j}(\M^{(k)}) \sum_{m=b-s+1}^b \Schur_{(m-b+s-1|b-s-i)}(\leftL) \bar{v}_m^{(b)}
\]
for $1 \leq j \leq k$, or, after reordering terms:
\[
\sum_{m=b-s+1}^b h_{m-j}(\leftM^{(k)}) \bar{v}_m^{(b)}
= \sum_{m=b-s+1}^b \sum_{i=j}^{b-s} (-1)^{b-s-i} h_{i-j}(\M^{(k)}) \Schur_{(m-b+s-1|b-s-i)}(\leftL) \bar{v}_m^{(b)}
\]
The latter is an application of Lemma~\ref{lemma:h reduction} with $\X =
\leftM^{(k)}$, $\Y = \leftL - \leftM^{(k)}$, $r=m-b+s$, and $c=b-s-j$. 
\end{proof}

\begin{cor}[Curved skein relation]
\label{cor:curvedMCCS} 
We have a homotopy equivalence in $\VS_{a,b}$ of the form
\[
\tw_{(\d')^c+(\Delta')^c}\left( \bigoplus_{s=0}^b 
	\qdeg^{s(b-1)} \tdeg^s \left\llbracket
\begin{tikzpicture}[scale=.4,smallnodes,rotate=90,anchorbase]
   \draw[very thick] (1,-1) to [out=150,in=270] (0,0); 
   \draw[line width=5pt,color=white] (0,-2) to [out=90,in=270] (.5,0) to [out=90,in=270] (0,2);
   \draw[very thick] (0,-2) node[right=-2pt]{$a$} to [out=90,in=270] (.5,0) 
   	to [out=90,in=270] (0,2) node[left=-2pt]{$a$};
   \draw[very thick] (1,1) to (1,2) node[left=-2pt]{$b$};
   \draw[line width=5pt,color=white] (0,0) to [out=90,in=210] (1,1); 
   \draw[very thick] (0,0) to [out=90,in=210] (1,1); 
   \draw[very thick] (1,-2) node[right=-2pt]{$b$} to (1,-1); 
   \draw[very thick] (1,-1) to [out=30,in=330] node[above=-2pt]{$s$} (1,1); 
   \end{tikzpicture}
		\right\rrbracket_\VS \right) \simeq 
	\qdeg^{b(a-b-1)}\tdeg^b\cdot \YK\left(
		 \left\llbracket
\begin{tikzpicture}[scale=.4,smallnodes,anchorbase,rotate=270]
\draw[very thick] (1,-1) to [out=150,in=270] (0,1) to (0,2) node[right=-2pt]{$b$}; 
\draw[line width=5pt,color=white] (0,-2) to (0,-1) to [out=90,in=210] (1,1);
\draw[very thick] (0,-2) node[left=-2pt]{$b$} to (0,-1) to [out=90,in=210] (1,1);
\draw[very thick] (1,1) to (1,2) node[right=-2pt]{$a$};
\draw[very thick] (1,-2) node[left=-2pt]{$a$} to (1,-1); 
\draw[very thick] (1,-1) to [out=30,in=330] node[below=-1pt]{$a{-}b$} (1,1); 
\end{tikzpicture}
		\right\rrbracket\right)
\]
where the twist on the left-hand side strictly increases the parameter $s$ and
$(\Delta')^c$ is $\V$-irrelevant.  
\end{cor}
\begin{proof} The right-hand side is a shift of $\YK(\MCS)\simeq
\YKMCSmin_{a,b}$, as observed in \eqref{eq:VKMCS}.
Proposition~\ref{prop:reduced-curved-skein-rel} shows this is homotopy
equivalent to 
$\tw_{\d^c+\Delta^c} \big(\bigoplus_{s=0}^b \qdeg^{s(b-1)} \tdeg^s \YMCCSmin_{a,b}^s\big)$
where the twist $\d^c+\Delta^c$ raises $s$-degree by one.
In Proposition~\ref{prop:MCCS topologically} we have
seen that $\YMCCSmin_{a,b}^s\simeq  \I^{(s)}(\VFT_{a,b-s})$. 
Homological perturbation (i.e. Proposition \ref{prop:HPT}) now extends 
the direct sum of these equivalences to a homotopy equivalence:
\[
\tw_{\d^c+\Delta^c} \big(\bigoplus_{s=0}^b \qdeg^{s(b-1)} \tdeg^s \YMCCSmin_{a,b}^s \big) 
\simeq 
\tw_{(\d')^c+(\Delta')^c} \big(\bigoplus_{s=0}^b \qdeg^{s(b-1)} \tdeg^s \I^{(s)}(\VFT_{a,b-s}) \big)
\]
for some twist $(\d')^c+(\Delta')^c$ that necessarily increases $s$-degree by 
the formula for the induced twist in Proposition \ref{prop:HPT}.
\end{proof}

\section{The splitting map}
\label{sec:splitting}

Recall from Convention \ref{conv:FT} that $\FTmin_{a,b} \defeq \MCCSmin_{a,b}^0$, 
which is a simplified model for the Rickard complex 
$\FT_{a,b} \defeq \MCCS_{a,b}^0$ of the $(a,b)$-colored full twist braid.
Let $\VFTmin_{a,b} := \YMCCSmin_{a,b}^0$ be its curved lift from Definition \ref{def:curved MCCS}.
In this section, we study a canonical closed morphism $\Sigma_{a,b} \in
\Hom^{0}_{\VSred_{a,b}}(\VFTmin_{a,b}, \oone_{a,b})$ that is the colored analogue
of the ``link splitting'' map from \cite{GH}. We will use the notation from
\S\ref{ss:two strand V cats} throughout.

\subsection{Characterization of the splitting map}
\label{ss:splitting map}

To begin, we show that, up to homotopy and scalar multiple, 
there is a unique closed degree-zero map $\VFTmin_{a,b} \to \oone_{a,b}$. 
This morphism admits an abstract description as a perturbation of the following.

\begin{definition}\label{def:splittingmapzero} 
Let the \emph{undeformed splitting map} $\splitt_{a,b} \in \Hom_{\CS_{a,b}}(\FTmin_{a,b}, \oone_{a,b})$ 
be the closed, degree-zero morphism given as the composite of the projection 
$\FTmin_{a,b}\rightarrow \qdeg^{ab} W_b$
with the unzip morphism $\un \colon \qdeg^{a b} W_b\rightarrow \oone_{a,b}$.
\end{definition}

We now have the following.

\begin{prop}\label{defthm:splittingmap} Precomposition with $\splitt_{a,b}$
gives an $\End_{\CS_{a,b}}(\oone_{a,b})$-linear homotopy equivalence
\[
\End_{\CS_{a,b}}(\oone_{a,b}) \to \Hom_{\CS_{a,b}}(\FTmin_{a,b}, \oone_{a,b}).
\]
Moreover, there exists a closed, degree-zero lift $\Sigma_{a,b} \in
\Hom_{\VS_{a,b}}(\VFTmin_{a,b}, \oone_{a,b})$ of $\splitt_{a,b}$ such that
precomposition with $\Sigma_{a,b}$ gives an
$\End_{\VS_{a,b}}(\oone_{a,b})$-linear homotopy equivalence
\[
\End_{\VS_{a,b}}(\oone_{a,b}) \to \Hom_{\VS_{a,b}}(\VFTmin_{a,b}, \oone_{a,b}).
\]
\end{prop}

\begin{proof}
Note that $\FTmin_{a,b}$ equals the finite one-sided twisted complex 
$\tw_{\d^v} \big(\bigoplus_{l=0}^b R_l \big)$,
where $R_l := \big( \bigoplus_{k=l}^b P_{k,l,0}, \d^h \big)$ 
and $\delta^v$ decreases the $l$-degree by one.
It follows that 
\[
\Hom_{\CS_{a,b}}(\FTmin_{a,b}, \oone_{a,b}) = 
\tw_{(\d^v)^*} \left(\bigoplus_{l=0}^b \Hom_{\CS_{a,b}}(R_{l},\oone_{a,b}) \right)
\]
is likewise a one-sided twisted complex ($(\d^v)^*$ now increases $l$-degree by one).
Considering the inclusion
\begin{equation}\label{eq:inHomFT}
\iota \colon \Hom_{\CS_{a,b}}(R_b, \oone_{a,b}) \hookrightarrow \Hom_{\CS_{a,b}}(\FTmin_{a,b}, \oone_{a,b})\, ,
\end{equation} 
we see that
\[
\cone(\iota) \simeq \tw_{(\d^v)^*} \left(\bigoplus_{l=0}^{b-1} \Hom_{\CS_{a,b}}(R_{l},\oone_{a,b}) \right).
\]
Now, \cite[\PropRows]{HRW1} gives that
\[
R_{l} \simeq
\left\llbracket
\begin{tikzpicture}[scale=.5,smallnodes,anchorbase]
	\draw[very thick] (-2,0) node[left]{$a$} to (0,0) \pr (2,1) node[right]{$b$};
	\draw[line width=5pt,color=white] (0,1) \pr (2,0);
	\draw[very thick] (-2,1) node[left]{$b$} to (0,1) \pr (2,0) node[right]{$a$};
	\draw[very thick] (-1.75,1) to (-1.5,.5)node[right,yshift=-1pt,xshift=-1pt]{$l$} to  (-1.25,0);
	\draw[very thick] (-.75,0) to (-.5,.5) to  (-.25,1);
\end{tikzpicture}
\right\rrbracket
\]
so Corollary~\ref{cor:negstab} (proved below) implies that
$\Hom_{\CS_{a,b}}(R_{l},\oone_{a,b}) \simeq 0$ when $0 \leq l \leq b-1$.
Proposition \ref{prop:HPT} allows for these equivalences to be applied to $\cone(\iota)$
term-wise (see Remark \ref{rem:onesided}), thus $\cone(\iota) \simeq 0$.
As a consequence, \eqref{eq:inHomFT} is a homotopy equivalence; 
observe that it is $\End_{\CS_{a,b}}(\oone_{a,b})$-linear.
Now, 
\[
R_{b} 
=
\qdeg^{ab}
\begin{tikzpicture}[smallnodes,rotate=90,anchorbase,scale=.5]
	\draw[very thick] (.5,.375) to [out=150,in=270] (0,1) node[left,xshift=2pt]{$a$};
	\draw[very thick] (.5,.375) to [out=30,in=270] (1,1) node[left,xshift=2pt]{$b$};
	\draw[very thick] (.5,-.375) to (.5,.375);
	\draw[very thick] (1,-1) node[right,xshift=-2pt]{$b$} to  [out=90,in=330] (.5,-.375);
	\draw[very thick] (0,-1) node[right,xshift=-2pt]{$a$} to [out=90,in=210] (.5,-.375);
\end{tikzpicture}
\]
so Corollary \ref{cor:basicHom} implies that
\[
\End_{\CS_{a,b}}(\oone_{a,b}) \to \Hom_{\CS_{a,b}}(R_b, \oone_{a,b})
\, , \quad
\phi \mapsto \phi \circ \un
\]
is an $\End_{\CS_{a,b}}(\oone_{a,b})$-linear isomorphism.
The first result now follows from Definition \ref{def:splittingmapzero}.

Next, we consider the $\V$-deformation. Let
\[
\begin{tikzcd}[cramped]
\Hom_{\CS_{a,b}}(\FTmin_{a,b}, \oone_{a,b})[\V]
\arrow[r,"f",bend left]
\arrow[loop left, "k"]
&
\End_{\CS_{a,b}}(\oone_{a,b})[\V]
\arrow[l,"g",bend left]
\arrow[loop right, "\overline{k}"]
\end{tikzcd}
\]
be the data giving the homotopy equivalence 
just constructed, with scalars extended to $\k[\V]$.
Note that all of the indicated maps are therefore $\End_{\CS_{a,b}}(\oone_{a,b})[\V]$-linear, 
and we have that $g(1)=\splitt_{a,b}$.
Next, note that
$\Hom_{\VS_{a,b}}(\VFTmin_{a,b}, \oone_{a,b}) 
= \tw_{\a}\big( \Hom_{\CS_{a,b}}(\FTmin_{a,b}, \oone_{a,b})[\V] \big)$
for
\begin{equation}\label{eq:alphposVdeg}
\a \in \End_{E_{a,b}}\big( \Hom_{\CS_{a,b}}(\FTmin_{a,b}, \oone_{a,b}) \big) \otimes_\k \k[\V]_{>0}
\end{equation}
with $\wt(\alpha) = \qdeg^0\tdeg^1$.
Here, we use the shorthand 
\begin{equation}\label{eq:firstE}
E_{a,b} := \End_{\CS_{a,b}}(\oone_{a,b})[\V] = \End_{\VS_{a,b}}(\oone_{a,b})\, .
\end{equation}
Since $\FTmin_{a,b}$ is bounded, this implies that $k \circ \alpha$ is nilpotent.
Proposition \ref{prop:HPT} then implies that $\tilde{g} = (1+k\circ \a)^{-1} \circ g$ 
is a homotopy equivalence
from $\tw_{\alpha'}\big( \End_{\CS_{a,b}}(\oone_{a,b})[\V] \big)$ to
$\Hom_{\VS_{a,b}}(\VFTmin_{a,b}, \oone_{a,b})$
for some twist $\alpha'$, which must be zero since $\oone_{a,b}$ is 
supported in a single cohomological degree. Thus,
\[
\tilde{g} \colon \End_{\VS_{a,b}}(\oone_{a,b})
\to
\Hom_{\VS_{a,b}}(\VFTmin_{a,b}, \oone_{a,b})
\]
is a homotopy equivalence.
Set $\Sigma_{a,b} = \tilde{g}(1)$.
Equation \eqref{eq:alphposVdeg} implies that $\tilde{g}(1) = g(1) \mod \k[\V]_{>0}$, 
so $\Sigma_{a,b}$ is indeed a curved lift of $g(1) = \splitt_{a,b}$.
Finally, since $g,k,$ and $\alpha$ are all $E_{a,b}$-linear, 
the same is true for $\tilde{g}$, thus for $\phi \in  \End_{\VS_{a,b}}(\oone_{a,b})$ 
we compute that
\[
\tilde{g}(\phi) = \phi \circ \tilde{g}(1) = \phi \circ \Sigma_{a,b}
\]
as desired.
\end{proof}

\begin{rem}\label{rem:uniqueSigma} Note that the morphism $\Sigma_{a,b}$
``constructed'' in the proof of Proposition \ref{defthm:splittingmap} is not
explicitly specified, since we have not specified the constituent morphisms
(this will be done in Definition~\ref{def:splitting map model} below). However,
any two closed, degree-zero lifts of $\splitt_{a,b}$ are necessarily homotopic,
so $\Sigma_{a,b}$ is uniquely determined up to homotopy. We call this map the
(deformed) \emph{splitting map}. Indeed, the difference between any two such
curved lifts of $\splitt_{a,b}$ is a closed, degree-zero element in the
$\V$-irrelevant submodule $\Hom_{\CS_{a,b}}(\VFTmin_{a,b}, \oone_{a,b})\otimes
\k[\V]_{>0} \subset \Hom_{\VS_{a,b}}(\VFTmin_{a,b}, \oone_{a,b})$. However,
$\Hom_{\VS_{a,b}}(\VFTmin_{a,b}, \oone_{a,b}) \simeq
\End_{\VS_{a,b}}(\oone_{a,b})$ and the only such element in the latter is zero
(which is null-homotopic).
\end{rem}

\begin{example}
For $a\geq b=1$, and using the notation introduced in \S\ref{ss:two strand
V cats}, the curvature element in $\VSred_{a,1}$ is $(x_{a+1}-x_{a+1}')\bar{v}_1$.
A curved lift $\Sigma_{a,1} \colon \VFTmin_{a,1} \to \oone_{a,1}$ of the splitting map
is then given by:
\begin{equation}\label{eq:a1LinkSplit}
\begin{tikzcd}[column sep = large]
\qdeg^a \;
\begin{tikzpicture}[smallnodes,rotate=90,anchorbase,scale=.5]
	\draw[very thick] (0,.375) to [out=150,in=270] (-.5,1);
	\draw[very thick] (0,.375) to [out=30,in=270] (.5,1);
	\draw[very thick] (0,-.375) to (0,.375);
	\draw[very thick] (.5,-1) node[right,xshift=-2pt]{$1$} to [out=90,in=330] (0,-.375);
	\draw[very thick] (-.5,-1) node[right,xshift=-2pt]{$a$} to [out=90,in=210] (0,-.375);
\end{tikzpicture}
\arrow[r, shift left, "x_{a+1} - x_{a+1}' "] \arrow[d, "\chi^+_0"]
&
\qdeg^{a-2} \tdeg \;
\begin{tikzpicture}[smallnodes,rotate=90,anchorbase,scale=.5]
	\draw[very thick] (0,.375) to [out=150,in=270] (-.5,1);
	\draw[very thick] (0,.375) to [out=30,in=270] (.5,1);
	\draw[very thick] (0,-.375) to (0,.375);
	\draw[very thick] (.5,-1) node[right,xshift=-2pt]{$1$} to [out=90,in=330] (0,-.375);
	\draw[very thick] (-.5,-1) node[right,xshift=-2pt]{$a$} to [out=90,in=210] (0,-.375);
\end{tikzpicture}
\arrow[r,"\chi^+_0"] \arrow[l, shift left, "\bar{v}_1"]
&
\qdeg^{-2} \tdeg^2 \;
\begin{tikzpicture}[scale=.5,smallnodes,rotate=90,anchorbase]
		\draw[very thick] (0,-2) node[right,xshift=-2pt]{$a$} to (0,0);
		\draw[very thick] (1,-2) node[right,xshift=-2pt]{$1$} to (1,0);
\end{tikzpicture}
\arrow[dll, "-\bar{v}_1"]
\\
\begin{tikzpicture}[scale=.5,smallnodes,rotate=90,anchorbase]
		\draw[very thick] (0,-2) node[right,xshift=-2pt]{$a$} to (0,0);
		\draw[very thick] (1,-2) node[right,xshift=-2pt]{$1$} to (1,0);
\end{tikzpicture} & &
\end{tikzcd}.
\end{equation}
Note that $\Sigma_{a,1} = \chi^+_0 \mod \bar{v}_1$.
\end{example}

\begin{rem}\label{rem:genSigma} Since $\VFTmin_{a,b} \simeq \VFT_{a,b}$, we also
obtain a degree-zero closed splitting map $\Sigma_{a,b} \colon \VFT_{a,b} \to
\oone_{a,b}$, that is unique up to homotopy. Using this, we can construct a
\emph{deformed splitting map} for the colored positive full twist braid on any
number of strands. Indeed, let $\FT_{\brc}$ denote the Rickard complex
associated to the full twist braid with strands colored $\bre_1,\ldots,\bre_m$.
The full twist braid can be written as a composition of (pure) braids of the
form:
\[
A_{i,j} = \agen_{j-1} \cdots \agen_{i+1} \agen_i^2 \agen_{i+1}^{-1} \cdots \agen_{j-1}^{-1} =
\begin{tikzpicture}[scale=.35,smallnodes,anchorbase]
	\draw[very thick] (0,0) to node[right=-3pt]{$\mydots$} (0,4);
	\draw[very thick] (1,0) to (1,4);	
	\draw[very thick] (4,0) to node[right=-3pt]{$\mydots$} (4,4);
	\draw[very thick] (5,0) to (5,4);	
	\draw[very thick] (7,0) to node[right=-3pt]{$\mydots$} (7,4);
	\draw[very thick] (8,0) to (8,4);	
	\draw[line width=5pt,color=white] (6,0) to [out=90,in=270] (2,2);
	\draw[very thick] (6,0) to [out=90,in=270] (2,2);
	\draw[line width=5pt,color=white] (2,0) to [out=90,in=270] (3,2);
	\draw[very thick] (2,0) to [out=90,in=270] (3,2) to [out=90,in=270] (2,4);
	\draw[line width=5pt,color=white] (2,2) to [out=90,in=270] (6,4);
	\draw[very thick] (2,2) to [out=90,in=270] (6,4);
\end{tikzpicture}
\]
(here, in the interest of space, we break our conventions and write the braid vertically). 
Thus, the $2$-strand splitting maps $\Sigma_{a,b} \colon\VFT_{a,b} \to \oone_{a,b}$
assemble to define the desired splitting map $\Sigma_{\brc} \colon \VFT_{\brc}
\to \oone_{\brc}$. Here, $\VFT_{\brc}$ is the unique curved lift of $\FT_{\brc}$
that is guaranteed to exist by Lemma \ref{lem:curvinginvertibles}. This
more-general splitting map plays a role in \S\ref{s:colored Hilb} below,
and is further generalized in \S\ref{sec:linksplit}.
\end{rem}

\subsection{Explicit description of the full twist splitting map}
\label{ss:explicit splitting map}

In this section, we explicitly
describe the splitting map $\Sigma_{a,b} \colon \VFTmin_{a,b}\rightarrow \oone_{a,b}$
as a morphism in $\VSred_{a,b}$.  We will adopt the convention for alphabets
labeling the web $W_k$ as in \eqref{eq:Walphabets} and the convention for
deformation parameters in $\VSred_{a,b}$ as in \S\ref{ss:two strand V cats}.
In particular, we consider deformation parameters 
$\overline{\V}^{(k)} = \{\bar{v}^{(k)}_1,\ldots,\bar{v}^{(k)}_k\}$
for each $k\geq 0$, as in 
the discussion following Remark \ref{rem:Vprime}. 
By convention we abbreviate
$\bar{v}_i:=\bar{v}_i^{(b)}$ and $\overline{\V}:=\overline{\V}^{(b)}$.
The alphabets $\overline{\V}^{(k)}$ and $\overline{\V}$ are therefore related by 
the substitution rule \eqref{eq:vbar l and vbar b}, with $\leftM^{(k)}=\X_{[a+1,a+k]}$,
and we identify each
$\k[\X_{[1,a+k]},\X_{[1,a+k]}',\overline{\V}^{(k)}]$ as a subalgebra of
$\k[\X,\rightX,\overline{\V}]$ accordingly.
For $a+1\leq i\leq a+k$, we have the elements
\[
\yred_i = \sum_{l=1}^k x_i^{l-1} \bar{v}_l^{(k)} 
\in \k[\X_{[1,a+k]},\X_{[1,a+k]}',\overline{\V}^{(k)}] \subset \k[\X,\rightX,\overline{\V}]
\]
which, as elements of $\k[\X,\rightX,\overline{\V}]$,
do not depend on $k$ so long as $1 \leq k \leq b$.

\begin{remark}\label{rmk:y acts on W} The algebra
$\k[\X,\rightX,\overline{\V}]^{\symg_{a}\times \symg_{k}\times
\symg_{b-k}}$ acts on the web $W_k$, thought of as an object of
$\CS_{a,b}\otimes\k[\overline{\V}]$.  In particular, $\bar{v}_i^{(k)}$ and
appropriately partially symmetric expressions in the $\yred_i$ may be regarded
as an endomorphism of $W_k$.
\end{remark}

We begin with a reformulation of Lemma \ref{lemma:v from y}.

\begin{lem}\label{lemma:v from y2}
For $1 \leq r \leq b$, we have
\begin{equation}\label{eq:ytov}
\bar{v}_r^{(k)}
= (-1)^{r-1} \partial_{a+1} \cdots \partial_{a+k-1}\big(e_{k-r}(\X_{[a+1,a+k-1]}) \cdot \yred_{a+k} \big)
\end{equation}
\end{lem}
\begin{proof}
Example \ref{exa:Sylvester} shows that 
$\big \{e_{k-i}(\X_{[a+1,a+k-1]}) \big\}_{i=1}^{k}$ and 
$\big\{ (-1)^{i-1} x_{a+k}^{i-1} \big\}_{i=1}^{k}$
are dual bases with respect to the Sylvester operator
\[
\partial_{a+1} \cdots \partial_{a+k-1} \colon 
\Sym(\X_{[a+1,a+k-1]} | \{x_{a+k}\}) \to \Sym(\X_{[a+1,a+k]})\, .
\]
We thus compute
\begin{align*}
\partial_{a+1} \cdots \partial_{a+k-1}\big(e_{k-r}(\X_{[a+1,a+k-1]}) \cdot \yred_{a+k} \big)
&= \sum_{l=1}^k \partial_{a+1} \cdots \partial_{a+k-1}
	\big(e_{k-r}(\X_{[a+1,a+k-1]}) \cdot x_{a+k}^{l-1} \bar{v}_l^{(k)} \big) \\
&= \sum_{l=1}^k \partial_{a+1} \cdots \partial_{a+k-1}
	\big(e_{k-r}(\X_{[a+1,a+k-1]}) \cdot x_{a+k}^{l-1} \big) \bar{v}_l^{(k)} \\
&= (-1)^{r-1} \bar{v}_r^{(k)}\, . \qedhere
\end{align*}
\end{proof}

We now give an explicit model for the deformed splitting map 
from Proposition \ref{defthm:splittingmap}.

\begin{definition}\label{def:splitting map model}
Let $\Sigma_{a,b} \in \Hom_{\VS_{a,b}}(\VFTmin_{a,b} , \oone_{a,b})$ have non-zero components 
as indicated by the following diagram:
\begin{equation}\label{eq:splitting map diagram}
\begin{tikzpicture}[baseline=8em]
\node (zz) at (0, 6) {$\oone_{a,b}$};
\node (aa) at (0,0) {$P_{0,0,0}$};
\node (ba) at (-2,0) {$P_{1,0,0}$};
\node (ca) at (-4,0) {$P_{2,0,0}$};
\node (da) at (-6,0) {$P_{3,0,0}$};
\node (ea) at (-8,0) {$\cdots$};
\node (bb) at (-2,2) {$P_{1,1,0}$};
\node (cb) at (-4,2) {$P_{2,1,0}$};
\node (db) at (-6,2) {$P_{3,1,0}$};
\node (eb) at (-8,2) {$\cdots$};
\node (cc) at (-4,4) {$P_{2,2,0}$};
\node (dc) at (-6,4) {$P_{3,2,0}$};
\node (ec) at (-8,4) {$\cdots$};
\node (dd) at (-6,6) {$P_{3,3,0}$};
\node (ed) at (-8,6) {$\cdots$};
\node at (-7,7) {$\ddots$};
\path[->,>=stealth,shorten >=1pt,auto,node distance=1.8cm]
%
%
(ea) edge node {$\d^h$} (da)
(da) edge node {$\d^h$} (ca)
(ca) edge node {$\d^h$} (ba)
(ba) edge node {$\d^h$} (aa)
(eb) edge node {$\d^h$} (db)
(db) edge node {$\d^h$} (cb)
(cb) edge node {$\d^h$} (bb)
(ec) edge node {$\d^h$} (dc)
(dc) edge node {$\d^h$} (cc)
(ed) edge node {$\d^h$} (dd)
%
%
([xshift=-2pt] dd.south) edge node[left]  {$\d^v$} ([xshift=-2pt] dc.north)
([xshift=-2pt] dc.south) edge node[left]  {$\d^v$} ([xshift=-2pt] db.north)
([xshift=-2pt] db.south) edge node[left]  {$\d^v$} ([xshift=-2pt] da.north)
([xshift=-2pt] cc.south) edge node[left]  {$\d^v$} ([xshift=-2pt] cb.north)
([xshift=-2pt] cb.south) edge node[left]  {$\d^v$} ([xshift=-2pt] ca.north)
([xshift=-2pt] bb.south) edge node[left] {$\d^v$} ([xshift=-2pt] ba.north)
%
%
([xshift=2pt] dc.north)  edge node[right] {$\Delta^v$} ([xshift=2pt] dd.south)
([xshift=2pt] db.north) edge node[right] {$\Delta^v$} ([xshift=2pt] dc.south)
([xshift=2pt] da.north) edge node[right] {$\Delta^v$} ([xshift=2pt] db.south)
([xshift=2pt] cb.north) edge node[right] {$\Delta^v$} ([xshift=2pt] cc.south)
([xshift=2pt] ca.north) edge node[right] {$\Delta^v$} ([xshift=2pt] cb.south)
([xshift=2pt] ba.north) edge node[right] {$\Delta^v$} ([xshift=2pt] bb.south)
%
%
(aa) edge node {$\Sigma_{a,b}^0$} (zz)
(bb) edge node {$\Sigma_{a,b}^1$} (zz)
(cc) edge node {$\Sigma_{a,b}^2$} (zz)
(dd) edge node {$\Sigma_{a,b}^3$} (zz);
\end{tikzpicture}
\end{equation}
The map $\Sigma_{a,b}^k$ is given as  
$(-1)^{\genfrac(){0pt}{2}{b}{2}}$ times the composition
\[
P_{k,k,0} \xrightarrow{\approx} W_k 
\xrightarrow{
\yred_{a+k+1}\cdots \yred_{a+b}}
W_k\xrightarrow{\thickchi^+} \oone_{a,b}.
\]
Here, the first map (denoted $\approx$) is a ``slanted identity'' 
$P_{k,k,0} = \qdeg^{k(a-b+k)-2(b-k)} \tdeg^{2(b-k)} W_k \xrightarrow{\id} W_k$ 
(which therefore has weight $\qdeg^{-k(a-b+k)+2(b-k)}\tdeg^{-2(b-k)}$).
The second is multiplication by $\yred_{a+k+1}\cdots \yred_{a+b}$,
which is a well-defined element in $\Sym(\B)\otimes \k[\overline{\V}]$.
The final morphism is 
\[
\thickchi^+ := \chi^+_{0}\cdots \chi^+_{k-1},
\]
which has weight $\qdeg^{k(a-b+k)}\tdeg^0$.
\end{definition}

\begin{remark}\label{rem:thickcap}
The morphism $\thickchi^+ \colon W_k \rightarrow \oone_{a,b}$
is given in extended graphical calculus as
\[
\CQGsgn{(-1)^{k(b-k)+\frac{1}{2}k(k-1)}}
\begin{tikzpicture}[baseline=0,smallnodes]
	\draw[CQG,ultra thick,<-] (0,-.5) node[below]{\scriptsize$k$}  to (0,.625);
	\draw[CQG,ultra thick] (.75,-.5) node[below]{\scriptsize$k$} to (.75,.625);
	\draw[CQG,thick, directed=.75] (.75,0) to [out=90,in=90] 
		node[black,yshift=-.5pt]{$\bullet$} node[below,black]{\scriptsize$k{-}1$} (0,0);
	\draw[CQG,thick, directed=.75] (.75,.55) to [out=90,in=90] 
		node[black,yshift=-.5pt]{$\bullet$} node[below,black]{\scriptsize$1$} (0,.55);
	\draw[CQG, thick,directed=.5] (.76,.625) to [out=90,in=0] (.375,1.125) to [out=180,in=90] (-.01,.625);
	\node at (.5,.375) {$\vdots$};
\end{tikzpicture}
=
(-1)^{\genfrac(){0pt}{2}{k}{2}} \cdot
\CQGsgn{(-1)^{k(b-k)}}
\begin{tikzpicture}[anchorbase, scale=.5,tinynodes, yscale=-1]
	\draw[ultra thick,CQG,->] (1,1.5) node[below=-2pt]{$k$} to [out=270,in=0] (.5,.5) to [out=180,in=270] (0,1.5);
\end{tikzpicture}
\]
In the perpendicular graphical calculus from \S\ref{ss:ssbim}, we have
\[
\begin{tikzpicture}[baseline=0em]
	\draw[FS,ultra thick,->] (0,-.25) node[below=-2pt]{\scriptsize$a{+}k$} to (0,.75) 
		node[above=-2pt]{\scriptsize $a$};
	\draw[FS,ultra thick,->] (.75,-.25) node[below=-2pt]{\scriptsize$b{-}k$} to (.75,.75) 
		node[above=-2pt]{\scriptsize $b$};
	\draw[FS,ultra thick, directed=.75] (0,0) to [out=90,in=270]  (.75,.5) ;
	\node[FS] at (0.25,0) {\scriptsize$k$};
\end{tikzpicture}
=
\CQGsgn{(-1)^{k(b-k)}}
\begin{tikzpicture}[anchorbase, scale=.5,tinynodes, yscale=-1]
	\draw[ultra thick,CQG,->] (1,1.5) node[below=-2pt]{$k$} to [out=270,in=0] (.5,.5) to [out=180,in=270] (0,1.5);
\end{tikzpicture} 
\]
thus the components of $\Sigma_{a,b}$ from Definition \ref{def:splitting map model}
are succinctly described as
\begin{equation}\label{eq:SigmaPGC}
\Sigma_{a,b}^k = 
(-1)^{\genfrac(){0pt}{2}{b}{2}+\genfrac(){0pt}{2}{k}{2}}
\begin{tikzpicture}[anchorbase]
	\draw[FS,ultra thick,->] (0,-.4) node[below=-2pt]{\scriptsize$a{+}k$}  to (0,.7)node[above=-2pt]{\scriptsize $a$};
	\draw[FS,ultra thick,->] (.75,-.4) node[below=-2pt]{\scriptsize$b{-}k$} to (.75,.7) node[above=-2pt]{\scriptsize $b$};
	\draw[FS,ultra thick, directed=.75] (0,0) to  [out=90,in=270]  (.75,.5) ;
	\node[FS] at (0.25,0) {\scriptsize$k$};
	\node at (.75,0) {$\bullet$};
	\node at (1.85,0) {\tiny $\yred_{a+k+1}\cdots \yred_{a+b}$};
\end{tikzpicture} .
\end{equation}
\end{remark}

\begin{thm}\label{thm:splittingmap}
The map $\Sigma_{a,b}$ from Definition \ref{def:splitting map model} is closed, has degree zero, 
and is a curved lift of $\splitt_{a,b}$.
As such, it is an explicit model for the deformed splitting map 
$\Sigma_{a,b}\colon \VFTmin_{a,b}\to \oone_{a,b}$ from Proposition \ref{defthm:splittingmap}.
\end{thm}

\begin{proof}
Since $\wt(\yred_{a+k+1}\cdots \yred_{a+b}) = \qdeg^{-2(b-k)}\tdeg^{2(b-k)}$, $\Sigma_{a,b}$
has degree zero. The component $\Sigma_{a,b}^b\colon P_{b,b,0} = \qdeg^{ab} W_b \to
\oone_{a,b}$ is $(-1)^{\genfrac(){0pt}{2}{b}{2}} \thickchi^+ = \un$, thus agrees with $\splitt_{a,b}$.
Since the sum of the remaining components is an element of
$\Hom_{\CS_{a,b}}(\FTmin_{a,b},\oone_{a,b}) \otimes \k[\V]_{>0}$, it remains to
check that the prescribed map is a chain map, i.e. that 
$\Sigma_{a,b}^k \circ \Delta^v = - \Sigma_{a,b}^{k-1} \circ \d^h$ as maps $P_{k,k-1,0} \to \oone_{a,b}$.

This verification is most-easily given in terms of the variables $\{\xi_i\}_{i=1}^b$, 
rather than the variables $\{\zeta_i^{(k)}\}_{i=1}^b$ that are adapted to the ``columns'' $\YK(W_k)$. 
Recall from \eqref{eq:xi and zeta} that the change of variables between these sets of variables 
is lower-triangular, so we have identifications
\[
P_{k,k,0} \approx W_k \otimes \xi_1\cdots \xi_k \quad \text{and} \quad
P_{k,k-1,0} \approx W_k \otimes \k \{ \xi_1\cdots \hat{\xi_i}\cdots  \xi_k \}_{i=1}^k\, .
\]

To begin, we compute
\[
\begin{aligned}
\Delta^v |_{P_{k,\bullet,0}}
&= \sum_{j=1}^k (-1)^{j-1} \zeta_j^{(k)} \sum_{l=j}^{b}  h_{l-j}(\M^{(k)}) \bar{v}_l \\
& \stackrel{\eqref{eq:xi and zeta}}{=}
\sum_{j=1}^k (-1)^{j-1} \sum_{i=1}^j (-1)^{i-1} e_{j-i} \big(\leftM^{(k)} \big) \xi_i \sum_{l=j}^{b}  h_{l-j}(\M^{(k)}) \bar{v}_l \\
&= \sum_{i=1}^k \left( \sum_{l=1}^b \left( \sum_{j=1}^k (-1)^{j-i} e_{j-i}(\leftM^{(k)}) h_{l-j}(\leftM^{(k)}) \right) \bar{v}_l \right) \xi_i\, .
\end{aligned}
\]
Next, we have that
\[
\sum_{j=1}^k (-1)^{j-i} e_{j-i}(\leftM^{(k)}) h_{l-j}(\leftM^{(k)}) =
\begin{cases}
1 & \text{if } l=i (\leq k) \\
(-1)^{k-i} \Schur_{(l-k-1 | k-i)}(\leftM^{(k)}) & \text{else}
\end{cases}
\]
using Lemma \ref{lemma:hook rewrite}. In the latter case, note that this is zero
unless $l > k$. Thus, \eqref{eq:vbar l and vbar b} shows that in 
the variables $\bar{v}_i^{(k)}$ the map $\Delta^v$ takes the simplified form
\[
\Delta^v |_{P_{k,\bullet,0}} = \sum_{i=1}^k \bar{v}_i^{(k)} \xi_i\, .
\]

Next, we compute $\d^h |_{P_{k,k-1,0}}$ in the variables $\{\xi_i\}_{i=1}^k$. 
This is simply given by the composition
\[
P_{k,k-1,0} \hookrightarrow \YK(W_k) \xrightarrow{\d^H} \YK(W_{k-1}) 
\twoheadrightarrow P_{k-1,k-1,0}
\]
where $\d^H$ is induced by $\chi_0$. In particular, $\d^H$ is determined in
Koszul degree $1$ by sending $W_k\otimes \xi_i \xrightarrow{\chi_0}
W_{k-1}\otimes \xi_i$. Lower-triangularity of the change of basis from
$\{\xi_i\}_{i=1}^{k-1}$ to $\{\zeta^{(k-1)}\}_{i=1}^{k-1}$ implies that $\d^h$
still acts by $\chi_0$ on terms indexed by (products of)
$\xi_1,\ldots,\xi_{k-1}$. However, for $\xi_k$, we compute
\[
\xi_k 
=  \sum_{j=1}^{k-1} (-1)^{j-1}h_{k-j}(\leftM^{(k-1)})\zeta_j^{(k-1)} + (-1)^{k-1} \zeta_k^{(k-1)}\, .
\]
The value of $\d^h(\xi_k)$ is obtained by projecting onto 
the span of $\{\zeta_j^{(k-1)}\}_{j=1}^{k-1}$, thus
\[
\begin{aligned}
\d^h |_{P_{k,\bullet,0}} \colon \xi_k 
&\xmapsto{\chi_0}
 \sum_{j=1}^{k-1} (-1)^{j-1}h_{k-j}(\leftM^{(k-1)})\zeta_j^{(k-1)} \\
&=
 \sum_{j=1}^{k-1} (-1)^{j-1}h_{k-j}(\leftM^{(k-1)}) \sum_{i=1}^j (-1)^{i-1} e_{j-i}(\leftM^{(k-1)}) \xi_i \\
&=  
 \sum_{i=1}^{k-1} \left( \sum_{j=i}^{k-1} (-1)^{j-i} h_{k-j}(\leftM^{(k-1)}) e_{j-i}(\leftM^{(k-1)}) \right) \xi_i \\
&= 
 \sum_{i=1}^{k-1} (-1)^{k-i+1} e_{k-i}(\leftM^{(k-1)}) \xi_i\, .
\end{aligned}
\]

Recalling that $P_{k,l,0} = \qdeg^{k(a-b+1)-2b} \tdeg^{2b-k}  \ W_k\otimes \largewedge^l[\zeta^{(k)}_1,\ldots, \zeta^{(k)}_k]$,
it therefore remains to verify that the diagram
\[
\begin{tikzcd}[column sep = 4cm]
\tdeg^{2b-k} W_k \otimes \xi_1\cdots \xi_i\cdots  \xi_k
\arrow[r,"\chi^+_{0}\cdots \chi^+_{k-1} \cdot \yred_{a+k+1}\cdots \yred_{a+b}"] 
&
\oone_{a,b}
\\
\tdeg^{2b-k} W_k \otimes \xi_1\cdots \hat{\xi_i}\cdots  \xi_k
\arrow[u,"(-1)^{i-1-k} \bar{v}_i^{(k)}"]
\arrow[r,"e_{k-i}(\leftM^{(k-1)}) \cdot \chi_0^+"]
&
\tdeg^{2b-k+1} W_{k-1} \otimes \xi_1\cdots \xi_i\cdots  \xi_{k-1}
\arrow[u,swap,"\chi^+_{0}\cdots \chi^+_{k-2} \cdot \yred_{a+k}\cdots \yred_{a+b}"]
\end{tikzcd},
\]
anti-commutes for $1 \leq i \leq k$ (here we have omitted the $\qdeg$-degree shifts
on the bimodules, and the overall factor of $(-1)^{\genfrac(){0pt}{2}{b}{2}}$ on
the components of $\Sigma_{a,b}$). To do so, we use the perpendicular graphical
calculus. Recall from \eqref{eq:perpchi} and Remark \ref{rem:thickcap} that
\[
\chi_m^+ \defeq
\begin{tikzpicture}[baseline=0em]
	\draw[FS,ultra thick,->] (0,-.5) node[below=-2pt]{\scriptsize$a{+}k$}  to (0,1);
	\draw[FS,ultra thick,->] (.75,-.5) node[below=-2pt]{\scriptsize$b{-}k$} to (.75,1);
	\draw[FS,thick, directed=.75] (0,0) to  [out=90,in=270]  (.75,.5) ;
	\node[FS] at (0.25,0) {\scriptsize$1$};
	\node at (.375,.25) {\scriptsize$\bullet$};
	\node at (.375,.45) {\scriptsize$m$};
\end{tikzpicture} \colon W_{k} \to W_{k-1}
\, , \quad
\thickchi^+ \defeq \chi^+_{0}\cdots \chi^+_{k-1} =
(-1)^{\genfrac(){0pt}{2}{k}{2}}
\begin{tikzpicture}[baseline=0em]
	\draw[FS,ultra thick,->] (0,-.5) node[below=-2pt]{\scriptsize$a{+}k$} to (0,1) 
		node[above=-2pt]{\scriptsize $a$};
	\draw[FS,ultra thick,->] (.75,-.5) node[below=-2pt]{\scriptsize$b{-}k$} to (.75,1) 
		node[above=-2pt]{\scriptsize $b$};
	\draw[FS,ultra thick, directed=.75] (0,0) to [out=90,in=270]  (.75,.5) ;
	\node[FS] at (0.25,0) {\scriptsize$k$};
\end{tikzpicture} \colon W_k\to W_0 \, .
\]
We then compute
\begin{multline*}
\chi^+_{0}\cdots \chi^+_{k-2} \cdot \yred_{a+k}\cdots \yred_{a+b} \cdot  e_{k-i}(\M^{(k-1)}) \cdot \chi_0 
=
(-1)^{\genfrac(){0pt}{2}{k-1}{2}}
\begin{tikzpicture}[anchorbase, xscale=1.25, smallnodes]
	\draw[FS,ultra thick,->] (0,-1.25) node[below=-2pt]{\scriptsize$a{+}k$}  to (0,.7)node[above=-2pt]{\scriptsize $a$};
	\draw[FS,ultra thick,->] (.75,-1.25) node[below=-2pt]{\scriptsize$b{-}k$} to (.75,.7) node[above=-2pt]{\scriptsize $b$};
	\draw[FS,ultra thick, directed=.75] (0,0) to [out=90,in=270] node[above=-1pt]{$k{-}1$} (.75,.5);
	\node at (.75,0) {$\bullet$};
	\node at (1.375,0) {\tiny $\prod_{j=k}^b \yred_{a+j}$}; 
	\node at (.375,-.375) {\tiny $\CQGbox{e_{k-i}}$};
	\draw[FS,thick, directed=.75] (0,-1) to  [out=90,in=270]  (.75,-.5) ;
	\node[FS] at (0.25,-1) {\scriptsize$1$};
\end{tikzpicture} \\
=
(-1)^{\genfrac(){0pt}{2}{k-1}{2}}
\begin{tikzpicture}[anchorbase,xscale=1.25, smallnodes]
	\draw[FS,ultra thick,->] (0,-1.25) node[below=-2pt]{\scriptsize$a{+}k$}  to (0,.7)node[above=-2pt]{\scriptsize $a$};
	\draw[FS,ultra thick,->] (.75,-1.25) node[below=-2pt]{\scriptsize$b{-}k$} to (.75,.7) node[above=-2pt]{\scriptsize $b$};
	\draw[FS,ultra thick, directed=.75] (0,-.25) to  [out=90,in=270]  (.75,.25) ;
	\node[FS] at (0.375,-.25) {\scriptsize$k{-}1$};
	\node at (.75,-1) {$\bullet$};
	\node at (1.5,-1) {\tiny $\prod_{j=k+1}^b \yred_{a+j}$}; 
	\node at (.375,.25) {\tiny $e_{k-i}$};
	\node at (.375,0) {$\bullet$};
	\draw[FS,thick, directed=.75] (0,-1) to  [out=90,in=270]  (.75,-.5) ;
	\node at (.375,-.5) {\tiny $\yred_{a+k}$};
	\node at (.375,-.75) {$\bullet$};
	\node[FS] at (0.25,-1) {\scriptsize$1$};
\end{tikzpicture}
=
(-1)^{\genfrac(){0pt}{2}{k-1}{2}}
\begin{tikzpicture}[anchorbase,smallnodes, xscale=1.25]
	\draw[FS,ultra thick,->] (-.45,-1.5) node[below=-2pt]{\scriptsize$a{+}k$}  to (-.45,.5)node[above=-2pt]{\scriptsize $a$};
	\draw[FS,ultra thick,->] (1.2,-1.5) node[below=-2pt]{\scriptsize$b{-}k$} to (1.2,.5) node[above=-2pt]{\scriptsize $b$};
	\draw[FS,ultra thick, directed=.8, directed=.2] (-.45,-1) \pu node[above=-2pt]{$k$} (.375,-.75) 
		\pu (.125,-.375) \pu (.375,0) \pu node[above=-2pt]{$k$} (1.2,.25);
	\draw[FS,thick] (.375,-.75) \pu (.625,-.375) \pu (.375,0);
	\node at (.875,-.5) {\tiny $\yred_{a+k}$};
	\node at (.625,-.375) {$\bullet$};
	\node at (1.2,-1.25) {$\bullet$};
	\node at (.5,-1.25) {\tiny $\prod_{j=k+1}^b \yred_{a+j}$}; 
	\node at (-.1,-.1) {\tiny $e_{k-i}$};
	\node at (.125,-.375) {\tiny $\bullet$};
\end{tikzpicture}.
\end{multline*}
Now, we may simplify this latter diagram since the middle portion is locally the composition
\[
\begin{tikzpicture}[scale=.35,smallnodes,rotate=90,anchorbase]
	\draw[very thick] (1,-1) to (1,1) node[left=-2pt]{$\leftM^{(k)}$} ;
\end{tikzpicture}
\xrightarrow{ \ \cre \ }
\begin{tikzpicture}[scale=.35,tinynodes,rotate=90,anchorbase]
	\draw[very thick] (1,1) to (1,2);
	\draw[very thick] (1,-2) to (1,-1); 
	\draw[very thick] (1,-1) to [out=150,in=210] node[below=0pt]{$\leftM^{(k-1)}$} (1,1);
	\draw[very thick] (1,-1) to [out=30,in=330] (1,1); 
\end{tikzpicture}
\xrightarrow{y_{a+k} \cdot e_{k-i}(\leftM^{(k-1)})}
\begin{tikzpicture}[scale=.35,tinynodes,rotate=90,anchorbase]
	\draw[very thick] (1,1) to (1,2);
	\draw[very thick] (1,-2) to (1,-1); 
	\draw[very thick] (1,-1) to [out=150,in=210] node[below=0pt]{$\leftM^{(k-1)}$} (1,1);
	\draw[very thick] (1,-1) to [out=30,in=330] (1,1); 
\end{tikzpicture}
\xrightarrow{ \ \col \ }
\begin{tikzpicture}[scale=.35,smallnodes,rotate=90,anchorbase]
	\draw[very thick] (1,-1) node[right=-2pt]{$\leftM^{(k)}$} to (1,1);
\end{tikzpicture}
\]
i.e. it is multiplication by the element
\[
\partial_{a+1} \cdots \partial_{a+k-1} \big( \yred_{a+k} \cdot e_{k-i}(\leftM^{(k-1)}) \big)
\stackrel{\eqref{eq:ytov}}{=} (-1)^{i-1} \bar{v}_i^{(k)} .
\]
This implies that
\[
(-1)^{\genfrac(){0pt}{2}{k-1}{2}}
\begin{tikzpicture}[anchorbase,smallnodes, xscale=1.25]
	\draw[FS,ultra thick,->] (-.45,-1.5) node[below=-2pt]{\scriptsize$a{+}k$}  to (-.45,.5)node[above=-2pt]{\scriptsize $a$};
	\draw[FS,ultra thick,->] (1.2,-1.5) node[below=-2pt]{\scriptsize$b{-}k$} to (1.2,.5) node[above=-2pt]{\scriptsize $b$};
	\draw[FS,ultra thick, directed=.8, directed=.2] (-.45,-1) \pu node[above=-2pt]{$k$} (.375,-.75) 
		\pu (.125,-.375) \pu (.375,0) \pu node[above=-2pt]{$k$} (1.2,.25);
	\draw[FS,thick] (.375,-.75) \pu (.625,-.375) \pu (.375,0);
	\node at (.875,-.5) {\tiny $\yred_{a+k}$};
	\node at (.625,-.375) {$\bullet$};
	\node at (1.2,-1.25) {$\bullet$};
	\node at (.5,-1.25) {\tiny $\prod_{j=k+1}^b \yred_{a+j}$}; 
	\node at (-.1,-.1) {\tiny $e_{k-i}$};
	\node at (.125,-.375) {\tiny $\bullet$};
\end{tikzpicture}
= 
(-1)^{\genfrac(){0pt}{2}{k-1}{2}+i-1}
\begin{tikzpicture}[anchorbase,xscale=1.25, smallnodes]
	\draw[FS,ultra thick,->] (0,-1) node[below=-2pt]{\scriptsize$a{+}k$} to (0,.75)node[above=-2pt]{\scriptsize $a$};
	\draw[FS,ultra thick,->] (.75,-1) node[below=-2pt]{\scriptsize$b{-}k$} to (.75,.75) node[above=-2pt]{\scriptsize $b$};
	\draw[FS,ultra thick, directed=.75] (0,-.25) to  [out=90,in=270]  (.75,.25) ;
	\node[FS] at (0.25,-.25) {\scriptsize$k$};
	\node at (.75,-.75) {$\bullet$};
	\node at (1.5,-.75) {\tiny $\prod_{j=k+1}^b \yred_{a+j}$}; 
	\node at (.375,.3) {\tiny $\bar{v}_i^{(k)}$};
	\node at (.375,0) {$\bullet$};
\end{tikzpicture}
= 
(-1)^{\genfrac(){0pt}{2}{k}{2}-k+i}
\begin{tikzpicture}[anchorbase,xscale=1.5, smallnodes]
	\draw[FS,ultra thick,->] (0,-1) node[below=-2pt]{\scriptsize$a{+}k$} to (0,.75)node[above=-2pt]{\scriptsize $a$};
	\draw[FS,ultra thick,->] (.75,-1) node[below=-2pt]{\scriptsize$b{-}k$} to (.75,.75) node[above=-2pt]{\scriptsize $b$};
	\draw[FS,ultra thick, directed=.75] (0,-.25) to  [out=90,in=270]  (.75,.25) ;
	\node[FS] at (0.25,.25) {\scriptsize$k$};
	\node at (.75,-.75) {$\bullet$};
	\node at (1.375,-.75) {\tiny $\prod_{j=k+1}^b \yred_{a+j}$}; 
	\node at (.375,-.5) {\tiny $\CQGbox{\bar{v}_i^{(k)}}$};
\end{tikzpicture}
\]
and the result follows. (Here, we used that $\genfrac(){0pt}{2}{k-1}{2} = \genfrac(){0pt}{2}{k}{2} - k + 1$.)
\end{proof}

We record an immediate consequence, which will be used below. 

\begin{cor}\label{cor:section}
Let $\psired_{a,b}\colon \oone_{a,b} \to \VFTmin_{a,b}$ be the morphism of weight $\qdeg^{-2b}\tdeg^{2b}$
given by the composition
\[
\oone_{a,b} \xrightarrow{\approx} \qdeg^{-2b}\tdeg^{2b} \oone_{a,b} = P_{0,0,0} \hookrightarrow \VFTmin_{a,b}\, ,
\]
then
$
\Sigma_{a,b} \circ \psired_{a,b} = \yred_{a+1}\cdots \yred_{a+b} \cdot \id_{\oone_{a,b}}
$ and $
\psired_{a,b} \circ \Sigma_{a,b} \sim \yred_{a+1}\cdots \yred_{a+b} \cdot
\id_{\VFTmin_{a,b}} $.
\end{cor}
\begin{proof} 
The first identity follows directly from the explicit description of $\Sigma_{a,b}$ in Theorem~\ref{thm:splittingmap}. 

Next, $\VFTmin_{a,b}$ is invertible since it is (homotopy equivalent to) a curved Rickard complex, so we have
\[
\End_{\VSred_{a,b}}(\VFTmin_{a,b}) \simeq \End_{\VSred_{a,b}}(\oone_{a,b}) 
= \Sym(\X_1|\X_2)[\overline{\V}]\, .
\]
Thus, $\psired_{a,b} \circ \Sigma_{a,b}$ is homotopic to some multiple of $\id_{\VFTmin_{a,b}}$, 
i.e. $\psired_{a,b} \circ \Sigma_{a,b} \sim c \cdot \id_{\VFTmin_{a,b}}$ for some 
$c \in \Sym(\X_1|\X_2)[\overline{\V}]$.
We then compute
\[
\yred_{a+1}\cdots \yred_{a+b} \cdot \Sigma_{a,b}\simeq (\Sigma_{a,b} \circ \psired_{a,b}) \circ \Sigma_{a,b} 
= \Sigma_{a,b} \circ (\psired_{a,b} \circ \Sigma_{a,b}) \simeq c\cdot \Sigma_{a,b}\, .
\]
Proposition \ref{defthm:splittingmap} implies that
$\Hom_{\VSred_{a,b}}(\VFTmin_{a,b}, \oone_{a,b})$ is a torsion-free
$\Sym(\X_1|\X_2)[\yred_{a+1}\cdots \yred_{a+b}]$-module, 
thus $c= \yred_{a+1}\cdots \yred_{a+b}$.
\end{proof}

\subsection{Splitting the skein relation}
\label{ss:splitting skein rel}

We now consider the curved splitting map $\Sigma_{a,b} \colon \VFTmin_{a,b}\rightarrow \oone_{a,b}$
in the context of the skein relation.
Recall that in the proof of Proposition~\ref{prop:MCCS topologically} (see \eqref{eq:MCCS curved}) 
we have established the isomorphism
\[
\VMCCSmin_{a,b}^s \cong \I^{(s)}(\VFTmin_{a,b-s})
\]
where $\I^{(s)}$ is the functor from Definition \ref{def:functor Ikb}.
We use this isomorphism to identify these complexes, which in turn
allows us to identify the left-hand side of equation \eqref{eq:curvedskeinmincx} 
with the complex:
\begin{equation}\label{eq:DCFT}
\VTDmin_b(a) :=
\tw_{\d^c+\Delta^c}\left(\bigoplus_{s=0}^b \qdeg^{s(b-1)} \tdeg^s \I^{(s)}(\VFTmin_{a,b-s})\right).
\end{equation}
In light of the topological interpretation of $\I^{(s)}(\VFTmin_{a,b-s}) \cong
\VMCCSmin_{a,b}^s$ afforded by Proposition \ref{prop:MCCS topologically}, 
we will refer to this as the curved complex of \emph{threaded digons}. 
The curved splitting map suggests the consideration of the analogous 
(curved) complex wherein each instance of $\VFTmin_{a,b-s}$ 
is replaced by the corresponding identity bimodule $\oone_{a,b-s}$. 

\begin{definition}\label{def:digoncomplex}
The complex of \emph{(unthreaded) digons} is the complex
\[
\oone_a \boxtimes \VTDmin_b(0) = \tw_{d}\left(\bigoplus_{s=0}^b \qdeg^{s(b-1)} \tdeg^s \I^{(s)}(\oone_{a,b-s})\right)
\]
with differential $d=\bigoplus_{s=0}^{b-1} d_s$ given by
\[
d_s \colon 
\begin{tikzpicture}[scale=.5,tinynodes,rotate=90,anchorbase]
	\draw[very thick] (0,-2) node[right=-2pt]{$a$} to (0,2) ;
	\draw[very thick] (1,1) to (1,2);
	\draw[very thick] (1,-2) node[right=-2pt]{$b$} to (1,-1); 
	\draw[very thick] (1,-1) to [out=150,in=210] (1,1);
	\draw[very thick] (1,-1) to [out=30,in=330] node[above=-3pt]{$s$} (1,1); 
\end{tikzpicture}
\xrightarrow{\cre}
\begin{tikzpicture}[scale=.5,tinynodes,rotate=90,anchorbase]
	\draw[very thick] (0,-2) node[right=-2pt]{$a$} to (0,2) ;
	\draw[very thick] (1,1) to (1,2);
	\draw[very thick] (1,-2) node[right=-2pt]{$b$} to (1,-1); 
	\draw[very thick] (1,-1) to [out=150,in=210] (1,1);
	\draw[very thick] (1,-1) to [out=30,in=330] node[above=-3pt]{$s$} (1,1);
	\draw[very thick] (.75,-.75) to [out=30,in=330] node[below=-2pt]{$1$} (.75,.75);
\end{tikzpicture}
\xrightarrow{\cong}
\begin{tikzpicture}[scale=.5,tinynodes,rotate=90,anchorbase]
	\draw[very thick] (0,-2) node[right=-2pt]{$a$} to (0,2) ;
	\draw[very thick] (1,1) to (1,2);
	\draw[very thick] (1,-2) node[right=-2pt]{$b$} to (1,-1); 
	\draw[very thick] (1,-1) to [out=150,in=210] (1,1);
	\draw[very thick] (1,-1) to [out=30,in=330] node[above=-3pt]{$s$} (1,1);
	\draw[very thick] (1.25,-.75) to [out=150,in=210] node[above=-3pt]{$1$} (1.25,.75);
\end{tikzpicture}
\xrightarrow{\col}
\begin{tikzpicture}[scale=.5,tinynodes,rotate=90,anchorbase]
	\draw[very thick] (0,-2) node[right=-2pt]{$a$} to (0,2) ;
	\draw[very thick] (1,1) to (1,2);
	\draw[very thick] (1,-2) node[right=-2pt]{$b$} to (1,-1); 
	\draw[very thick] (1,-1) to [out=150,in=210] (1,1);
	\draw[very thick] (1,-1) to [out=30,in=330] node[above=-2pt]{$s{+}1$} (1,1); 
\end{tikzpicture}
\]
\end{definition}

We introduce the following notation:
\[
\Dig^s_{a,b} := \I^{(s)}(\oone_{a,b-s}) 
=
\begin{tikzpicture}[scale=.35,smallnodes,rotate=90,anchorbase]
	\draw[very thick] (0,-2) node[right=-2pt]{$a$} to (0,2) ;
	\draw[very thick] (1,1) to (1,2);
	\draw[very thick] (1,-2) node[right=-2pt]{$b$} to (1,-1); 
	\draw[very thick] (1,-1) to [out=150,in=210] (1,1);
	\draw[very thick] (1,-1) to [out=30,in=330] node[above,yshift=-2pt]{$s$} (1,1); 
\end{tikzpicture}
= \oone_{a} \boxtimes \big( \E^{(s)} \F^{(s)}\oone_{b,0} \big)
\]
for the digon webs appearing in $\oone_a \boxtimes \VTDmin_b(0)$. 
In this notation, the components of the differential $d$ take the form 
$d_s = \id_{\oone_a} \boxtimes d'_s$, 
where $d'_s$ admits the following descriptions in perpendicular and extended graphical calculus:
\[
d'_s :=
\begin{tikzpicture}[anchorbase,smallnodes]
	\draw[FS,ultra thick,->] (0,-.75) node[below=-2pt]{$b{-}s$}  to (0,.25);
	\draw[FS,ultra thick,->] (.75,-.75) node[below=-2pt]{$s$} to (.75,.25);
	\draw[FS,thick, directed=.5] (0,-.5) node[right]{$1$} to [out=90,in=270] (.75,0) ;
\end{tikzpicture}
=
\CQGsgn{(-1)^{s}}
\begin{tikzpicture}[anchorbase,smallnodes]
	\draw[CQG,ultra thick,->] (0,-.75) node[below=-2pt]{$s$} to (0,0.25);
	\draw[CQG,ultra thick,<-] (.75,-.75) node[below=-2pt]{$s$} to (.75,0.25);
	\draw[CQG,thick, rdirected=.5] (0,-.125) to [out=270,in=270]  (.75,-.125);
	\node at (1.25,0) {$(b,0)$};
\end{tikzpicture}.
\]

We now observe that the complex $\oone_a \boxtimes \VTDmin_b(0)$ is homotopically trivial.

\begin{lemma}\label{lemma:Dig complex} 
The complex $\oone_a \boxtimes \VTDmin_b(0)$ is contractible, 
with null homotopy given by $k=\bigoplus_{s=1}^b k_s$ 
with $k_s=\oone_a\boxtimes k'_s$ where 
\[ 
k'_s:=
(-1)^{b-s}
\begin{tikzpicture}[anchorbase,smallnodes]
	\draw[FS,ultra thick,->] (0,-.75) node[below]{\scriptsize$b{-}s$}  to (0,.25);
	\draw[FS,ultra thick,->] (.75,-.75) node[below,yshift=-1pt]{\scriptsize$s$} to (.75,.25);
	\draw[FS,thick, directed=.75] (.75,-.5) node[left]{$1$} to  [out=90,in=270] node[black]{$\bullet$} 
		node[above=-1pt,black]{$b{-}1$}  (0,0) ;
\end{tikzpicture}
=
(-1)^{b-s}
\CQGsgn{\cdot (-1)^{b-1}}
\begin{tikzpicture}[anchorbase,smallnodes]
	\draw[CQG,ultra thick,->] (0,-.75) node[below,yshift=-1pt]{\scriptsize$s$} to (0,0.25);
	\draw[CQG,ultra thick,<-] (.75,-.75) node[below,yshift=-1pt]{\scriptsize$s$} to (.75,0.25);
	\draw[CQG,thick, directed=.75] (0,-.375) to [out=90,in=90] node[black]{$\bullet$} 
		node[above=-1pt,black]{$b{-}1$} (.75,-.375);
	\node at (1.25,0) {$(b,0)$};
\end{tikzpicture}
\]
Moreover, the components satisfy  $k_{s-1}\circ k_s=0$.
\end{lemma}
	
\begin{proof}
We illustrate the chain complex and the homotopy as:
\[
\left(
\begin{tikzcd}
\Dig_{a,b}^0 \ar[r,shift left=2pt,"d_0"]
	& \qdeg^{b-1} \tdeg^1 \Dig_{a,b}^1 \ar[r,shift left=2pt,"d_1"] \ar[l,shift left=2pt,"k_1"]
	& \cdots \ar[r,shift left=2pt,"d_{b-1}"] \ar[l,shift left=2pt,"k_2"]
	&\qdeg^{b(b-1)} \tdeg^b \Dig_{a,b}^b \ar[l,shift left=2pt,"k_b"]
\end{tikzcd}
\right)
\]

It suffices to show that
$k'_{s+1}\circ d'_s + d'_{s-1}\circ k'_s = \id_{\E^{(s)}F^{(s)} \oone_{b,0}}$. 
In extended graphical calculus, this is
\begin{equation}
(-1)^{b-s-1} \CQGsgn{(-1)^{b-1+s}}
\begin{tikzpicture}[anchorbase,smallnodes]
	\draw[CQG,ultra thick,->] (0,-.75) node[below=-2pt]{$s$} to (0,.75);
	\draw[CQG,ultra thick,<-] (.75,-.75) node[below=-2pt]{$s$} to (.75,.75);
	\draw[CQG,thick, rdirected=.35] (.75,.125) to [out=90,in=90] 
		node[black,yshift=-1pt]{\normalsize$\bullet$} node[black,above]{$b{-}1$} (0,.125);
	\draw[CQG,thick, rdirected=.5] (0,-.125) to [out=270,in=270] (.75,-.125);
	\node at (1.25,.5) {\normalsize$(b,0)$};
\end{tikzpicture}
+
(-1)^{b-s} \CQGsgn{(-1)^{b-1+s-1}}
\begin{tikzpicture}[anchorbase,smallnodes]
	\draw[CQG,ultra thick,->] (0,-.75) node[below=-2pt]{$s$} to (0,.75);
	\draw[CQG,ultra thick,<-] (.75,-.75) node[below=-2pt]{$s$} to (.75,.75);
	\draw[CQG,thick, rdirected=.35] (.75,-.5) to [out=90,in=90] 
		node[black,yshift=-1pt]{\normalsize$\bullet$} node[above,black]{$b{-}1$} (0,-.5);
	\draw[CQG,thick, rdirected=.5] (0,.5) to [out=270,in=270] (.75,.5);
	\node at (1.25,.5) {\normalsize$(b,0)$};
\end{tikzpicture}
=
\begin{tikzpicture}[anchorbase,smallnodes]
	\draw[CQG,ultra thick,->] (0,-.75) node[below=-2pt]{$s$} to (0,.75);
	\draw[CQG,ultra thick,<-] (.75,-.75) node[below=-2pt]{$s$} to (.75,.75);
	\node at (1.25,.5) {\normalsize$(b,0)$};
\end{tikzpicture}
\end{equation}
which is a specialization of \cite[Lemma 4.6.4]{KLMS}. 
The relation $k'_{s-1}\circ k'_s=0$ similarly follows from a computation in extended graphical calculus.
\end{proof}

Note that $\Dig_{a,b}^s$ can be viewed as an object in $\VSred_{a,b}$, 
since the relevant curvature acts on it by zero.
Applying the functor $\I^{(s)}$ to the curved splitting maps 
$\Sigma_{a,b-s} \colon \VFTmin_{a,b-s} \to \oone_{a,b-s}$ gives a family of maps
\[
\VMCCSmin_{a,b}^s \xrightarrow{\I^{(s)}(\Sigma_{a,b-s})} \Dig_{a,b}^s
\]
The main goal of this section is to assemble these into a chain map
$\VTDmin_b(a) \to \oone_a \boxtimes \VTDmin_b(0)$ as follows.

\begin{thm}[Skein relation splitter]\label{thm:skein split}
There is a degree-zero closed morphism 
\[
\Phi \colon \VTDmin_b(a) \to \oone_a \boxtimes \VTDmin_b(0)
\]
whose component 
$\Phi^{r,s} \colon \qdeg^{s(b-1)}\tdeg^s \VMCCSmin_{a,b}^{s} \rightarrow \qdeg^{r(b-1)}\tdeg^r \Dig_{a,b}^r$ is 
homotopic to the splitting map $\I^{(s)}(\Sigma_{a,b-s})$ if $s=r$, and zero otherwise.
\end{thm}

The proof of this theorem will occupy the remainder of this section.
We begin with a number of preparatory lemmata.
First, the following implies that Theorem \ref{thm:skein split} 
holds in the undeformed setting.

\begin{lem}\label{lem:skeinsplituncurved}
The square
\begin{equation}\label{eq:skeinsplituncurved}
\begin{tikzcd}
\I^{(s)}(\FTmin_{a,b-s}) 
\arrow[r,"\d^c"] 
\arrow[d,"\I^{(s)}(\splitt_{a,b-s})"]
&
\I^{(s+1)}(\FTmin_{a,b-s-1}) 
\arrow[d,"\I^{(s+1)}(\splitt_{a,b-s-1})"]\\
\I^{(s)}(\oone_{a,b-s}) 
\arrow[r,"d"] 
&
\I^{(s+1)}(\oone_{a,b-s-1}) 
\end{tikzcd}
\end{equation}
commutes in $\CS_{a,b}$. By slight abuse of notation, 
$\d^c$ here denotes the morphism
\[
\I^{(s)}(\FTmin_{a,b-s}) \cong \MCCSmin_{a,b}^s 
\xrightarrow{\d^c}
\MCCSmin_{a,b}^{s+1} \cong \I^{(s+1)}(\FTmin_{a,b-s-1})\, .
\]
\end{lem}

\begin{proof}
Since $\splitt_{a,b-s}$ is supported on the minimal $\tdeg$-degree summand, the same is
true for $I^{(s)}(\splitt_{a,b-s})$. In $\MCCSmin_{a,b}^s \cong
\I^{(s)}(\FTmin_{a,b-s})$, that summand is $\qdeg^{a(b-s)-b(b+1)} \tdeg^{b}
W_{b-s} \zeta_1^{(b-s)} \cdots \zeta_{b-s}^{(b-s)} \cong \qdeg^{a(b-s)} W_{b-s}$
(see e.g. Definition \ref{def:Pkls} and equation \eqref{def:YMCCSmin}).  By
definition (see Proposition \ref{prop:KMCSdiffs}) the connecting differential
restricted to such summands is given by
\[
\qdeg^{a(b-s)-b(b+1)} \tdeg^{b} W_{b-s} \zeta_1^{(b-s)} \cdots \zeta_{b-s}^{(b-s)}
\xrightarrow{\chi_0^+} \qdeg^{a(b-s)-b(b+1)} \tdeg^{b} W_{b-s} \zeta_1^{(b-s-1)} \cdots \zeta_{b-s-1}^{(b-s)}\, .
\]
Definition \ref{def:splittingmapzero}
then implies that it
thus suffices to verify that the square
\[
\begin{tikzcd}
\qdeg^{a(b-s)} W_{b-s}
\arrow[r,"\chi_0^+"] 
\arrow[d,"\un"]
&
\qdeg^{a(b-s-1)} W_{b-s-1}
\arrow[d,"\un"]\\
\Dig_{a,b}^s
\arrow[r,"d"] 
&
\Dig_{a,b}^{s+1}
\end{tikzcd}
\]
commutes. This follows from the definitions of $\chi_0^+$ and $d$.
Indeed, (ignoring the shifts) the above square can be expanded to the diagram
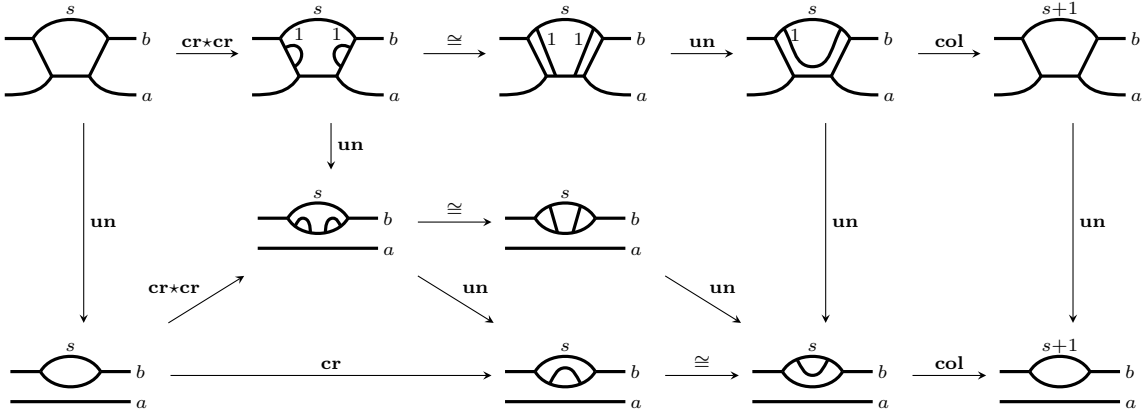
\[
\begin{tikzcd}
{\begin{tikzpicture}[rotate=90,anchorbase,smallnodes]
	\draw[very thick] (0,.25) to [out=150,in=270] (-.25,.875);
	\draw[very thick] (.5,.5) to (.5,.875);
	\draw[very thick] (0,.25) to (.5,.5);
	\draw[very thick] (0,-.25) to (0,.25);
	\draw[very thick] (.5,-.5) to [out=30,in=330] node[above=-3pt]{$s$} (.5,.5);
	\draw[very thick] (0,-.25) to (.5,-.5);
	\draw[very thick] (.5,-.875) node[right=-3pt]{$b$} to (.5,-.5);
	\draw[very thick] (-.25,-.875) node[right=-3pt]{$a$} to [out=90,in=210] (0,-.25);
\end{tikzpicture}}
\ar[r,"\cre \hComp \cre"] \ar[dd,"\un"]
&
{\begin{tikzpicture}[rotate=90,anchorbase,smallnodes]
	\draw[very thick] (0,.25) to [out=150,in=270] (-.25,.875);
	\draw[very thick] (.5,.5) to (.5,.875);
	\draw[very thick] (0,.25) to (.5,.5);
	\draw[very thick] (.125,.3125) to [out=300,in=225] (.375,.25) 
		node[above,yshift=-2pt]{$1$} to [out=45,in=300] (.375,.4375);
	\draw[very thick] (0,-.25) to (0,.25);
	\draw[very thick] (.5,-.5) to [out=30,in=330] node[above=-3pt]{$s$} (.5,.5);
	\draw[very thick] (.125,-.3125) to [out=60,in=135] (.375,-.25) 
		node[above,yshift=-2pt]{$1$} to [out=315,in=60] (.375,-.4375);	
	\draw[very thick] (0,-.25) to (.5,-.5);
	\draw[very thick] (.5,-.875) node[right=-3pt]{$b$} to (.5,-.5);
	\draw[very thick] (-.25,-.875) node[right=-3pt]{$a$} to [out=90,in=210] (0,-.25);
\end{tikzpicture}}
\ar[r,"\cong"] \ar[d,"\un"]
&
{\begin{tikzpicture}[rotate=90,anchorbase,smallnodes]
	\draw[very thick] (0,.25) to [out=150,in=270] (-.25,.875);
	\draw[very thick] (.5,.5) to (.5,.875);
	\draw[very thick] (0,.25) to (.5,.5);	
	\draw[very thick] (0,.125) to node[above,yshift=-2pt,xshift=2pt]{$1$} (.625,.375);
	\draw[very thick] (0,-.25) to (0,.25);
	\draw[very thick] (.5,-.5) to [out=30,in=330] node[above=-3pt]{$s$} (.5,.5);
	\draw[very thick] (0,-.125) to node[above,yshift=-2pt,xshift=-2pt]{$1$} (.625,-.375);
	\draw[very thick] (0,-.25) to (.5,-.5);
	\draw[very thick] (.5,-.875) node[right=-3pt]{$b$} to (.5,-.5);
	\draw[very thick] (-.25,-.875) node[right=-3pt]{$a$} to [out=90,in=210] (0,-.25);
\end{tikzpicture}}
\ar[r,"\un"]
&
{\begin{tikzpicture}[rotate=90,anchorbase,smallnodes]
	\draw[very thick] (0,.25) to [out=150,in=270] (-.25,.875);
	\draw[very thick] (.5,.5) to (.5,.875);
	\draw[very thick] (0,.25) to (.5,.5);	
	\draw[very thick] (0,-.25) to (0,.25);
	\draw[very thick] (.5,-.5) to [out=30,in=330] node[above=-3pt]{$s$} (.5,.5);
	\draw[very thick] (.625,.375) to [out=210,in=90] 
		node[above]{$1$} (.125,0) to [out=270,in=150] (.625,-.375);
	\draw[very thick] (0,-.25) to (.5,-.5);
	\draw[very thick] (.5,-.875) node[right=-3pt]{$b$} to (.5,-.5);
	\draw[very thick] (-.25,-.875) node[right=-3pt]{$a$} to [out=90,in=210] (0,-.25);
\end{tikzpicture}}
\ar[r,"\col"] \ar[dd,"\un"]
&
{\begin{tikzpicture}[rotate=90,anchorbase,smallnodes]
	\draw[very thick] (0,.25) to [out=150,in=270] (-.25,.875);
	\draw[very thick] (.5,.5) to (.5,.875);
	\draw[very thick] (0,.25) to (.5,.5);
	\draw[very thick] (0,-.25) to (0,.25);
	\draw[very thick] (.5,-.5) to [out=30,in=330] node[above=-3pt]{$s{+}1$} (.5,.5);
	\draw[very thick] (0,-.25) to (.5,-.5);
	\draw[very thick] (.5,-.875) node[right=-3pt]{$b$} to (.5,-.5);
	\draw[very thick] (-.25,-.875) node[right=-3pt]{$a$} to [out=90,in=210] (0,-.25);
\end{tikzpicture}} 
\ar[dd,"\un"]
\\
&
{\begin{tikzpicture}[scale=.4,smallnodes,rotate=90,anchorbase]
	\draw[very thick] (0,-2) node[right=-3pt]{$a$} to (0,2) ;
	\draw[very thick] (1,1) to (1,2);
	\draw[very thick] (1,-2) node[right=-3pt]{$b$} to (1,-1); 
	\draw[very thick] (1,-1) to [out=150,in=210] (1,1);
	\draw[very thick] (1,-1) to [out=30,in=330] node[above=-3pt]{$s$} (1,1); 
	\draw[very thick] (.75,.75) to [out=330,in=90] (1,.5) to [out=270,in=0] (.5,.25);
	\draw[very thick] (.75,-.75) to [out=30,in=270] (1,-.5) to [out=90,in=0] (.5,-.25);
\end{tikzpicture}}
\ar[rd,"\un"] \ar[r,"\cong"]
&
{\begin{tikzpicture}[scale=.4,smallnodes,rotate=90,anchorbase]
	\draw[very thick] (0,-2) node[right=-3pt]{$a$} to (0,2) ;
	\draw[very thick] (1,1) to (1,2);
	\draw[very thick] (1,-2) node[right=-3pt]{$b$} to (1,-1); 
	\draw[very thick] (1,-1) to [out=150,in=210] (1,1);
	\draw[very thick] (1,-1) to [out=30,in=330] node[above=-3pt]{$s$} (1,1); 
	\draw[very thick] (1.375,.5) to (.5,.25);
	\draw[very thick] (1.375,-.5) to (.5,-.25);
\end{tikzpicture}}
\ar[rd,"\un"]
& &
\\
{\begin{tikzpicture}[scale=.4,smallnodes,rotate=90,anchorbase]
	\draw[very thick] (0,-2) node[right=-3pt]{$a$} to (0,2) ;
	\draw[very thick] (1,1) to (1,2);
	\draw[very thick] (1,-2) node[right=-3pt]{$b$} to (1,-1); 
	\draw[very thick] (1,-1) to [out=150,in=210] (1,1);
	\draw[very thick] (1,-1) to [out=30,in=330] node[above=-3pt]{$s$} (1,1); 
\end{tikzpicture}}
\ar[ur,"\cre \hComp \cre"] \ar[rr,"\cre"]
& &
{\begin{tikzpicture}[scale=.4,smallnodes,rotate=90,anchorbase]
	\draw[very thick] (0,-2) node[right=-3pt]{$a$} to (0,2) ;
	\draw[very thick] (1,1) to (1,2);
	\draw[very thick] (1,-2) node[right=-3pt]{$b$} to (1,-1); 
	\draw[very thick] (1,-1) to [out=150,in=210] (1,1);
	\draw[very thick] (1,-1) to [out=30,in=330] node[above=-3pt]{$s$} (1,1); 
	\draw[very thick] (.625,.5) to [out=330,in=90] (1.125,0) to [out=270,in=30] (.625,-.5);
\end{tikzpicture}}
\ar[r,"\cong"]
&
{\begin{tikzpicture}[scale=.4,smallnodes,rotate=90,anchorbase]
	\draw[very thick] (0,-2) node[right=-3pt]{$a$} to (0,2) ;
	\draw[very thick] (1,1) to (1,2);
	\draw[very thick] (1,-2) node[right=-3pt]{$b$} to (1,-1); 
	\draw[very thick] (1,-1) to [out=150,in=210] (1,1);
	\draw[very thick] (1,-1) to [out=30,in=330] node[above=-3pt]{$s$} (1,1); 
	\draw[very thick] (1.375,.5) to [out=210,in=90] (.875,0) to [out=270,in=150] (1.375,-.5);
\end{tikzpicture}}
\ar[r,"\col"]
&
{\begin{tikzpicture}[scale=.4,smallnodes,rotate=90,anchorbase]
	\draw[very thick] (0,-2) node[right=-3pt]{$a$} to (0,2) ;
	\draw[very thick] (1,1) to (1,2);
	\draw[very thick] (1,-2) node[right=-3pt]{$b$} to (1,-1); 
	\draw[very thick] (1,-1) to [out=150,in=210] (1,1);
	\draw[very thick] (1,-1) to [out=30,in=330] node[above=-3pt]{$s{+}1$} (1,1); 
\end{tikzpicture}}
\end{tikzcd} .
\]
All of the squares are readily seen to commute. The Frobenius extension
structure discussed in \S\ref{sec:Frobenius} implies that the triangle
commutes, and commutativity of the hexagon can be verified using an explicit
computation (or is a consequence of \cite[Equation (3.8)]{QR} and foam isotopy).
\end{proof}

Adjoining the squares \eqref{eq:skeinsplituncurved} for $0\leq s \leq b$ gives a
chain map 
\[
\KMCSmin_{a,b} \defeq \tw_{\d^c} \! \left(\bigoplus_{s=0}^b \qdeg^{s(b-1)} \tdeg^s \MCCSmin_{a,b}^s\right) 
\longrightarrow 
\oone_a \boxtimes \VTDmin_b(0)
\]
which is the uncurved analogue of Theorem \ref{thm:skein split}. 
It remains to deform this map, and our argument will proceed in two steps. 
First, we will show that the a deformation of the square \eqref{eq:skeinsplituncurved} 
commutes up to homotopy. Second, we will ``straighten'' these maps to give the desired closed morphism.

To aid in the former, we next establish a particular instance of an 
adjunction involving the functors $I^{(s)}$.
Set
\[
\overline{E}_{a,(b-s,s)} := \End_{\VSred_{a,b}}(\Dig_{a,b}^s) 
= \Sym(\X_1|\L|\B) \otimes \k[\bar{v}_1^{(b)}, \ldots, \bar{v}_b^{(b)}]
\]
Here, we follow Convention \ref{conv:I alphabets} and denote
the alphabet on the $s$-labeled edge in $\Dig_{a,b}^s$ by $\B$ and 
the alphabet on the $(b-s)$-labeled edge by $\L$.
We view $\overline{E}_{a,(b-s,s)}$ as a module over
\[
\overline{E}_{a,b-s} := \End_{\VSred_{a,b-s}}(\oone_{a,b-s}) 
= \Sym(\X_1|\L) \otimes \k[\bar{v}_1^{(b-s)}, \ldots, \bar{v}_{b-s}^{(b-s)}]
\]
via the inclusion $\Sym(\X_1|\L) \hookrightarrow \Sym(\X_1|\L|\B)$ and the map 
\[
\k[\bar{v}_1^{(b-s)}, \ldots, \bar{v}_{b-s}^{(b-s)}] \to \Sym(\L) \otimes \k[\bar{v}_1^{(b)}, \ldots, \bar{v}_b^{(b)}]
\]
given in \eqref{eq:vbar l and vbar b}.

\begin{lem}\label{lem:homs from I(X) to 1}
Let $X$ be a curved complex in $\VSred_{a,b-s}$, then there is an isomorphism
\[
\qdeg^{-s(b-s)} 
\Hom_{\VSred_{a,b-s}}(X, \oone_{a,b-s}) \otimes_{\overline{E}_{a,b-s}} \overline{E}_{a,(b-s,s)}
\xrightarrow{\cong} 
\Hom_{\VSred_{a,b}}(I^{(s)}(X) , \oone_{a,b})
\]
of dg $\overline{E}_{a,(b-s,s)}$-modules that is natural in $X$.
\end{lem}
\begin{proof}
First, we consider the undeformed setting.
Let $X$ be a $1$-morphism in $\CS_{a,b}$, then we have that
\[
I^{(s)}(X) \defeq (\oone_a \boxtimes {}_{b}M_{b-s,s}) \hComp (X \boxtimes \oone_s) 
	\hComp (\oone_a \boxtimes {}_{b-s,s}S_{b}).
\]
Using Proposition \ref{prop:dualityonMS} and Corollary \ref{cor:I(X)Hom}, 
we compute
\begin{equation}\label{eq:IsoInC}
\begin{aligned}
\Hom_{\CS_{a,b}}(I^{(s)}(X) , \oone_{a,b}) 
&= \Hom_{\CS_{a,b}}\big( (\oone_a \boxtimes {}_{b}M_{b-s,s}) \hComp (X \boxtimes \oone_s) 
	\hComp (\oone_a \boxtimes {}_{b-s,s}S_{b}) , \oone_{a,b} \big) \\ 
&\cong \qdeg^{-2s(b-s)} \Hom_{\CS_{a,b}} \big(X \boxtimes \oone_s , 
	\oone_{a} \boxtimes ({}_{b-s,s}S_{b} \hComp {}_{b}M_{b-s,s}) \big) \\
&\cong \qdeg^{-s(b-s)} \Hom_{\CS_{a,b}} \big(X , \oone_{a,b-s})
	\otimes \End_{\CS_{a,b}}(\oone_s) \, .
\end{aligned}
\end{equation}

In our present notation, we have that $\End_{\CS_{a,b}}(\oone_s) = \Sym(\B)$.
Set $\overline{E}_{a,b-s}(\B) := \overline{E}_{a,b-s} \otimes \Sym(\B)$, then
Definition \ref{def:functor Ikb} then gives that
\[
\begin{aligned}
\Hom_{\VSred_{a,b}}(I^{(s)}(X) , \oone_{a,b}) 
&\defeq \Hom_{\CS_{a,b}}(I^{(s)}(X) , \oone_{a,b}) \otimes \k[\bar{v}_1^{(b)},\ldots,\bar{v}_b^{(b)}] \\
&\cong \Hom_{\CS_{a,b}}(I^{(s)}(X) , \oone_{a,b}) \otimes_{\Sym(\X_1|\L|\B)} \overline{E}_{a,(b-s,s)} \\
&\cong \Hom_{\CS_{a,b}}(I^{(s)}(X) , \oone_{a,b}) \otimes_{\Sym(\X_1|\L|\B)} \overline{E}_{a,b-s}(\B) 
	\otimes_{\overline{E}_{a,b-s}(\B)} \overline{E}_{a,(b-s,s)}\, .
\end{aligned}
\]
Using the isomorphism in \eqref{eq:IsoInC}, we obtain
\[
\begin{aligned}
\Hom_{\VSred_{a,b}}(I^{(s)}(X) , \oone_{a,b}) 
&\cong \qdeg^{-s(b-s)} \big( \Hom_{\CS_{a,b-s}}(X , \oone_{a,b-s}) \otimes \Sym(\B) \big) \otimes_{\Sym(\X_1|\L|\B)} \overline{E}_{a,b-s}(\B) 
	\otimes_{\overline{E}_{a,b-s}(\B)} \overline{E}_{a,(b-s,s)} \\
&\cong \qdeg^{-s(b-s)} \Hom_{\CS_{a,b-s}}(X , \oone_{a,b-s}) \otimes_{\Sym(\X_1|\L)} \overline{E}_{a,b-s} \otimes_{\overline{E}_{a,b-s}} 
	\overline{E}_{a,(b-s,s)}\, .
\end{aligned}
\]
Since
\[
\Hom_{\VSred_{a,b-s}}(X, \oone_{a,b-s}) 
\cong \Hom_{\CS_{a,b-s}}(X, \oone_{a,b-s}) \otimes_{\Sym(\X_1,\L)} \overline{E}_{a,b-s}
\]
this implies our result. Indeed, it remains to see that the isomorphism is an isomorphism of complexes, 
and this follows from Definition \ref{def:functor Ikb}, which implies that
the differential on both complexes is induced from that on $X$.
\end{proof}

\begin{cor}\label{cor:HomFromVFT}
There is an $\overline{E}_{a,b}$-linear homotopy equivalence
\[
\Hom_{\VSred_{a,b}}(\I^{(s)}(\oone_{a,b-s}),\oone_{a,b}) 
\xrightarrow{\simeq} 
\Hom_{\VSred_{a,b}}(\I^{(s)}(\VFTmin_{a,b-s}),\oone_{a,b})
\]
defined by sending $f\mapsto f \circ \I^{(s)}(\Sigma_{a,b-s})$.
\end{cor}
\begin{proof}
The preceding lemma gives the commutative diagram
\[
\begin{tikzcd}
\qdeg^{-s(b-s)} \Hom_{\VSred_{a,b-s}}(\oone_{a,b-s}, \oone_{a,b-s}) \otimes_{\overline{E}_{a,b-s}} \overline{E}_{a,(b-s,s)}
\ar[r,"\cong"] \ar[d,"(-) \circ \Sigma_{a,b-s}"]
& \Hom_{\VSred_{a,b}}(I^{(s)}(\oone_{a,b-s}) , \oone_{a,b}) \ar[d, "(-) \circ I^{(s)}(\Sigma_{a,b-s}) "] \\
\qdeg^{-s(b-s)} \Hom_{\VSred_{a,b-s}}(\VFTmin_{a,b-s}, \oone_{a,b-s}) \otimes_{\overline{E}_{a,b-s}} \overline{E}_{a,(b-s,s)}
\ar[r,"\cong"]
& \Hom_{\VSred_{a,b}}(I^{(s)}(\VFTmin_{a,b-s}) , \oone_{a,b})\, .
\end{tikzcd}
\]
Proposition \ref{defthm:splittingmap} gives that the left vertical map is a homotopy equivalence,
thus the right vertical map is as well.
\end{proof}

We now have a curved analogue of Lemma \ref{lem:skeinsplituncurved}.

\begin{lemma}\label{lem:nonstrictladder}
The square
\begin{equation}\label{eq:nonstrictladder}
\begin{tikzcd}
\I^{(s)}(\VFTmin_{a,b-s}) 
\arrow[r,"\d^c+\Delta^c"] 
\arrow[d,"\I^{(s)}(\Sigma_{a,b-s})"]
&
\I^{(s+1)}(\VFTmin_{a,b-s-1}) 
\arrow[d,"\I^{(s+1)}(\Sigma_{a,b-s-1})"]\\
\I^{(s)}(\oone_{a,b-s}) 
\arrow[r,"d"] 
&
\I^{(s+1)}(\oone_{a,b-s-1}) 
\end{tikzcd}
\end{equation}
in $\VSred_{a,b}$ commutes up to homotopy.
\end{lemma}

\begin{proof}
First, observe that \eqref{eq:nonstrictladder} specializes to
\eqref{eq:skeinsplituncurved}, thus by Lemma \ref{lem:skeinsplituncurved} we
have that the chain maps $d \circ \I^{(s)}(\Sigma_{a,b-s})$ and
$\I^{(s+1)}(\Sigma_{a,b-s}) \circ (\d^c+\Delta^c)$ are curved lifts in
$\Hom_{\VSred_{a,b}}\big(\I^{(s)}(\VFTmin_{a,b-s}),\I^{(s+1)}(\oone_{a,b-s-1})
\big)$ of the same morphism in
$\Hom_{\CS_{a,b}}\big(\I^{(s)}(\FTmin_{a,b-s}),\I^{(s+1)}(\oone_{a,b-s-1})
\big)$. Hence, the morphism 
\[
d \circ \I^{(s)}(\Sigma_{a,b-s}) - \I^{(s+1)}(\Sigma_{a,b-s-1}) \circ (\d^c+\Delta^c)
\] 
is a closed, degree-zero element of the $\V$-irrelevant
submodule of
$\Hom_{\VSred_{a,b}}\big(\I^{(s)}(\VFTmin_{a,b-s}),\I^{(s+1)}(\oone_{a,b-s-1})
\big)$.

Now, since
\[
\I^{(s+1)}(\oone_{a,b-s-1}) = \Dig_{a,b}^{s+1} \cong \bigoplus_{\qbinom{b}{s+1}} \oone_{a,b} \, ,
\]
Corollary \ref{cor:HomFromVFT} gives an $\overline{E}_{a,b}$-linear homotopy equivalence
\[
\Hom_{\VSred_{a,b}}\big(\I^{(s)}(\VFTmin_{a,b-s}),\I^{(s+1)}(\oone_{a,b-s-1}) \big)
\simeq
\Hom_{\VSred_{a,b}}\big(\I^{(s)}(\oone_{a,b-s}),\I^{(s+1)}(\oone_{a,b-s-1}) \big) \, .
\]
As in Remark \ref{rem:uniqueSigma}, any degree-zero element in the latter that
is $\V$-irrelevant is zero, hence null-homotopic. Thus 
\[
d \circ \I^{(s)}(\Sigma_{a,b-s}) - \I^{(s+1)}(\Sigma_{a,b-s-1}) \circ (\d^c+\Delta^c) \sim 0 
\]
in $\Hom_{\VSred_{a,b}}\big(\I^{(s)}(\VFTmin_{a,b-s}),\I^{(s+1)}(\oone_{a,b-s-1}) \big)$, 
as desired.
\end{proof}

Using this, we can now give the following.

\begin{proof}[Proof of Theorem \ref{thm:skein split}]
We will prove that there exist (strictly) commutative squares of the form
\begin{equation}\label{eq:strictsquare}
\begin{tikzcd}
\I^{(s)}(\VFTmin_{a,b-s}) 
\arrow[r,"\d^c+\Delta^c"] 
\arrow[d,"g_s"]
&
\I^{(s+1)}(\VFTmin_{a,b-s-1}) 
\arrow[d,"g_{s+1}"]\\
\I^{(s)}(\oone_{a,b-s}) 
\arrow[r,"d"] 
&
\I^{(s+1)}(\oone_{a,b-s-1}) 
\end{tikzcd}
\end{equation}
wherein $g_s \sim \I^{(s)}(\Sigma_{a,b-s})$.
Commutativity of these diagrams then implies that the $g_s$ assemble to give the 
requisite closed morphism. Set
\[
D_s := \d^c+\Delta^c \in \Hom_{\VSred_{a,b}}\big(\I^{(s)}(\VFTmin_{a,b-s}) ,\I^{(s+1)}(\VFTmin_{a,b-s})  \big)\, ,
\]
and let $k_s \in \Hom_{\VSred_{a,b}}\big(\I^{(s)}(\oone_{a,b-s}) ,\I^{(s-1)}(\oone_{a,b-s+1}) \big)$
be the homotopy from Lemma \ref{lemma:Dig complex}. 
Define
\[
g_s := d_{s-1} \circ k_s \circ \I^{(s)}(\Sigma_{a,b-s}) + k_{s+1} \circ \I^{(s+1)}(\Sigma_{a,b-s-1}) \circ D_s\, ,
\]
then we compute that
\[
d_s \circ g_s = d_s \circ k_{s+1} \circ \I^{(s+1)}(\Sigma_{a,b-s-1}) \circ D_s
=  g_{s+1} \circ D_s
\]
so the square \eqref{eq:strictsquare} indeed strictly commutes.
Finally, Lemma \ref{lem:nonstrictladder} gives that 
$d_s \circ \I^{(s)}(\Sigma_{a,b-s}) \sim \I^{(s+1)}(\Sigma_{a,b-s-1}) \circ D_s$, 
so
\begin{align*}
g_s &\defeq  d_{s-1} \circ k_s \circ \I^{(s)}(\Sigma_{a,b-s}) + k_{s+1} \circ \I^{(s+1)}(\Sigma_{a,b-s-1}) \circ D_s \\
&\sim d_{s-1} \circ k_s \circ \I^{(s)}(\Sigma_{a,b-s}) + k_{s+1} \circ d_s \circ \I^{(s)}(\Sigma_{a,b-s}) 
= \I^{(s)}(\Sigma_{a,b-s})
\end{align*}
as desired.
\end{proof}

\section{Colored full twists and Hilbert schemes}
\label{s:colored Hilb}

As conjectured in \cite{GNR} and shown in \cite{GH}, 
the relation between Soergel bimodules and Hilbert schemes is mediated by the 
(positive) full twist braid.  
In this section, we speculate on the extension to the colored setting, 
culminating in Conjecture \ref{conj:FT ideal}.
In the following \S\ref{s:hopf link}, we establish this conjecture in the $2$-strand case.

\subsection{The full twist ideals}
\label{ss:FT setup}

Let $\brc = (\bre_1,\ldots,\bre_m)$ be an object in $\SSBim$. Extending the
notation from \S \ref{ss:two strand V cats}, we let $\CS_{\brc}:=
\CS({}_{\brc}\SSBim_{\brc}),$ and denote the corresponding alphabets of
cardinality $\bre_i$ acting on the left and right by $\leftX_i$ and $\rightX_i$,
respectively (for $1 \leq i \leq m$). Following the notation in
\eqref{eq:Vnotation}, we denote the corresponding dg category of curved
complexes by $\VS_{\brc}$. Recall that the deformation parameters therein are
denoted by $\V_i = \{v_{i,1},\ldots, v_{i,\bre_i}\}$ and the curvature element
is
\[
(\d_X+\Delta)^2 = \sum_{i=1}^m \sum_{r=1}^{\bre_i} h_r(\leftX_i - \rightX_i) v_{i,r}\, .
\] 
Let $\FT_{\brc}\in \CS_{\brc}$ denote the Rickard complex associated to the
$\brc$-colored positive full twist braid, and let $\VFT_{\brc}$ denote its
(unique, up to homotopy) lift to $\VS_{\brc}$.

We now introduce the main objects of interest.  
\begin{definition}\label{def:EM} Extending \eqref{eq:firstE}, we let
\[
E_{\brc} := \End_{\VS_{\brc}}(\oone_{\brc}) = \Sym(\X_1|\cdots|\X_m)[\V_1,\ldots,\V_m]
\]
and set
\[
M_{\brc}:=\Hom_{\VS_{\brc}}(\oone_{\brc},\VFT_{\brc})\, ,
\] 
which we regard as a (differential) bigraded $E_{\brc}$-module. 
\end{definition}

Note that $M_{\brc}$ is a bigraded complex whose homology $H(M_{\brc})$
is the lowest Hochschild-degree summand of the deformed homology of the
$\brc$-colored $(m,m)$-torus link $T(m,m;\brc)$, up to a $\qdeg$-shift. 
We can also regard $E_{\brc}$ as a bigraded complex with zero differential.
The splitting map $\Sigma_\brc\colon \VFT_{\brc}\rightarrow \oone_{\brc}$ from Remark \ref{rem:genSigma}
induces an $E_{\brc}$-linear morphism 
\begin{equation}\label{eq:MtoE}
M_{\brc} \xrightarrow{\Sigma_\brc\circ -} E_{\brc},
\end{equation} 
which is a chain map, since $\Sigma_{\brc}$ is degree-zero and closed.

\begin{definition}\label{def:J}
Let $H(\Sigma_{\brc})\colon H(M_{\brc}) \to E_{\brc}$ denote the map induced on homology by 
the chain map \eqref{eq:MtoE}. We set $J_{\brc} := \im(H(\Sigma_{\brc}))$ and call 
$J_{\brc}\vartriangleleft E_{\brc}$ the $\brc$-\emph{colored full twist ideal}. 
(By construction, it is an ideal in $E_{\brc}$.)
\end{definition}

\begin{remark}
Recall that the (two strand) splitting map $\Sigma_{a,b}\colon \VFT_{a,b} \to \oone_{a,b}$ is only canonical 
only up to homotopy. Since the general splitting map is built from $\Sigma_{a,b}$'s, 
replacing the latter with a homotopic map would in turn replace $\Sigma_{\brc}$ with a homotopic map.
However, since $E_{\brc}$ has zero differential, the full twist ideal $J_{\brc}$ remains well-defined, 
regardless of our specific model for $\Sigma_{a,b}$.
\end{remark}

It is an important problem to compute $H(M_{\brc})$ and $J_{\brc}$ explicitly. 
Indeed, in the case that $\brc=1^m := (1,\ldots,1)$, they were shown to be isomorphic, 
and the latter ideal was described explicitly by E.~Gorsky and the first-named author \cite{GH}.
In turn, this computation was used to provide an explicit connection between triply-graded
Khovanov--Rozansky homology and the geometry of the Hilbert scheme of points in $\C^2$.
As such, we anticipate that an explicit presentation of $J_{\brc}$ for general $\brc$
will provide an analogous relation between \emph{colored} Khovanov--Rozansky homology and 
the geometry of the Hilbert scheme.

The explicit description of $H(M_{1^m}) \cong J_{1^m}$ in \cite{GH} relies on
work of Elias and the first-named author \cite{MR3880028}. Therein, they show
that the \emph{uncolored} $(m,m)$-torus link $T(m,m)$ is \emph{parity} (Definition \ref{def:parity}).  
Based on this, and Conjecture \ref{conj:cables}, we propose the following. 

\begin{conj}\label{conj:FT is parity} 
The $\brc$-colored $(m,m)$-torus link $T(m,m;\brc)$ is parity. 
More generally, all colored positive torus links are parity.
\end{conj}

As Theorem~\ref{thm:parityfree} shows, 
the following is an immediate consequence of parity.

\begin{proposition}\label{prop:M is free general} If Conjecture \ref{conj:FT is
parity} holds, then there is an isomorphism of bigraded $\k[\V]$-modules:
\begin{equation*}
	\pushQED{\qed}
	\YS H_{\KR}(T(m,m;\brc)) \cong H_{\KR}(T(m,m;\brc)) \otimes \k[\V] \, .\qedhere
	\popQED
\end{equation*}
\end{proposition}

We next record an important consequence of the expected ``flatness'' statement in 
Proposition~\ref{prop:M is free general}. 

\begin{cor}\label{cor:M is J general} 
Let $T(m,m;\brc)$ be the $\brc$-colored $(m,m)$-torus link, 
and let $U(\brc)$ be the $\brc$-colored unlink.
If Conjecture \ref{conj:FT is parity} holds, then the map 
$\YS H_{\KR}(\Sigma_{\brc})\colon \YS H_{\KR}\big( T(m,m;\brc) \big) \to \YS H_{\KR}\big( U(\brc) \big)$
induced by the splitting map $\Sigma_{\brc}\colon \VFT_{\brc} \to \oone_{\brc}$ 
is injective.
In particular, the map 
$H(\Sigma_\brc) \colon H(M_{\brc}) \rightarrow E_{\brc}$ is injective and
thus an isomorphism onto its image $J_{\brc}\subset E_{\brc}$. 
\end{cor}

\begin{proof}
Consider the $m=2$ case (i.e. $\brc=(a,b)$).
Corollary \ref{cor:section} shows that there exists a morphism 
$\psired_{a,b} \in \Hom_{\VSred_{a,b}}(\oone_{a,b},\VFT_{a,b})$ 
that satisfies 
\[
\psired_{a,b} \circ \Sigma_{a,b} \sim \yred_{a+1}\cdots \yred_{a+b} \cdot \id_{\VFT_{a,b}}
\]
(in our current notation, the $\yred$ variables correspond to the alphabet $\V_2$ 
corresponding to the $b$-labeled strand).
It follows that the induced map on homology satisfies
\[
\YS H_{\KR}(\psired_{a,b}) \circ \YS H_{\KR}(\Sigma_{a,b}) 
= \yred_{a+1}\cdots \yred_{a+b} \cdot \id_{\YS H_{\KR}( T(2,2;(a,b)))}\, .
\]

Assuming Conjecture \ref{conj:FT is parity},
Proposition \ref{prop:M is free general} implies that $H(M_{(a,b)})$ is $\V$-torsion free. 
Since
\[
\yred_j = \sum_{k=1}^b x_j^{k-1} \bar{v}_k 
\]
we have that $\yred_{a+1}\cdots \yred_{a+b}$ is the sum of the monomial 
$(\bar{v}_1)^b$ and terms involving the alphabet $\X_2$ that are lower order with respect to 
the alphabet $\overline{\V}$ (if we impose the monomial order $\bar{v}_1 > \cdots > \bar{v}_b$).
Since\footnote{In the $2$-strand notation established in \S\ref{ss:two strand V cats},
we denote $V_1 = \V_L^{(a)}$ and $\V_2=\V_R^{(b)}$.} 
$\k[\V] = \k[\V_L^{(a)},\V_R^{(b)}] = \k[\V_L^{(a)},\overline{\V}]$, 
this implies that multiplication by $\yred_{a+1}\cdots \yred_{a+b}$ is injective.
Hence, $\YS H_{\KR}(\Sigma_{a,b})$ is injective as well.

The general case follows analogously, using the description of the full twist
braid as the horizontal composition of two-strand full twists from Remark
\ref{rem:genSigma}. See also Theorem \ref{thm:parityinjective} below for a more
general result.
\end{proof}

Assuming the validity of Conjecture \ref{conj:FT is parity}, 
the computation of the deformed colored homology of the colored torus link $T(m,m;\brc)$ 
in lowest Hochschild-degree amounts to finding a presentation of the ideal $J_{\brc} \lhd E_{\brc}$. 
We now aim to formulate a precise conjectural description of $J_{\brc}$.
To help motivate this, we now recall the description in the uncolored case from \cite{GH}. 
We begin by establishing some conventions.

\begin{conv}\label{conv:J E uncolored} 
As above, we write $1^N$ for the sequence $(1,\ldots,1)$ of length $N$. 
If $\brc= 1^N$ then each alphabet $\X_i$ consists of a single variable $x_i$
and we have a single deformation parameter for each $i$ that we denote by $y_i := v_{i,1}$. 
The latter thus satisfy $\wt (y_i) = \qdeg^{-2}\tdeg^2$.
In this case, define the \emph{total alphabets} by
\[
\X = \bigsqcup_{i=1}^m \X_i = \{x_1,\ldots,x_N\} \, , \quad 
\Y = \bigsqcup_{i=1}^m \V_i = \{y_1,\ldots,y_N\}
\] 
so $E_{1^N} = \k[\X,\Y]$.
\end{conv}

\begin{conv}\label{conv:ideals}
Below, we will consider various algebras $E$ and ideals $I \vartriangleleft E$ that are 
generated by collections of elements $S \subset E$. 
Since we will consider different algebras that contain the same subsets $S$, 
we will denote such ideals by $I = E \cdot S$ in order to make clear the algebra in 
which each ideal lives.
\end{conv}

The following combines the torus link computations of \cite{MR3880028} 
(see also \cite{mellit2017homology, hogancamp2019torus}) 
with \cite[Theorem 6.16 and Corollary 6.17]{GH}.

\begin{thm}\label{thm:FT uncolored} 
The parity conjecture (Conjecture \ref{conj:FT is parity}) holds when $\brc=1^N$.
Further, the full twist ideal $J_N:=J_{1^N}$ 
is equal to the ideal $I_N \lhd \k[\X,\Y]$ generated by polynomials in $\k[\X,\Y]$ 
that are antisymmetric with respect to the diagonal $\symg_N$-action.
\end{thm}

Inspired by this theorem, and also our results in \S\ref{s:hopf link} below, 
we now formulate a conjecture concerning the ideals $J_{\brc}$ 
for general $\brc=(\bre_1,\ldots,\bre_m)$ that extends Theorem \ref{thm:FT uncolored}.

\begin{conv}\label{conv:total}
We extend Convention \ref{conv:J E uncolored} to the case of general $\brc = (\bre_1,\ldots,\bre_m)$
by defining the \emph{total alphabets}
\[
\X := \bigsqcup_{i=1}^m \X_i \, , \quad \V=\bigsqcup_{i=1}^m \V_i\, .
\]
We write $\X = \{x_1,\ldots,x_N\}$ where $N:=\bre_1+\cdots+\bre_m$
and thus identify the alphabets $\X_i$ (of cardinality $\bre_i$) with the subsets
\[
\X_i = \{x_j \mid \bre_1+\cdots + \bre_{i-1}<j\leq \bre_1+\cdots + \bre_{i} \} \subset \X\, .
\]
It follows that
\[
E_\brc = \Sym(\X_1|\cdots|\X_m)[\V_1,\ldots,\V_m]
= \k[\X,\V]^{\symg_{\bre_1}\times \cdots \times \symg_{\bre_m}}\, ,
\]
provided we let $\symg_{\bre_1}\times \cdots \times \symg_{\bre_m}$ act on the alphabet $\V$ trivially.
\end{conv}

For the duration of this section, let $\brc = (\bre_1,\ldots,\bre_m)$ and set
$N=\bre_1+\cdots+\bre_m$. Given an index $1 \leq j \leq N$, we write 
$\varpi(j) = i$ if 
$\bre_1+\cdots + \bre_{i-1} +1 \leq j \leq \bre_1+\cdots + \bre_{i-1}+\bre_i$. 

\begin{definition} \label{def:ys and vs}
Consider the alphabets
\[
\Y_i := \{y_{\bre_1+\cdots + \bre_{i-1} +1},\ldots, y_{\bre_1+\cdots + \bre_{i-1} +\bre_i}\}
\subset \Y = \{y_1,\ldots,y_N\}\, ,
\]
and regard $\k[\X_i,\Y_i]$ as a subalgebra of $\k[\X_i,\V_i]$ via the assignment 
\begin{equation}\label{eq:y to v 2}
y_j\mapsto \sum_{r=1}^{\bre_{\varpi(j)}} x_{j}^{r-1} v_{\varpi(j),r}
\end{equation}
that is analogous to \eqref{eq:y and v}.
Let $\symg_N$ act diagonally on $\k[\X,\Y]$, 
which we view as a subalgebra of $\k[\X,\V]$.
Following Convention \ref{conv:ideals}, define the ideal
\begin{equation}\label{eq:Idef}
I_{\brc}:= E_{\brc}\cdot \left\{\frac{f(\X,\Y)}{\Delta(\X_1)\cdots \Delta(\X_m)} \mid
f\in \k[\X,\Y] \text{ is antisymmetric for } \symg_{N} \right\}
\end{equation}
where $\Delta(\X_i)$ denotes the Vandermonde determinant in the alphabet $\X_i$.
\end{definition}

Set $\symg_\brc := \symg_{\bre_1}\times\cdots\times \symg_{\bre_m}$. 
The inclusion $\k[\X,\Y] \hookrightarrow \k[\X,\V]$ given by the 
assignment \eqref{eq:y to v 2} is $\symg_\brc$-equivariant, 
provided we let this group act trivially on the alphabet $\V$.
It follows that polynomials $f\in \k[\X,\Y]$ that are 
antisymmetric for the action of $\symg_N$ are 
antisymmetric for the action of $\symg_\brc$ 
when viewed as elements in $\k[\X,\V]$. 
Thus, such $f$ are divisible by the product of Vandermonde
determinants $\Delta(\X_1)\cdots \Delta(\X_m)$ and the resulting quotient 
is $\symg_\brc$-invariant.
Thus, the ideal $I_\brc$ is indeed well-defined.

Recall from \S\ref{ss:haiman dets} that any
set $S$ of monic monomials in $\k[x,y]$ with $|S|=N$ determines 
a Haiman determinant $\Delta_S(\X,\Y) \in \k[\X,\Y]$. 
The latter are antisymmetric for the diagonal $\symg_N$ action on $\k[\X,\Y]$, 
thus determine elements
\begin{equation}\label{eq:HaimanRatio}
\frac{\Delta_S(\X,\Y)}{\Delta(\X_1)\cdots\Delta(\X_m)} \in I_\brc\, .
\end{equation}
In \S\ref{s:hopf link} below, we show that when $m=2$ (i.e. when $\brc = (a,b)$), 
certain elements of the form \eqref{eq:HaimanRatio} generate the ideal $J_\brc$. 
Motivated by this, 
we propose the following as a $\brc$-colored analogue of Theorem \ref{thm:FT uncolored}.

\begin{conj}\label{conj:FT ideal} 
The $\brc$-colored full twist ideal $J_{\brc}\vartriangleleft E_{\brc}$ 
agrees with the ideal $I_\brc$ from Definition~\ref{def:ys and vs}.
\end{conj}

When $\brc=1^N$, this is simply a restatement of Theorem \ref{thm:FT uncolored}, 
since by definition $I_N = I_{1^N}$.
We will prove the $\brc = (a,b)$ case of Conjecture \ref{conj:FT ideal} below in Theorem \ref{thm:Jab} 
by pairing explicit computations of certain Haiman determinants (from \S \ref{ss:haiman dets})
with an elaborate inductive argument. In this case, Corollary \ref{cor:JHaiman}
identifies a specific set of $2^b$ Haiman determinants that generate $I_{a,b}$.

\subsection{Full twists and Hilbert schemes}\label{ss:connection to Hilbert schemes}

Pioneering work of Haiman \cite{Haiman} shows that the 
ideal
\[
I_N = \k[\X,\Y] \cdot \left\{ f(\X,\Y) \mid f\in \k[\X,\Y] \text{ is antisymmetric for } \symg_{N} \right\}
\]
appearing in Theorem \ref{thm:FT uncolored} plays a crucial role in the study of the 
Hilbert scheme $\Hilb_N(\C^2)$ of $N$ points in $\C^2$.
Recall that the \emph{isospectral Hilbert scheme} $X_N$ is defined to be the 
reduced fiber product
\[
\begin{tikzcd}
X_N \arrow[r] \arrow[d] & (\C^2)^N \arrow[d] \\
\Hilb_N(\C^2) \arrow[r] & S^N(\C^2)
\end{tikzcd}
\]
and that in \cite[Proposition 3.4.2]{Haiman} it is shown that $X_N$ is isomorphic to the 
blowup of $(\C^2)^N$ at the ideal $\C \otimes_\k I_N \vartriangleleft \C[\X,\Y]$.
Haiman goes on to to show that $X_N$ is normal, Cohen-Macaulay, and Gorenstein.
The following seemingly (but not) straightforward result plays a role in these considerations, 
and will be used below.

\begin{lem}[{\cite[Corollary 3.8.3]{Haiman}}]\label{lem:HaimanI}
$\displaystyle I_N = \bigcap_{1\leq i<j \leq N} \k[\X,\Y]\cdot\{x_i - x_j, y_i-y_j\}.$
\end{lem}

Inspired by the connections between uncolored, triply-graded link homology and 
the isospectral Hilbert scheme $X_N$, we propose the following as a ``colored'' 
analogue of the latter.

\begin{definition}\label{def:singular Hilb}
For $\brc=(\bre_1,\ldots,\bre_m)$, 
let $X_{\brc}$ be the blowup of $\Spec(E_\brc)$ at the ideal $I_\brc \vartriangleleft E_\brc$. 
In other words, set
\[
\cal{A}_{\brc} := \bigoplus_d I_{\brc}^d z^d \subset E_\brc [z]
\]
and let $X_{\brc}:=\Proj(\cal{A}_{\brc})$.
\end{definition}

We expect that the relation between Soergel bimodules and 
the (isospectral) Hilbert scheme in \cite{GNR,ObRozSBim,GH}
has an analogue for singular Soergel bimodules, with $X_\brc$ playing 
the role of ``colored'' isospectral Hilbert scheme.
In particular, let $\cal{B}_\brc$ be the colored ``full-twist dg algebra'',
\[
\cal{B}_\brc := \bigoplus_d \Hom_{\VS_\brc}(\oone_{\brc},\VFT_{\brc}^d)
\]
Post-composing with the splitting maps $\VFT_\brc^d\rightarrow \oone_{\brc}$
gives us a map of dg algebras $\cal{B}_\brc \rightarrow \cal{A}_{\brc}$. 

\begin{conj}
The map of dg algebras $\cal{B}_{\brc}\rightarrow A_{\brc}$ is a quasi-isomorphism. 
\end{conj}

We save investigations along these lines for future work.

\section{Homology of the colored Hopf link}
\label{s:hopf link}
In this section we prove Conjecture \ref{conj:FT is parity} and Conjecture
\ref{conj:FT ideal} for the (positive) Hopf link, using the explicit description of the
splitting map from \S\ref{ss:explicit splitting map}, as well as the
curved skein relation from \S\ref{s:curved skein rel}. After the general
setup from \S\ref{s:colored Hilb} we now return to the notational
conventions for the 2-strand case that were established in \S\ref{ss:two strand V cats}.

We start by showing that the colored Hopf link is parity, 
which implies that the link splitting map to the corresponding colored unlink homology is injective. 
In \S\ref{sec:cornermaps} we compute generators for the Hopf link homology 
in lowest $\adeg$-degree and in \S\ref{ss:haimandethopf} we show 
that their images under the splitting map are (reduced) Haiman determinants. 
This implies the inclusion $J_{a,b}\subset I_{a,b}$ between the ideals 
from Definitions \ref{def:J} and \ref{def:ys and vs}, in the $2$-strand case.
We will refer to the ideal $J_{a,b}$ as the \emph{Hopf link ideal}, 
since it is isomorphic to the (lowest Hochschild summand of the) deformed colored homology of 
the $(a,b)$-colored Hopf link.
In Sections \ref{sec:families} and \ref{sec:induction}, we establish the opposite inclusion, 
thus proving Conjecture \ref{conj:FT ideal} when $\brc = (a,b)$.
The argument proceeds via induction using a family of ideals $J_{a,(b-s,s)}$
that are related to the curved colored skein relation.

Since the positive $(a,b)$-colored Hopf link is isotopic to the $(b,a)$-colored Hopf link, 
without loss of generality we assume that $a \geq b$ for the duration.

\subsection{Parity of the colored Hopf link}
\label{sec:Hopf-parity}

Here, we compute the (undeformed) Khovanov--Rozansky homology of the $(a,b)$-colored Hopf link
in lowest $\adeg$-degree and show that it is supported in cohomological degrees
congruent to $ab$ modulo 2. 
In other words, we show that the colored Hopf link is \emph{parity}. 
The argument is analogous to the one given in \cite[Example 4.14]{Wed1}. 
By Definition \ref{def:Rickardcx}, we have
\[
C_{b,a}\hComp C_{a,b} = 
\tw_{\d\hComp \id}\left( \bigoplus_{l=0}^b \qdeg^{-l} \tdeg^{l} \F^{(b-l)}\E^{(a-l)} \hComp C_{a,b} \right).
\]
We begin with the following computation.

\begin{lem}\label{lem:Hopfrow}
There is a homotopy equivalence
\[
\HH_\bullet( \F^{(b-l)}\E^{(a-l)} \hComp C_{a,b} ) 
\simeq
\adeg^{-l} \qdeg^{ab} \tdeg^{l} 
\HH_\bullet \left(
\begin{tikzpicture}[scale=.5,smallnodes,anchorbase]
	\draw[very thick] (-2,0) node[left]{$a{+}b{-}l$} to (-1.25,0);
	\draw[very thick] (-1.25,0) to [out=300,in=240] node[below]{$a{-}l$} (1.25,0);
	\draw[very thick] (-1.25,0) to [out=60,in=180] (-.5,.5);
	\draw[very thick] (-.5,.5) to [out=300,in=240] node[below]{$l$} (.5,.5);
	\draw[very thick] (-.5,.5) to [out=60,in=120] node[above,yshift=-2pt]{$b{-}l$} (.5,.5);
	\draw[very thick] (1.25,0) to [out=120,in=0] (.5,.5);
	\draw[very thick] (1.25,0) to (2,0) node[right]{$a{+}b{-}l$};
\end{tikzpicture}
\right)
\]
of complexes of bigraded vector spaces.
\end{lem}

\begin{proof}
We compute 
\[
\begin{aligned}
\HH_\bullet( \F^{(b-l)}\E^{(a-l)} \hComp C_{a,b} ) 
&= \HH_\bullet\left(
\left\llbracket
	\begin{tikzpicture}[scale=.5,smallnodes,anchorbase]
		\draw[very thick] (-2,0) node[left]{$a$} to (0,0) \pr (2,1) node[right]{$b$};
		\draw[line width=5pt,color=white]  (0,1) \pr (2,0);
		\draw[very thick] (-2,1) node[left]{$b$} to node[above,yshift=-2pt]{$l$}  (0,1)  \pr (2,0) node[right]{$a$};
		\draw[very thick] (-1.75,1) to (-1.5,.5)node[right,yshift=-1pt,xshift=-1pt]{} to  (-1.25,0);
		\draw[very thick] (-.75,0) to (-.5,.5) to  (-.25,1);
	\end{tikzpicture}
\right\rrbracket	
\right)
\cong 
\HH_\bullet\left(
\left\llbracket
	\begin{tikzpicture}[scale=.5,smallnodes,anchorbase]
		\draw[very thick] (-1,0) node[left]{$a{+}b{-}l$} to (0,0) \pr (2,1) to (3,1) node[right]{$l$};
		\draw[line width=5pt,color=white] (0,1) \pr (2,0);
		\draw[very thick] (-1,1) node[left]{$l$} to (0,1) \pr (2,0)  to (3,0)node[right]{$a{+}b{-}l$};
		\draw[very thick] (2.25,1) to (2.5,.5)node[right,yshift=-1pt,xshift=-1pt]{$b{-}l$} to  (2.75,0);
		\draw[very thick] (-.75,0) to (-.5,.5) to  (-.25,1);
	\end{tikzpicture}
\right\rrbracket	
\right) \\
&\simeq 
\HH_\bullet\left(
\left\llbracket
	\begin{tikzpicture}[scale=.5,smallnodes,anchorbase]
		\draw[very thick] (-1,0) node[left]{$a{+}b{-}l$} to (-.25,0) \pr (1.75,1) to (3,1) node[right]{$l$};
		\draw[very thick] (-.25,0) to (.125,0) \pr (2.125,.75) 
			to [out=0,in=120] (2.5,.5)node[right,yshift=-1pt,xshift=-1pt]{$b{-}l$} to (2.75,0);
		\draw[line width=3pt,color=white]  (-0.125,.75) \pr (1.875,0) (0.25,1) \pr (2.25,0);
		\draw[very thick] (-1,1) node[left]{$l$} to (0.25,1) \pr (2.25,0)  to (3,0)node[right]{$a{+}b{-}l$};
		\draw[very thick] (-.75,0) to [out=60,in=180]  (-0.125,.75) \pr (1.875,0) to (2.25,0);
	\end{tikzpicture}
\right\rrbracket	
\right) \\
&\simeq
\adeg^{-l}  \qdeg^{l^2} \tdeg^{l} 
\HH_\bullet\left(
\left\llbracket
	\begin{tikzpicture}[scale=.5,smallnodes,anchorbase]
		\draw[very thick] (-1,0) node[left]{$a{+}b{-}l$} to (-.25,0) \pr (1,.75);
		\draw[very thick] (-.25,0) to (.125,0) \pr (2.125,.75) 
			to [out=0,in=120] (2.5,.5) node[right,yshift=1pt,xshift=-2pt]{$b{-}l$} to (2.75,0);
		\draw[line width=3pt,color=white]  (-0.125,.75) \pr (1.875,0) (1,.75) \pr (2.25,0);
		\draw[very thick] (1,.75) \pr (2.25,0)  to (3,0)node[right]{$a{+}b{-}l$};
		\draw[very thick] (-.75,0) to [out=60,in=180]  (-0.125,.75) \pr (1.875,0) to (2.25,0);
		\node at (1,1) {$l$};
	\end{tikzpicture}
\right\rrbracket	
\right) \\
&\simeq
\adeg^{-l} \qdeg^{l^2 + b(a-l) + l(b-l)} \tdeg^{l}  
\HH_\bullet\left(
\left\llbracket
	\begin{tikzpicture}[scale=.5,smallnodes,anchorbase]
		\draw[very thick] (-2,0) node[left]{$a{+}b{-}l$} to (-1.25,0);
		\draw[very thick] (-1.25,0) to [out=300,in=240] node[below]{$a{-}l$} (1.25,0);
		\draw[very thick] (-1.25,0) to [out=60,in=180] (-.5,.5);
		\draw[very thick] (-.5,.5) to [out=300,in=240] node[below]{$l$} (.5,.5);
		\draw[very thick] (-.5,.5) to [out=60,in=120] node[above,yshift=-2pt]{$b{-}l$} (.5,.5);
		\draw[very thick] (1.25,0) to [out=120,in=0] (.5,.5);
		\draw[very thick] (1.25,0) to (2,0) node[right]{$a{+}b{-}l$};
	\end{tikzpicture}
\right\rrbracket	
\right)
\end{aligned}
\]
where we have used 
Propositions~\ref{prop:forkslide} and \ref{prop:HH tracelike 2}
and Lemma~\ref{lem:Markov2}.
\end{proof}

We also record the following useful result, which has a completely analogous proof. 
\begin{prop}
\label{prop:negstab}
The complex of bigraded vector spaces $\HH^\bullet( C^\vee_{b,a}\hComp
\F^{(a-b+l)}\E^{(l)})$ has homology supported in strictly positive
$\adeg$-degrees when $0 \leq l \leq b-1$.
\end{prop}

\begin{cor} \label{cor:negstab}
For $0 \leq l \leq b-1$, there is a homotopy equivalence of dg $\End_{\CS_{a,b}}(\oone_{a,b})$-modules
$\Hom_{\CS_{a,b}}(\oone_{a,b}\F^{(l)}\E^{(a-b+l)} \hComp C_{a,b},\oone_{a,b}) \simeq 0$.
\end{cor}

\begin{proof}
The complex $\Hom_{\CS_{a,b}}(\oone_{a,b}\F^{(l)}\E^{(a-b+l)} \hComp C_{a,b},\oone_{a,b})$ 
inherits an action of $\End_{\CS_{a,b}}(\oone_{a,b})$ via its central action on the left in $\CS_{a,b}$.
We then have:
	\begin{align*}
		\Hom_{\CS_{a,b}}(\oone_{a,b}\F^{(l)}\E^{(a-b+l)} \hComp
	C_{a,b},\oone_{a,b})
	& \cong
	\Hom_{\CS_{a,b}}(\oone_{a,b}\F^{(l)}\E^{(a-b+l)}
	,C^\vee_{b,a})\\
	& \cong
	\Hom_{\CS_{a,b}}(\oone_{a,b}
	,C_{b,a}^\vee\hComp \F^{(a-b+l)}\E^{(l)})\\
	&= \HH^0(C^\vee_{b,a}\hComp \F^{(a-b+l)}\E^{(l)}) \simeq 0
	\end{align*}
where we have used duality, Proposition~\ref{prop:dualityonMS}, and Proposition~\ref{prop:negstab}. 
\end{proof}

The following is now an immediate consequence of Lemma~\ref{lem:Hopfrow}
together with Proposition \ref{prop:HPT}.

\begin{proposition} \label{prop:Hopfparityandgens}
The $(a,b)$-colored Hopf link is parity.
\end{proposition}

\begin{proof}
Up to a global shift of magnitude $\adeg^{ab} \qdeg^{-2ab} \tdeg^{-ab}$ (see
Definition~\ref{def:HHH}), the complex $C_{\KR}(\widehat{\FT_{a,b}})$ associated
to the closure of the $(a,b)$-colored full twist is given by
\[
\HH_\bullet( C_{b,a}\hComp C_{a,b} ) = 
\tw_{\a}
\left( \bigoplus_{l=0}^b  \qdeg^{-l}\tdeg^l \HH_\bullet(\F^{(b-l)}\E^{(a-l)} \hComp C_{a,b}) \right)
\]
for some twist $\a$ which strictly increases the index $l$.
Here, the right-hand side is a finite one-sided twisted complex (see Remark \ref{rem:onesided}), 
hence Proposition \ref{prop:HPT} (the homological perturbation lemma) 
allows us to apply the homotopy equivalences from Lemma~\ref{lem:Hopfrow} term-wise.
We thus obtain
\[
\HH_\bullet( C_{b,a}\hComp C_{a,b} ) 
\simeq \tw_{\b} \left(\bigoplus_{l=0}^b \adeg^{-l} \qdeg^{ab - l} \tdeg^{2l} 
\HH_\bullet \left(
\begin{tikzpicture}[scale=.5,smallnodes,anchorbase]
	\draw[very thick] (-2,0) node[left]{$a{+}b{-}l$} to (-1.25,0);
	\draw[very thick] (-1.25,0) to [out=300,in=240] node[below]{$a{-}l$} (1.25,0);
	\draw[very thick] (-1.25,0) to [out=60,in=180] (-.5,.5);
	\draw[very thick] (-.5,.5) to [out=300,in=240] node[below]{$l$} (.5,.5);
	\draw[very thick] (-.5,.5) to [out=60,in=120] node[above,yshift=-2pt]{$b{-}l$} (.5,.5);
	\draw[very thick] (1.25,0) to [out=120,in=0] (.5,.5);
	\draw[very thick] (1.25,0) to (2,0) node[right]{$a{+}b{-}l$};
\end{tikzpicture}
\right)
\right)
\]
for some Maurer--Cartan element $\b$. 
Since the Maurer--Cartan element $\b$ has $\tdeg$-degree one, it must be zero
since it is acting on a complex which is supported in even cohomological degrees.
\end{proof}

For the reader's convenience, we include also the Hochschild cohomology version of the above 
(in particular note that the shift $\adeg^{-l}$ disappears):
\[
\HH^\bullet(C_{b,a}\hComp C_{a,b}) \ \simeq \ \bigoplus_{l\geq 0}
\qdeg^{2(a-l)(b-l)-2l} \tdeg^{2l} \left(\qdeg^{ab-l^2} \HH^\bullet \left(
\begin{tikzpicture}[scale=.5,smallnodes,anchorbase]
	\draw[very thick] (-2,0) node[left]{$a{+}b{-}l$} to (-1.25,0);
	\draw[very thick] (-1.25,0) to [out=300,in=240] node[below]{$a{-}l$} (1.25,0);
	\draw[very thick] (-1.25,0) to [out=60,in=180] (-.5,.5);
	\draw[very thick] (-.5,.5) to [out=300,in=240] node[below]{$l$} (.5,.5);
	\draw[very thick] (-.5,.5) to [out=60,in=120] node[above,yshift=-2pt]{$b{-}l$} (.5,.5);
	\draw[very thick] (1.25,0) to [out=120,in=0] (.5,.5);
	\draw[very thick] (1.25,0) to (2,0) node[right]{$a{+}b{-}l$};
\end{tikzpicture}	
\right) \right) 
\]

\begin{cor}\label{cor:FT-split-inj}
Let $T(2,2;a,b)$ be the $(a,b)$-colored Hopf link, and let $U(a,b)$ be the $(a,b)$-colored unlink.  
The map 
$\YS H_{\KR}(\Sigma_{a,b}) \colon \YS H_{\KR}(T(2,2;a,b) \to \YS H_{\KR}(U(a,b))$ is injective. 
In particular, the map
$H(\Sigma_{a,b}) \colon H(M_{a,b}) \rightarrow E_{a,b}$ is injective.
\end{cor}
\begin{proof}
This is an immediate consequence of Corollary \ref{cor:M is J general}.
\end{proof}

Below, we will compute the image of the map 
$H(\Sigma_{a,b}) \colon H(M_{a,b}) \rightarrow E_{a,b}$.
Recall from \S \ref{ss:FT setup} that 
\[
E_{a,b} \defeq \End_{\VS_{a,b}}(\oone_{a,b}) \, , \quad
M_{a,b}\defeq \Hom_{\VS_{a,b}}(\oone_{a,b},\VFTmin_{a,b})\, ,
\]
and that the homology of the latter is isomorphic to the lowest Hochschild
degree summand of $\YS H_{\KR}(T(2,2;a,b)$ (up to shift). 

\subsection{Corner maps}\label{sec:cornermaps}

It will be useful to have an explicit set of generators for the homology of 
$M_{a,b} \defeq \Hom_{\VS_{a,b}}(\oone_{a,b},\VFTmin_{a,b})$.
Recall from Definition \ref{def:curved MCCS} that
\[
\VFTmin_{a,b} = \tw_{\d^h+\d^v+\Delta^v} \left( \bigoplus_{b\geq k\geq r\geq 0} P_{k,r,0} \right)
\]
where 
$P_{k,r,0}:= \qdeg^{k(a-b+1)-2b} \tdeg^{2b-k}  W_k \otimes \largewedge^r[\xi_1,\ldots,\xi_k]$.
Motivated (visually) by \eqref{eq:splitting map diagram},
we will refer to the summands $P_{k,k,0} \subset \VFTmin_{a,b}$ as \emph{corners}.
Reindexing by taking $k=b-l$, we have that
\begin{align*}
P_{b-l,b-l,0}  &= \qdeg^{(b-l)(a-b+1)-2b} \tdeg^{b+l} W_k\otimes \xi_1\cdots \xi_{b-l} \\
&\cong \qdeg^{(b-l)(a-b+1)-2b} \tdeg^{b+l} \qdeg^{(b-l)(b-l+1)} \tdeg^{-b+l} W_{b-l}
= \qdeg^{(b-l)(a-l)-2l} \tdeg^{2l}  W_{b-l}.
\end{align*}

\begin{definition}\label{def:Psi and corner} 
For $0\leq l\leq b$ let 
$\PSI_l \colon  \Hom_{\SSBim}(\oone_{a,b},W_{b-l}) \to 
\Hom_{\VS_{a,b}}(\oone_{a,b},\VFTmin_{a,b})$ 
be the map sending 
$\phi\colon \oone_{a,b}\rightarrow W_{b-l}$ 
to the composite
\[
\oone_{a,b} \xrightarrow{\phi} W_{b-l} \xrightarrow{\approx} \qdeg^{(b-l)(a-l)-2l} \tdeg^{2l}  
W_{b-l} \xrightarrow{\cong}  P_{b-l,b-l,0} \hookrightarrow \VFTmin_{a,b}\, ,
\]
where (as before) the map denoted $\approx$ is a slanted identity of $W_{b-l}$. 
\end{definition}

We can regard $\PSI_l$ as a degree-zero map 
\[
\qdeg^{(b-l)(a-l)-2l} \tdeg^{2l} \Hom_{\SSBim}(\oone_{a,b},W_{b-l})\rightarrow
\Hom_{\VS_{a,b}}(\oone_{a,b},\VFTmin_{a,b})\, .
\]

\begin{lemma}\label{lemma:including corners} 
If we regard
$\Hom_{\SSBim}(\oone_{a,b},W_{b-l})$ as a complex with zero differential, 
then the morphism $\PSI_l$ from Definition \ref{def:Psi and corner} is closed.
Furthermore, the induced map of complexes
\[
\bigoplus_{l=0}^b \qdeg^{(b-l)(a-l)-2l} \tdeg^{2l} \Hom_{\SSBim}(\oone_{a,b},W_{b-l})\otimes\k[\V] 
\xrightarrow{\bigoplus_l \PSI_l} \Hom_{\VS_{a,b}}(\oone_{a,b},\VFTmin_{a,b})
\]
is surjective in homology.
\end{lemma}
\begin{proof}
The differential of $\PSI_l$ sends $\phi\in \Hom_{\SSBim}(\oone_{a,b},W_{b-l})$ to
$(\d^h+\d^v+\Delta^v)\circ \phi$; here we identify
$P_{b-l,b-l,0}$ with $W_{b-l}$ (up to a shift). 
Each of the differentials
$\Delta^v$ and $\d^h$ restrict to zero at the corners, 
 thus to see that $\PSI_l$ is closed it suffices to show that 
$\d^v\circ \phi=0$ for all $\phi\in \Hom(\oone_{a,b},W_{b-l})$.
This follows since $\d^v$ is the Koszul differential associated to the action of 
$h_i(\X_2-\X_2')$ on $W_{b-l}$ for $1\leq i\leq b-l$, 
and the central elements $h_i(\X_2-\X_2')$ act by zero 
on $\Hom_{\SSBim}(\oone_{a,b},W_{b-l})$.
This proves the first statement.

For the second statement, Corollary \ref{cor:M is J general} 
implies that post-composing with the splitting map 
\[
\Sigma_{a,b} \colon \VFT_{a,b} \rightarrow \oone_{a,b}
\]
gives a map 
$
M_{a,b} \defeq \Hom_{\VS_{a,b}}(\oone_{a,b},\VFT_{a,b}) 
\to \End_{\VS_{a,b}}(\oone_{a,b}) \defeq E_{a,b}
$
which is injective in homology.
Thus $H(\Sigma_{a,b}) \colon H(M_{a,b}) \to E_{a,b}$ 
is an isomorphism onto its image $J_{a,b}$.
By Definition \ref{def:splitting map model}, the splitting map $\Sigma_{a,b}$ 
is supported on the corners, i.e. $\Sigma|_{P_{k,r,0}} =0$ unless $k=r$.
Thus, $J_{a,b}$ is spanned by elements of the form $\Sigma_{a,b} \circ \PSI_l(\phi)$,
and the result follows from injectivity of $H(\Sigma_{a,b})$.
\end{proof}

Fix $0 \leq l \leq b$.
We now explicitly describe the $\Hom$-spaces $\Hom_{\SSBim}(\oone_{a,b}, W_{b-l})$. 
\begin{definition}\label{def:corner maps}
Consider the following elements of $\Hom_{\SSBim}(\oone_{a,b}, W_{b-l})$,
\[
\phi_{a,b,l}(\lambda) := 
\begin{tikzpicture}[anchorbase, smallnodes]
		\draw[FS,ultra thick,->] (0,-.5) node[below]{$a$}  to (0,.7)node[above=-1pt]{$a{+}b{-}l$};
		\draw[FS,ultra thick,->] (.75,-.5) node[below]{$b$} to (.75,.7)node[above=-1pt]{$l$};
		\draw[FS,ultra thick, directed=.75] (.75,-.25) to [out=90,in=270] node[below]{$b{-}l$} (0,.25) ;
		\node at (.75,.25) {$\bullet$};
		\node at (1,.25) {$\Schur_{\lambda}$};
	\end{tikzpicture}
\colon
\begin{tikzpicture}[rotate=90,anchorbase,smallnodes]
		\draw[very thick] (.4,-1) to  node[above,xshift=2pt]{$b$}  (.4,1);
		\draw[very thick] (-.3,-1) to node[below,yshift=-3pt]{$a$} (-.3,1);
		\draw[densely dotted, FS] (-.5,0) to (.6,0);
\end{tikzpicture}
\longrightarrow
\begin{tikzpicture}[rotate=90,anchorbase,smallnodes]
		\draw[very thick] (0,.25) to [out=150,in=270] (-.25,1) 
			node[left,xshift=2pt]{$a$};
		\draw[very thick] (.5,.5) to (.5,1) node[left,xshift=2pt]{$b$};
		\draw[very thick] (0,.25) to  (.5,.5);
		\draw[very thick] (0,-.25) to node[below,yshift=-3pt]{$a{+}b{-}l$} (0,.25);
		\draw[very thick] (.5,-.5) to [out=30,in=330] 
			node[above,yshift=-2pt]{$l$} (.5,.5);
		\draw[very thick] (0,-.25) to  (.5,-.5);
		\draw[very thick] (.5,-1) node[right,xshift=-2pt]{$b$} to (.5,-.5);
		\draw[very thick] (-.25,-1)node[right,xshift=-2pt]{$a$} 
			to [out=90,in=210] (0,-.25);
		\draw[densely dotted, FS] (-.4,0) to (1.1,0);
\end{tikzpicture}
\]
where $\lambda \in P(l,b-l)$ is a partition in the $l \times b-l$ rectangle.
\end{definition}

Here, we have written $\phi_{a,b,l}(\lambda)$ using the perpendicular graphical
calculus from \S\ref{ss:ssbim}.  In particular, taking $\lambda=\emptyset$
gives us the canonical map 
$\phi_{a,b,l}(\emptyset)  \colon \oone_{a,b} \to W_{b-l}$ of weight $\qdeg^{(a-l)(b-l)}$ 
given by $\cre$ (digon creation on the $b$-labeled strand) followed by $\zip$.

We now show that these maps span $\Hom_{\SSBim}(\oone_{a,b}, W_{b-l})$.
Recall the alphabet labeling conventions from Convention \ref{conv:I alphabets}, 
and observe that
\[
W_{b-l} = I^{(l)}({}_{a,b-l}S_{a+b-l}M_{a,b-l}) =
\begin{tikzpicture}[rotate=90,anchorbase,smallnodes]
		\draw[very thick] (0,.25) to [out=150,in=270] (-.25,1) 
			node[left,xshift=2pt]{$\leftX_1$};
		\draw[very thick] (.5,.5) to (.5,1) node[left,xshift=2pt]{$\leftX_2$};
		\draw[very thick] (0,.25) to node[left=-1pt,yshift=-1pt]{$\leftL$}  (.5,.5);
		\draw[very thick] (0,-.25) to (0,.25);
		\draw[very thick] (.5,-.5) to [out=30,in=330] 
			node[above,yshift=-2pt]{$\B$} (.5,.5);
		\draw[very thick] (0,-.25) to node[right=-1pt,yshift=-1pt]{$\rightL$} (.5,-.5);
		\draw[very thick] (.5,-1) node[right,xshift=-2pt]{$\rightX_2$} to (.5,-.5);
		\draw[very thick] (-.25,-1)node[right,xshift=-2pt]{$\rightX_1$} 
			to [out=90,in=210] (0,-.25);
\end{tikzpicture}
\]
This gives an action of $\Sym(\leftX_1 | \leftL | \B)$ on $W_{b-l}$, 
and hence on $\Hom_{\SSBim}(\oone_{a,b},W_{b-l})$ by post-composition.
In our current situation, $|\B| = l$ and $|\L| = b-l$.

For use here and below, we record the following.

\begin{lem}\label{lem:easy 1 to I(X)}
Let $X$ be a complex in $\CS_{a,b-l}$, then there is an isomorphism
\[
\Hom_{\CS_{a,b}}(\oone_{a,b},I^{(l)}(X)) \cong 
\qdeg^{-l(b-l)} \Hom_{\CS_{a,b-l}}(\oone_{a,b-l} , X) \otimes \Sym(\B)
\]
of dg $\Sym(\leftX_1 | \leftL | \B)$-modules that is natural in $X$.
\end{lem}

\begin{proof}
Proposition \ref{prop:dualityonMS} gives that
\[
\begin{aligned}
\Hom_{\CS_{a,b}}(\oone_{a,b},\I^{(l)}(X))
&= \Hom_{\CS_{a,b}}\big(\oone_{a,b} , 
	(\oone_a \boxtimes {}_{b}M_{b-l,l}) \hComp (X \boxtimes \oone_l) \hComp (\oone_a \boxtimes {}_{b-l,l}S_{b}) \big) \\
&\cong \qdeg^{-2l(b-l)} \Hom_{\CS_{a,b}} \big(\oone_{a,b-l,l} , 
	(\oone_a \boxtimes ({}_{b-l,l}S_{b} \hComp {}_{b}M_{b-l,l})) \hComp (X \boxtimes \oone_l) \big)\, .
\end{aligned}
\]
Corollary \ref{cor:I(X)Hom} then gives
\[
\Hom_{\CS_{a,b}}(\oone_{a,b},\I^{(l)}(X))
\cong \qdeg^{-l(b-l)} \Hom_{\CS_{a,b-l}}(\oone_{a,b-l} , X) \otimes \Sym(\B)\, .
\]
The result follows since each isomorphism is $\Sym(\leftX_1 | \leftL | \B)$-linear 
and natural in $X$.
\end{proof}

\begin{cor}\label{cor:homs to W}
There is an isomorphism 
\[
\qdeg^{(a-l)(b-l)} \Sym(\leftX_1 | \leftL | \B) \cong \Hom_{\SSBim}(\oone_{a,b},W_{b-l})
\]
of $\Sym(\leftX_1 | \leftL | \B)$-modules sending $1 \to \phi_{a,b,l}(\emptyset)$.
Consequently, $\Hom_{\SSBim}(\oone_{a,b},W_{b-l})$ is a free
$\Sym(\X_1|\X_2)$-module with basis given by the
morphisms $\phi_{a,b,l}(\lambda)$ from Definition \ref{def:corner maps}.
\end{cor}
\begin{proof}
Lemma \ref{lem:easy 1 to I(X)} and Corollary \ref{cor:basicHom} give that
\[
\begin{aligned}
\Hom_{\SSBim}(\oone_{a,b},W_{b-l}) &\cong 
\qdeg^{-l(b-l)} \Hom_{\CS_{a,b-l}}(\oone_{a,b-l} , {}_{a,b-l}S_{a+b-l}M_{a,b-l}) \otimes \Sym(\B) \\
& \cong \qdeg^{(a-l)(b-l)} \Sym(\leftX_1 | \leftL) \otimes \Sym(\B)
\end{aligned}
\]
and show that under this isomorphism $\phi_{a,b,l}(\emptyset)$ is sent to $1$.
The second statement then follows from Example~\ref{exa:Sylvester}.
\end{proof}

Lemma \ref{lemma:including corners} and Corollary \ref{cor:homs to W} immediately 
imply the following.

\begin{prop}\label{prop:corners span hopf} As a module over $E_{a,b}$, the
homology of $M_{a,b} = \Hom_{\VS_{a,b}}(\oone_{a,b},\VFTmin_{a,b})$ is spanned by the 
classes of the elements $\PSI(\phi_{a,b,l}(\lambda))$, 
where $\PSI := \bigoplus_L \PSI_l$ 
is the inclusion from Lemma \ref{lemma:including corners}.\qed
\end{prop}

Recall the reduction functor $\pi\colon \VS_{a,b}\rightarrow \VS_{a,b}$ from Definition~\ref{def:reduction}. 
This induces an algebra endomorphism of $E_{a,b}=\End_{\VS_{a,b}}(\oone_{a,b})$ 
that we denote by the same symbol. By Definition~\ref{def:reduction}, it maps 
\[
v_{L,i}^{(a)}\mapsto 0 \, , \quad v_{R,i}^{(b)}\mapsto v_{R,i}^{(b)} - v_{L,i}^{(b)}
\] 
since the actions of $\leftX_2$ and $\rightX_2$ appearing in \eqref{eq:vred}
are identified on $E_{a,b}$. 
We will also use $\pi$ to refer to the $\k[\X_1,\X_2]$-linear endomorphism of
$\k[\X_1,\X_2,\V_L^{(a)},\V_R^{(b)}]$ determined by the same formulas. This map sends:
\[
y_i\mapsto 
\begin{cases}
	0 & 1\leq i\leq a \\
	\yred_i & a< i \leq a+b
\end{cases}
\]
when we express $y_i$ as in \eqref{eq:def-yi-2strand} and $\yred_i$ as in \eqref{eq:ybar and vbar}.

\begin{lemma}\label{lemma:Jab reduced}
We have $J_{a,b} = \langle \pi(J_{a,b})\rangle$ as ideals in $E_{a,b}$.
\end{lemma}
\begin{proof}
We claim that $J_{a,b}$ admits a set of generators, which is preserved by $\pi$.
Indeed, as defined, $\VFTmin_{a,b}$ and the splitting map 
$\Sigma_{a,b} \colon \VFTmin_{a,b}\rightarrow \oone_{a,b}$ are both reduced, hence fixed by $\pi$.
Moreover, the elements $\PSI(\phi_{a,b,l}(\lambda))$ from Proposition \ref{prop:corners span hopf} 
are fixed by $\pi$ since the deformation parameters $\V$ do not enter into their definition.  
It follows that $\Sigma_{a,b} \circ \PSI(\phi_{a,b,l}(\lambda)) \in E_{a,b}$ is fixed by $\pi$.  By
Proposition~\ref{prop:corners span hopf}, these elements generate $J_{a,b}$.
\end{proof}

\subsection{Haiman determinants and the Hopf link}
\label{ss:haimandethopf}

We now compute the generators 
$\Sigma_{a,b} \circ \PSI(\phi_{a,b,l}(\lambda))\in J_{a,b}$ explicitly. 
We will see that these elements are special cases of the 
Haiman determinants from \S\ref{ss:haiman dets}.

\begin{definition}\label{def:keyshape}
For $a\geq b\geq l\geq 0$ and a partition $\l\in P(l,b-l)$,
the associated \emph{key shape} $\Key_l(\l)$ is defined to be
the set of monic monomials in $\k[x,y]$:
\[
\Key_l(\l):=\{x^{a+b-l-1},\ldots,x,1\}\cup\{x^{\l_1+l-1}y,\ldots,x^{\l_l}y\}
\]
ordered as indicated.
\end{definition}

\begin{conv}\label{conv:key}
We will identify finite sets $S$ of monic monomials in $\k[x,y]$ with
finite subsets of $\Z_{\geq 0}\times \Z_{\geq 0}$, 
illustrated by collections of boxes living in the $4^{th}$ quadrant.
In general, we will call such collections of boxes \emph{shapes}. 

For example, for $a=5$, $b=3$ and $l=2$ there are three key shapes:
\[ 
\begin{tikzpicture}[scale=.7,smallnodes,anchorbase]
	\draw (0,0) rectangle (1,1);
	\node at (.5,.5) {$1$};
	\draw (1,0) rectangle (2,1);
	\node at (1.5,.5) {$x$};
	\draw (2,0) rectangle (3,1);
	\node at (2.5,.5) {$x^2$};
	\draw (3,0) rectangle (4,1);
	\node at (3.5,.5) {$x^3$};
	\draw (4,0) rectangle (5,1);
	\node at (4.5,.5) {$x^4$};
	\draw (5,0) rectangle (6,1);
	\node at (5.5,.5) {$x^5$};
	\draw (0,-1) rectangle (1,0);
	\node at (.5,-.5) {$y$};
	\draw (1,-1) rectangle (2,0);
	\node at (1.5,-.5) {$xy$};
	\node at (0,-1.4) {$\l_2$};
	\node at (1,-1.4) {$\l_1{+}1$};
	\draw[dotted] (3,0) to (3, -1);
	\node at (3,-1.4) {$b$};
	\draw[dotted] (5,0) to (5, -1);
	\node at (5,-1.4) {$a$};
	\draw[dotted] (6,0) to (6, -1);
	\node at (6,-1.4) {$a{+}b{-}l$};
\end{tikzpicture}	
,\quad
\begin{tikzpicture}[scale=.7,smallnodes,anchorbase]
	\draw (0,0) rectangle (1,1);
	\node at (.5,.5) {$1$};
	\draw (1,0) rectangle (2,1);
	\node at (1.5,.5) {$x$};
	\draw (2,0) rectangle (3,1);
	\node at (2.5,.5) {$x^2$};
	\draw (3,0) rectangle (4,1);
	\node at (3.5,.5) {$x^3$};
	\draw (4,0) rectangle (5,1);
	\node at (4.5,.5) {$x^4$};
	\draw (5,0) rectangle (6,1);
	\node at (5.5,.5) {$x^5$};
	\draw (0,-1) rectangle (1,0);
	\node at (.5,-.5) {$y$};
	\draw (2,-1) rectangle (3,0);
	\node at (2.5,-.5) {$x^2y$};
	\node at (0,-1.4) {$\l_2$};
	\node at (2,-1.4) {$\l_1{+}1$};
	\draw[dotted] (5,0) to (5, -1);
	\node at (5,-1.4) {$a$};
	\draw[dotted] (6,0) to (6, -1);
	\node at (6,-1.4) {$a{+}b{-}l$};
\end{tikzpicture}
,\quad
\begin{tikzpicture}[scale=.7,smallnodes,anchorbase]
	\draw (0,0) rectangle (1,1);
	\node at (.5,.5) {$1$};
	\draw (1,0) rectangle (2,1);
	\node at (1.5,.5) {$x$};
	\draw (2,0) rectangle (3,1);
	\node at (2.5,.5) {$x^2$};
	\draw (3,0) rectangle (4,1);
	\node at (3.5,.5) {$x^3$};
	\draw (4,0) rectangle (5,1);
	\node at (4.5,.5) {$x^4$};
	\draw (5,0) rectangle (6,1);
	\node at (5.5,.5) {$x^5$};
	\draw (1,-1) rectangle (2,0);
	\node at (1.5,-.5) {$xy$};
	\draw (2,-1) rectangle (3,0);
	\node at (2.5,-.5) {$x^2y$};
	\node at (1,-1.4) {$\l_2$};
	\node at (2,-1.4) {$\l_1{+}1$};
	\draw[dotted] (5,0) to (5, -1);
	\node at (5,-1.4) {$a$};
	\draw[dotted] (6,0) to (6, -1);
	\node at (6,-1.4) {$a{+}b{-}l$};
\end{tikzpicture}		
\]
which are associated with the partitions $\emptyset$, $(1)$ and $(1,1)$ inside the $2\times 1$ box, 
respectively.

In a similar way, every partition $\l$ specifies a finite set of monic
monomials via the coordinates of the boxes in the Young diagram for $\l$.
We caution the reader that this set of monomials na\"{i}vely associated to a partition $\l$ 
is unrelated to both the key shape $\Key_l(\l)$ and the
set of monomials $\mathcal{M}_N(\l)$ from Definition~\ref{def:monomial list}.
\end{conv}

\begin{example}\label{ex:keys} The Haiman determinant associated to a key shape takes the form
\begin{equation}\label{eq:key general}
\Delta_{\Key_l(\l)} =
\begin{vmatrix}
x^{a+b-l-1}_1 &  \cdots & x^{a+b-l-1}_a & x^{a+b-l-1}_{a+1} &\cdots & x^{a+b-l-1}_{a+b} \\ 
\vdots & \ddots & \vdots & \vdots  & \ddots & \vdots\\
1 &  \cdots & 1 & 1 &\cdots & 1 \\
x^{\l_1+l-1}_1 y_1 &  \cdots & x^{\l_1+l-1}_a y_a & x^{\l_1+l-1}_{a+1} y_{a+1}&\cdots & x^{\l_1+l-1}_{a+b} y_{a+b}\\
\vdots & \ddots & \vdots & \vdots  & \ddots & \vdots\\
x^{\l_l}_1 y_1 &  \cdots & x^{\l_l}_a y_a & x^{\l_l}_{a+1} y_{a+1}&\cdots & x^{\l_l}_{a+b} y_{a+b} 
\end{vmatrix}
\end{equation}
\end{example}

\begin{lemma}[Key lemma] \label{lem:keylemma}
For $a \geq b\geq l \geq 0$ and a partition $\lambda \in P(l,b-l)$, 
we have the following identity in $E_{a,b}$:
\begin{equation} \label{eq:keylemma}
	\Sigma_{a,b} \circ \PSI(\phi_{a,b,k}(\lambda)) 
	= \pm \frac{\pi(\Delta_{\Key_l(\lambda)}(\X,\Y))}{\Delta(\X_1)\Delta(\X_2)} \cdot \id \, .
\end{equation}
Here, $\X = \X_1 \cup \X_2 = \{x_1,\ldots,x_{a+b}\}$, $\Y = \{y_1,\ldots,y_{a+b}\}$, 
and $\pi$ is the reduction functor from Definition \ref{def:reduction}
(as discussed before Lemma \ref{lemma:Jab reduced} above).
\end{lemma}

\begin{proof}
We will use Definition \ref{def:monomial list} to help abbreviate parts of the proof.  
Specifically, for each $c\geq 1$ and each weakly decreasing sequence
$\b=(\b_1\geq \cdots \b_c\geq 0)$ of length $c$, 
let $\mathcal{M}_{c}(\b) = \{x^{\b_1+c-1},\ldots, x^{\b_{c-1}-1},x^{\b_c}\}$ 
be the associated list of monomials.  
We note the following properties of $\mathcal{M}_{c}(\b)$:
\begin{enumerate}
\item $\mathcal{M}_{c}(\emptyset) = \{x^{c-1},\ldots,x,1\}$.
\item $\hdet\mathcal{M}_{c}(\b) = \Schur_\b(\X)\Delta(\X)$, where $|\X|=c$.
\item If $\b\in P(c,d)$, then the dual complementary partition $\hat{\b}\in P(d,c)$ satisfies 
\[
\mathcal{M}_{d}(\hat{\b}) = \{x^{c+d-1},\ldots,x,1\}\smallsetminus \mathcal{M}_{c}(\b)\, .
\]
\end{enumerate}

Now, consider the determinant $\Delta_{\Key_l(\l)}$, depicted in \eqref{eq:key general}, 
regarded as an element of $\k[\X,\V]$ via \eqref{eq:y to v 2}. 
If we apply the algebra endomorphism $\pi$, the result is:
\begin{equation}\label{eq:keys reduced}
\pi(\Delta_{\Key_l(\l)}) =
\begin{vmatrix}
x^{a+b-l-1}_1 &  \cdots & x^{a+b-l-1}_a & x^{a+b-l-1}_{a+1} &\cdots & x^{a+b-l-1}_{a+b} \\ 
\vdots & \ddots & \vdots & \vdots  & \ddots & \vdots\\
1 &  \cdots & 1 & 1 &\cdots & 1 \\
0 &  \cdots & 0 & x^{\l_1+l-1}_{a+1} \yred_{a+1}&\cdots & x^{\l_1+l-1}_{a+b} \yred_{a+b}\\
\vdots & \ddots & \vdots & \vdots  & \ddots & \vdots\\
0 &  \cdots &0 & x^{\l_l}_{a+1} \yred_{a+1}&\cdots & x^{\l_l}_{a+b} \yred_{a+b} 
\end{vmatrix}
\end{equation}
This determinant will be computed by 
(multiple column) Laplace expansion in the first $a$ columns. 
Let $M$ denote the matrix
appearing on the right-hand side of \eqref{eq:keys reduced}, and let
$C=\{1,\dots, a\}$ be the indexing set for the selected columns.  The result of
the expansion is of the form:
\[|M| = \sum_{\substack{R\subset \{1,\dots, a+b\}\\ |R|=a}}
(-1)^{\epsilon_{R,C}} |M_{R,C}||M_{R^c,C^c}|\] where the sum ranges over size
$a$ subsets of rows, denoted $R$.
The summands are the product of the $a\times a$-minor 
determined by the pair $R,C$ and the complementary $b\times b$-minor, 
with a sign determined by the permutation shuffling $C$ to $R$. 

It is clear that $|M_{R,C}|=0$ unless $R\subset\{1,\dots, a+b-l\}$.  In other words, $R$ 
corresponds to a selection of $a$ of the first $a+b-l$ rows.  Equivalently, the choice of $R$ 
corresponds to the choice of $a$ monomials from $\{x^{a+b-l-1},\ldots,x,1\}$.  Such choices 
are parametrized exactly by partitions $\b\in P(a,b-l)$, via the assignment $\b\mapsto 
\mathcal{M}_a(\b)\subset \{x^{a+b-l-1},\ldots,x,1\}$.  Thus:
\[
|M_{R,C}| = 
\begin{vmatrix}
x_1^{\b_1+a-1} & \dots & x_a^{\b_1+a-1} \\ 
\vdots &  & \vdots \\ 
x_1^{\b_a} & \dots & x_a^{\b_a}  
\end{vmatrix} = \hdet(\mathcal{M}_a(\b))\, .
\]
The complementary choice of monomials 
$\{x^{a+b-l-1},\ldots,x,1\}\smallsetminus \mathcal{M}_a(\b)$ coincides with $\mathcal{M}_{b-l}(\hat{\b})$ for 
the dual complementary partition $\hat{\b}\in P(b-l,a)$, so:
\[
|M_{R^c,C^c}| = \begin{vmatrix}
x_{a+1}^{\hat{\b}_1+b-l-1} & \dots & x_{a+b}^{\hat{\b}_1+b-l-1} \\ 
\vdots &  & \vdots \\ 
x_{a+1}^{\hat{\b}_{b-l}} & \dots & x_{a+b}^{\hat{\b}_{b-l}}  \\
x^{\l_1+l-1}_{a+1} \yred_{a+1}&\cdots & x^{\l_1+l-1}_{a+b} \yred_{a+b}\\
\vdots  & \ddots & \vdots\\
x^{\l_l}_{a+1} \yred_{a+1}&\cdots & x^{\l_l}_{a+b} \yred_{a+b} 
\end{vmatrix} = \hdet(\mathcal{M}_{b-l}(\hat{\b})\cup \mathcal{M}_l(\l) \yred)\, .
\]
Thus, Laplace expansion yields the following identity:
\[
|M|= \sum_{\b\in P_{a,b-l}} (-1)^{|\hat{\b}|} \hdet\Big(\mathcal{M}_a(\b)\Big) 
	\hdet\Big(\mathcal{M}_{b-l}(\hat{\b})\cup  \mathcal{M}_l(\l)\yred\Big)\, ,
\]
where the sign is obtained from shuffling $\{x^{a+b-l-1},\ldots,x,1\}$ into
$\mathcal{M}_a(\b)\cup \mathcal{M}_{b-l}(\hat{\b})$.

We now compute the left-hand side of \eqref{eq:keylemma}
using perpendicular graphical calculus. 
Up to the $\pm$ sign coming from \eqref{eq:SigmaPGC}, 
this equals
\begin{equation}\label{eq:keylhs}
	\begin{tikzpicture}[anchorbase,smallnodes]
		\draw[FS,ultra thick,->] (0,-1.25) node[below]{$a$}  to (0,.7)node[above]{$a$};
		\draw[FS,ultra thick,->] (.75,-1.25) node[below]{$b$} to (.75,.7) node[above]{$b$};
		\draw[FS,ultra thick, directed=.75] (0,0) to [out=90,in=270] node[below]{$b{-}l$}  (.75,.5) ;
		\node at (.75,0) {$\bullet$};
		\node at (1.25,0) {$\prod_j \yred_j$};
		\node at (.75,-.5) {$\bullet$};
		\node at (1,-.5) {$\Schur_{\lambda}$};
		\draw[FS,ultra thick, directed=.75] (.75,-1) to [out=90,in=270] node[below]{$b{-}l$} (0,-.5) ;
	\end{tikzpicture}\!\!\!\!
=
	\!\!\!\! \sum_{\beta\in P(a,k)} \!\!\!(-1)^{|\hat{\beta}|} \!
	\begin{tikzpicture}[anchorbase,smallnodes]
		\draw[FS,ultra thick,->] (0,-1.25) node[below]{$a$}  to (0,.7)node[above]{$a$};
		\draw[FS,ultra thick,->] (1,-1.25) node[below]{$b$} to (1,.7) node[above]{$b$};
		\draw[FS,ultra thick, directed=.75] (.3,0) to [out=90,in=270] node[above,pos=.3]{$b{-}l$} (1,.5);
		\node at (1,0) {$\bullet$};
		\node at (1.5,0) {$\prod_j \yred_j$};
		\node at (1,-.5) {$\bullet$};
		\node at (1.25,-.5) {$\Schur_{\lambda}$};
		\draw[FS,ultra thick] (1,-1) to  [out=90,in=270]  (.3,-.5) to (.3,0) ;
		\node at (.3,-.3) {$\bullet$};
		\node at (.6,-.3) {$\Schur_{\hat{\beta}}$};
		\node at (0,-.3) {$\bullet$};
		\node at (-.3,-.3) {$\Schur_{\beta}$};
	\end{tikzpicture}
=
	\!\!\!\! \sum_{\beta\in P(a,k)} \!\!(-1)^{|\hat{\beta}|}\!
	\begin{tikzpicture}[anchorbase,smallnodes]
		\draw[FS,ultra thick,->] (0,-1.25) node[below]{$a$}  to (0,.7)node[above]{$a$};
		\draw[FS,ultra thick] (1.5,-1.25) node[below]{$b$} to (1.5,-1) ;
		\draw[FS,ultra thick,->] (1.5,.5) to (1.5,.7) node[above]{$b$};
		\draw[FS] (1.48,-1) \pu (.5,-.25) \pu (1.48,.5) ;
		\draw[FS] (1.49,-1) \pu (.75,-.25) \pu (1.49,.5) ;
		\draw[FS] (1.5,-1) \pu (1,-.25) \pu (1.5,.5) ;
		\draw[FS] (1.5,-1) \pu (2,-.25) \pu (1.5,.5) ;
		\draw[FS] (1.51,-1) \pu (2.25,-.25) \pu (1.51,.5) ;
		\draw[FS] (1.52,-1) \pu (2.5,-.25) \pu (1.52,.5) ;
		\node at (.75,-.25) {$\CQGbox{\mathcal{M}_{k}(\hat{\b})}$};
		\node at (2.25,-.25) {$\CQGbox{\mathcal{M}_l(\l)\yred}$};
		\node at (0,-.25) {$\bullet$};
		\node at (-.25,-.25) {$\Schur_{\beta}$};
	\end{tikzpicture}
\end{equation}
where $\prod_j \yred_j = \prod_{j=a+b-l+1}^{a+b} \yred_j$.
Here, $\mathcal{M}_l(\l)\yred := \{x^{\lambda_1+l-1} \yred,\ldots,x^{\lambda_l} \yred\}$, 
and we use the equality
\[
\begin{tikzpicture}[anchorbase,smallnodes]
\draw[FS,ultra thick,->] (0,-.75) node[below]{$l$} to 
	node[pos=.3,black]{$\bullet$} node[pos=.3,right,black]{$\Schur_\lambda$} 
	node[pos=.7,black]{$\bullet$} node[pos=.7,right,black]{$\prod_j \yred_j$} 
	(0,.75) node[above=-2pt]{$l$};
\end{tikzpicture}
=
\begin{tikzpicture}[anchorbase,smallnodes]
\draw[FS,ultra thick,->] (0,.375) to
	node[pos=.4,black]{$\bullet$} node[pos=.4,right,black]{$\prod_j \yred_j$} 
	(0,.75) node[above=-2pt]{$l$};
\draw[FS,thick] (0,-.375) to [out=150,in=270] node[black,pos=.7]{$\bullet$} 
	node[black,pos=.7,left=-2pt,yshift=-2pt]{\tiny$\lambda_1{+}l{-}1$} (-.5,0) 
	to [out=90,in=210] (0,.375);
\draw[FS,thick] (0,-.375) to [out=30,in=270] node[black,pos=.7]{$\bullet$} 
	node[black,pos=.7,right=-2pt,yshift=-2pt]{\tiny$\lambda_l$} (.5,0) 
	to [out=90,in=330] (0,.375);
\node at (0,0) {$\cdots$};
\draw[FS,ultra thick] (0,-.75) node[below]{$l$} to (0,-.375);
\end{tikzpicture}
=
\begin{tikzpicture}[anchorbase,smallnodes]
\draw[FS,ultra thick,->] (0,.375) to (0,.75) node[above=-2pt]{$l$};
\draw[FS,thick] (0,-.375) to [out=150,in=270] node[black,pos=.7]{$\bullet$} 
	node[black,pos=.7,left=-2pt,yshift=-2pt]{\tiny$\lambda_1{+}l{-}1$} (-.5,0) 
	to [out=90,in=210] node[black]{$\bullet$} 
	node[black,left=-2pt,yshift=2pt]{$\yred_{a{+}b{-}l{+}1}$} (0,.375);
\draw[FS,thick] (0,-.375) to [out=30,in=270] node[black,pos=.7]{$\bullet$} 
	node[black,pos=.7,right=-2pt,yshift=-2pt]{\tiny$\lambda_l$} (.5,0) 
	to [out=90,in=330] node[black]{$\bullet$} 
	node[black,right=-2pt,yshift=2pt]{$\yred_{a{+}b}$} (0,.375);
\node at (0,0) {$\cdots$};
\draw[FS,ultra thick] (0,-.75) node[below]{$l$} to (0,-.375);
\end{tikzpicture}
=:
\begin{tikzpicture}[anchorbase,smallnodes]
\draw[FS,ultra thick,->] (0,.375) to (0,.75) node[above=-2pt]{$l$};
\draw[FS,thick] (0,-.375) to [out=150,in=270] (-.25,0) to [out=90,in=210] (0,.375);
\draw[FS,thick] (0,-.375) to [out=30,in=270] (.25,0) to [out=90,in=330] (0,.375);
\node at (0,0) {$\CQGbox{\mathcal{M}_l(\l)\yred}$};
\draw[FS,ultra thick] (0,-.75) node[below]{$l$} to (0,-.375);
\end{tikzpicture}
\]
In the middle step, we can slide $\prod_j \yred_j$ through the 
top vertex (which acts via the Demazure operator associated with the 
longest element from Example \ref{exa:Demazure}) 
since it is $\symg_l$-symmetric.
Now, we can express the right-hand side of \eqref{eq:keylhs} as:
\begin{equation} \label{eq:keymid}
\eqref{eq:keylhs} = \!\!\! \sum_{\beta\in P(a,k)} \!\!(-1)^{|\hat{\beta}|}
\left(\frac{\hdet(\mathcal{M}_a(\b))}{\Delta(\X_1)}\right)
\left(\frac{\pi(\hdet(\mathcal{M}_{b-l}(\hat{\b})\cup \mathcal{M}_l(\l)y))}{\Delta(\X_2)}\right)
= \frac{|M|}{\Delta(\X_1)\Delta(\X_2)}\, .
\end{equation}
and the result follows.
\end{proof}

\begin{example} Let $a\geq b$ and consider the determinant associated to the
``maximal'' key shape $\Key_b(\emptyset)$. We compute
\[
  \pi(\Delta_{\Key_b(\emptyset)}(\X,\Y))= \begin{vmatrix} x_1^{a-1} & \dots &
  x_a^{a-1} & x_{a+1}^{a-1} &\dots & x_{a+b}^{a-1}\\  
      \vdots &  & \vdots  & \vdots & &\vdots\\
      1 & \dots & 1 & 1 &\dots & 1\\ 0 & \dots & 0 & x_{a+1}^{b-1} \yred_{a+1}
      &\dots & x_{a+b}^{b-1} \yred_{a+b} \\ \vdots &  & \vdots & \vdots & &
      \vdots \\ 0 & \dots & 0 &  \yred_{a+1} &\dots & \yred_{a+b} \\
      \end{vmatrix} = \yred_{a+1}\dots \yred_{a+b}  \Delta(\X_1)\Delta(\X_2) 
\]
by expanding the determinant in the first $a$ columns. The product
$\yred_{a+1}\dots \yred_{a+b}$ is familiar from Corollary~\ref{cor:section} 
(which, in fact, had already shown that $\yred_{a+1}\dots \yred_{a+b} \in J_{a,b}$).
\end{example}

\subsection{Reduced vs. unreduced}

Recall that Conjecture \ref{conj:FT ideal} proposes an explicit algebraic description of the full twist 
ideal. In the present (two-strand) case, it posits that the Hopf link ideal $J_{a,b}$ is equal to the ideal
\begin{equation}\label{eq:IabDef}
I_{a,b} = E_{a,b} \cdot \left\{\frac{f(\X,\Y)}{\Delta(\X_1)\Delta(\X_2)} \mid
f\in \k[\X,\Y] \text{ is antisymmetric for } \symg_{a+b} \right\} .
\end{equation}
Here, $\X = \X_1 \cup \X_2 = \{x_1,\ldots,x_{a+b}\}$ and $\Y = \{y_1,\ldots,y_{a+b}\}$. 
In order to establish a relation between the full twist ideal $J_{a,b}$ and the ideal $I_{a,b}$, 
we now relate (certain) reduced and unreduced Haiman determinants.

To begin, we recall the method of computing determinants via the \emph{Schur complement}.
\begin{prop}\label{prop:SC}
Let $A,B,C,$ and $D$ be $n\times n$, $n\times m$, $m\times n$, and $m\times m$ matrices (respectively) 
with coefficients in a commutative ring, and consider the $(n+m) \times (n+m)$ matrix
\[
M = 
\left(
\begin{array}{c|c}
A & B \\
\hline
C & D
\end{array}
\right)
\]
If $A$ is invertible, then
$\det(M)= \det(A) \det(D-CA^{-1}B)$.
\end{prop}
\begin{proof}
Gaussian elimination.
\end{proof}

This immediately implies the following. 

\begin{cor}\label{cor:SCkey}
Let $\Key_l(\l)$ be a key shape and let
\begin{equation}\label{eq:SCmatrices}
\begin{gathered}
A = 
\begin{pmatrix}
1 & \cdots & 1 \\
x_1 & \cdots & x_a \\
\vdots & & \vdots \\
x_1^{a-1} & \cdots & x_a^{a-1}
\end{pmatrix} 
\, , \quad
B=
\begin{pmatrix}
1 & \cdots & 1 \\
x_{a+1} & \cdots & x_{a+b} \\
\vdots & & \vdots \\
x_{a+1}^{a-1} & \cdots & x_{a+b}^{a-1}
\end{pmatrix} \\
C_l(\lambda) :=
\begin{pmatrix}
x_1^{a} & \cdots & x_a^{a} \\
\vdots & & \vdots \\
x_1^{a+b-l-1} & \cdots & x_a^{a+b-l-1} \\
x_1^{\lambda_l} y_1 & \cdots & x_a^{\lambda_l} y_a \\
\vdots & & \vdots \\
x_1^{\lambda_1+l-1} y_1 & \cdots & x_a^{\lambda_1+l-1} y_a
\end{pmatrix} 
\, , \quad
D_l(\lambda) :=
\begin{pmatrix}
x_{a+1}^{a} & \cdots & x_{a+b}^{a} \\
\vdots & & \vdots \\
x_{a+1}^{a+b-l-1} & \cdots & x_{a+b}^{a+b-l-1} \\
x_{a+1}^{\lambda_l} y_{a+1} & \cdots & x_{a+b}^{\lambda_l} y_{a+b} \\
\vdots & & \vdots \\
x_{a+1}^{\lambda_1+l-1} y_{a+1} & \cdots & x_{a+b}^{\lambda_1+l-1} y_{a+b}
\end{pmatrix} 
\end{gathered}
\end{equation}
then
$
\Delta_{\Key_l(\l)}(\X,\Y) = \pm \Delta(\X_1) \det(D_l(\lambda) - C_l(\lambda)A^{-1}B).
$
\end{cor}

Note that the $\pm$ sign occurs here since our ordering of monomials differs 
from the conventions for Haiman determinants (established in Definition \ref{def:shapes and dets})
that is used in Definition \ref{def:keyshape}.
We will continue to use this ordering for the remainder of this section, 
since, for our current considerations, Proposition \ref{prop:SC} is best-adapted to this ordering.
Corollary \ref{cor:SCkey} motivates the study of the matrix $D_l(\lambda) - C_l(\lambda)A^{-1}B$. 
Our next result computes the entries of this matrix.

\begin{lem}\label{lem:FunkyMonomials}
Let $z$ be an indeterminate and let $r,s \geq 0$.
If $A$ is the (Vandermonde) matrix from \eqref{eq:SCmatrices}, 
then
\[
m_{\X_1,\V_L}^{r,s}(z) :=
\begin{pmatrix}
x_1^r y_1^s & \cdots & x_a^r y_a^s
\end{pmatrix}
\cdot A^{-1} \cdot
\begin{pmatrix}
1 \\
z \\
\vdots \\
z^{a-1}
\end{pmatrix}
\]
is a polynomial in $E_{a,0}[z]$ that satisfies
\begin{equation}\label{eq:Minterpolation}
m_{\X_1,\V_L}^{r,s}(x_i) = x_i^r y_i^s
\end{equation}
for all $x_i \in \X_1$ (thus is the unique such polynomial).
\end{lem}
\begin{proof}
Note that $E_{a,0} = \Sym(\X_1)[\V_L]$. By equation \eqref{eq:y to v 2} and
Corollary \ref{cor:monomial reduction}, there exists a polynomial $m(z) \in
E_{a,0}[z]$ of degree $\leq a-1$ that satisfies $m(x_i) = x_i^ry_i^s$ for all
$x_i \in \X_1$. Now, by definition, $m_{\X_1,\V_L}^{r,s}(z)$ is a polynomial
with coefficients in the field of fractions of $\k[\X_1,\Y_1]$ 
of degree $\leq a-1$ satisfying
\eqref{eq:Minterpolation}. Since there is a unique such polynomial and $E_{a,0}$
is a subring of this field of fractions, we have $m_{\X_1,\V_L}^{r,s}(z) =
m(z)$. Thus $m_{\X_1,\V_L}^{r,s}(z) \in E_{a,0}[z]$.
\end{proof}

We now relate the entries of $D_l(\lambda) - C_l(\lambda)A^{-1}B$ in the 
unreduced and reduced setting.

\begin{lem}\label{lem:MonomialDifference}
Let $x_j \in \X_2$ and let $r \geq 0$, 
then
\[
x_j^r y_j - m_{\X_1,\V_L}^{r,1}(x_j) =
x_j^r \yred_j +
\sum_{k=a-r+1}^a \big( x_j^{r+k-1} - m_{\X_1,\V_L}^{r+k-1,0}(x_j) \big) v_{L,k}^{(a)}\, .
\]
(By convention, the summation on the right-hand side is zero when $r=0$.)
\end{lem}
\begin{proof}
To begin, note that Corollary \ref{cor:monomial reduction} implies that
\begin{equation}\label{eq:Mc0}
m_{\X_1,\V_L}^{c,0}(z) = \sum_{t=1}^a (-1)^{a-t} \Schur_{(c-a|a-t)}(\X_1) z^{t-1}
\end{equation}
when $c \geq a$.
Next, suppose that $x_i \in \X_1$, then
\[
\begin{aligned}
x_i^r y_i &\stackrel{\eqref{eq:y to v 2}}{=} x_i^r \sum_{k=1}^{a} x_i^{k-1} v_{L,k}^{(a)} = \sum_{k=1}^{a} x_i^{r+k-1} v_{L,k}^{(a)} \\
&\stackrel{\eqref{eq:Mc0}}{=} \sum_{k=1}^{a-r} x_i^{r+k-1} v_{L,k}^{(a)} 
	+ \sum_{k=a-r+1}^{a} \left( \sum_{t=1}^a (-1)^{a-t} \Schur_{(r+k-1-a|a-t)}(\X_1) x_i^{t-1} \right) v_{L,k}^{(a)} \, .
\end{aligned}
\]
This implies that 
\[
m_{\X_1,\V_L}^{r,1}(z) =  \sum_{k=1}^{a-r} z^{r+k-1} v_{L,k}^{(a)} 
	+ \sum_{k=a-r+1}^{a} \left( \sum_{t=1}^a (-1)^{a-t} \Schur_{(r+k-1-a|a-t)}(\X_1) z^{t-1} \right) v_{L,k}^{(a)}\, .
\]
On the other hand, recall from \eqref{eq:vred} that we have
\[
\begin{aligned}
\yred_j &= \sum_{k=1}^b x_j^{k-1} \bar{v}_k = \sum_{k=1}^b x_j^{k-1} ( v_{R,k}^{(b)} - v_{L,k}^{(b)})
	= y_j - \sum_{k=1}^b x_j^{k-1} v_{L,k}^{(b)} \\
&= y_j - \sum_{k=1}^b x_j^{k-1} \left( v_{L,k}^{(a)}+(-1)^{b-k} \sum_{i=1}^{a-b} \Schur_{(i-1|b-k)}(\X_2) v_{L,b+i}^{(a)} \right)
\end{aligned}
\]
(since $\leftX_2 = \rightX_2$ in $E_{a,b}$).
We thus compute that
\begin{multline}\label{eq:MonomialDifference}
x_j^r \yred_j - x_j^r y_j + m_{\X_1,\V_L}^{r,1}(x_j) = 
\sum_{k=1}^{a-r} x_j^{r+k-1} v_{L,k}^{(a)} - \sum_{k=1}^b x_j^{r+k-1} v_{L,k}^{(a)} \\
	+ \sum_{\substack{a-r+1\leq k \leq a \\ 1 \leq t \leq a}} (-1)^{a-t} \Schur_{(r+k-1-a|a-t)}(\X_1) x_j^{t-1} v_{L,k}^{(a)}
	- \sum_{\substack{1 \leq k \leq b \\ 1 \leq i \leq a-b}} (-1)^{b-k} \Schur_{(i-1|b-k)}(\X_2) x_j^{r+k-1}  v_{L,b+i}^{(a)} \, .
\end{multline}

There are now two cases.
First, suppose that $b \geq a-r$, then
\[
\begin{aligned}
\eqref{eq:MonomialDifference} 
&= \sum_{k=a-r+1}^b 
\left( \left(\sum_{t=1}^{a} (-1)^{a-t} \Schur_{(r+k-1-a|a-t)}(\X_1) x_j^{t-1} \right) - x_j^{r+k-1} \right) v_{L,k}^{(a)} \\
& \qquad + \sum_{i=1}^{a-b} 
\left( \sum_{t=1}^{a} (-1)^{a-t} \Schur_{(r+b+i-1-a|a-t)}(\X_1) x_j^{t-1} 
	- \sum_{k=1}^{b} (-1)^{b-k} \Schur_{(i-1|b-k)}(\X_2) x_j^{r+k-1} \right) v_{L,b+i}^{(a)} \\
&\stackrel{\eqref{eq:Mc0}}{=} \sum_{k=a-r+1}^b \left( m_{\X_1,\V_L}^{r+k-1,0}(x_j) - x_j^{r+k-1} \right) v_{L,k}^{(a)} \\
& \qquad + \sum_{i=1}^{a-b} 
\left( m_{\X_1,\V_L}^{r+b+i-1,0}(x_j)
	- x_j^r \sum_{k=1}^{b} (-1)^{b-k}  \Schur_{(i+b-1-b|b-k)}(\X_2) x_j^{k-1} \right) v_{L,b+i}^{(a)} \\
&= \sum_{k=a-r+1}^b \left( m_{\X_1,\V_L}^{r+k-1,0}(x_j) - x_j^{r+k-1} \right) v_{L,k}^{(a)} 
+ \sum_{i=1}^{a-b} \left( m_{\X_1,\V_L}^{r+b+i-1,0}(x_j) - x_j^r \cdot x_j^{i+b-1} \right) v_{L,b+i}^{(a)} \\
&= \sum_{k=a-r+1}^a \left( m_{\X_1,\V_L}^{r+k-1,0}(x_j) - x_j^{r+k-1} \right) v_{L,k}^{(a)}
\end{aligned}
\]
as desired. When $b \leq a-r$, the computation is similar:
\begin{align*}
\eqref{eq:MonomialDifference} 
&= \sum_{k=b+1}^{a-r} 
\left( x_j^{r+k-1} - \sum_{t=1}^{b} (-1)^{b-t} \Schur_{(k-1-b|b-t)}(\X_2) x_j^{r+t-1} \right) v_{L,k}^{(a)} \\
& \qquad + \sum_{k=a-r+1}^{a} 
\left( \sum_{t=1}^{a} (-1)^{a-t} \Schur_{(r+k-1-a|a-t)}(\X_1) x_j^{t-1} 
	- \sum_{t=1}^{b} (-1)^{b-t} \Schur_{(k-1-b|b-t)}(\X_2) x_j^{r+t-1} \right) v_{L,k}^{(a)} \\
&\stackrel{\eqref{eq:Mc0}}{=} \sum_{k=b+1}^{a-r} x_j^r
\left( x_j^{k-1} - \sum_{t=1}^{b} (-1)^{b-t} \Schur_{(k-1-b|b-t)}(\X_2) x_j^{t-1} \right) v_{L,k}^{(a)} \\
& \qquad + \sum_{k=a-r+1}^{a} 
\left(m_{\X_1,\V_L}^{r+k-1,0}(x_j)
	- x_j^r \sum_{t=1}^{b} (-1)^{b-t} \Schur_{(k-1-b|b-t)}(\X_2) x_j^{t-1} \right) v_{L,k}^{(a)} \\
&= 0 + \sum_{k=a-r+1}^{a} 
\left(m_{\X_1,\V_L}^{r+k-1,0}(x_j) - x_j^r \cdot x_j^{k-1} \right) v_{L,k}^{(a)}
= \sum_{k=a-r+1}^a \left( m_{\X_1,\V_L}^{r+k-1,0}(x_j) - x_j^{r+k-1} \right) v_{L,k}^{(a)} \, .
\qedhere
\end{align*}
\end{proof}

Lemmata \ref{lem:FunkyMonomials} and \ref{lem:MonomialDifference} immediately 
relate the (unreduced) Haiman determinants $\Delta_{\Key_l(\l)}(\X,\Y)$ to certain
reduced Haiman determinants $\pi(\Delta_S(\X,\Y))$.
Namely, recall from Convention \ref{conv:key} that we identify sets of monic monomials $S$ 
with finite subsets of $\Z_{\geq 0} \times \Z_{\geq 0}$ 
which are illustrated as collections of boxes that we call shapes.

\begin{defi}\label{def:Smonomial}
for $0 \leq l \leq b$, let $\mathsf{S}_l$ be the collection of all subsets of monomials $S$ of the form
\[
\{1,x,\ldots,x^{a-1}\} \cup
\{x^{s_1},\ldots,x^{s_{b-l}}\} \cup \{x^{r_1}y,\ldots,x^{r_l}y\}
\]
with $a \leq s_1 < \cdots < s_{b-l} \leq a+b$ and $0 \leq r_1 < \cdots < r_{l} \leq b$.
\end{defi}

In other words, $\mathsf{S}_l$ consists of shapes that only have boxes in the first two rows, 
with exactly $l$ boxes in the second row confined to the first $b$ positions,
and with $a+b-l$ boxes in the first row confined to the first $a+b$ positions and 
necessarily occupying the first $a$ boxes in that row.
In particular, note that $\Key_l(\l) \in \mathsf{S}_l$.

\begin{prop}\label{prop:UNvsRED}
Let $S \in \mathsf{S}_l$, then 
\[
\Delta_S(\X,\Y) - \pi(\Delta_S(\X,\Y)) = \sum_{R \in \mathsf{S}_{\leq l-1}} c_{S,R} \cdot \pi(\Delta_R(\X,\Y))
\]
for $c_{S,R} \in \k[\V_L]$ and $\mathsf{S}_{\leq l-1} = \bigcup_{k=0}^{l-1} \mathsf{S}_{k}$.
\end{prop}
\begin{proof}
Proposition \ref{prop:SC} gives an analogue of Corollary \ref{cor:SCkey} 
for the Haiman determinant $\Delta_S(\X,\Y)$. 
Namely, it gives that 
\[
\Delta_S(\X,\Y) = \det(A) \cdot \det(M_S)
\]
for the $b \times b$ matrix
\[
M_S = 
\begin{pmatrix}
x_{a+1}^{s_1} & \cdots & x_{a+b}^{s_1} \\
\vdots & & \vdots \\
x_{a+1}^{s_{b-l}} & \cdots & x_{a+b}^{s_{b-l}} \\
x_{a+1}^{r_1} y_{a+1} & \cdots & x_{a+b}^{r_1} y_{a+b} \\
\vdots & & \vdots \\
x_{a+1}^{r_l} y_{a+1} & \cdots & x_{a+b}^{r_l} y_{a+b}
\end{pmatrix} 
-
\begin{pmatrix}
x_{1}^{s_1} & \cdots & x_{a}^{s_1} \\
\vdots & & \vdots \\
x_{1}^{s_{b-l}} & \cdots & x_{a}^{s_{b-l}} \\
x_{1}^{r_1} y_{1} & \cdots & x_{a}^{r_1} y_{a} \\
\vdots & & \vdots \\
x_{1}^{r_l} y_{1} & \cdots & x_{a}^{r_l} y_{a}
\end{pmatrix} 
A^{-1}
B
\]
with $A$ and $B$ as in \eqref{eq:SCmatrices}.
Lemma \ref{lem:FunkyMonomials} then implies that
\begin{equation}\label{eq:MDblock}
M_S = 
\begin{pmatrix}
x_{a+1}^{s_1} - m_{\X_1,\V_L}^{s_1,0}(x_{a+1}) & \cdots & x_{a+b}^{s_1} - m_{\X_1,\V_L}^{s_1,0}(x_{a+b}) \\
\vdots & & \vdots \\
x_{a+1}^{s_{b-l}} - m_{\X_1,\V_L}^{s_{b-1},0}(x_{a+1}) & \cdots & x_{a+b}^{s_{b-l}} - m_{\X_1,\V_L}^{s_{b-1},0}(x_{a+1}) \\
x_{a+1}^{r_1} y_{a+1} - m_{\X_1,\V_L}^{r_1,1}(x_{a+1}) & \cdots & x_{a+b}^{r_1} y_{a+b} - m_{\X_1,\V_L}^{r_1,1}(x_{a+b}) \\
\vdots & & \vdots \\
x_{a+1}^{r_l} y_{a+1} - m_{\X_1,\V_L}^{r_l,1}(x_{a+1}) & \cdots & x_{a+b}^{r_l} y_{a+b} - m_{\X_1,\V_L}^{r_l,1}(x_{a+b}).
\end{pmatrix} 
\end{equation}
The result now follows by applying Lemma \ref{lem:MonomialDifference} to the
last $l$ rows of \eqref{eq:MDblock}. Indeed, this expresses each row of $M_S$ as
the sum of the corresponding row in $\pi(M_S)$ and a $\k[\V_L]$-linear
combination of rows corresponding to monomials $x^s$ for $a \leq s \leq a+b-1$.
Thus, we have that
\[
\det(M_S) = \det(\pi(M_S)) + \sum_{R \in \mathsf{S_{\leq l-1}}} c_{S,R} \cdot \det(\pi(M_R))
\]
for some $c_{R,S} \in \k[\V_L]$ and the result follows. (Informally, when we use
Lemma \ref{lem:MonomialDifference} and linearly expand $\det(M_S)$ along the
relevant rows, in each term a box in the second row either: becomes reduced, or
moves to row one in a position between the $a^{th}$ and the $(a+b-1)^{st}$ 
with a coefficient in $\k[\V_L]$. 
Note that there is no distinction between reduced versus unreduced monomials in the first row.)
\end{proof}

\begin{cor}\label{cor:RedToUn}
For any key shape $\Key_l(\l)$, we have that
\[
\pi(\Delta_{\Key_l(\l)}(\X,\Y)) = \sum_{R \in \mathsf{S}_{\leq l}} c_{l,\l,R} \cdot \Delta_R(\X,\Y)
\]
for $c_{l,\l,R} \in \k[\V_L]$ and $\mathsf{S}_{\leq l} = \bigcup_{k=0}^{l} \mathsf{S}_{k}$.
\end{cor}
\begin{proof}
Proposition \ref{prop:UNvsRED} shows that there is a unitriangular matrix with coefficients in $\k[\V_L]$ relating 
$\{ \Delta_S(\X,\Y) \mid S \in \mathsf{S}_{\leq l} \}$ and $\{ \pi(\Delta_S(\X,\Y)) \mid S \in \mathsf{S}_{\leq l} \}$. 
The result then follows since $\pi(\Delta_{\Key_l(\l)}(\X,\Y)) \in \{ \pi(\Delta_S(\X,\Y)) \mid S \in \mathsf{S}_{\leq l} \}$.
\end{proof}

Combining this with Lemma \ref{lem:keylemma} and the results from \S\ref{sec:cornermaps} 
gives the following.

\begin{prop}\label{prop:JinI}
$J_{a,b} \subset I_{a,b}$
\end{prop}

\begin{proof}
Proposition \ref{prop:corners span hopf} and Lemmata \ref{lemma:Jab reduced} and \ref{lem:keylemma} 
show that $J_{a,b}$ is generated by the elements $\frac{\pi(\Delta_{\Key_l(\l)}(\X,\Y))}{\Delta(\X_1)\Delta(\X_2)}$ 
with $0 \leq l \leq b$ and $\lambda \in P(l,b-l)$. 
Corollary \ref{cor:RedToUn} expresses each such generator as
\[
\sum_{R \in \mathsf{S_{\leq l}}} c_{l,\l,R} \cdot \frac{\Delta_R(\X,\Y)}{\Delta(\X_1)\Delta(\X_2)}
\]
with $c_{l,\l,R} \in \k[\V_L] \subset E_{a,b}$. 
Since every $\Delta_R(\X,\Y)$ is $\symg_{a+b}$-antisymmetric, 
we see that $\frac{\pi(\Delta_{\Key_l(\l)}(\X,\Y))}{\Delta(\X_1)\Delta(\X_2)} \in I_{a,b}$.
\end{proof}

\subsection{Families of ideals}\label{sec:families}

Our proof of the opposite inclusion $I_{a,b} \subset J_{a,b}$ for $b >1$
requires an inductive argument that makes use of the curved skein relation via
Theorem \ref{thm:skein split}. We begin by introducing the relevant rings and
ideals. First, we generalize the Hopf link ideal $J_{a,b}$ to a family of ideals
that arise from the splitting map
\[
I^{(s)}(\Sigma) \colon
\I^{(s)}(\VFTmin_{a,b-s}) \to \I^{(s)}(\oone_{a,b-s})
\]
for threaded digons. 
For the duration of this section, 
we will use Convention \ref{conv:I alphabets} for notation relevant 
to the functor $I^{(s)}$. 
In particular, we will consider $|\B|=s$ and $|\L|=\ell$ 
(often we have $\ell=b-s$).

\begin{definition}\label{def:Eals}
Let
\[
E_{a,(\ell,s)} := \Sym(\X_1|\L|\B)[\V_L^{(a)},\V_R^{(\ell+s)}]
\]
where $\V_L^{(a)}$ and $\V_R^{(\ell+s)}$ are alphabets as in \eqref{eq:2strandVvars}.
Further, we identify $E_{a,(\ell,0)} = E_{a,\ell}$ 
(from Definition~\ref{def:EM}, with $\L$ in place of $\X_2$) 
and regard this as a subalgebra of $E_{a,(\ell,s)}$ 
via the identification:
\[
v_{R,i}^{(\ell)} = v_{R,i}^{(\ell+s)} + (-1)^{\ell-i}\sum_{j=\ell+1}^{\ell+s} \Schur_{(j-\ell-1|\ell-i)}(\L) v_{R,j}^{(\ell+s)}
\]
that is analogous to \eqref{eq:vbar l and vbar b}.
\end{definition}

The relevance of this algebra stems from the following result.

\begin{lemma}\label{lemma:homs from 1 to I(X)} 
Let $X$ be a curved complex in $\VS_{a,\ell}$, then there is an isomorphism
\[
\Hom_{\VS_{a,\ell+s}}(\oone_{a,\ell+s},\I^{(s)}(X)) \cong 
\qdeg^{-\ell s} \Hom_{\VS_{a,\ell}}(\oone_{a,\ell},X)\otimes_{E_{a,(\ell,0)}} E_{a,(\ell,s)}
\]
of dg $E_{a,(\ell,s)}$-modules that is natural in $X$.
\end{lemma}
\begin{proof}
The endomorphism algebra
$\End_{\VS_{a,b}}(\I^{(s)}(X))$ is a 
$\Sym(\leftX_1 | \rightX_1 | \leftL | \rightL | \B)[\V_L^{(a)},\V_R^{(\ell+s)}]$-module.
(As a reminder, we denote alphabets as in Convention \ref{conv:I alphabets}.)
The induced action 
on $\Hom_{\VS_{a,b}}(\oone_{a,b},\I^{(s)}(X))$ factors through the quotient in
which we identify $\leftX_1=\rightX_1$ and $\leftX_2=\rightX_2$, and thus also $\leftL=\rightL$.
Hence, we see that $\Hom_{\VS_{a,\ell+s}}(\oone_{a,\ell+s},\I^{(s)}(X))$ is indeed 
a dg $E_{a,(\ell,s)}$-module. 
(The differential is induced from the differential on $X$, 
which commutes with this action.)

We now establish the isomorphism. 
Lemma \ref{lem:easy 1 to I(X)} gives that
\[
\Hom_{\CS_{a,\ell+s}}(\oone_{a,\ell+s},\I^{(s)}(X))
\cong \qdeg^{-\ell s} \Hom_{\CS_{a,\ell}}(\oone_{a,\ell} , X) \otimes \Sym(\B)\, .
\]
and in the deformed setting, this gives
\[
\begin{aligned}
\Hom_{\VS_{a,\ell+s}}(\oone_{a,\ell+s},\I^{(s)}(X))
&= \Hom_{\CS_{a,\ell+s}}(\oone_{a,\ell+s},\I^{(s)}(X))[\V_L^{(a)},\V_R^{(\ell+s)}] \\
&\cong \qdeg^{-\ell s} \Hom_{\CS_{a,\ell}}(\oone_{a,\ell} , X) \otimes \Sym(\B)[\V_L^{(a)},\V_R^{(\ell+s)}] \\
&\cong \qdeg^{-\ell s} \Hom_{\VS_{a,\ell}}(\oone_{a,\ell},X)\otimes_{E_{a,(\ell,0)}} E_{a,(\ell,s)}\, .
\end{aligned}
\]
The result follows since Definitions \ref{def:functor Ikb} and \ref{def:Eals} 
show that this is an isomorphism of complexes.
\end{proof}

Recall the notation $\Dig_{a,b}^s \defeq I^{(s)}(\oone_{a,b-s})$ from 
\S\ref{ss:splitting skein rel}. Our last result has the following important consequence.

\begin{cor}\label{cor:Mals=Jals}
Let $a \geq b \geq 0$, then
$\Hom_{\VS_{a,b}}(\oone_{a,b} , \Dig_{a,b}^s) 
\cong \qdeg^{-s(b-s)} E_{a,(b-s,s)}$
and the map 
\[
H(\I^{(s)}(\Sigma)) \colon
H\big( \Hom_{\VS_{a,b}}(\oone_{a,b} , \I^{(s)}(\VFTmin_{a,b-s})) \big)
\to 
E_{a,(b-s,s)}
\]
on homology induced by the splitting map is injective.
\end{cor}
\begin{proof}
Applying Lemma \ref{lemma:homs from 1 to I(X)} to the (co)domain of the 
splitting map $\Sigma_{a,b-s}: \VFTmin_{a,b-s} \to \oone_{a,b-s}$ gives the 
commutative diagram:
\begin{equation}\label{eq:splitsquare}
\begin{tikzcd}
\Hom_{\VS_{a,b}}(\oone_{a,b},\I^{(s)}(\VFTmin_{a,b-s}))
\arrow[r,"\cong"]
\arrow[d, "{\Hom(\oone_{a,b},\I^{(s)}(\Sigma_{a,b-s}))}" ]
&
\qdeg^{-s(b-s)} \Hom_{\VS_{a,b-s}}(\oone_{a,b-s},
	\VFTmin_{a,b-s})\otimes_{E_{a,(b-s,0)}} E_{a,(b-s,s)} 
\arrow[d, "{\Hom(\oone_{a,b} , \Sigma_{a,b-s})}"] \\
\Hom_{\VS_{a,b}}(\oone_{a,b},\I^{(s)}(\oone_{a,b-s}))
\arrow[r,"\cong"]
&
\qdeg^{-s(b-s)} \Hom_{\VS_{a,b-s}}(\oone_{a,b-s},
	\oone_{a,b-s})\otimes_{E_{a,(b-s,0)}} E_{a,(b-s,s)}
\end{tikzcd}
\end{equation}
The first assertion follows since the bottom-right corner of this diagram is 
$\qdeg^{-s(b-s)} E_{a,(b-s,s)}$.
For the second, we note that Corollary \ref{cor:M is J general} and 
Proposition \ref{prop:Hopfparityandgens} imply that the right vertical map 
induces an injective map in homology. 
Commutativity of the diagram implies that the same is true for the left vertical map.
\end{proof}

We next introduce notation that extends Definitions \ref{def:EM} and \ref{def:J}.

\begin{definition}\label{def:Jfam}
For $a \geq b \geq 0$, let
$M_{a,(b-s,s)} := \qdeg^{s(b-s)} \Hom_{\VS_{a,b}}(\oone_{a,b}, \I^{(s)}(\VFTmin_{a,b-s}))$
and set
\[
J_{a,(b-s,s)} := \im \big( H(M_{a,(b-s,s)}) \xrightarrow{H(\I^{(s)}(\Sigma))} E_{a,(b-s,s)} \big) 
\vartriangleleft E_{a,(b-s,s)}\, .
\]
\end{definition}

In the next section, we will achieve our goal of showing that $J_{a,b} = I_{a,b}$ 
by using an inductive argument involving the collection $\{ J_{a,(b-s,s)} \}_{s=0}^b$. 
To this end, we now observe that $J_{a,(\ell,0)} = J_{a,\ell}$ 
completely determines $J_{a,(\ell,s)}$.

\begin{lemma}\label{lemma:J is stable} 
For $\ell+s \leq a$, 
we have $J_{a,(\ell,s)}= E_{a,(\ell,s)} \cdot  J_{a,(\ell,0)}$.
\end{lemma}
\begin{proof}
Lemma \ref{lemma:homs from 1 to I(X)} gives that 
\[
M_{a,(\ell,s)} \cong M_{a,(\ell,0)} \otimes_{E_{a,(\ell,0)}} E_{a,(\ell,s)}\, . 
\]
Taking homology and applying Corollary \ref{cor:Mals=Jals} gives
\[
J_{a,(\ell,s)} \cong J_{a,(\ell,0)} \otimes_{E_{a,(\ell,0)}} E_{a,(\ell,s)}
\]
which is a restatement of the desired result.
\end{proof}

Motivated by this, we introduce the following family of ideals 
that generalize the ideal $I_{a,b} \vartriangleleft E_{a,b}$.

\begin{definition}\label{def:Ials}
Let $\X^{(a+\ell)} = \X_1 \cup \L = \{x_1,\ldots,x_{a+\ell}\}$ and 
$\Y^{(a+\ell)} = \{y_1,\ldots,y_{a+\ell}\}$ and set
\[
I_{a,(\ell,s)} 
:= E_{a,(\ell,s)} \cdot \left\{\frac{f(\X^{(a+\ell)},\Y^{(a+\ell)})}{\Delta(\X_1) \Delta(\L)} \mid
f\in \k[\X^{(a+\ell)},\Y^{(a+\ell)}] \text{ is antisymmetric for } \symg_{a+\ell} \right\} .
\]
\end{definition}

Analogous to Lemma \ref{lemma:J is stable}, the following holds (essentially by definition):

\begin{lem}\label{lemma:I is stable}
For $\ell+s \leq a$, 
we have $I_{a,(\ell,s)} = E_{a,(\ell,s)} \cdot  I_{a,(\ell,0)}$.
\end{lem}
\begin{proof}
The ideals $I_{a,(\ell,0)} = I_{a,\ell}$ and $I_{a,(\ell,s)}$ both have generators of the form 
$\frac{f(\X^{(a+\ell)},\Y^{(a+\ell)})}{\Delta(\X_1) \Delta(\L)}$, which give elements of 
$E_{a,\ell}$ and $E_{a,(\ell,s)}$ by expanding $\{y_i\}_{i=1}^{a+\ell}$ in the alphabets
$\V_L^{(a)} \cup \V_R^{(\ell)}$ and $\V_L^{(a)} \cup \V_R^{(\ell+s)}$.
The result now follows from Definition \ref{def:Eals} and Proposition~\ref{prop:va and vb}, 
which shows that the triangle:
\[
\begin{tikzcd}
\left\{\frac{f(\X^{(a+\ell)},\Y^{(a+\ell)})}{\Delta(\X_1) \Delta(\L)} \mid
f\in \k[\X^{(a+\ell)},\Y^{(a+\ell)}] \text{ is antisymmetric for } \symg_{a+\ell} \right\}
\arrow[r,hook]
\arrow[dr,hook]
&
E_{a,(\ell,0)}
\arrow[d,hook] \\
&
E_{a,(\ell,s)}
\end{tikzcd}
\]
commutes.
\end{proof}

\subsection{The inductive argument}\label{sec:induction}

In this section, we prove the following result, which establishes Conjecture \ref{conj:FT ideal} in the $2$-strand case.

\begin{thm}\label{thm:Jab}
Let $a \geq b \geq 0$, then $J_{a,b} = I_{a,b}$.
\end{thm}

More generally, Lemmata \ref{lemma:J is stable} and \ref{lemma:I is stable} then immediately give:

\begin{cor}\label{cor:Jals}
Let $a \geq b \geq s \geq 0$, then $J_{a,(b-s,s)} = I_{a,(b-s,s)}$. \qed
\end{cor}

Proposition \ref{prop:JinI} implies that Theorem \ref{thm:Jab} 
will follow by showing  that $I_{a,b} \subset J_{a,b}$. 
To begin, we establish the $b=1$ case.

\begin{lem}\label{lem:Ja1=Ia1}
$I_{a,1} \subset J_{a,1}$ (and thus $J_{a,1} = I_{a,1}$).
\end{lem}

\begin{proof}
Recall that generators for $I_{a,1}$ take the form
$\frac{f(\X,\Y)}{\Delta(\X_1)}$ where $f(\X,\Y)$ is anti-symmetric for the
diagonal action of $\symg_{a+1}$ on $\k[\X,\Y]$. Here,
$\X=\{x_1,\ldots,x_{a+1}\}$, $\X_1 = \X \smallsetminus \{x_{a+1}\}$, and
$\Y=\{y_1,\ldots,y_{a+1}\}$. Fix such a generator and note that $f(\X,\Y) \in
I_{a+1}$ (see \S \ref{ss:connection to Hilbert schemes}). Lemma
\ref{lem:HaimanI} gives have that
\[
I_{a+1} = \bigcap_{1\leq i<j \leq a+1} \k[\X,\Y]\cdot\{x_i - x_j, y_i-y_j\} 
\subset \bigcap_{1\leq i<j \leq a+1} \k[\X,\V]\cdot\{x_i - x_j, y_i-y_j\}\, .
\]
Now, in $\k[\X,\V]$, for $1\leq i < j \leq a$ we have
\[
y_i - y_j = \sum_{r=1}^a \big( x_i^{r-1} - x_j^{r-1} \big) v_{L,r}^{(a)} \in \k[\X,\V] \cdot \{x_i-x_j\}
\] 
and for $1 \leq i \leq a$ we have
\[
y_i - y_{a+1} = \sum_{r=1}^a \big( x_i^{r-1} - x_{a+1}^{r-1} \big) v_{L,r}^{(a)}
			+ \left( \sum_{r=1}^a x_{a+1}^{r-1} v_{L,r}^{(a)} \right) - y_{a+1}
= \left( \sum_{r=1}^a \big( x_i^{r-1} - x_{a+1}^{r-1} \big) v_{L,r}^{(a)} \right) - \yred_{a+1}\, .
\]
Hence,
\[
\bigcap_{1\leq i<j \leq a+1} \k[\X,\V]\cdot\{x_i - x_j, y_i-y_j\} 
= \left( \bigcap_{1 \leq i < j \leq a} \k[\X,\V]\cdot \{x_i - x_j\} \right) \cap
	\left( \bigcap_{1 \leq i \leq a} \k[\X,\V]\cdot \{x_i - x_{a+1}, \yred_{a+1}\}  \right).
\]
Now, we have that
\[
\bigcap_{1 \leq i < j \leq a} \k[\X,\V]\cdot \{x_i - x_j\} 
	= \k[\X,\V] \cdot \Delta(\X_1)
\]
and
\begin{align*}
\bigcap_{1 \leq i \leq a} \k[\X,\V]\cdot \{x_i - x_{a+1}, \yred_{a+1}\} 
&= \k[\X,\V] \cdot \yred_{a+1} + \bigcap_{1 \leq i \leq a} \k[\X,\V]\cdot \{x_i - x_{a+1}\} \\
&= \k[\X,\V] \cdot \{e_a(\X_1 - \{x_{a+1}\}), \yred_{a+1}\}
\end{align*}

Thus, we have that
\begin{align*}
\bigcap_{1\leq i<j \leq a+1} \k[\X,\V]\cdot\{x_i - x_j, y_i-y_j\} 
&= \big( \k[\X,\V] \cdot \Delta(\X_1) \big) 
	\cap \big( \k[\X,\V] \cdot \{e_a(\X_1 - \{x_{a+1}\}), \yred_{a+1}\} \big) \\
&= \Delta(\X_1) \k[\X,\V] \cdot \{e_a(\X_1 - \{x_{a+1}\}), \yred_{a+1}\}
\end{align*}
and so 
\[
\frac{f(\X,\Y)}{\Delta(\X_1)} \in \k[\X,\V] \cdot \{e_a(\X_1 - \{x_{a+1}\}), \yred_{a+1}\}\, .
\]
Further, since $\frac{f(\X,\Y)}{\Delta(\X_1)} \in \k[\X,\V]$ is invariant 
under the $\symg_{a}$-action, this implies that in fact
\[
\frac{f(\X,\Y)}{\Delta(\X_1)} \in 
\Sym(\X_1| \{x_{a+1}\})[\V] \cdot \{e_a(\X_1 - \{x_{a+1}\}), \yred_{a+1}\} 
= \pi(J_{a,1}) \subset J_{a,1}
\]
and thus $I_{a,1} \subset J_{a,1}$, as desired.
\end{proof}

Our proof of Theorem \ref{thm:Jab} for $b \geq 2$ relies on
Theorem \ref{thm:skein split}, which has the following consequence.

\begin{lemma}\label{lemma:Js in Es} 
There is a (contractible) complex of the form
\[
E_{a,(b-\bullet,\bullet)} := 
\left(
\begin{tikzcd}
E_{a,(b-0,0)} \ar[r,"d_0"]
	& \cdots \ar[r,"d_{s-1}"]
	& \qdeg^{s(s-1)} \tdeg^s E_{a,(b-s,s)} \ar[r,"d_{s}"]
	& \cdots \ar[r,"d_{b-1}"]
	&\qdeg^{b(b-1)} \tdeg^b E_{a,(0,b)}
\end{tikzcd}
\right)
\]
with differentials induced from those in Definition \ref{def:digoncomplex}.
Further, $d_s(J_{a,(b-s,s)}) \subset J_{a,(b-s-1,s+1)}$, 
thus the ideals $J_{a,(b-s,s)} \vartriangleleft E_{a,(b-s,s)}$ form a subcomplex
\[
J_{a,(b-\bullet,\bullet)} := 
\left(
\begin{tikzcd}
J_{a,(b-0,0)} \ar[r,"d_0"]
	& \cdots \ar[r,"d_{s-1}"]
	& \qdeg^{s(s-1)} \tdeg^s J_{a,(b-s,s)} \ar[r,"d_{s}"]
	& \cdots \ar[r,"d_{b-1}"]
	&\qdeg^{b(b-1)} \tdeg^b J_{a,(0,b)}
\end{tikzcd}
\right)
\]
of $E_{a,(b-\bullet,\bullet)}$.
\end{lemma}

We will occasionally abbreviate the notation of these complexes to
$J_\bullet \subset E_\bullet$ when there is no confusion as to the values of 
$a$ and $b$.

\begin{proof}
The complex $E_\bullet$ is obtained by applying the functor $\Hom_{\VS_{a,b}}
(\oone_{a,b},-)$ to the complex
\[
\oone_a \boxtimes \VTDmin_b(0)
\defeq
\left(
\begin{tikzcd}
\Dig_{a,b}^0 \ar[r,"d_0"]
	& \cdots \ar[r,"d_{s-1}"]
	& \qdeg^{s(b-1)} \tdeg^s \Dig_{a,b}^s \ar[r,"d_s"]
	& \cdots \ar[r,"d_{b-1}"]
	&\qdeg^{b(b-1)} \tdeg^b \Dig_{a,b}^b
\end{tikzcd}
\right)
\]
from Definition~\ref{def:digoncomplex} 
and using Lemma~\ref{lemma:homs from 1 to I(X)}, which gives that
\[
\Hom_{\VS_{a,b}}(\oone_{a,b},\qdeg^{s(b-1)} \tdeg^s \Dig_{a,b}^s) \cong \qdeg^{s(s-1)} E_{a,(b-s,s)}\, .
\]
The component of the differential 
$\qdeg^{s(s-1)} \tdeg^s E_{a,(b-s,s)} \xrightarrow{d_s} \qdeg^{(s+1)s} \tdeg^s E_{a,(b-s-1,s+1)}$ 
is explicitly given by:
\begin{align*}
\qdeg^{s(s-1)} E_{a,(b-s,s)} \defeq& \qdeg^{s(s-1)} \Sym(\X_1 | \L | \B)[\V_L^{(a)},\V_R^{(b)}] \\
\hookrightarrow &
\qdeg^{s(s-1)}\Sym(\X_1 | \L \smallsetminus \{ x_{a+b-s}\} | \{ x_{a+b-s}\} | \B)[\V_L^{(a)},\V_R^{(b)}] \\
\xrightarrow{\partial_{1,s}}&
\qdeg^{(s+1)s} \Sym(\X_1 | \L \smallsetminus \{ x_{a+b-s}\} | \{ x_{a+b-s}\} \cup \B)[\V_L^{(a)},\V_R^{(b)}] 
	\defeq \qdeg^{(s+1)s} E_{a,(b-s-1,s+1)}
\end{align*}
where here $\partial_{1,s}$ is (the tensor product of identity morphisms with) the Sylvester operator 
\[
\partial_{1,s} \colon
\Sym(\{x_{a+b-s}\} | \B) \to \Sym(\{ x_{a+b-s}\} \cup \B)
\]
from Example \ref{exa:Sylvester}, which has degree $-2|\B|=-2s$. (The complex
$E_\bullet$ is contractible by Lemma \ref{lemma:Dig complex}.) Finally, the
differential on $E_{a,(b-\bullet,\bullet)}$ restricts to a map $J_{a,(b-s,s)}
\rightarrow J_{a,(b-s-1,s+1)}$ by Theorem \ref{thm:skein split}.
\end{proof}

We next investigate the homology of the subcomplex $J_{a,(b-\bullet,\bullet)}$, since we are 
interested in identifying the quotient complex $J_{a,(b,0)}$ of the latter. 
We now work towards the proof of the following:

\begin{prop}\label{prop:Jexact}
The sequence 
\[
\begin{tikzcd}
0 \ar[r]
&J_{a,(b-0,0)} \ar[r,"d_0"]
& \cdots \ar[r,"d_{s-1}"]
& \qdeg^{s(s-1)} \tdeg^s J_{a,(b-s,s)} \ar[r,"d_{s}"]
& \cdots \ar[r,"d_{b-1}"]
&\qdeg^{b(b-1)} \tdeg^b J_{a,(0,b)}
\end{tikzcd}
\]
is exact.
\end{prop}

The proof will follow almost immediately from the following two lemmata.

\begin{lem}\label{lem:Im=J}
The complex $J_{a,(b-\bullet,\bullet)}$ is equal to
the image of the complex $\Hom_{\VS_{a,b}}(\oone_{a,b},\YKMCSmin_{a,b})$ under the chain map 
to $E_{a,(b-\bullet,\bullet)}$ induced by the map 
\[
\YKMCSmin_{a,b} \cong 
\VTDmin_b(a) \xrightarrow{\Phi} \oone_a \boxtimes \VTDmin_b(0)
\]
from Theorem \ref{thm:skein split}.
\end{lem}

\begin{proof}
By definition, $J_{a,(b-s,s)}$ is the image of the homology 
\[
H(M_{a,(b-s,s)}) = H\big(\Hom_{\VS_{a,b}}(\oone_{a,b},\qdeg^{s(b-s)} \I^{(s)}(\VFTmin_{a,b-s})) \big)
\]
under the map induced by $\I^{(s)}(\Sigma_{a,b-s})$.
By Theorem \ref{thm:skein split}, this is the same as the image under the map induced 
on homology by the component
\[
\Phi^{s,s} \colon \qdeg^{s(b-s)} \I^{(s)}(\VFTmin_{a,b-s}) \to \qdeg^{s(b-s)} \Dig_{a,b}^s
\]
of $\Phi$. (Recall from \eqref{eq:DCFT} that $\I^{(s)}(\VFTmin_{a,b-s})$ is the 
summand of $\VTDmin_b(a)$ in $\tdeg$-degree $s$.)
Since $E_{a,(b-s,s)}$ has trivial differential, 
Theorem \ref{thm:skein split} implies that 
it suffices to show that every element in the image of the map 
\[
\Phi^{s,s} \colon \Hom_{\VS_{a,b}}(\oone_{a,b},\qdeg^{s(b-s)} \I^{(s)}(\VFTmin_{a,b-s})) 
\to \Hom_{\VS_{a,b}}(\oone_{a,b},\qdeg^{s(b-s)}\Dig_{a,b}^s)
\]
is the image of a cycle in $\Hom_{\VS_{a,b}}(\oone_{a,b},\qdeg^{s(b-s)} \I^{(s)}(\VFTmin_{a,b-s}))$.

To this end, recall from Definition \ref{def:splitting map model}
that the deformed splitting map $\Sigma_{a,b} \colon \VFTmin_{a,b} \to \oone_{a,b}$ is 
supported on the ``corners'' $\bigoplus_{l=0}^b P_{l,l,0} \subset \VFTmin_{a,b}$. 
Lemma \ref{lemma:homs from 1 to I(X)} and \eqref{eq:splitsquare} 
imply that the same is true 
for $\I^{(s)}(\VFTmin_{a,b-s})$, i.e. $\I^{(s)}(\Sigma_{a,b-s})$ is supported on the corners
\[
P_{\bullet,\bullet,s} :=\bigoplus_{l=0}^{b-s} P_{l,l,s} \subset \I^{(s)}(\VFTmin_{a,b-s})\, .
\]
In fact, this implies the same statement with $\I^{(s)}(\Sigma_{a,b-s})$ replaced by $\Phi^{s,s}$.
Indeed, the proof of Theorem \ref{thm:skein split} gives the explicit formula 
\[
\Phi^{s,s} = d_{s-1} \circ k_s \circ \I^{(s)}(\Sigma_{a,b-s}) 
	+ k_{s+1} \circ \I^{(s+1)}(\Sigma_{a,b-s-1}) \circ (\delta^c + \Delta^c)
\]
with $d_s$ and $k_s$ as in Definition \ref{def:digoncomplex} and Lemma \ref{lemma:Dig complex}.
Since $\I^{(s)}(\Sigma_{a,b-s})$ is supported on the corners, the same is true for the first summand.
For the second summand, 
note that Proposition \ref{prop:KMCSdiffs} and equation \eqref{eq:twists} 
imply that the connecting differential 
\[
\delta^c + \Delta^c \colon \I^{(s)}(\VFTmin_{a,b-s}) \to \I^{(s+1)}(\VFTmin_{a,b-s-1})
\]
never sends ``non-corners'' (i.e. summands $P_{k,l,s}$ with $l <k$) to corners. 
Since $\I^{(s+1)}(\Sigma_{a,b-s-1})$ is supported on the corners, the second summand is as well.

Finally, it follows from the description of the vertical component $\d^v$ of the differential on 
$\I^{(s)}(\VFTmin_{a,b-s}) \cong \YMCCSmin_{a,b}^s$
given by Proposition \ref{prop:KMCSdiffs} that 
\[
\Hom_{\VS_{a,b}}(\oone_{a,b},P_{\bullet,\bullet,s}) \subset 
\Hom_{\VS_{a,b}}(\oone_{a,b},\I^{(s)}(\VFTmin_{a,b-s})) 
\]
is a subcomplex with zero differential. Thus every element in the support of
$\Phi^{s,s}$ is a cycle, so its image is spanned by cycles.
\end{proof}

\begin{lem}
The inclusion $J_{a,(0,b)}\hookrightarrow J_{a,(b-\bullet,\bullet)}$ is surjective in homology.
\end{lem}
\begin{proof}
Lemma \ref{lem:Im=J} gives a chain map
\begin{equation}\label{eq:PhiToJ}
\Phi \colon \Hom_{\VS_{a,b}}(\oone_{a,b},\YKMCSmin_{a,b}) \to J_{a,(b-\bullet,\bullet)}\, .
\end{equation}
Observe that the complex
\[
\Hom_{\VS_{a,b}}(\oone_{a,b},\YKMCSmin_{a,b}) \cong 
\tw_{\d^c+\Delta^c}\left(\bigoplus_{s=0}^b \Hom_{\VS_{a,b}} \big( \oone_{a,b} , \qdeg^{s(b-1)} \tdeg^s \I^{(s)}(\VFTmin_{a,b-s}) \big) \right)
\]
is filtered by $s$-degree, and
we can likewise view $J_{a,(b-\bullet,\bullet)} = \tw_d \left( \bigoplus_{s=0}^b J_{a,(b-s,s)} \right)$ as $s$-filtered.
Theorem \ref{thm:skein split} then implies that $\Phi$ is a filtered chain map. 
Further, by Corollary \ref{cor:Mals=Jals} and Definition \ref{def:Jfam}, 
the chain map \eqref{eq:PhiToJ} induces an isomorphism in homology 
for the associated graded complexes.
Since the $s$-filtration is bounded, a straightforward argument using the 
long exact sequence associated to a short exact sequence of chain complexes and the five lemma 
implies\footnote{Alternatively, this follows from the standard fact that if a morphism of spectral sequences 
is an isomorphism on a certain page, then it is an isomorphism on all subsequent pages.}
that $\Phi$ induces an isomorphism
\begin{equation}\label{eq:specseq}
H \big( \Hom_{\VS_{a,b}}(\oone_{a,b},\YKMCSmin_{a,b}) \big) \cong
H \big( J_{a,(b-\bullet,\bullet)} \big)\, .
\end{equation}

We are therefore interested in the complex: 
\begin{equation}\label{eq:hom 1 to VKMCS}
\Hom_{\VS_{a,b}}(\oone_{a,b},\YKMCSmin_{a,b}) \cong 
\tw_{\sum \bar{v}_i \xi_i} \left(\Hom_{\CS_{a,b}}(\oone_{a,b},\MCSmin_{a,b})\otimes \k[ \V_L^{(a)},\V_R^{(b)} ]
	\otimes \largewedge[\xi_1,\ldots,\xi_b]\right)
\end{equation}
Observe that the inclusion of the summand
\begin{equation}\label{eq:subsurH}
\Hom_{\CS_{a,b}}(\oone_{a,b},\MCSmin_{a,b})\otimes \k[ \V_L^{(a)},\V_R^{(b)} ]\otimes \xi_1\cdots \xi_b 
\hookrightarrow 
\text{right-hand side of } \eqref{eq:hom 1 to VKMCS}
\end{equation}
is a chain map, since the components of the differential that leave 
the summand
\[
\MCSmin_{a,b} \otimes \k[ \V_L^{(a)},\V_R^{(b)} ]\otimes \xi_1\cdots \xi_b
\subset \KMCSmin_{a,b}
\]
become zero upon applying $\Hom_{\CS_{a,b}}(\oone_{a,b},-)$.
Further, \eqref{eq:subsurH} is surjective in homology, 
because the twist in the right-hand side of \eqref{eq:hom 1 to VKMCS} 
is the Koszul differential associated to the action of the 
regular sequence $\bar{v}_1,\ldots,\bar{v}_{b}$ on
$\Hom_{\CS_{a,b}}(\oone_{a,b},\MCSmin_{a,b})\otimes \k[ \V_L^{(a)},\V_R^{(b)} ]$. 

Now, let $\iota \in \Hom_{\CS_{a,b}}(\oone_{a,b},\MCSmin_{a,b})$ 
be the inclusion of $\oone_{a,b}$ into $\MCSmin_{a,b}$ as the right-most chain group,
i.e. $\iota$ spans the chain group $\Hom_{\CS_{a,b}}^b(\oone_{a,b},\MCSmin_{a,b})$.
It follows from Proposition \ref{prop:MCS} and 
equations \eqref{eq:trace linearity} and \eqref{eq:R1HochCo} that
the homology of $\Hom_{\CS_{a,b}}(\oone_{a,b}, \MCSmin_{a,b})$ is supported in
$\tdeg$-degree $b$, hence is spanned by the class of $\iota$.
Thus, the composition 
\[
\End_{\CS_{a,b}}(\oone_{a,b})\otimes \k[ \V_L^{(a)},\V_R^{(b)} ]  
\rightarrow\Hom_{\CS_{a,b}}(\oone_{a,b},\MCSmin_{a,b})\otimes \k[ \V_L^{(a)},\V_R^{(b)} ] 
\hookrightarrow  \Hom_{\VS_{a,b}}(\oone_{a,b},\YKMCSmin_{a,b})
\]
sending $1\mapsto \iota \otimes \xi_1\cdots\xi_b$ is surjective in homology.
Upon inspection, we see that this is simply the inclusion
\[
\Hom_{\VS_{a,b}}( \oone_{a,b} , \I^{(b)} (\oone_{a,0}) )
\hookrightarrow
\Hom_{\VS_{a,b}}( \oone_{a,b} , \YKMCSmin_{a,b} )\, .
\]
It then follows, via 
\eqref{eq:specseq}, 
that the inclusion of $J_{a,(0,b)}$ into $J_{a,(b-\bullet,\bullet)}$ is surjective in homology.
\end{proof}

\begin{proof}[Proof (of Proposition \ref{prop:Jexact}).]
The homology of the complex 
\[
\begin{tikzcd}
0 \ar[r]
& J_{a,(b-0,0)} \ar[r,"d_0"]
& \cdots \ar[r,"d_{s-1}"]
& \qdeg^{s(s-1)} \tdeg^s J_{a,(b-s,s)} \ar[r,"d_{s}"]
& \cdots \ar[r,"d_{b-1}"]
& \qdeg^{b(b-1)} \tdeg^b J_{a,(0,b)} \ar[r]
& 0
\end{tikzcd}
\]
is supported on the far right. 
\end{proof}

We now establish Theorem \ref{thm:Jab} and Corollary \ref{cor:Jals}.

\begin{proof}[Proof of Theorem \ref{thm:Jab}.]
We show that $I_{a,\ell} = J_{a,\ell}$ for $1 \leq \ell < a$ implies that $I_{a,\ell+1} = J_{a,\ell+1}$. 
The result then follows inductively from the base case\footnote{Our use of Proposition \ref{prop:Jexact} 
later in the proof will make clear that we cannot simply induct up from 
the obvious equality $I_{a,0} = J_{a,0}$.} 
$I_{a,1} = J_{a,1}$ that was established above in Lemma \ref{lem:Ja1=Ia1}.

To begin, we first claim that the inclusion $E_{a,(b,0)} \hookrightarrow E_{a,(b-1,1)}$ 
restricts to an inclusion $I_{a,(b,0)} \hookrightarrow I_{a,(b-1,1)}$, 
i.e. it sends elements in the former to the latter.
Indeed, recall that the ideal $I_{a,(b,0)}$ is generated over $E_{a,(b,0)}$ by
expressions 
\[
\frac{f(\X,\Y)}{\Delta(\X_1) \Delta(\X_2)} 
\]
where $f \in \k[\X,\Y]$ is anti-symmetric for the diagonal action of $\symg_{a+b}$. 
Similarly, Definition \ref{def:Ials} implies that $I_{a,(b-1,1)}$ 
is generated over $E_{a,(b-1,1)}$ by expressions 
\[
\frac{g(\X\smallsetminus\{x_{a+b}\},\Y\smallsetminus\{y_{a+b}\})}
	{\Delta(\X_1) \Delta(\X_2\smallsetminus\{x_{a+b}\})}\, ,
\]
where $g\in \k[\X,\Y]$ is anti-symmetric for the diagonal action of $\symg_{a+b-1}\times \symg_1$. 
Thus, for $\symg_{a+b}$-anti-symmetric $f \in \k[\X,\Y]$, it suffices to show that
\begin{equation}\label{eq:IinI}
f(\X,\Y) \in {E_{a,(b-1,1)}\text{-}\spann} 
\left\{ \frac{\Delta(\X_2)}{\Delta(\X_2\smallsetminus\{x_{a+b}\})} \cdot
g(\X,\Y) \mid g\in \k[\X, \Y] \text{ antisymmetric for } \symg_{a+b-1}\times \symg_{1}.
\right\}
\end{equation}
Note that $\Delta(\X_2)\Delta(\X_2\smallsetminus\{x_{a+b}\})^{-1}=\prod_{j=a+1}^{a+b-1} (x_{j}-x_{a+b})$.

By the $E_{a,(b,0)}$-linearity of the inclusion $E_{a,(b,0)} \hookrightarrow E_{a,(b-1,1)}$, 
it suffices to check \eqref{eq:IinI} in the case when $f$ is an antisymmetrized monomial, 
i.e. when $f(\X,\Y)=\Delta_S(\X,\Y)=\hdet(S)$ for some collection 
\[
S=\{m_i(x,y)\}_{i=1}^{a+b}\subset \k[x,y] 
\]
of monic monomials. 
In this case, we now see that \eqref{eq:IinI} follows from Laplace expansion in the last column ($j=a+b$), 
after performing certain column operations on $f(\X,\Y)=|m_i(x_j,y_j)|_{i,j=1}^{a+b}$. 

To describe these column operations we use the identity
\[
\sum_{j=a+1}^{a+b} (-1)^{a+b-j} \frac{\Delta(\X_2\smallsetminus\{x_j\})}{\Delta(\X_2)} \cdot m_i(x_{j},y_{j})
= \partial_{a+1}\cdots \partial_{a+b-1}(m_i(x_{a+b},y_{a+b})) \in E_{(a,b),0}\subset E_{(a,b-1),1}
\]
which is straightforward to verify by induction in $b$. 
(The containment in $E_{(a,b),0}$ holds since the Sylvester operator 
$\partial_{a+1}\cdots \partial_{a+b-1}$ maps 
$\Sym(\X_2\smallsetminus \{x_{a+b}\}|x_{a+b})[\V_R^{(b)}] \to \Sym(\X_2)[\V_R^{(b)}]$). 
The desired column operations change the entries of the last column as follows:
\begin{equation}\label{eq:colop}
\begin{aligned}
m_i(x_{a+b},y_{a+b}) 
\mapsto&  m_i(x_{a+b},y_{a+b}) + \sum_{j=a+1}^{a+b-1} (-1)^{a+b-j} 
	\frac{\Delta(\X_2\smallsetminus\{x_j\})}{\Delta(\X_2\smallsetminus\{x_{a+b}\})} \cdot m_i(x_{j},y_{j}) \\
&= \frac{\Delta(\X_2)}{\Delta(\X_2\smallsetminus\{x_{a+b}\})}
	\partial_{a+1}\cdots \partial_{a+b-1}(m_i(x_{a+b},y_{a+b}))\, .
\end{aligned}
\end{equation}
Laplace expansion in the last column then gives
\[
|m_i(x_j,y_j)|_{i,j=1}^{a+b}\\
= \sum_{k=1}^{a+b} (-1)^{a+b+k} \partial_{a+1}\cdots \partial_{a+b-1}(m_k(x_{a+b},y_{a+b})) 
\frac{\Delta(\X_2)}{\Delta(\X_2\smallsetminus\{x_{a+b}\})} |m_i(x_j,y_j)|_{i\neq k,j\neq a+b}\, ,
\]
which verifies \eqref{eq:IinI}.
(Note that we have to slightly extend scalars by inverting the expression
$\Delta(\X_2\smallsetminus\{x_{a+b}\})$ to perform \eqref{eq:colop}; 
the value of the determinant remains unaffected regardless.)

Suppose now that $2 \leq \ell \leq b$ and that we have shown that 
$I_{a,m} = J_{a,m}$ for $1 \leq m < \ell$. 
In particular, Lemmata \ref{lemma:J is stable} and \ref{lemma:I is stable} imply that 
$I_{a,(\ell-s,s)} = J_{a,(\ell-s,s)}$ for $1 \leq s \leq \ell-2$. 
We thus consider the following commutative diagram:
\[
\begin{tikzcd}
& I_{a,(\ell,0)} \ar[r,hook,"d_0"] & I_{a,(\ell-1,1)} \arrow[r,"d_1"] \ar[d,equals] & I_{a,(\ell-2,2)} \ar[d,equals]  \\
0 \arrow[r] & J_{a,(\ell,0)} \ar[r,hook,"d_0"] \arrow[u,hook] & J_{a,(\ell-1,1)} \arrow[r,"d_1"]  & J_{a,(\ell-2,2)} 
\end{tikzcd}
\]
where the vertical inclusion is given by Proposition \ref{prop:JinI}. 
Since $d_1 \circ d_0 = 0$ (by Lemma \ref{lemma:Js in Es}), 
we have that $I_{a,(\ell,0)} \subset \ker \big( I_{a,(\ell-1,1)} \xrightarrow{d_1} I_{a,(\ell-2,2)} \big)$. 
However, 
\[
\ker \big( I_{a,(\ell-1,1)} \xrightarrow{d_1} I_{a,(\ell-2,2)} \big) = 
\ker \big( J_{a,(\ell-1,1)} \xrightarrow{d_1} J_{a,(\ell-2,2)} \big)
\]
and Proposition \ref{prop:Jexact} implies that the latter equals $J_{a,(\ell,0)}$. 
Thus $I_{a,\ell} = I_{a,(\ell,0)} = J_{a,(\ell,0)} = J_{a,\ell}$, 
which establishes the result by induction.
\end{proof}

We record a consequence of Theorem \ref{thm:Jab}, 
which gives a generating set for the ideal $I_{a,b} \vartriangleleft E_{a,b}$ 
of cardinality $2^b$.
Note that its statement (unlike its proof) appears to have nothing to do with link homology!

\begin{cor}\label{cor:JHaiman}
$I_{a,b} = E_{a,b} \cdot \left\{ \frac{\Delta_{\Key_l(\lambda)}(\X,\Y)}{\Delta(\X_1)\Delta(\X_2)} 
	\mid 0\leq l \leq b \, , \ \lambda \in P(l,b-l) \right\}$\, .
\end{cor}
\begin{proof}
Theorem \ref{thm:Jab} allows us to upgrade the statement of 
Proposition \ref{prop:UNvsRED} to the following:
\begin{equation}\label{eq:UNvsRED}
\Delta_{\Key_l(\lambda)}(\X,\Y) - \pi(\Delta_{\Key_l(\lambda)}(\X,\Y))
= \sum_{\substack{m \leq l-1 \\ \mu \in P(m,b-m)}} 
	c_{\lambda,\mu} \cdot \pi(\Delta_{\Key_m(\mu)}(\X,\Y))
\end{equation}
for $c_{\lambda,\mu} \in E_{a,b}$.
Indeed, for $R \in \mathsf{S}_{l-1}$ (recall Definition \ref{def:Smonomial}), 
we have that $\frac{\Delta_R(\X,\Y)}{\Delta(\X_1)\Delta(\X_2)} \in I_{a,b} = J_{a,b}$, 
thus $\frac{\pi(\Delta_R(\X,\Y))}{\Delta(\X_1)\Delta(\X_2)} \in \pi(J_{a,b}) \subset J_{a,b}$.
By Proposition \ref{prop:corners span hopf} and Lemma \ref{lem:keylemma}, 
this implies that 
\[
\frac{\pi(\Delta_R(\X,\Y))}{\Delta(\X_1)\Delta(\X_2)} 
= \sum_{\substack{m \leq l-1 \\ \mu \in P(m,b-m)}} 
	c_{R,\mu} \cdot \frac{\pi(\Delta_{\Key_m(\mu)}(\X,\Y))}{\Delta(\X_1)\Delta(\X_2)}
\]
where the bound $m \leq l-1$ follows by comparing $y$-degree. 
This gives \eqref{eq:UNvsRED}.

Now, \eqref{eq:UNvsRED} implies that there is a unitriangular matrix with coefficients 
in $E_{a,b}$ relating
\[
\{\Delta_{\Key_l(\lambda)}(\X,\Y) \mid 0 \leq l \leq b \, , \ \lambda \in P(l,b-l)\}
\]
and
\[
\{\pi(\Delta_{\Key_l(\lambda)}(\X,\Y)) \mid 0 \leq l \leq b \, , \ \lambda \in P(l,b-l)\} \, .
\] 
Thus,
\begin{align*}
I_{a,b} = J_{a,b} 
&= E_{a,b} \cdot \left\{ \frac{\pi(\Delta_{\Key_l(\lambda)}(\X,\Y))}{\Delta(\X_1)\Delta(\X_2)} 
	\mid 0\leq l \leq b\, , \ \lambda \in P(l,b-l) \right\} \\
&= E_{a,b} \cdot \left\{ \frac{\Delta_{\Key_l(\lambda)}(\X,\Y)}{\Delta(\X_1)\Delta(\X_2)} 
	\mid 0\leq l \leq b \, , \ \lambda \in P(l,b-l) \right\}
\end{align*}
where in the first line we again use 
Proposition \ref{prop:corners span hopf} and Lemma \ref{lem:keylemma}.
\end{proof}

\section{Link splitting and effective thickness}
\label{sec:linksplit}

The explicit description of the $2$-strand full twist ideal allows for the
extension of the link splitting results from \cite{BS,GH} to the setting of
colored, triply-graded link homology. In particular, in \S
\ref{ss:crossingchange} we establish an analogue of \cite[Corollary 4.4]{GH}
under the condition that certain elements are invertible. We call these elements
\emph{transparifers}, and
their invertibility is the colored analogue of the condition from \cite{GH} that
the differences of deformation parameters $y$ on distinct uncolored link
components are invertible. However, in the colored setting, Definitions
\ref{def:bundling v's} and \ref{def:YHHH} endow each link component with an
alphabet of deformation parameters of cardinality equal to the color of the link
component. Consequently, there are numerous other specializations of interest
that do not invert the transparifer. We save a complete investigation for future
work, but briefly discuss a heuristic for these deformations in \S
\ref{ss:effective thickness}. This makes contact with Conjecture \ref{conj:E}

\subsection{Crossing change morphisms and general splitting maps}
\label{ss:crossingchange}

Much of the machinery of \cite[Section 4]{GH} applies mutatis mutandis, 
using the following analogue of Corollary~\ref{cor:section}.

\begin{lemma}
	\label{lem:section}
Let $a \geq b \geq 0$. There exists a closed morphism
$\psi_{a,b}\colon \oone_{a,b} \to \VFTmin_{a,b}$ of
weight $\qdeg^{-2b}\tdeg^{2b}$ such that
\[
\Sigma_{a,b} \circ \psi_{a,b} = 
\frac{\Delta_{\Key_b(\emptyset)}(\X,\Y)}{\Delta(\X_1)\Delta(\X_2)} 
\cdot \id_{\oone_{a,b}}
\quad \text{and} \quad
\psi_{a,b} \circ \Sigma_{a,b} \sim \frac{\Delta_{\Key_b(\emptyset)}(\X,\Y)}{\Delta(\X_1)\Delta(\X_2)}  \cdot \id_{\VFTmin_{a,b}} \, .
\]
\end{lemma}
\begin{proof} Since
$\Delta_{\Key_b(\emptyset)}(\X,\Y)\Delta(\X_1)^{-1}\Delta(\X_2)^{-1}
\in I_{a,b}=J_{a,b}=\im(H(\Sigma_{a,b}))$ 
we can find closed $\psi_{a,b}\in M_{a,b}$ satisfying the first equation. 
The second relation follows as in the proof of Corollary~\ref{cor:section}.
\end{proof}

As a reminder,
$\X = \X_1 \cup \X_2 = \{x_1,\ldots,x_{a+b}\}$ and $\Y = \{y_1,\ldots,y_{a+b}\}$. 
Further, recall from Remark \ref{rem:subtleV} that expressions of the form
$g \cdot \id_{X}$ with $g \in \End_{\YS(\SSBim)}(\oone_{a,b})$
are shorthand for the morphism
\[
X \cong \oone_{a,b} \hComp X \xrightarrow{g \hComp \Id} \oone_{a,b} \hComp X  \cong X \, .
\]
The morphism $\psi_{a,b}$ from Lemma \ref{lem:section} is the unreduced analogue
of $\psired_{a,b}$ from Corollary~\ref{cor:section}.
We work with it here, since it treats the deformation parameters 
on the $a$- and $b$-labeled strands more ``democratically'' 
than $\psired_{a,b}$.

\begin{remark} As Corollary \ref{cor:section} and Lemma \ref{lem:section} show, 
the ``sections'' of the $2$-strand full twist splitting map are not unique in the colored case. 
The maps $\psi_{a,b}$ and $\psired_{a,b}$ are two valid choices.
\end{remark}

Given alphabets $\X_i ,\X_j$ and associated alphabets $\Y_i$, $\Y_j$, set
$a=\max(|\X_i|,|\X_j|)$ and $b=\min(|\X_i|,|\X_j|)$. We will abbreviate
\begin{equation} \label{eqn:sectionval}
\sectionval_{i,j}:= \frac{\Delta_{\Key_b(\emptyset)}(\X_i+\X_j,\Y_i+\Y_j)}{\Delta(\X_i)\Delta(\X_j)} 
\end{equation}
and call this the \emph{transparifer} of the morphism $\psi_{a,b}$. Whenever the
transparifer is invertible, the morphism $\psi_{a,b}$ from
Lemma~\ref{lem:section} becomes a homotopy equivalence between the identity
bimodule and the $2$-strand full twist. 
Thus, strands in the latter become \emph{transparent} to each other.

\begin{example} \label{exa:transpval} 
Under the substitution $v_{i,r}\mapsto 0$ and $v_{j,r}\mapsto 0$
for $r>1$, the transparifer reduces to $\pm(v_{j,1}-v_{i,1})^b$. 
If we further specialize $v_{j,1}\mapsto z_j u$ for some scalars $z_j$ 
and invertible $u$ as in 
Example~\ref{exa:coefficients} \eqref{exa:coefficients3}, 
then the transparifer becomes invertible whenever $z_i\neq z_j$.   
\end{example}

\begin{lemma}
	\label{lem:crossingchangemaps}
Consider the $i^{th}$ colored Artin generator $(\agen_i)_{\brc}$ in $\Br_m(\Z_{\geq 1})$ 
and let $\perm \in \symg_m$. 
Set $b:=\min(b_i,b_{i+1})$, then there exists a pair of degree zero morphisms:
\[
\psi^{+} \colon \YS C((\agen_i)_{\brc,\perm}) 
	\to \YS C((\agen_i\inv)_{\brc,\perm})
\, , \quad 
\psi^{-} \colon \YS C((\agen_i\inv)_{\brc,\perm}) 
	\to \qdeg^{2b} \tdeg^{-2b} \YS C((\agen_i)_{\brc,\perm})
\]
such that
$\psi^{+}\circ \psi^{-} \sim \sectionval_{\perm(i),\perm(i+1)} \cdot \id_{\YS C(\agen_i\inv)}$, 
and $\psi^{-}\circ \psi^{+} \sim \sectionval_{\perm(i),\perm(i+1)} \cdot \id_{\YS C(\agen_i)}$.
\end{lemma}
\begin{proof} This follows from Lemma~\ref{lem:section} by horizontal
	composition (on the right) with the curved Rickard complex for the inverse Artin generator.
\end{proof}

More generally, we have the following.

\begin{proposition}\label{prop:braidsplit}
	Let $\beta^+_{\brc}$ and $\beta^-_{\brc}$ be
colored braids in $\Br_m(\Z_{\geq 1})$ that differ in a single crossing, 
i.e.:
\[
\beta^{\pm}_{\brc} := \beta' \agen_i^{\pm} \beta''_{\brc}
\]
for some $\beta', \beta'' \in \Br_m(\Z_{\geq 1})$. 
Let $b:=\min(\beta''(\brc)_i,\beta''(\brc)_{i+1})$ 
(i.e. the minimal color involved in the crossing in the 
Artin generator $\agen_i^{\pm}$) and let $\perm \in \symg_m$.
Then, there exists a pair of degree zero morphisms
\[
\psi^{+}	
	\colon \YS C(\beta^+_{\brc,\perm}) 
	\to \YS C(\beta^-_{\brc,\perm})
	\, , \quad 
\psi^{-}	
	\colon \YS C(\beta^-_{\brc,\perm}) 
	\to \qdeg^{2b} \tdeg^{-2b} \YS C(\beta^+_{\brc,\perm})
\]
such that
\[
\psi^{+} \circ \psi^{-}
	\sim \id_{\YS C(\beta')}\hComp (\sectionval_{(\beta''\perm)(i),(\beta''\perm)(i+1)} 
	\cdot \id_{\YS C(\agen_i\inv)}) \hComp \id_{\YS C(\beta''_{\brc,\perm})}
\]
and
\[
\psi^{-} \circ \psi^{+}
	\sim \id_{\YS C(\beta')}\hComp (\sectionval_{(\beta''\perm)(i),(\beta''\perm)(i+1)} 
	\cdot \id_{\YS C(\agen_i)}) \hComp \id_{\YS C(\beta''_{\brc,\perm})} \, .
\]
\end{proposition}
\begin{proof}
This is an immediate consequence of 
Lemma \ref{lem:crossingchangemaps} using horizontal composition 
in $\YS(\SSBim)$.
\end{proof}

Note that if the crossing in Proposition \ref{prop:braidsplit} occurs between 
strands that correspond to the same component of the closure of a balanced, 
colored braid, then the relevant transparifer acts null-homotopically on the complex 
$\YS C_{\KR}(\b_{\brc})$ from Definition \ref{def:YHHH} 
(since the alphabets $\X_i$ and $\X_j$ become identified, as do $\Y_i$ and $\Y_j$).
Thus, in the setting of link homology, 
Proposition \ref{prop:braidsplit} is most interesting when the crossing is between 
strands in distinct components of the closure.

Now, given a link $\LB$ presented as the closure $\hat{\beta}$ of a braid $\beta$, 
recall that it is possible to unlink a component of $\LB$ from the remaining 
components via a sequence of crossing changes.
More precisely, if $\LB$ has components $\{\LB_i\}_{i=1}^\ell$, then there is a sequence of 
crossing changes that take the braid $\beta$ to a braid $\beta'$ where
\[
\hat{\beta} = \LB = \LB_1 \cup \cdots \cup \LB_\ell
\]
and
\[
\hat{\beta'} = (\LB_1 \cup \cdots \cup \LB_{i-1} \cup \LB_{i+1} \cup \cdots \cup \LB_\ell) 
	\sqcup \LB_i \, .
\]
In the latter display, the square cup $\sqcup$ denotes the \emph{split union}, 
which is given by placing two links in disjoint $3$-balls in $S^3$. 
Passing to the complex $\YS C_{\KR}(\beta_{\brc})$ from Definition \ref{def:YHHH} 
(which is the relevant complex when considering braid closures), 
now immediately gives the following.

\begin{proposition}\label{prop:CYcrossingchangemaps} 
Let $\LB$ and $\LB'$ be colored links presented as the closures 
of colored braids $\b_{\brc}$ and $\b'_{\brc}$, and suppose that 
there is a sequence of crossing changes taking $\b_{\brc}$ to $\b'_{\brc}$.
Let $p_{i,j}$ and $n_{i,j}$ denote the number of positive-to-negative 
and negative-to-positive crossing changes between the components $\LB_i$ and $\LB_j$, 
then there exist closed morphisms
\[
\psi_{\LB\to \LB'} \colon \YS C_{\KR}(\b_{\brc}) \to \YS C_{\KR}(\b'_{\brc})
\, , \quad
\psi_{\LB'\to \LB} \colon \YS C_{\KR}(\b'_{\brc}) \to \YS C_{\KR}(\b_{\brc})
\]
so that 
\[
\psi_{\LB'\to \LB}\circ \psi_{\LB\to \LB'} \sim 
\prod_{i,j}\sectionval_{i,j}^{p_{i,j}+n_{i,j}} \cdot \id_{\YS C_{\KR}(\b_{\brc})}
\]
and
\[
\psi_{\LB\to \LB'}\circ \psi_{\LB'\to \LB}  \sim 
\prod_{i,j}\sectionval_{i,j}^{p_{i,j}+n_{i,j}}\cdot \id_{\YS C_{\KR}(\b'_{\brc})} \, .
\]
\end{proposition}
(Here, by slight abuse of notation, 
$\sectionval_{i,j}$ denotes the transparifer evaluated at the relevant quadrupel 
of alphabets associated with the $i^{th}$ and $j^{th}$ components of $\LB$.)
\begin{proof}
This is essentially an immediate consequence of Proposition \ref{prop:braidsplit}. 
The only subtle point is that we must identify the actions of transparifers ``in the middle'' of 
curved Rickard complexes with those acting on the left, up to homotopy. 
This follows since the actions of all the $\X$ alphabets along a strand are homotopic, 
thus the same holds for the associated $\Y$ alphabets.
\end{proof}

\begin{rem}
With notation as in Proposition \ref{prop:CYcrossingchangemaps}, 
set
\[
p=\sum_{i,j} p_{i,j} \min(\LB_i,\LB_j) \, , \quad n=\sum_{i,j} n_{i,j} \min(\LB_i,\LB_j)
\]
where $\min(\LB_i,\LB_j)$ denotes the smaller of the colors of these components, then
\[
\wt(\psi_{\LB \to \LB'}) = \qdeg^{-2n}\tdeg^{2n}
\, , \quad
\wt(\psi_{\LB' \to \LB}) = \qdeg^{-2p}\tdeg^{2p} \, .
\]
\end{rem}

Combining Definition \ref{def:YC(b,M)}
with Proposition \ref{prop:CYcrossingchangemaps} then gives the following.

\begin{cor}\label{cor:split}
Retain notation as in Proposition \ref{prop:CYcrossingchangemaps}.
Suppose that $M$ is a coefficient module as in Definition \ref{def:YC(b,M)}
such that each occurring $\sectionval_{i,j}$ acts invertibly on $M$, 
then $\psi_{\LB\to\LB'}$ induces a homotopy equivalence 
$\YS C_{\KR}(\b_{\brc},M) \simeq \YS C_{\KR}(\b'_{\brc},M)$ 
and thus an isomorphism $\YS H_{\KR}(\LB,M)\cong \YS H_{\KR}(\LB',M)$. \qed
\end{cor}

In particular, the hypotheses of Corollary \ref{cor:split} hold under the specialization 
from Example \ref{exa:transpval}. In this special case, we recover 
\cite[Theorem 6.3]{MR4178751}.

Next, given a colored link $\LB$, 
let $\spli(\LB)$ denote the colored link obtained as the split union of its components.
In other words, if $\LB$ has components $\{\LB_i\}_{i=1}^\ell$, then
\[
\LB = \LB_1 \cup \cdots \cup \LB_\ell
\]
while
\[
\spli(\LB) = \LB_1 \sqcup \cdots \sqcup \LB_\ell \, .
\]
By the discussion preceding Proposition \ref{prop:CYcrossingchangemaps}, 
there exists a sequence of crossing changes taking a braid presentation for $\LB$ 
to one for $\spli(\LB)$.

\begin{definition} 
Let $\LB$ be a colored link that is presented as the closure of a colored braid $\b_{\brc}$.
Given a sequence of crossing changes between distinct components that transforms $\b_{\brc}$ 
into a braid $\spli(\b_{\brc})$ whose closure is $\spli(\LB)$, 
the closed morphism
\[
\psi_{\LB\to \spli(\LB)}\colon \YS C_{\KR}(\b_{\brc}) \to \YS C_{\KR}(\spli(\b_{\brc}))
\] 
associated to this sequence of crossing changes via Proposition \ref{prop:CYcrossingchangemaps}
will be called a \emph{splitting map} for $\LB$.
\end{definition}

Note that splitting maps and, more generally, the morphisms from
Proposition~\ref{prop:CYcrossingchangemaps} may depend on the sequence of crossing
changes (i.e. not just on the (co)domain complexes).

\begin{thm}\label{thm:parityinjective} 
Suppose that $\YS H_{\KR}(\LB)$ is free over $\k[\V^{\pi_0(\LB)}]$, 
e.g. for parity reasons as in Theorem~\ref{thm:parityfree}, 
then any splitting map induces an injective map on homology 
\[
\psi_{\LB\to \spli(\LB)}\colon \YS H_{\KR}(\LB) \to \YS H_{\KR}(\spli(\LB)).
\]
\end{thm}
\begin{proof}
By Proposition \ref{prop:CYcrossingchangemaps}, 
composing $\psi_{\LB\to \spli(\LB)}$ with the reverse map $\psi_{\spli(\LB)\to \LB}$ 
produces an endomorphism of $\YS H_{\KR}(\LB)$ that acts by multiplication 
with a product of transparifers.
Denote this product of transparifers by $T$, and note that it is a homogeneous 
polynomial in $\Sym(\X^{\pi_{0}(\LB)})[\V^{\pi_0(\LB)}]$.
Moreover, Example~\ref{exa:transpval} shows that 
\[
T = T_0 + T_{>0}
\]
where $T_{>0}$ is an element of the $\X$-irrelevant ideal
(specifically, $T_0$ is a polynomial in the subalphabet $\{v_{c,1}\}_{c \in \pi_{0}(\LB)} \subset \V^{\pi_0(\LB)}$). 
By assumption, $\YS H_{\KR}(\LB)$ is free over $\k[\V^{\pi_0(\LB)}]$, 
so the injectivity of the action of $T$ follows from the injectivity of the action of $T_0$. 
\end{proof}

Note that if $\LB$ is the closure of a \emph{pure} colored braid $\b_{\brc}$, then each of its components are 
unknots. Thus, any link splitting map gives a map
\[
\YS H_{\KR}(\LB) \to E_\brc\, .
\]
Further, if $\LB$ is parity, Theorem \ref{thm:parityinjective} implies that this map is injective, 
and thus identifies the former with an ideal in the latter. 
As such, Theorem \ref{thm:parityinjective} generalizes Corollary \ref{cor:M is J general}.

We conclude this section with an application to the undeformed colored Khovanov--Rozansky 
homology of parity links.

\begin{thm}\label{thm:parity and split}  
Let $\LB$ be a colored link that is parity. 
Upon collapsing the trigrading $(\deg_{\adeg}, \deg_{\qdeg}, \deg_{\tdeg})$ to the 
bigrading $(\deg_{\adeg}, \deg_{\qdeg}+\deg_{\tdeg})$, there is an isomorphism

\[
H_{\KR}(\LB)\cong H_{\KR}(\spli(\LB))
\]
of bigraded vector spaces.
\end{thm}
\begin{proof}
Let $\LB = \LB_1\cup \cdots \cup \LB_\ell$ be an $\ell$-component colored link, 
presented as the closure of a colored braid $\b_{\brc}$, and consider $\YS C_{\KR}(\b_{\brc})$. 
View $M = \k$ as a module over $\k[\V^{\pi_0(\LB)}]$
on which $v_{c,k}$ act by zero if $k>1$, and $v_{1,1},\ldots,v_{\ell,1}$ act by distinct scalars.
This implies that $v_{c,1} - v_{c',1}$ is invertible for $c\neq c'$ in $\pi_0(\LB)$, 
thus Example \ref{exa:transpval} implies that the transparifers associated 
with distinct link components are invertible.
Since $\wt(v_{c,1}) = \qdeg^{-2}\tdeg^{2}$, 
we have that $\deg_{\qdeg}(v_{c,1})+\deg_{\tdeg}(v_{c,1})=0$. 
Thus, we may regard 
$M$ as a bigraded $\k[\V^{\pi_0(\LB)}]$-module, 
with gradings $\deg_{\adeg}$ and $\deg_{\qdeg}+\deg_{\tdeg}$.

Tensoring Lemma \ref{lemma:simplifying YC} with $M=\k$ gives
\begin{align*}
\YS C_{\KR}(\b_\brc,\k)
&= \YS C_{\KR}(\b_\brc)\otimes_{\k[\V^{\pi_0(\LB)}]} \k \\
&\simeq \tw_{\Delta''}  (H_{\KR}(\LB) \otimes \k[\V^{\pi_0(\LB)}])\otimes_{\k[\V^{\pi_0(\LB)}]} \k \\
&\cong \tw_{\Delta''}  (H_{\KR}(\LB) \otimes \k[\V^{\pi_0(\LB)}]\otimes_{\k[\V^{\pi_0(\LB)}]} \k)
\cong \tw_{\Delta''}(H_{\KR}(\LB))
\end{align*}
for some twist $\Delta''$.
Since $\LB$ is parity, the twist $\Delta''$ is zero on $H_{\KR}(\LB)$ for degree reasons, 
thus $\tw_{\Delta''}(H_{\KR}(\b_\brc)) = H_{\KR}(\LB)$.
On the other hand, Corollary \ref{cor:split} and Proposition \ref{prop:knots are boring} imply that
\begin{align*}
\YS C_{\KR}(\b_\brc,\k) 
&\simeq \YS C_{\KR}(\spli(\b_\brc),\k) \\
&= \YS C_{\KR}(\spli(\b_\brc))\otimes_{\k[\V^{\pi_0(\LB)}]} \k \\
&\simeq C_{\KR}(\spli(\b_\brc)) \otimes \k[\V^{\pi_0(\LB)}] \otimes_{\k[\V^{\pi_0(\LB)}]} \k 
= C_{\KR}(\spli(\b_\brc)) \, .
\end{align*}
Since complexes of vector spaces 
are homotopy equivalent to their homologies, we have
\[
H_\KR(\spli(\LB)) \simeq C_{\KR}(\spli(\b_\brc)) \simeq H_{\KR}(\LB) \, . \qedhere
\]
\end{proof}

\subsection{Effective thickness}
\label{ss:effective thickness}

In \S \ref{ss:crossingchange}, we focused on specializations of the deformation parameters 
that inverted the transparifer;
however, there are various other specializations of interest. 
This is already apparent in the description of $\VFTmin_{a,b}$ afforded by 
Proposition \ref{prop:KMCSdiffs curved} (see Example \ref{exa:KMCS} for the $a=2=b$ case).
Indeed, upon inverting the elements 
\[
\sum_{l=k+1}^b h_{l-k-1}(\M^{(k+i)}) \bar{v}_l
\]
acting in the $(k+i)^{th}$ column of $\VFTmin_{a,b}$ for all $i \geq 1$ 
(e.g. by inverting $\bar{v}_{k+1}$ and setting $\bar{v}_l = 0$ for $l>k+1$),
each of these columns becomes null-homotopic.
In turn, this forces $\VFTmin_{a,b}$ to closely resemble the complex $\VFTmin_{a,k}$. 

A more-precise formulation of this phenomenon is as follows.
Let $C(\b_\brc)$ be the (uncurved) complex of singular Soergel bimodules associated to a colored braid $\b_\brc$.
Let $\mathsf{s}$ be a connected component of $\b_\brc$ (i.e. a strand), 
and let $\leftX$ and $\rightX$ be the associated alphabets on the left and right, respectively.
Intuitively speaking, ``turning on'' strand-wise curvature of the form 
$h_{a'+1}(\leftX-\rightX) v_{a'+1}$ with $v_{a'+1}$ invertible 
makes $\mathsf{s}$ behave as if its color is $\leq a'$.
We might say that such a strand now has \emph{effective thickness} $\leq a'$. 

Let us illustrate this with a concrete example.
Extending our notation from \S \ref{ss:splitting skein rel}, 
let $\TD_b(a)$ denote the complex of (uncurved) threaded digons
appearing on the left-hand side in the \emph{undeformed} colored skein relation \eqref{eq:convrecall0}.
Omitting all explicit grading shifts, this is
\begin{equation}\label{eq:detect}
\TD_b(a) :=
\left(
\left\llbracket
\begin{tikzpicture}[scale=.35,smallnodes,rotate=90,baseline=.4em]
	\draw[very thick] (1,-1) to [out=150,in=270] (0,0); 
	\draw[line width=5pt,color=white] (0,-2) to [out=90,in=270] (.5,0) 
		to [out=90,in=270] (0,2);
	\draw[very thick] (0,-2) node[right=-2pt]{$a$}to [out=90,in=270] (.5,0) 
		to [out=90,in=270] (0,2) node[left=-2pt]{$a$};
	\draw[line width=5pt,color=white] (0,0) to [out=90,in=210] (1,1); 
	\draw[very thick] (0,0) to [out=90,in=210] (1,1); 
	\draw[very thick] (1,-2) node[right=-2pt]{$b$} to (1,-1); 
	\draw[dotted] (1,-1) to [out=30,in=330] node[above,yshift=-2pt]{$0$} (1,1); 
	\draw[very thick] (1,1) to (1,2) node[left=-2pt]{$b$};
\end{tikzpicture}
\right\rrbracket 
\to
\left\llbracket
\begin{tikzpicture}[scale=.35,smallnodes,rotate=90,baseline=.4em]
	\draw[very thick] (1,-1) to [out=150,in=270] (0,0); 
	\draw[line width=5pt,color=white] (0,-2) to [out=90,in=270] (.5,0) 
		to [out=90,in=270] (0,2);
	\draw[very thick] (0,-2) node[right=-2pt]{$a$}to [out=90,in=270] (.5,0) 
		to [out=90,in=270] (0,2) node[left=-2pt]{$a$};
	\draw[line width=5pt,color=white] (0,0) to [out=90,in=210] (1,1); 
	\draw[very thick] (0,0) to [out=90,in=210] (1,1); 
	\draw[very thick] (1,-2) node[right=-2pt]{$b$} to (1,-1); 
	\draw[very thick] (1,-1) to [out=30,in=330] node[above,yshift=-2pt]{$1$} (1,1); 
	\draw[very thick] (1,1) to (1,2) node[left=-2pt]{$b$};
\end{tikzpicture}
\right\rrbracket 
\to \cdots \to
\left\llbracket
\begin{tikzpicture}[scale=.35,smallnodes,rotate=90,baseline=.4em]
	\draw[dotted] (1,-1) to [out=150,in=270] (0,0); 
	\draw[line width=5pt,color=white] (0,-2) to [out=90,in=270] (.5,0) 
		to [out=90,in=270] (0,2);
	\draw[very thick] (0,-2) node[right=-2pt]{$a$}to [out=90,in=270] (.5,0) 
		to [out=90,in=270] (0,2) node[left=-2pt]{$a$};
	\draw[dotted] (0,0) to [out=90,in=210] (1,1); 
	\draw[very thick] (1,-2) node[right=-2pt]{$b$} to (1,-1); 
	\draw[very thick] (1,-1) to [out=30,in=330] node[above,yshift=-2pt]{$b$} (1,1); 
	\draw[very thick] (1,1) to (1,2) node[left=-2pt]{$b$};
\end{tikzpicture}
\right\rrbracket 
\right) \, .
\end{equation}
By \cite[\shiftedRickard \ and \HRWSkeinrel]{HRW1}, 
$\TD_b(a) \simeq 0$ when $a<b$, 
so \eqref{eq:detect} detects the ``thickness'' of the $a$-labeled strand. 
If instead $a\geq b$, this complex is not contractible, 
but becomes so after deforming to give the $a$-labeled strand effective thickness 
smaller than $b$.
Thus, this deformation forces the $a$-labeled strand 
to ``act'' as if its label was smaller than $b$.

To see this, it is convenient to replace $\TD_b(a)$ by the (homotopy equivalent) 
complex on the right-hand side of the colored skein relation \eqref{eq:convrecall0}
\[
\KMCS_{a,b} = 
K\left(\left\llbracket
\begin{tikzpicture}[scale=.35,smallnodes,anchorbase,rotate=270]
\draw[very thick] (1,-1) to [out=150,in=270] (0,1) to (0,2) node[right=-2pt]{$b$}; 
\draw[line width=5pt,color=white] (0,-2) to (0,-1) to [out=90,in=210] (1,1);
\draw[very thick] (0,-2) node[left=-2pt]{$b$} to (0,-1) to [out=90,in=210] (1,1);
\draw[very thick] (1,1) to (1,2) node[right=-2pt]{$a$};
\draw[very thick] (1,-2) node[left=-2pt]{$a$} to (1,-1); 
\draw[very thick] (1,-1) to [out=30,in=330] node[below=-1pt]{$a{-}b$} (1,1); 
\end{tikzpicture}
\right\rrbracket 
\right) =
\tw_{\d}\big( \MCS_{a,b} \otimes \largewedge[\xi_1,\ldots,\xi_b] \big) \, .
\]
(Again, we are suppressing grading shifts.)
Recall from Definition \ref{def:dv} that 
$\d=\sum_{1\leq i\leq b} h_i(\leftX_2-\rightX_2)\otimes \xi_i^\ast$.
Fix an integer $a'\geq 0$ and ``turn on'' a curved Maurer--Cartan element $\Delta$
with curvature $h_{a'+1}(\X_1-\X_1') v_{a'+1}$ with $v_{a'+1}$ invertible. 
If $a' < b$, we claim that this deformed complex is contractible.
Indeed, since $\leftX_1+\leftX_2 = \rightX_1+\rightX_2$ on $\KMCS_{a,b}$, 
equation \eqref{eq:somerelations4} gives that
\[
h_r(\leftX_1-\rightX_1) = h_r(\rightX_2-\leftX_2) = - \sum_{j=1}^r h_{r-j}(\rightX_2 - \leftX_2) h_j(\leftX_2 - \rightX_2)
\]
so we may take
\[
\Delta = -v_{a'+1} \sum_{j=1}^{a'+1} h_{a'+1-j}(\rightX_2 - \leftX_2) \xi_j
= -v_{a'+1} \xi_{a'+1} - \sum_{j=1}^{a'} h_{a'+1-j}(\rightX_2 - \leftX_2) \xi_j \, .
\]
It follows that
\[
\tw_\Delta(\KMCS_{a,b}) \cong 
\tw_{- \sum_{j=1}^{a'} h_{a'+1-j}(\rightX_2 - \leftX_2) \xi_j} 
\big(\tw_{-v_{a'+1} \xi_{a'+1}} (\KMCS_{a,b}) \big) \simeq 0 \, .
\]
Here, we have used that $\tw_{-v_{a'+1} \xi_{a'+1}} (\KMCS_{a,b})$ is contractible 
when $v_{a'+1}$ is invertible, together with Proposition \ref{prop:HPT}, 
which show that further twisting by 
$- \sum_{j=1}^{a'} h_{a'+1-j}(\rightX_2 - \leftX_2) \xi_j$ does not break contractibility.

Since the present paper is already quite long, 
we save further investigations along these lines for future work. 
However, to come full circle, 
we do comment that our discussion here informs Conjecture \ref{conj:E} 
from the introduction. 
Recall that part of this conjecture asserts that a generalization of \eqref{eq:detect} 
should be interpreted as the $b^{th}$ ``elementary symmetric function'' 
of certain braids associated with the strands threading the digons. 
Some motivation is thus: 
as with \eqref{eq:detect}, 
when the (effective) size of the inputs to 
the elementary symmetric function $e_b(-)$ is smaller than $b$,
it vanishes.

\appendix

\section{Some \texorpdfstring{$\Hom$}{Hom}-space computations}

We first recall the following.

\begin{lem}[{\cite[Lemma 4.10]{RTub2}}]
\label{lem:TrMS}
Let $b,c \geq 0$, 
then there is an isomorphism
\[
\Tr^{c}({}_{b,c}S_{b+c} \hComp {}_{b+c}M_{b,c}) \cong 
\prod_{i=1}^{c} (\qdeg^{b}+\adeg \qdeg^{-b-2i}) \cdot \oone_{b} \otimes 
\End_{\SSBim}(\oone_c)
\]
of $\End_{\SSBim}(\oone_c)$-modules.
(The $\End_{\SSBim}(\oone_c)$ action on the left-hand side is induced from 
the action on the left $c$-labeled boundary of ${}_{b,c}S_{b+c} \hComp {}_{b+c}M_{b,c}$.)
\end{lem}

We now use the colored partial trace to give quick proofs of the following results.
To save space, we abbreviate $\CS=\CS(\SSBim)$, $\CS(\DS) = \CS(\DS(\Bim))$, 
and ${}_{b,c}S_{b+c}M_{b,c} = {}_{b,c}S_{b+c} \hComp {}_{b+c}M_{b,c}$.

\begin{cor}\label{cor:I(X)Hom}
Let $a,b,c \geq 0$ and let $X,Y \in \CS({}_{a,b}\SSBim_{a,,b})$ be complexes of singular Soergel bimodules.
There is an isomorphism of $\End_{\SSBim}(\oone_c)$-modules
\[
\Hom_{\CS}\big( X \boxtimes \oone_c , 
	(\oone_a \boxtimes {}_{b,c}S_{b+c}M_{b,c}) \hComp (Y \boxtimes \oone_c) \big) 
\cong \qdeg^{bc} \Hom_{\CS}(X,Y) \otimes \End_{\CS}(\oone_c)
\]
natural in both $X$ and $Y$.
\end{cor}
\begin{proof}
We compute using Propositions \ref{prop:TrAdj} and \ref{prop:TrBi} and Lemma \ref{lem:TrMS}:
\[
\begin{aligned}
\Hom_{\CS}\big( X \boxtimes \oone_c , 
	(\oone_a \boxtimes {}_{b,c}S_{b+c}M_{b,c}) \hComp (Y \boxtimes \oone_c) \big)
&= \Hom^{\adeg=0}_{\CS(\DS)}\big( X \boxtimes \oone_c , 
	(\oone_a \boxtimes {}_{b,c}S_{b+c}M_{b,c}) \hComp (Y \boxtimes \oone_c) \big) \\
&\cong \Hom^{\adeg=0}_{\CS(\DS)}\big( X, 
	(\oone_a \boxtimes \Tr^{c}({}_{b,c}S_{b+c}M_{b,c})) \hComp Y \big) \\
&\cong \Hom^{\adeg=0}_{\CS(\DS)}\big( X, 
	\prod_{i=1}^{c} (\qdeg^{b}+\adeg \qdeg^{-b-2i}) Y \otimes \End_{\CS}(\oone_c) \big) \\
&= \qdeg^{bc} \Hom_{\CS}(X,Y) \otimes \End_{\CS}(\oone_c)\, .
\end{aligned}
\]
The result follows since all of the constituent isomorphisms are natural in $X$ and $Y$.
\end{proof}

\begin{cor}\label{cor:basicHom}
Let $b,c \geq 0$, then the $\Hom$-spaces
\[
\Hom_{\SSBim}(\oone_{b,c} , {}_{b,c}S_{b+c} \hComp {}_{b+c}M_{b,c})
\quad \text{and} \quad
\Hom_{\SSBim}({}_{b,c}S_{b+c} \hComp {}_{b+c}M_{b,c} , \oone_{b,c})
\]
are free $\End_{\SSBim}(\oone_{b,c})$-modules generated by the 
zip and unzip morphisms, respectively.
\end{cor}
\begin{proof}
Let $\X_1$ and $\X_2$ be alphabets with $|\X_1|=b$ and $|\X_2|=c$ 
and identify $\Sym(\X_1|\X_2) = \End_{\SSBim}(\oone_{b,c})$.
The map
\[
\End_{\SSBim}(\oone_{b,c}) \to \Hom_{\SSBim}(\oone_{b,c} , {}_{b,c}S_{b+c}M_{b,c})
\, , \quad
f \mapsto \zip \circ f
\]
is injective since $\un \circ \zip =  \Schur_{c^b}(\X_1 - \X_2)$ and $\Sym(\X_1|\X_2)$ is torsion-free.
To see that it is surjective, we compare graded dimensions.
For this, Corollary \ref{cor:I(X)Hom} gives that
\[
\Hom_{\SSBim}(\oone_{b,c} , {}_{b,c}S_{b+c}M_{b,c})
= \Hom_{\CS}(\oone_{b,c} , {}_{b,c}S_{b+c}M_{b,c})
\cong \qdeg^{bc} \End_{\CS}(\oone_b) \otimes \End_{\CS}(\oone_c)
= \qdeg^{bc} \End_{\SSBim}(\oone_{b,c})
\]
as desired. 

The other assertion follows either using a similar argument
(here, we use that $\zip \circ \un = \Schur_{c^b}(\leftX_1 - \rightX_2)$ 
is injective on ${}_{b,c}S_{b+c} \hComp {}_{b+c}M_{b,c}$), 
or by applying duality in $\SSBim$.
\end{proof}

\end{document}